\documentclass[12pt]{amsart}
\usepackage{setspace}
\usepackage{amsmath,amscd, amsthm}
\usepackage{amsfonts}
\usepackage[english]{babel}
\usepackage{amssymb}
\usepackage{bbm}
\usepackage[top=1in, bottom=1in, left=.9in, right=.9in]{geometry}
\usepackage{enumerate}
\usepackage{mathrsfs}
\usepackage{parskip}
\usepackage{xypic}
\input xy-pic.sty

\usepackage{hyperref}
\hypersetup{
  colorlinks   = true, 
  urlcolor     = blue,
  linkcolor    = blue, 
  citecolor   = blue 
}
\usepackage{cleveref}

\usepackage{latexsym}

\usepackage{fancyhdr}

\newtheoremstyle{thmstyle}
	{7pt} 
	{0pt} 
	{\itshape} 
	{} 
	{\bfseries} 
	{:} 
	{.5em} 
	{} 

\theoremstyle{thmstyle} 
\newtheorem{theorem}{\bf Theorem}[]
\newtheorem{proposition}[theorem]{\bf Proposition}
\newtheorem{corollary}[theorem]{\bf Corollary}
\newtheorem{lemma}[theorem]{\bf Lemma}
\newtheorem{claim}[theorem]{\bf Claim}
\newtheorem{remark}[theorem]{\bf Remark}
\newtheorem{definition}[theorem]{\bf Definition}

\numberwithin{equation}{section}

\usepackage{thmtools}
\usepackage{aliascnt}
\crefname{theorem}{Theorem}{Theorems}
\crefname{lemma}{Lemma}{Lemmas}
\usepackage{verbatim}
\usepackage{wasysym}
\usepackage{cancel}

\usepackage{tikz}
\usepackage{tikz-cd}
\usepackage{smartdiagram}

\usepackage{multicol}
\usepackage{siunitx}

\usepackage{graphicx}
\usepackage{enumitem}

\title{Spectral decomposition of pseudo-cuspforms, and meromorphic continuation of Eisenstein series, on $\mathbb{Q}$-rank one arithmetic quotients}
\author{Iver Walkoe}
\date{10.29.2019} 
\begin{document}

%
%

\pagestyle{plain} 
\setcounter{page}{1}
\pagenumbering{roman}
\maketitle

\begin{quote}{\sc Abstract:}
We extend Lax-Phillips' theorem on discreteness of pseudo-cuspforms, in the style of Colin de Verdi{\`e}re's use of the Friedrichs self-adjoint extension of a restriction of the Laplace-Beltrami operator, as opposed to the use of semigroup methods. We use this to prove meromorphic continuation of Eisenstein series in several $\mathbb{Q}$-rank one cases, again following Colin de Verdi{\`e}re, as opposed to the semigroup-oriented viewpoint of Lax-Phillips and W. Mueller.
\\ \end{quote}

\setcounter{tocdepth}{2}
\tableofcontents

%
%
%
%
%
%
\newpage
\setcounter{page}{1}
\pagenumbering{arabic}
\pagestyle{plain}

\section{Introduction and synopsis}\label{Intro-and-Motivation}

This paper is a modified version of the author's dissertation submitted to the University of Minnesota (\cite{Walkoe-2019}) . As elaborated below, we follow Lax-Phillips' discreteness argument for pseudo-cuspforms, making use of the Friedrichs self-adjoint extension of a restriction of the Laplace-Beltrami operator as outlined in \cite{CdV-pseudo}. Then, following \cite{CdV-Eis}, we prove meromorphic continuation of Eisenstein series in many $\mathbb{Q}$-rank one cases. Technically, the primary distinction is the use of the approach of Colin de Verdi{\`e}re, as opposed to the semigroup-oriented viewpoint of Lax-Phillips and W. Mueller.

The method is to examine properties of the Casimir operator acting on automorphic Sobolev spaces on $G/K$ for the $\mathbb{Q}$-rank one groups being investigated. The analytic set-up is first used to establish a Rellich-like compactness result. This in turn is used to show the discreteness of the spectrum of the Friedrichs self-adjoint extension $\widetilde{\Delta}_a$ of the Lax-Phillips operator $\Delta_a$ (defined below) and, consequently, the discrete decomposition of the space of $L^2$ pseudo-cuspforms on $\Gamma\backslash G/K$. These results are then used to demonstrate the meromorphic continuation of Eisenstein series up-to and past the critical line $\text{Re}(s) = \frac{1}{2}$.

We begin by examining three concrete groups: $O(r,1)$, $U(r,1)$ and $Sp^\ast(r,1)$ (described below), for which we provide complete development of the steps outlined above. Our discussion concludes by extending the above arguments to a more general class of $\mathbb{Q}$-rank one orthogonal groups. This generalization requires invocation of the compactness of anisotropic quotients and units theorem (recalled in an appendix).

We hope the exposition is of value both by highlighting the clean lines of reasoning in Colin de Verdi{\`e}re's use of the Friedrichs self-adjoint extension to establish meromorphic continuation of Eisenstein series, and by providing explicit, if occasionally gritty, computations that seem difficult to find in the literature.

I want to express my appreciation for the guidance and insight provided by Professor Paul Garrett in directing my dissertation program and for suggesting this problem. I also want to thank Professors Brubaker, Diaconu and McGehee of the University of Minnesota for agreeing to serve on my thesis committee.

The argument largely breaks down into the following steps:
\begin{enumerate}
\item Deriving a suitable expression for the Casimir operator $\Omega_\frak{g}$ on $G/K$ by:
\begin{enumerate}
	\item Expressing the Lie algebra $\frak{g}$ and its dual $\frak{g}^\ast$ in suitable Iwasawa coordinates and invoking $\Omega_\frak{g} = \sum_{x \in \frak{g}} x \cdot x^\prime$.
	\item Identifying, and dropping, the terms of the Casimir operator  {\it parallel to $K$} since these terms act trivially on functions on $G/K$.
\end{enumerate}
\item Several convenient simplifications are obtained by invoking reduction theory:
\begin{itemize}
	\item In the $\mathbb{Q}$-rank one cases we consider, reduction theory implies that the geometry of $G/K$ has a particularly tractable form. Namely, {\it cusps are points}: a neighborhood of a cusp on $\Gamma\backslash G/K$ is of the form ``{\it compact manifold} $\times$ {\it ray}'' where the ray (i.e., a one-dimensional manifold homeomorphic to $(0, \infty)$) corresponds to a ``{\it height}'' parameter that will be the primary focus of the analysis (cf. \hyperref[sec:Reduction-Theory-Siegel-Sets]{\S 2.2} and \hyperref[sec:Remarks-Number-of-Cusps]{\S 2.3}). 
	\item In these cases, which do not include Siegel modular forms, minimal parabolic subgroups are also maximal (proper) parabolic subgroups and there is a {\it finite number of rational $\Gamma$-conjugacy classes}. That is, there is a {\it finite number} of cusps.
	\item The analysis requires bounds on Sobolev-like norms of families parameterized in terms of the {\it ray} coordinate for each cusp. Since there are only finitely many cusps, the appropriate maximum or minimum will thus work for all the cusps. As such, keeping track of the individual cusps only introduces a notational burden; we simplify the presentation by treating the situation as though there were only a single cusp with parameter $y$ which does not change the mathematical content.
	\item The study of the functional equation(s), that is, of the scattering matrix, naturally requires attention to the interactions of the Eisenstein series attached to the various cusps. However, here we consider only the logically prior analytical problem of meromorphic continuation which does not involve interactions between the constant terms on the various cusps so there is no lost content by treating individual cusps.
\end{itemize}
\item The Lax-Phillips approach requires not only the positivity of $-\Omega_\frak{g}$ but also of its factors parallel to the compact and ray mentioned above, which we establish.
\item The key analytical estimates involve proving that a suitable Sobolev norm is bounded by the $L^2_a$ norm where $a$ is a value of the height parameter $y$. Our proof is a more complete elaboration of the arguments sketched in pp. 204 - 206 in \cite{LP}. The use of this bound is discussed in item \#6 below.
\begin{itemize}
	\item {\it Automorphic test functions} $\mathscr{D}_a$ are conventional test functions (smooth, compactly supported) that are also in $L^2_a$. We define an {\it automorphic Sobolev spaces} $\frak{B}^1_a$ as the closure of  $\mathscr{D}_a$ in the Lax-Phillips space $L^2_a$ with respect to an inner-product expressed in terms of the Casimir/Laplacian operator.
	\item We establish the density of the automorphic test functions in the Lax-Phillips space $L^2_a$ (cf. \hyperref[sec:Density-of-Automorphic-Test-Functions]{\S 7 ``For $a \gg 1 , \; \mathscr{D}_a \text{ is dense in } L^2_a(\Gamma \backslash G /K)$''}). 
	\item $L^2_a$ norms of truncated tails of elements of  $\frak{B}^1_a$ are then shown to vanish {\it strongly} by showing that the Lax-Phillips $L^2_a$ norm of the tail is bounded by a vanishing parameter on the tail times the element's global Sobolev $\frak{B}^1_a$ norm (cf. \hyperref[sec:L2-norms-zero-strongly]{\S 8 ``$L^2$ norms of truncated tails go to $0$ strongly''}).
	\item Last, and necessary for the validity of the Rellich lemma, we show and that this relationship is preserved under smooth truncations (cf. \hyperref[sec:Global-B1a-bounds-of-smooth-trunc-L2a-tail-norms]{\S 9 ``$\frak{B}^1$ norms of tails are bounded by global $\frak{B}^1$ norms''}).
\end{itemize}	
\item The positivity results imply the existence of a Friedrichs self-adjoint extension to the Lax-Phillips operator  (cf. \hyperref[sec:Remarks-Symmetric-versus-Self-Adjoint]{\S 2.6}, \hyperref[sec:Friedrichs-extension-defn]{\S 2.7} and \hyperref[sec:Unbounded-ops-with-compact-resolvent]{\S 2.8}).
\item The bounding of the Sobolev norm by the $L^2_a$ norm provides a Rellich lemma: the inclusion of the automorphic Sobolev space $\frak{B}^1_a$ into the Lax-Phillips space $L^2_a$ is compact. We apply this to conclude that the Friedrichs self-adjoint extension $\widetilde{\Delta}_a$ of $\Delta_a$ has a compact resolvent (cf. \hyperref[sec:compact-inclusion-compact-resolvent]{``compact inclusions imply compact resolvents''}). These results combine to show the discreteness of the spectrum of $\widetilde{\Delta}_a$ and the corresponding discrete decomposition of pseudo-cuspforms.
\item Meromorphic continutation of Eisenstein series $E_s$ is then established by:
\begin{enumerate}
\item Creating a pseudo-Eisentstein series $h_s$ from the height function $\eta$ and a smooth cutoff $\tau$. We show that $h_s$ is entire as a function-valued function of $s$.
\item A new function $\widetilde{E}_s$ is defined by taking an image of $h_s$ under the compact resolvent $(\widetilde{\Delta}_a - \lambda_s)^{-1}$ and subtracting this image from $h_s$ (where $\lambda_s = c \cdot s(s - 1) \text{ for suitable } c \in \mathbb{R}$).  $\widetilde{E}_s$ is seen to agree with $E_s$ for $\text{Re}(s) > 1$ and $\widetilde{E}_s$ extends $E_s$ past the critical line. 
\end{enumerate}
\end{enumerate}
Conceptually, the gist of the overall argument is, for suitably large values of the cusp height $a$, to demonstrate the positivity of the tangential and nontangenial components of $\Delta_a$ which establishes the existence of the Friedrichs extension $\widetilde{\Delta}_a$ of the restriction of the Laplacian $\Delta_a$. Next, the bounding of the Sobolev norms in the cusp parameter $a$ of smoothly truncated functions by their global Sobolev norms proves the Rellich lemma which is used to demonstrate the discreteness of the spectrum of $\widetilde{\Delta}_a$ and the compactness of its resolvent $(\widetilde{\Delta}_a - \lambda_s)^{-1}$.

Then, though the resolvent $(\widetilde{\Delta}_a - \lambda_s)^{-1}$ is not a true projection operator, $\widetilde{E}_s$ is created in a process analogous to orthogonalization: the pseudo-Eisenstein series $h_s$ is not an eigenfunction of $\Delta_a$, but by snipping off the image of $h_s$ under $(\widetilde{\Delta}_a - \lambda_s)^{-1}$ from $h_s$, we create an eigenfunction which agrees with the genuine Eisenstein series $E_s$ in the domain of $E_s$ and extends $E_s$ past $\text{Re}(s) = 1$ as a consequence of the entirety of $h_s$ and the holomorphic operator $(\widetilde{\Delta}_a - \lambda_s)^{-1}$.

\section{Motivation for the approach and some history}\label{Intro-and-Motivation}

The spectral methods employed in \cite{LP} are based on semigroup methods as expounded in \cite{Hille-Phillips-1957}. These results, though approached from a different viewpoint as explicated in \cite{PBG}, will be used to prove the discreteness of pseudo-cuspforms and the meromorphic continuation of Eisenstein series on several families of rank one groups. The perspective developed in \cite{PBG} clarifies the methodology sketched in \cite{CdV-Eis} and \cite{CdV-pseudo}. We use \cite{PBG} as our primary reference for background material and reproduce some of the results there for the convenience of the reader.

First is the case of $O(r,1)$, where the unipotent radical is abelian. Then, the ideas will be extended to the simplest rational forms of $U(r,1)$ and $Sp^\ast (r,1)$, which present somewhat greater technical challenges. Then we treat a fairly general case of $\mathbb{Q}$-rank one orthogonal groups.
The key element is systematic use of modern analysis, notably operator theory and global automorphic Sobolev spaces. 

\subsection{Motivation: discrete decomposition of $L^2(\Gamma \backslash G/K)$ cuspforms}\label{sec:Motivation-discrete-decomposition}

For the groups $G$ being considered, with discrete (arithmetic) subgroup $\Gamma$ and (maximal) compact subgroup $K$, automorphic forms will be $\mathbb{C}$-valued functions on $\Gamma \backslash G/K$, meeting further conditions depending on the situation. 
A key role is played by the constant term $c_P f$ of an automorphic form $f$ on $\Gamma \backslash G/K$ for parabolic $P$ with unipotent radical $N$, a function on $N \Gamma_\infty \backslash G/K$ defined by  
$$
(\text{constant term})\; f(x) = c_P f(x) = \int_{(N \cap \Gamma) \backslash N} f(n \cdot x) \; dn
$$
The left-invariant Haar measure $dn$ on $N$ descends to the quotient as a right-invariant measure since $N \cap \Gamma$ is discrete. The first example, $G = O(r,1)$, is such that the group $N$, the unipotent radical of a (minimal) parabolic subgroup, is abelian (isomorphic to $\mathbb{R}^{r-1}$), which will not be the case in subsequent examples. In all the cases considered, the quotient $(N \cap \Gamma)\backslash N$ is compact. 
The quotient has the unique compatible measure for {\it winding-up} and {\it unwinding}
$$
\int_{(N \cap \Gamma) \backslash N} \bigg( \sum_{\gamma \in N \cap \Gamma} \phi(\gamma n) \bigg) dn = \int_N \phi(n) \; dn \qquad (\text{for all } \phi \in C^o_c(N))
$$

The integral defining the constant term shows that $c_P f$ is a left $N$-invariant function on $G/K$:
$$
c_P f(n^\prime x) = \int_{(N \cap \Gamma) \backslash N} f(n \cdot n^\prime x) \; dn= \int_{(N \cap \Gamma) \backslash N} f((n n^\prime) \cdot x)  \; dn = c_P f(x) \enspace (\text{for } n^\prime \in N)
$$

Let $P$ be the parabolic with unipotent radical $N$, so that $P = NM$ where $M$ is a Levi-Malcev subgroup of $P$. $P$ fits into an Iwasawa decomposition $G = PK$. We have that $\Gamma_\infty = P \cap \Gamma$ normalizes $N \cap \Gamma$, noting that in the simpler examples  $N \cap \Gamma$ is of finite index in $P \cap \Gamma$. The constant term is still left $\Gamma_\infty$-invariant. Together, with the normality of $N$ in $P$, these observations show $c_P f$ is left $N \Gamma_\infty$-invariant. Thus, constant terms of functions $f$ on $\Gamma \backslash G/K$ are left $\Gamma \cap M$-invariant functions on
$$
N\backslash G/K = N\backslash (N A^+ K)/K \approx A^+ \times (\Gamma \cap M^1)\backslash M^1/(M^1 \cap K)
$$
(a fuller discussion of the groups occurring above, including $M^1$ is in section \cref{sec:Remarks-Number-of-Cusps}). In the first cases we examine, this simplifies to a ray
$$
N\backslash G/K = N\backslash (N A^+ K)/K \approx A^+ \approx (0, \infty)
$$

Some care is called for in that $f \in L^2(\Gamma\backslash G/K)$ does not imply that $c_P f \in L^2(N \Gamma_\infty \backslash G/K)$. However, if $f$ is locally $L^1$, so that $|f|$ has finite integrals over compact subsets of $\Gamma \backslash G$, Fubini's theorem implies that a compactly-supported integral of $f$ in one of several variables is again locally $L^1$. This applies to $x \times y \to f(n_x a_y)$ in Iwasawa coordinates. This aspect of the constant term map will be clarified later.

Cuspforms are automorphic forms $f$ meeting the Gelfand condition $c_P f = 0$. The term cuspform can be used in a strong sense that further requires a cuspform to be a $\Delta$-eigenfunction in $L^2(\Gamma\backslash G/K)$, but this usage often proves too restrictive. An additional complication is that $L^2$ functions do not have good pointwise values, so vanishing of the constant term must must be taken in an almost everywhere sense for $L^2$ automorphic functions. It turns out often better to consider the constant term map as a map on distributions, and the Gelfand condition interpreted as a distributional vanishing condition, as will be done later when the constant term map is examined in greater detail. Let
$$
L^2_o(\Gamma\backslash G/K) = \{L^2\text{-cuspforms}\} = \{f \in L^2(\Gamma\backslash G/K) \; : \; c_P f = 0 \}
$$
A classic result is the discrete decomposition of the space of cuspforms: 
%
%
%
%
\begin{theorem}
The space $L^2_o(\Gamma \backslash G/K)$ of square-integrable cuspforms is a closed subspace of $L^2(\Gamma\backslash G/K)$, and has an orthonormal basis of $\Delta$-eigenfunctions. Each eigenspace is finite-dimensional, and the number of eigenvalues below a given bound is finite.
\end{theorem}
\begin{proof}
cf. \cite{Selberg-1956}, \cite{Langlands-SLN544-1967-1976}, \cite{Gelfand-PS-1963}
\end{proof}
The closed-ness of the space of $L^2$ cuspforms comes from recharacterization of it in terms of pseudo-Eisenstein series, the basic theory of which is described later. In contrast, the full space $L^2(\Gamma\backslash G/K)$ does not have a basis of $\Delta$-eigenfunctions. In the concrete cases first examined of $O(r,1)$, $U(r,1)$ and $Sp^\ast(r,1)$, the orthogonal complement of cuspforms in $L^2(\Gamma\backslash G/K)$ consists primarily of integrals of non-$L^2$ eigenfunctions for $\Delta$, namely the standard Eisenstein series $E_{s}$. However, in the more general $\mathbb{Q}$-rank one case, the orthogonal complement of cuspforms in $L^2(\Gamma\backslash G/K)$ consists primarily of integrals of the {\it cuspidal-data} Eisenstein series $E_{s,f}$ where $f$ is a function on $(\Gamma \cap M^1)\backslash M^1/(M^1 \cap K)$. Letting $\eta$ be the {\it height} function (see \cref{sec:Remarks-Number-of-Cusps}), in the concrete cases first considered, the constant term is $\eta^s + \eta^{1-s} = y^s + y^{1-s}$; in  the more general $\mathbb{Q}$-rank one case, the constant term is $(\eta^s + c_{s,f}\cdot \eta^{1-s})\cdot f(m^\prime)$ where $m^\prime \in M^1$ (cf. \cite{Langlands-SLN544-1967-1976}, \cite{Moeglin-Waldspurger-1995} and \cite{PBG}).

The operator $\Delta$ presents technical issues that require precise treatment. For example, while $L^2 (\Gamma\backslash G/K)$ lies inside the collection of distributions on $\Gamma\backslash G/K$, and distributional interpretation of the action of $\Delta$  would make it well-defined on all of $L^2(\Gamma\backslash G/K)$, $\Delta$ would not stabilize $L^2(\Gamma\backslash G/K)$. This would ostensibly seem to obstruct use of the symmetry or self-adjointness of $\Delta$ as an (unbounded) operator on a Hilbert space. To address this complication, precise invocation of the theory of unbounded operators on Hilbert spaces is used, notably of the Friedrichs extension of the restriction of $\Delta$ to subspaces of its domain as explicated in \cite{PBG}.

Although $\Gamma\backslash G/K$ may fail to be smooth, $\Gamma\backslash G$ is always smooth, because $\Gamma$ is discrete, and the equivalence
$$
C^\infty_c(\Gamma\backslash G/K) = C^\infty_c(\Gamma\backslash G)^K = \text{ right } K \text{-fixed test functions on } \Gamma\backslash G
$$
avoids issues of smoothness of $\Gamma\backslash G/K$ by using $K$-invariant functions on the smooth space $\Gamma\backslash G$. Questions regarding existence and sufficiency of $K$-invariant test functions are resolved by use of the Gelfand-Pettis vector-valued integral (cf. \cite{PBG} chapter 14). Specifically, the Gelfand-Pettis integral maps all test functions to right $K$-invariant ones since $K$ is compact and the action of $G$ on the right is continuous. Additionally, if the measure of $K$ is normalized to be $1$, then the averaging map is the identity on already-$K$-invariant functions.

We will introduce a more general setting when discussing $\mathbb{Q}$-rank one orthogonal algebraic groups, but will first consider real Lie groups admitting simpler coordinate descriptions (see \cref{sec:Remarks-Number-of-Cusps} for an overview of the general case considered). In these coordinates, for $g \in G$, a key role is played by a function $\eta(g)$ called the {\it height} of $g$ which is a function of a single real parameter in the Iwasawa coordinates of $g$. Let $g=n_x m_y = n_x m^\prime \ell_y$ be Iwasawa coordinates of $g$, with $n_x \in N$ and $m_y = m^\prime \ell_y \in M = A^+ \cdot M^1$ (the subgroups $A^+$ and $M^1$ and their relationship are discussed in \cref{sec:Remarks-Number-of-Cusps}). Define a function $\eta(g)$ called the {\it height} of $g$ by $\eta(g) = \eta(n_x m^\prime \ell_y) = \eta(\ell_y)$. In the cases of $O(r,1)$, $U(r,1)$ and $Sp^\ast(r,1)$, the height is a power of the {\it ray} coordinate $y$. 

\begin{remark}
Arithmetic quotients of $\mathbb{Q}$-rank one groups have cusps which are essentially points as opposed to the more geometrically complicated structures that arise when treating higher rank groups. By reduction theory, in the rank one case, a finite number of Siegel sets cover the quotient and the quotient can be expressed as a union of a compact set (with potentially complicated geometry) and a finite number of cusps which admit a simpler description. Notably, also by reduction theory, cylindrical {\it collar} neighborhoods of the cusps can be chosen corresponding to sufficiently large values of a height parameter so that the analysis on the neighborhood of each cusp is effectively independent of the analysis on the other cusps, as described in the section \hyperref[sec:Reduction-Theory-Siegel-Sets]{reduction theory and Siegel sets} (see also \cite{PBG}, \cite{Borel-1991}, \cite{Platonov-Rapinchuk-1994} and \cite{Springer-1994}). The net effect, which we will continually exploit, is that, while there might be more than one cusp, because the number of cusps is finite, in each situation where a bound on the height parameter is needed, a single value can be found that will work for all of the cusps individually. This allows us to treat the (finite) collection of cusps as though there was just one cusp which has the added value of reducing notational clutter and highlighting the key height parameter. There is no mathematical substance skirted by this approach, but the method is justified by clarifying the description and by simplifying the notation.
\end{remark}

For $a \geq 0$, define the Lax-Phillips space $L^2_a(\Gamma\backslash G/K)$ of pseudo-cuspforms by
$$
\{ f \in L^2(\Gamma\backslash G/K) : c_P f(m_y) = 0 \text{ for }m_y = m^\prime \ell_y, \; m^\prime \in M^1, \ell_y \in A^+ \text{ and } \eta(\ell_y) \geq a \}
$$
Let  
$$
\Delta_a = \Delta \bigg|_{C^\infty_c (\Gamma\backslash G/K) \cap L^2_a (\Gamma\backslash G/K )}
$$
and $\tilde{\Delta}_a$ its Friedrichs extension (see \cref{sec:Friedrichs-extension-defn}), called a {\it pseudo-Laplacian} attached to the restriction of $\Delta$ to $L^2_a(\Gamma\backslash G/K)$. We show
\begin{theorem}
For cut-off height a sufficiently large depending on $\Gamma$, $\tilde{\Delta}_a$ has purely discrete spectrum.
\end{theorem}

The proof will be an extension of the Lax-Phillips argument \cite{LP} but in the context of Friedrichs extensions rather than semigroups. For $O(r,1)$, it is convenient that the unipotent radical $N$ is abelian, so a literal Fourier expansion can be used as in \cite{LP}. For the other groups considered, $U(r,1)$ and $Sp^\ast(r,1)$, the unipotent radical is two-step nilpotent. For $U(r,1)$ it is a Heisenberg group. We prove a result extending the idea of \cite{CdV-Eis}. Compare also \cite{Muller-1996} and the appendix in \cite{Moeglin-Waldspurger-1995}.
%
%
\begin{theorem}
For $\mathbb{Q}$-rank one arithmetic quotients, the Eisenstein series $E_{s,f}$ associated to the cusps of $\Gamma$ and $\Omega^1$ eigenfunctions $f$ on $(\Gamma \cap M^1)\backslash M^1/(M^1 \cap K)$ have meromorphic continuations.
\end{theorem}
\noindent
Although the current methods prove meromorphic continuation without further efforts for multiple cusps, the proof of functional equations {\it is} more complicated because the functional equation is not localized to individual cusps but has subtle arithmetic information involving the interaction between cusps.

Again following the approach developed in Lax-Phillips but using Friedrichs extensions as used in \cite{CdV-pseudo} and \cite{CdV-Eis} rather than semigroups, we show
\begin{theorem}
For $\mathbb{Q}$-rank one arithmetic quotients, the new eigenfunctions (termed ``exotic'' in \cite{PBG}) for $\widetilde{\Delta}_a$ with $\lambda < -\frac{1}{4}$ are certain truncated Eisenstein series $\wedge^a_\kappa E_{s,f,\kappa}$ associated to the cusps $\kappa$ of $\Gamma$ 
where $P_\kappa$ is the parabolic stabilizing $\kappa$ with Levi-Malcev decomposition:

$$P_\kappa= N_\kappa\cdot M_\kappa = N_\kappa\cdot M^1_\kappa\cdot A^+_\kappa$$.	

Then, with $M_\kappa \ni m_a=m'\cdot \ell_a$ with $m'\in M^1_\kappa$ and $\ell_a\in A^+_\kappa$ ($M^1$ and $A^+$ will be described completely below), and where $c_{P_\kappa}$ is the constant term along $P_\kappa$, the exotic eigenfunctions are characterized by
$$c_{P_\kappa}E_{s,f,\kappa}(m_a)=0$$
\end{theorem} 

In the case of a single cusp, the condition simplifies to $c_P E_{s,f}(m_y) = 0$ where $M \ni m = m^\prime \cdot \ell_a$, $m^\prime \in M^1$, $\ell_a \in A^+$.

The notation is hopefully suggestive as to the intended meaning and interpretation. We follow the notational conventions of \cite{PBG} to help comparison with the more general treatment developed there. For additional detail and  development, see \cite{PBG} \S\S \; 1.11, 2.10, 3.14, 11.6 and 11.11.

%
%
%
%
%

\section{Background and overview of methods}\label{primary-tools}

The relevant background and methods used in the spectral analysis of automorphic forms are typically found in substantial volumes whose purpose is to function as references and often have broad scope. Rather than point the reader to these large works for background and context, we begin with an overview to establish terminology and perspective. We use \cite{PBG} as our primary reference and reproduce selected material for the convenience of the reader. The bibliography in \cite{PBG} is extensive and contains numerous additional references including many original sources. Other contemporary references are \cite{Iwaniec} and \cite{Moeglin-Waldspurger-1995}. 

We use techniques from modern analysis to identify solutions of self-adjoint extensions to invariant symmetric operators restricted to subspaces of functions defined on quotient spaces $\Gamma\backslash G$ and $\Gamma\backslash G/K$. Analysis of these operators is applied to the discrete decomposition of the space of pseudo-cuspforms and the meromorphic continuation of Eisenstein series $E_s$ by solving differential equations for smoothly truncated Eisenstein series. The tail terms will be inhomogeneous parts for distributional solutions of Poisson-type equations attached to the associated Friedrichs extension. Specifically, the operators will be used to define Sobolev-like spaces of functions with square-integrable derivatives that will be shown to have a Rellich-type compactness property. The Rellich compactness property will then be used to prove the discreteness of the spectrum of the operators which will in turn be used to establish a discrete decomposition of associated pseudo-cuspforms and the analytic continuation of certain types of Eisenstein series.

%
%
%

\subsection{Remarks on pseudo-Eisenstein series}\label{sec:Remarks-on-Pseudo-Eis}

The space of pseudo-Eisenstein series arises as the orthogonal complement to the space of cuspforms. For $L^2(\Gamma \backslash G/K)$, a goal could be to express the orthogonal complement to cuspforms $L^2_o(\Gamma \backslash G/K)$ in terms of $\Delta$-eigenfunctions. To exhibit explicit examples of $L^2$ functions spanning the complement, the Gelfand vanishing condition can be recast in a distributional fashion. First, for $f \in L^2(\Gamma \backslash G)$, the constant term $c_P f$ is a left $N\Gamma_\infty$-invariant function on $G$ which vanishes as a distribution if and only if 
$$
\int_{N\Gamma_\infty \backslash G} \phi \cdot c_P f = 0 \qquad (\text{for all } \phi \in C^\infty_c(N\Gamma_\infty\backslash G)
$$
with right $G$-invariant measure on $N\Gamma_\infty\backslash G$: the modular function of $G$ restricted to $N\Gamma_\infty$ equals the modular function of $N\Gamma_\infty$
$$
\delta_G \big|_{N\Gamma_\infty} = \delta_{N\Gamma_\infty}
$$
Since $f$ is right $K$-invariant, $c_P f$ is a right $K$-invariant distribution so it only needs to be evaluated on test functions $\phi \in C^\infty_c(N\Gamma_\infty \backslash G)^K$. Making use of the Iwasawa decomposition, there are isomorphisms
$$
N\backslash G/K \approx N \backslash (N A^+ M^1 K)/K \approx A^+ \cdot M^1/(K \cap M^1)
$$
that identify $N\backslash G/K$ with $A^+ \cdot M^1/(K \cap M^1)$ and also identify right $K$-invariant functions $\phi$ on $N\backslash G$ with functions of the height parameter $\textit{height}(n_x m_y k) = \textit{height}(n_x m^\prime \ell_y k) = \eta(g)$ (where $\ell_y \in A^+$ and $m^\prime \in M^1$). For $f \in L^2$, $f$ is locally integrable, and thus Fubini's theorem implies its constant term $c_P f$ is locally integrable. As such, $c_P f$ can be integrated against test functions on $N \Gamma_\infty \backslash G/K$.

A spectral theory involves the expression of an element of a vector space in terms of a basis of well-understood elements. The finite-dimensional case is of course completely understood and classical while in the infinite-dimensional, topological vector space setting many subtleties arise. In the function space setting, particularly those defined on homogeneous spaces, we typically prefer a basis comprised of eigenfunctions of a natural translation invariant differential operator. A common complication is the existence of important eigenfunctions that are not integrable in a desired sense. This arises already in the most basic case of the real line $\mathbb{R}$, the invariant one-dimensional Laplacian $d^2/dx^2$ and exponential functions $f(t) = e^{2 \pi i \xi t}$ for $\xi \in \mathbb{R}$. In the case of automorphic forms on $\Gamma\backslash G/K$, a comparable role is played by the {\it genuine Eisenstein series}. In the simpler, concrete cases considered the Eisenstein series are given by
$$
E_s(g) = \sum_{\gamma \in \Gamma_\infty\backslash \Gamma} \eta(\gamma \cdot g)^s
$$
where $\eta$ is the height function. In the more general case, including the $\mathbb{Q}$-rank one case we consider, Eisenstein series are given by
$$
E_s(g) = \sum_{\gamma \in \Gamma_\infty\backslash \Gamma}f(m^\prime(\gamma \cdot g))^s
$$
where $f$ is function on $(\Gamma \cap M^1)\backslash M^1/(M^1 \cap K)$.
Genuine Eisenstein series have many desirable properties. They are eigenfunctions for the invariant Laplacian and, for $\text{re}(s) > 1$, converge absolutely and uniformly on compact sets and are of moderate growth (cf. \cite{PBG} and references therein).  However, they are not in $L^2(\Gamma\backslash G/K)$. A collection of $L^2(\Gamma\backslash G/K)$ functions expressible in terms of integrals of genuine Eisenstein series is provided by {\it pseudo-Eisenstein series}.

Given a test function $\phi$ in $C^\infty_c(N\Gamma_\infty \backslash G)^K$, the corresponding pseudo-Eisenstein series $\Psi_\phi$ in $C^\infty_c(\Gamma \backslash G/K)$ is given by
$$
\Psi_\phi (g) = \sum_{\gamma \in \Gamma_\infty \backslash \Gamma} \phi(\gamma \cdot g)
$$
which fits into an adjunction:
$$
\int_{N\Gamma_\infty \backslash G} \phi \cdot c_P f = \int_{\Gamma \backslash G} \Psi_\phi \cdot f \qquad (\text{for } f \in L^2(\Gamma \backslash G))
$$
This adjunction is useful in aspects of unwinding and winding-up. Convergence of the sum is proven below. 

The expression for $\Psi_\phi$ can be obtained by direct computation using the $N\Gamma_\infty$-invariance of $\phi$ and the $\Gamma$-invariance of $f$, as follows. First, unwind $c_P f$
\begin{align*}
\int_{N\Gamma_\infty \backslash G} \phi \; \cdot \; c_P f =& \int_{N\Gamma_\infty \backslash G} \phi(g) \bigg( \int_{(N \cap \Gamma_\infty)\backslash N} f(ng) \; dn \bigg) \; d\mu(g) \\
 =& \int_{\Gamma_\infty \backslash G} \phi(g) \; f(g) \; d\mu(g)
\end{align*}
Using the $\Gamma$-invariance of $f$, wind up the integral on the right
\begin{align*}
\int_{\Gamma_\infty \backslash G} \phi(g) \; f(g) \; d\mu(g) =& \int_{\Gamma \backslash G} \sum_{\gamma \in \Gamma_\infty \backslash \Gamma } f(\gamma \cdot g) \phi(\gamma \cdot g) \;  d\mu(g) \\
=& \int_{\Gamma \backslash G} f(g) \bigg( \sum_{\gamma \in \Gamma_\infty \backslash \Gamma } \phi(\gamma \cdot g) \bigg) \; d\mu(g)
\end{align*}
which exhibits the convergence of the pseudo-Eisenstein series $\Psi_\phi$ associated to $\phi$ as all integrals are finite. This is the preparation for the standard result (cf. \cite{PBG} \S 1.8):
\begin{lemma}
The series for a pseudo-Eisenstein series $\Psi_\phi$ is locally finite: for $g$ in a fixed compact set in $G$, there are only finitely-many non-zero summands in $\Psi_\phi(g) = \sum_{\gamma \in \Gamma_\infty \backslash \Gamma} \phi(\gamma g)$. This implies $\Psi_\phi \in C^\infty_c(\Gamma \backslash G)$.	
\end{lemma}
\noindent
Also implying (cf. {\it ibid}).
\begin{corollary}
	Square-integrable cuspforms are the orthogonal complement in $L^2(\Gamma \backslash G /K)$ to the subspace of $L^2(\Gamma \backslash G/K)$ spanned by the pseudo-Eisenstein series $\Psi_\phi$ with $\phi \in C^\infty_c(N \Gamma_\infty\backslash G /K)$. The map $f \to c_P f$ is continuous from $L^2(\Gamma \backslash G /K)$ to distributions on $N\Gamma_\infty\backslash G/K$.
\end{corollary}

%
%

\subsection{Reduction theory and Siegel sets}\label{sec:Reduction-Theory-Siegel-Sets}

We will need only a few results from reduction theory regarding existence and basic properties of Siegel sets. Siegel sets, to be made precise below, are geometrically-simple sets with convenient covering properties in terms of translations by a discrete subgroup $\Gamma$. Fix a number field $k$ and $k$-bilinear $k$-valued form $S$ on a $k$-vector space V; in the $\mathbb{Q}$-rank one case ($\mathbb{Q}$-rank defined in the \hyperref[sec:Remarks-Number-of-Cusps]{next section}), there is a single $\mathbb{Q}$-conjugacy class of $\mathbb{Q}$-rational parabolics, given by the $k$-stabilizer of a $k$-isotropic $k$-line in the $k$-vector space. In this case, there is a single $\mathbb{Q}$-conjugacy class of $\mathbb{Q}$-parabolics so that there is no hierarchy to $\mathbb{Q}$-parabolics. Thus, cusps can be thought of as points which are topologically unrelated to each other. A characteristic that proves decisive is that a finite number of Siegel sets can surject onto the quotient. In those Siegel sets, a height parameter from the split component of the standard Levi-Malcev subgroup can be identified so that subsets of the Siegel sets, determined by conditions on the height parameter, have a particularly useful decomposition, as described below. Depending on context, Siegel sets may be considered to be subsets of $G$ or of $G/K$ as the key properties of Siegel sets do not involve the compact subgroup $K$. 

Slightly abusing notation, and since our primary focus is on real Lie groups, we will use $G$ both for the algebraic group and its real points. If we need to specify the local points over a particular field or ring, we will use standard notation as in the case of the rational points of $G$, namely $G(\mathbb{Q}) = G_\mathbb{Q}$. When emphasizing a collection of places (e.g., archimedean), we will use a subscript; thus the points of $G$ at the place $\nu$ will be denoted $G_\nu$. We will also use the convention that $\Gamma$ is an arithmetic subgroup of $G$.

In all of the cases considered, $G$ will be a $\mathbb{Q}$-rank one algebraic group with an Iwasawa decomposition of its real points
$$
G = PK = NMK
$$
where $K$ is a maximal compact of $G$, $P$ the standard minimal parabolic, $N$ the unipotent radical of $P$, and $M = A^+ \cdot M^1$ the corresponding standard Levi-Malcev component (the subgroups $A^+$ and $M^1$ are defined in \cref{sec:Remarks-Number-of-Cusps}). For our purposes, $G$ will be assumed to have a single cusp (for reasons elaborated below). For $g \in G$, let $g = n_g m_g k_g$ be $g$'s Iwasawa coordinates.  We recall the definition and some properties of Siegel sets: 
\begin{itemize}
	\item A (standard) Siegel set is a subset of $G$ given by a compact set $C \subset N$,  a compact set $D \subset M^1$, and a (height) parameter $t$
	$$
	\frak{S}_{t,C,D} = \{ n m k : \text{height of } m = \eta(m) = \eta(m^\prime \ell_y) = \eta(\ell_y) = \eta(y) \geq t\}
	$$
	where $n \in C \subset N$, $m = m^\prime \ell_y$ ($m^\prime \in D \subset M^1$, $\ell_y \in A^+$), and $k \in K$.
	\item Conveniently, in the first examples we treat of $O(r,1)$, $U(r,1)$ and $Sp^\ast(r,1)$, {\it $M^1$ will be compact so the prescription of $D$ will not be necessary and Siegel sets can be specified as $\frak{S}_{t,C}$}.
	\item Reduction theory implies that there is a sufficiently small $t$ and compact sets $C_i \subset N$ and $D_i \subset M^1$, along with a finite number of elements $g_i \in G(\mathbb{Q})$ so that $\Gamma$-translates of the union of the $g_i$-translates of the standard Siegel sets $\frak{S}_i = \frak{S}_{t,C_i,D_i}$ cover $G$:
	$$
	\Gamma \cdot \bigg( \bigcup_i g_i \cdot \frak{S}_i \bigg) = G
	$$
	Such a collection of Siegel sets clearly surjects to the quotients $\Gamma\backslash G$ and $\Gamma\backslash G/K$.
	\item Given such a collection $\frak{S}$ of Siegel sets, reduction theory then guarantees the existence of a height $t_o \gg t$ so that the subsets $\frak{S}^i_{t_o}$ of $\frak{S}^i_{t,C,D}$ given by the height condition
	$$
	\frak{S}^i_{t_o} = \big\{ g  \in \frak{S}^i_{t,C,D} \big| \enspace g = n \cdot m \cdot k \enspace \text{such that } m = m^\prime a_y \text{ and the height of } a_y = \eta(a_y) > t_o \big\}
	$$
	satisfy, {\it for all $i$ and $j$}, $\frak{S}^i_{t_o} \cap \gamma \frak{S}^j_{t_o} \ne \emptyset$ implies $\gamma \in \Gamma \cap P = \Gamma_\infty$. This implies the existence of a sufficiently large height so that high-enough portions of the Siegel sets in $\frak{S}$ do not interact in the quotient: the associated cusps can be treated separately. Since there is a finite collection of Siegel sets, common bounds can be taken to work for all the Siegel sets. In particular and without loss of generality, this allows us to simplify our notation by examining a single Siegel set as the results apply to any finite collection of cusps.
	\item The existence of suitable Siegel sets provides a way to use {\it separation of variables} to simplify analysis on the quotient $\Gamma\backslash G / K$ by making {\it tail} estimates on the ray $A^+$ more tractable. Using the simplification described above, this is implemented by examining a Siegel set $\frak{S}$ assumed for simplicity to surject to the quotient $\Gamma\backslash G$. Let $g=n_x m_y = n_x m^\prime \ell_y$ be Iwasawa coordinates with $n_x \in N$ and $m_y = m^\prime \ell_y \in M = A^+ \cdot M^1$. For $c \in \mathbb{R}$, let $\eta(g) = \eta(n_x m^\prime \ell_y) = \eta(\ell_y)$ be the height of $g$. Let $Y_o$ and $Y_\infty$ be the respective images of $\{g \in \frak{S}  :  \eta(g) \leq c+1\}$ and $\{g \in \frak{S} : \eta(g) \geq c \}$ in $\Gamma \backslash G/K$. By construction, the interiors of $Y_o$ and $Y_\infty$ cover $\Gamma\backslash G/K$. In the cases considered, while $Y_o$ will be a compact set with possibly complicated geometry, for sufficiently large $c \geq t_o$, $Y_\infty$ will conveniently decompose as a product of a compact manifold and a ray:
		$$
		Y_\infty \approx (N \cap \Gamma)\backslash N \times (M^1 \cap \Gamma)\backslash M^1/(M^1 \cap K) \times (c, \infty)
		$$
	\item The groups under consideration all admit Iwasawa decompositions with Levi-Malcev subgroups $M$ with simple split components $A^+$. As described in the next section, $A^+$ is a single {\it ray} (i.e., a subset of the real line homeomorphic to $(0,\infty)$).
\end{itemize}

%
%

\subsection{Analysis on neighborhoods of cusps of $\Gamma\backslash G/K$}\label{sec:Remarks-Number-of-Cusps}

We treat the quotients as though {\it there is only one cusp in $\Gamma\backslash G/K$} since when working with a finite number of Siegel sets corresponding to point cusps, a common bound can be found to assure disjointness. 

An additional complication, not of immediate interest to us here is that, when there is more than one cusp, the functional equation of Eisenstein series implies the Eisenstein series associated to a cusp is not mapped to itself, but gets smeared over the other cusps. The scattering matrix, relating the Eisenstein series at the cusps at $1-s$ to the Eisenstein series at the cusps at $s$, is nontrivial to determine. From a classical viewpoint the scattering matrix presents additional difficulties. Since those difficulties are not our focus, nor impact our calculations, we do not need to account for separate cusps and behave as though there was a just a single cusp.

Let $k$ be a number field (finite extension of $\mathbb{Q}$), and $S$ a symmetric $(r+1) \times (r+1)$ matrix over $k$ with non-zero diagonal entries. At different archimedean places of $k$, the local signature of $S$ can vary. Let the ``$k$-rank'' of a $k$-valued quadratic form on a $k$-vectorspace be the dimension of a (hence every, by Witt) maximal totally isotropic subspace.

We recall that a quadratic form over a non-archimedean field (or for characteristic not $2$) in $5$ or more variables has an isotropic vector (cf. \cite{Borevich-Shafarevich-1966} \cite{OMeara-2000}). This, in conjunction with the Hasse principle, that a quadratic form over a number field has a global non-trivial isotropic vector if and only if it has a local isotropic vector everywhere, implies that in dimension greater than or equal to $5$ that $S$ is $k$-anisotropic if and only if there is a real archimedean place where it is anisotropic. Thus, the $k$-rational rank in higher dimensions is mostly controlled by what happens at archimedean places.

That is, while we cannot require conditions at complex places, if there is a real archimedean place at which $S$ has signature $(r,1)$, and at every other real archimedean place the signature of $S$ is $(p,q)$ with $\text{min}(p,q)\ge 1$, and $r+1 \ge 5$, then the $k$-rank is $1$. In particular, there cannot be any real place of $k$ where the signature of $S$ is $(r+1,0)$ -- {\it $S$ is nowhere anisotropic} -- and there is at least one real place of $k$ where the signature of $S$ is $(r,1)$. In this case there are $k$-coordinates in which
$$
S = 
\renewcommand\arraystretch{1.25}
\begin{pmatrix}
0 & 0 & 1 \\
0 & S^\prime & 0 \\
1 & 0 & 0
\end{pmatrix}
$$
with $S^\prime$ anisotropic over $k$.

Thus, when the $\mathbb{Q}$-rank of $G$ is one, the cusps of $\Gamma\backslash G/K$ correspond to $\mathbb{Q}$-conjugacy classes of $\mathbb{Q}$-parabolics of $G$. Hasse-Minkowski and the assumption that the dimension of the vector space is $r+1 \geq 5$, implies that if the $\mathbb{R}$-rank is one at one real place and at least one at all real places, then the global rank is one.

This implies the pictures of standard proper $\mathbb{Q}$-rational parabolic $P$, unipotent radical $N$, and Levi-Malcev component $M$ exhibit the same structural features brought to bear in the more elementary versions of the $O(r,1)$, $U(r,1)$ and $Sp^\ast(r,1)$ examples.

However, in this case the Levi component $M$ is of the form
$$
\begin{pmatrix}
u & 0 & 0 \\
0 & h & 0 \\
0 & 0 &  u^{-1}
\end{pmatrix}
\qquad h \in O(S^\prime), u \in GL(1, k)
$$
and is not the cartesian product of a ray and a compact group but rather decomposes as a product $M = A^+ \cdot M^1$
$$
A^+  =  \bigg\{ m_y = 
\renewcommand\arraystretch{1.5}
\begin{pmatrix}
\ell & 0 & 0 \\
0 & 1_{r-1} & 0 \\
0 & 0 & \ell^{-1}
\end{pmatrix}
\ : \; \ell \in (0,\infty) \bigg\}
$$
Where the ray $(0,\infty)$ is embedded diagonally in $k_\infty^\times$. The complement to $A^+$ in the Levi component $M$ is
$$
M^1  =  \bigg\{ 
\renewcommand\arraystretch{1.5}
\begin{pmatrix}
u & 0 & 0 \\
0 & m^\prime & 0 \\
0 & 0 &  u^{-1}
\end{pmatrix}
\ : \; \text{where } N_{k/\mathbb{Q}}(u) = 1 \text{ and } m^\prime \in O(S^\prime) \bigg\}
$$
That is, the Levi component of P is $M = k_\infty^\times \times O(S')$ and extending the Galois norm $N_{k/\mathbb{Q}}$ to be suitably multilinear, $M^1$ is
$$
\{ b \in k_\infty^\times : N_{k/\mathbb{Q}}(b)=1  \} \times O(S^\prime)
$$
It is still the case though that $M_k\backslash M_\mathbb{A}$ is a cartesian product of a ray and a compact set by the compactness of anisotropic quotients. Thus the previous methods apply.

Viewed as a real Lie group, $G$ has an Iwasawa decomposition
$$
G = PK = NMK
$$
where $P$ is a parabolic subgroup, $K$ a (maximal) compact subgroup, $N$ the unipotent radical of $P$ and $M$ a Levi-Malcev complement to $N$. This gives
$$
G/K \approx N \cdot M/(M \cap K)
$$
In later sections, $G$ will be viewed as an algebraic orthogonal group, thought of as a functor from commutative $\mathbb{Q}$-algebras to groups. $G$ will correspond to a non-degenerate $k$-valued quadratic form $S$ of $k$-rank one on a $k$-vectorspace $V$. In this context, the group $G(\mathbb{Q}) = G_\mathbb{Q}$ of $\mathbb{Q}$-points of $G$ is the collection of $k$-linear automorphisms of $V$ preserving the form $S$. 

For a field extension $E$ of $\mathbb{Q}$, we can make a corresponding $E$-vectorspace $V\otimes_k(E \otimes_\mathbb{Q} k)$, extend the form $S$ bilinearly, and let $G(E)$ be the group of $E \otimes_\mathbb{Q} k$-linear automorphisms of $V\otimes_k(E \otimes_\mathbb{Q} k)$. For our purposes, $E$ will often be $\mathbb{R}$ or $\mathbb{Q}_p$ for varying $p$. 

For $\mathbb{Q}$-rank one, the number of cusps is the number of $\Gamma$-conjugacy classes of $\mathbb{Q}$-parabolics. This number is finite but is difficult to determine. For instance, for $\Gamma=SL(2,\frak{o})$ with ring of integers $\frak{o}$ of a number field, the number of cusps of $\Gamma\backslash G$ is the class number of $\frak{o}$.

However, it is convenient that the {\it collar} neighborhoods used in our analysis on cusps do not interact. That is, we have that finitely many Siegel sets suffice to cover the quotient (i.e., attached to $\mathbb{Q}$-parabolic conjugacy classes), and, given one such Siegel set $\frak{S}$ attached to minimal parabolic $P$, and a different minimal parabolic $Q$, there is a sufficiently small Siegel set $\frak{S}^\prime$ attached to $Q$ so that for $\gamma \in \Gamma$
$$
\gamma \cdot \frak{S} \cap \frak{S}^\prime \neq \emptyset \text{ if and only } \gamma \cdot P \cdot \gamma^{-1}=Q
$$

That is, for $\mathbb{Q}$-rank one orthogonal groups, reduction theory guarantees the existence of a sufficiently large height value so that the $\Gamma$-translates of the collar neighborhoods of the cusps do not interact, allowing each cusp to be treated individually. Since there is a finite number of cusps, this lets us choose the greatest (or least, depending on context) suitable height parameter which will then work for all of the cusps, thus allowing us to essentially treat the cusps separately.

In classical terms, for the representatives $P$ for each of the finitely-many $\Gamma$-conjugacy classes of rational parabolics, we have a corresponding unipotent radical $N$ and constant term.

Thus, up to choice of normalization and coordinates, the split component is isomorphic to the connected component of the identity in 
$$
P/N M^1 \approx \mathbb{R}^\times
$$
that is, the connected component being the ``ray''. Note that $M^1$ may not be compact but $(M^1 \cap \Gamma)\backslash M^1$ will be compact by the compactness of anisotropic quotients.

Adelically, Witt's theorem shows that all parabolics of the same type are $k$-conjugate which, together with adelic reduction theory, establishes that a single Siegel set can cover the quotient, so that there is just one cusp adelically.

%
%

\subsection{Truncation operators}\label{sec:Remarks-Truncation-Operators}

The genuine Eisenstein series are not in $L^2(\Gamma\backslash G/K)$, but from the theory of the constant term (see \cite{PBG} \S 8.1) the only obstruction is the constant term, which can be altered by truncation, removing this obstacle. 
The spectral decomposition of cuspforms (see \cite{Moeglin-Waldspurger-1995} and \cite{Langlands-SLN544-1967-1976}, along with references in \cite{PBG}) and the theory of the constant term establish that the $\Delta$-eigenfunction cuspforms are of rapid decay, and that the residues of Eisenstein series are in $L^2(\Gamma\backslash G/K)$ and are orthogonal to cuspforms.
We want truncation to produce automorphic forms. Naive truncation of the constant term for large values
$$
\text{naive truncation of }f(g) =
\begin{cases} 
f(g) &(\text{for } \eta(g) \leq T) \\
f(g) - c_P f(g) &(\text{for } \eta(g) > T)
\end{cases}
$$
does not produce an automorphic form: on a single Siegel set $\frak{S}_{t,C}$ this description functions correctly but it does not extend to $G/K$ or $G$ as it is not $\Gamma$-invariant. For sufficiently large (depending on the reduction theory) $T$ we can remedy this and achieve $\Gamma$-invariance by first defining the tail $c^T_P f$ of the constant term $c_P f$ of $f$ to be
$$
c^T_P f(g) =
\begin{cases} 
0 &(\text{for } \eta(g) \leq T) \\
c_P f(y) &(\text{for } \eta(g) > T)
\end{cases}
$$
For legibility, we may replace a subscript by an argument in parentheses in the notation for pseudo-Eisenstein series when the function $\phi$ has a more lengthy expression in which case we write
$$
\Psi(\phi) = \Psi_\phi
$$
Although $c^T_P f$ need not be smooth, nor compactly supported, by design (that is, for $T$ sufficiently large) its support is sufficiently high so that we have control over the analytical issues:
\begin{claim}
For $T$ sufficiently large, the pseudo-Eisenstein series $\Psi(c^T_P f)$ is a locally finite sum, hence, uniformly convergent on compacts.
\end{claim}
\begin{proof}
The tail $c^T_P f$ is left $N$-invariant. Reduction theory (cf. \cite{Borel-1965-66b}, \cite{PBG} and \cite{Springer-1994}) shows that, given $t_o$, for large-enough $t$, a set $\{n a_y k \; : \; y > t_o \}$ does not meet $\gamma \cdot \{ n a_y k \; :\; y > t \}$ unless $\gamma \in \Gamma_\infty$. Thus, for large-enough $T$, $\{ n a_y k \; : \; y > T \} $ does not meet $\gamma \cdot \{n a_y k \; : \; y>T \}$ unless $\gamma \in \Gamma_\infty$. Thus, $\gamma_1 \cdot \{ n a_y k \; : \; y > T \}$ does not meet $\gamma_2 \cdot \{ n a_y k \; : \; y > T \}$ unless $\gamma_1 \Gamma_\infty = \gamma_2 \Gamma_\infty$ $\square$
\\
\end{proof}
The translates of a standard Siegel set $\frak{S}_{t,C}$ cover $G$. The description of general Siegel sets is outside the scope of our discussion (cf. \cite{Borel-1965-66b} and \cite{PBG} and \cite{Springer-1994}). For our purposes, in a real-rank one group, a standard Siegel set is
$$
\frak{S}_{t,C} = \{ n a_y k : n \in C \text{ compact } \subset N, k \in K, a_y \in \text{ split component of } M, y \geq t \}
$$
This will require an extension for the more general cases of $\mathbb{Q}$-rank one groups. The above result leads similarly to
\begin{claim}
On a standard Siegel set $\frak{S}_{t,C}$, $\Psi(c^T_P f) = c^T_P f$ for all $T$ sufficiently large depending on $t$.
\end{claim}
\begin{proof}
By reduction theory, a set $\{n a_y k : y > t_o \}$ does not meet $\gamma \cdot \{ n a_y k : y > T \}$ unless $\gamma \in \Gamma_\infty$, for large-enough $T$ depending on $t_o$. Thus, for large-enough $T$, $\{ n a_y k  : y > T \}$ does not meet $\frak{S}_{t_o , C}$ unless $\gamma \in \Gamma_\infty$. That is, the only non-zero summand in $\Psi(c^T_P f)$ is the term $c^T_P f$ itself. $\square$
\\
\end{proof}
\noindent
Thus, we find that the proper definition of the truncation operator $\wedge^T$ is
$$
\wedge^T f = f - \Psi(c^T_P f)
$$
As desired, a critical effect of the truncation procedure is:
\begin{claim}
For $s$ away from poles, the truncated Eisenstein series $\wedge^T E_s$ is of rapid decay in Siegel sets.
\end{claim}
\begin{proof}
From the theory of the constant term, $E_s - c_P E_s$ is of rapid decay (cf. \cite{PBG} chapter 8 and \S 13.7) in a standard Siegel set. By the previous claim, $(E_s - c^T_P E_s)(g) = (E_s - c_P E_s)(g)$ for $\eta(g) \geq T$ , so it is is also of rapid decay. $\square$
\end{proof}

%
%

\subsection{Invariant description of the Casimir Operator}\label{sec:Remarks-The-Casimir-Operator}

The Killing form $B(X,Y)$ is $G$-invariant by $\textit{Ad}$ (and thus also $\frak{g}$-invariant by $\textit{ad}$), non-degenerate, and $G$-equivariantly identifies the dual $\frak{g}^\ast$ with $\frak{g}$. 
Using this, the Casimir operator can be characterized as usual as the element in the universal enveloping algebra $U\frak{g}$ of $\frak{g}$ which is the image of $1_\frak{g} \in \textrm{End}_R(\frak{g})$ under the chain of $G$-equivariant maps:
$$
\xymatrix{
1_\frak{g} \ar@{.>}[d]_{\rotatebox[origin=c]{270}{$\in$}} \ar[rrrr]                          &&&&                      \Omega = \Omega_\frak{g} \ar@{.>}[d]^{\rotatebox[origin=c]{270}{$\in$}} \\
\mathrm{End}_\mathbb{R}(\frak{g})  \ar[r]^\approx & \frak{g} \otimes \frak{g}^\ast \ar[r]^\approx_B & \frak{g} \otimes \frak{g} \ar[r]^{inc} & \otimes^\bullet \frak{g} \ar[r]^{quotient} & U\frak{g}}
$$

The identity map $1_\frak{g} \in \textrm{End}_\mathbb{R}(\frak{g})$ commutes with all automorphisms $\textrm{Ad}g$ from $g \in G$. Since all maps are $G$-equivariant, the image $\Omega_\frak{g}$ commutes with the action of $G$. This implies that the Casimir element $\Omega = \Omega_\frak{g}$ is an element of the sub-algebra $(U \frak{g})^G$ of $G$-invariant elements of $U\frak{g}$.

The $G$-invariance of $\Omega$ allows evaluation on functions on $G$ as either differential operators on the {\it left} or the {\it right}. Further, the images $\Omega f$ of right $K$-invariant functions $f$ are again right $K$-invariant. Thus, to evaluate $\Omega f (g)$ for $f$ on $G/K$ it suffices to evaluate $\Omega f (x)$ for a set of {\it representatives} $x$ for $G/K$.

There are at least two natural choices for representatives for $G/K$. Many models, such as those for hyperbolic $n$-space, use an $Iwasawa$ decomposition $G = PK = NMK$, where $N$ is the unipotent radical, $M$ is a Levi component of the parabolic subgroup $P$, and $K$ is a maximal compact subgroup, giving $G/K \approx NM/(M \cap K)$. A complication is that the Lie algebra of the parabolic $P = NM$ is not orthogonal to the Lie algebra $\frak{k}$ of $K$ with respect to the Killing form, and subsequent computations must take this into account.

To decompose the tangent space of $G$ with respect to $G/K$, coordinates from the {\it Cartan} decomposition are useful, as follows. Let $\tau$ be a Cartan involution on $\frak{g}$, that is, an involutory Lie algebra anti-automorphism such that $\frak{k}$ is the $(+1)$-eigenspace. Let $\frak{s}$ be the $(-1)$-eigenspace. Then the Cartan decomposition is $G = \textrm{exp}(\frak{s}) \cdot K$. Even though $\frak{s}$ is not a Lie subalgebra and $\textrm{exp}(\frak{s})$ is not a group, conveniently $\frak{s} \perp \frak{k}$ with respect to the invariant pairing $B( \cdot , \cdot)$.

For every basis $\{x_i\}$ of $\frak{g}$ and dual basis $\{x^\prime_i\}$  of $\frak{g}^\ast$ with $x^\prime_i = B(x_i, \cdot)$, the image of $1_\frak{g} \in \textrm{End}_\mathbb{R}(\frak{g})$ in $\frak{g} \otimes \frak{g}$ is $\sum_i x_i \otimes x^\prime_i$ so the Casimir operator $\Omega_\frak{g}$ can be expressed as
$$\Omega = \sum_{i=1}^n x_i x_i^\prime = \sum_{i=1}^n x_i^\prime x_i$$
The invariance of the Casimir element means $\Omega$ descends to both $\Gamma \backslash G$ and $\Gamma \backslash G/K$, giving the invariant Laplacian on the latter.

We will explicitly express the Casimir element in a two-step fashion: first expressing $\Omega_\frak{g} \in U \frak{g}$ using a basis for $\frak{g}$ and $\frak{g}^\ast$, the latter via identification of a dual basis by the Killing form as above,  and then making use of the natural coordinate systems provided by the Iwasawa decomposition to simplify the resulting expression by noting that operators associated with the maximal compact $K$ act by $0$ on functions on $G/K$.

%
%

\subsection{Symmetric versus self-adjoint operators}\label{sec:Remarks-Symmetric-versus-Self-Adjoint}

This is preparation for eigenfunction decompositions of Hilbert spaces by operators closely related to invariant Laplacians. The exposition below is based on \cite{PBG} (cf. Chapter 9).

Resolvents $R_\lambda = (T - \lambda)^{-1}$ can exist, as everywhere-defined, continuous linear maps on a Hilbert space, even for $T$ unbounded and only densely-defined. This requires that $T$ is symmetric, in the sense that $\langle T v, w \rangle = \langle v , T w \rangle$ for $v$, $w$ in the domain $D_T$ of $T$, and semi-bounded in the sense that there is a constant $C$ such that either $\langle Tv,v \rangle \geq C \cdot \langle v,v \rangle$ for all $v$ in $D_T$ or $\langle Tv,v \rangle \leq C \langle v,v \rangle$ for all $v$ in $D_T$. Under these conditions, $T$ has a Friedrichs self-adjoint extension, which provides several useful features, as will be described below.

In common scenarios, we may anticipate that a given unbounded operator is self-adjoint when extended suitably, and a simple version of the operator can be defined on an easily described, small, dense domain, where it specifies a symmetric operator. Then a self-adjoint extension is shown to exist, as in Friedrichs' theorem below.

A not-necessarily continuous, that is, not-necessarily bounded, linear operator $T$, defined on a dense subspace $D_T$ of a Hilbert space $V$, is called an {\it unbounded} operator on $V$, even though it is likely not defined or definable on all of $V$. We consider mostly {\it symmetric} unbounded operators $T$, meaning that $\langle Tv,w \rangle = \langle v,Tw \rangle$ for $v$,$w$ in the domain $D_T$ of $T$. 

For unbounded operators on $V$, description of the domain is essential: an unbounded operator $T$ on $V$ requires specifying a subspace $D$ of $V$ and a linear map $T :  D \to V$. By convention, explicit declaration of the domain of an unbounded operator is often suppressed, instead writing $T_1 \subset T_2$ when $T_2$ is an extension of $T_1$, in the sense that the domain of $T_2$ contains that of $T_1$, and the restriction of $T_2$ to the domain of $T_1$ agrees with $T_1$. Unlike self-adjoint operators on finite-dimensional spaces, and unlike  self-adjoint bounded operators on Hilbert spaces, symmetric unbounded operators, even when densely defined, usually need to be extended in order to behave more like self-adjoint operators in finite-dimensional and bounded-operator situations.

We see below that the adjoint $T^\ast$ of a symmetric operator $T$ is not symmetric unless already $T$ is self-adjoint, that is, unless $T = T^\ast$. In particular, existence of adjoints for symmetric, densely-defined operators $T$ does not immediately imply existence of $(T^\ast)^\ast$. Paraphrasing the notion of symmetry: a densely-defined operator $T$ is symmetric when $T \subset T^\ast$, and self-adjoint when $T = T^\ast$. These comparisons refer to the domains of these not-everywhere-defined operators. In the following claim and its proof, the domain of a map $S$ on $V$ is incorporated in a reference to its graph 
$$
\text{graph } S = \{ v \oplus Sv \; : \; v \in \text{domain} S \} \subset V \oplus V
$$
The direct sum $V \oplus V$ is a Hilbert space with natural inner product $\langle v \oplus v^\prime , w \oplus w^\prime \rangle = \langle v, v^\prime \rangle + \langle w, w^\prime \rangle$. Define an isometry $U :  V \oplus V \to V \oplus V$ by $v \oplus w \to -w \oplus v$. This is preparation for the standard (\cite{Friedrichs} \cite{PBG}) result concerning the existence of the Friedrichs self-adjoint extension for semi-bounded symmetric operators. We say that an operator $T^\prime$, $D^\prime$ is a sub-adjoint to a symmetric operator $T$, $D$ when 
$$
\langle Tv,w \rangle = \langle v,T^\prime w \rangle (\text{for } v \in D, w \in D^\prime)
$$
 For dense domain $D$, for given $D^\prime$ there is at most one $T^\prime$ meeting the sub-adjointness condition.

\begin{lemma}\label{Friedrichs-Extension-Existence}(Existence of Friedrichs extensions)
Given $T$ with dense domain $D$, there is a unique maximal $T^\ast$, $D^\ast$ among all sub-adjoints to $T$, $D$. The adjoint $T^\ast$ is closed, in the sense that its graph is closed in $V \oplus V$. In fact, the adjoint is characterized by its graph, the orthogonal complement in $V \oplus V$ to the image of the graph of $T$ under $U$, namely, 
$$
\text{graph } T^\ast = \text{ orthogonal complement of } U(\text{graph }T)
$$
\end{lemma}
\begin{corollary}
For $T_1 \subset T_2$ with dense domains, $T^\ast_2 \subset T^\ast_1$. $\square$
\end{corollary}
\begin{corollary}
$T \subset T^{\ast\ast}$ for densely-defined, symmetric $T$.
\end{corollary}
\begin{corollary}
A densely-defined self-adjoint operator has a closed graph.	
\end{corollary}
\begin{proof}	
Self-adjointness of densely-defined $T$ includes equality of domains $T = T^\ast$. Again, since the graph of
$T^\ast$ is an orthogonal complement, it is closed. $\square$
\end{proof}

Closed-ness of the graph of a self-adjoint operator is essential in proving existence of resolvents, below.
\begin{corollary}
The adjoint $T^\ast$ of a symmetric densely-defined operator $T$ is also symmetric if and only if $T = T^\ast$. $\square$
\end{corollary}

\begin{proposition}
Eigenvalues for symmetric operators $T$, $D$ are real.
\end{proposition}
\noindent
We have the standard (cf. \cite{PBG})
\begin{theorem}
Theorem: Let $T$ be self-adjoint with dense domain $D$. For $\lambda \in \mathbb{C}$, $\lambda \notin \mathbb{R}$, the image $(T - \lambda )D$ is the whole Hilbert space $V$. The resolvent $R_\lambda$ exists. For $T$ positive, for $\lambda \notin [0, +\infty)$, the image $(T - \lambda )D$ is the whole space $V$, and $R_\lambda$ exists.
\end{theorem}
\noindent
and (cf. {\it ibid})
\begin{theorem}
(Hilbert) For $T$ self-adjoint, for points $\lambda$, $\mu$ off the real line, or, for $T$ positive self-adjoint and $\lambda$, $\mu$ off $[0, +\infty)$,
$$
R_\lambda - R_\mu = (\lambda - \mu) R_\lambda R_\mu
$$
For the operator-norm topology, $\lambda \to R_\lambda$, is holomorphic at such points.
\end{theorem}
%

%
%

\subsection{Friedrichs' canonical self-adjoint extensions}\label{sec:Friedrichs-extension-defn}

%
Following \cite{PBG} (see \cite{Friedrichs} for the original development), semi-bounded operators are more tractable than general unbounded symmetric operators. A densely-defined symmetric operator $T$, $D$ is {\it positive} (or non-negative), denoted $T \geq 0$, when $\langle Tv,v \rangle \geq 0$ for all $v \in D$. All the eigenvalues of a positive operator are non-negative real. Similarly, $T$ is negative when $\langle T v, v \rangle \leq 0$ for all $v$ in the (dense) domain of $T$ . Generally, if there is a constant $c \in \mathbb{R}$ such that $\langle T v, v \rangle \geq c \cdot \langle v, v \rangle$ (written $T \geq c$), or $\langle T v, v \rangle \leq c \cdot \langle v, v \rangle$ (written $T \leq c$), $T$ is said to be semi-bounded. 

The following argument for positive operators can easily be adapted to the general semi-bounded situation.
For positive, symmetric $T$ on $V$ with dense domain $D$, define a hermitian form $\langle , \rangle_1$ and corresponding norm $| \cdot |_1$ by 
$$
\langle v,w \rangle_1 = \langle v,w \rangle + \langle Tv,w \rangle = \langle v,(1 + T)w \rangle = \langle (1 + T)v,w \rangle \qquad (\text{for } v, w \in D)
$$
The symmetry and positivity of $T$ make $\langle , \rangle_1$ positive-definite hermitian on $D$, and $\langle v, w \rangle_1$ has sense whenever at least one of $v$, $w$ is in $D$. Let $V^1$ be the Hilbert-space completion of $D$ with respect to the metric $d_1$ induced by the norm $| \cdot |_1$ on $D$. The completion $V^1$ continuously injects to $V$.
\begin{theorem}
	(Friedrichs) A positive, densely-defined, symmetric operator $T$ with domain $D$ dense in Hilbert space $V$ has a positive self-adjoint extension $\widetilde{T}$  with domain $D \subset V^1$, characterized by
$$
\langle (1 + T) v , (1 + \widetilde{T})^{-1} w \rangle = \langle v, w \rangle \qquad (\text{for } v \in D \text{ and } w \in V)
$$
The bound $\langle T v,v \rangle \geq 0$ for $v$ in the domain $D$ of $T$ is preserved. The resolvent $(1 + T )^{-1}  :  V \to V^1$ is
continuous with respect to the finer topology on $V^1$.
\end{theorem}
\begin{proof}
Since the Friedrichs extension is so important to us, we recall some details. First, let $j$ be the continuous linear map $j : V^1 \to V$ obtained by extending by continuity the identity map $D \to D$, with the source being given the $| \cdot |_1$ topology and the target being given the $| \cdot |$ topology. We claim that $j$ is an injection. By construction, $\langle v,w \rangle_1 = \langle jv,Tw \rangle$ for $v \in V^1$ and $w \in D$. For $0 \neq v \in V^1$, since $D$ is dense in $V^1$, there exists $w \in D$ such that $\langle v,w \rangle_1 \neq 0$. For that $v$,
$$
0 \neq \langle v,w \rangle_1 = \langle jv,Tw \rangle
$$
Thus, $jv \neq 0$ for $0 \neq v \in V^1$, and $j$ is indeed injective. We may identify $V^1$ with its image in $V$, noting that $V^1$ has a finer topology than that induced from $V$.

For $h \in V$ and $v \in V^1$, the functional $\lambda_h :  v \to \langle v,h \rangle$ has a bound 
$$
|\lambda_h v| \leq |v| \cdot |h| \leq |v|_1 \cdot |h|
$$
so the norm of the functional $\lambda_h$ on $V^1$ is at most $|h|$. By Riesz-Fr{\' e}chet, there is unique $Bh$ in the Hilbert space $V^1$ with $|Bh|_1 \leq |h|$, such that $\lambda_h(v) = \langle v, Bh \rangle_1$ for $v \in V^1$, and then $|Bh| \leq |Bh|_1 \leq |h|$. The map $B  : V \to V^1$ is verifiably linear. There is a symmetry of $B$:
\begin{align*}
\langle Bv,w \rangle =& \lambda_w (Bv) = \langle Bv,Bw \rangle_1 = \langle Bw,Bv \rangle_1 = \lambda_v (Bw) \\
=& \langle Bw,v \rangle = \langle v,Bw \rangle \quad (\text{for } v,w \in V)
\end{align*}
Positivity of $B$ is similar:
$$
\langle v,Bv \rangle = \lambda_v(Bv) = \langle Bv,Bv \rangle_1 \geq \langle Bv,Bv \rangle \geq 0 
$$
$B$ is injective: for $Bw = 0$, for all $v \in V^1$
$$
0 = \langle v,0 \rangle_1 = \langle v,Bw \rangle_1 = \lambda_w(v) = \langle v,w \rangle
$$
Since $V^1$ is dense in $V$, this gives $w = 0$. The image of $B$ is dense in $V^1$: if $w \in V^1$ is such that
$\langle Bv,w \rangle_1 = \lambda v(w) = 0$ for all $v \in V$, taking $v = w$ gives
$$
0 = \lambda_w(w) = \langle w,Bw \rangle_1 = \langle Bw,Bw \rangle
$$
and by injectivity $w = 0$. Thus, $B :  V \to V_1 \subset V$ is bounded, symmetric, positive, injective, with dense image. In particular, $B$ is self-adjoint.

Thus, $B$ has a possibly unbounded positive, symmetric inverse $A$. Since $B$ injects $V$ to a dense subset $V_1$, necessarily $A$ surjects from its domain (inside $V_1$) to $V$. We claim that $A$ is self-adjoint. Let $S : V \oplus V \to V \oplus V$ by $S(v \oplus w) = w \oplus v$. Then $\text{graph } A = S(\text{graph } B)$. In computing orthogonal complements $X^\perp$, clearly
$$
(SX)^\perp = S(X^\perp) 
$$
since the domain of $B^\ast$ is the domain of $B$. Thus, $A$ is self-adjoint.
We claim that for $v$ in the domain of $A$, $\langle Av, v \rangle \geq \langle v, v \rangle$. Indeed, letting $v = Bw$, 
$$
\langle v,Av \rangle = \langle Bw,w \rangle  = \lambda_w Bw = \langle Bw,Bw \rangle_1 \geq \langle Bw,Bw \rangle = \langle v,v \rangle
$$
Similarly, with $v^\prime = Bw^\prime$, and $v \in V^1$,
$$
\langle v,Av^\prime \rangle = \langle v,w^\prime \rangle = \lambda_{w^\prime} v = \langle v,Bw^\prime \rangle_1 = \langle v,v^\prime \rangle_1 \qquad (v \in V^1, v^\prime \text{ in the domain of } A)
$$
Last, we show that $A$ is an extension of $S = 1 + T$. By the above,
$$
\langle v, Sw \rangle = \lambda_{Sw}v = \langle v, BSw \rangle_1 \qquad (\text{for } v,w \in D)
$$
but, by definition of $\langle , \rangle_1$
$$
\langle v, Sw \rangle = \langle v,w \rangle_1 \qquad (\text{for } v,w \in D)
$$
so that
$$
\langle v, w - BSw \rangle_1 = 0 \qquad (\text{for all } v,w \in D)
$$
Since $D$ is dense in $V_1$ (i.e., in the $| \cdot |_1$ topology), $BSw = w$ for $w \in D$. Thus, $w \in D$ is in the range of $B$, so is in the domain of $A$, and
$$
Aw = A(BSw) = Sw
$$
Thus, the domain of $A$ contains that of $S$ and extends $S$, so the domain of $A$ is dense in $V^1$. In fact, $B = (1 + T )^{-1}$ maps $V \to V^1$ continuously even with the finer $\langle , \rangle_1$-topology on $V^1$: the relation $\langle v, Bw \rangle_1 = \langle v,w \rangle$ for $v \in V^1$ with $v = Bw$ gives
$$
|Bw|^2_1 = \langle Bw, Bw \rangle_1 = \langle Bw,w \rangle \leq |Bw| \cdot |w| \leq |Bw|_1  \cdot |w|
$$
The resulting $|Bw|_1 \leq |w|$ gives continuity in the finer topology. $\square$
\end{proof}

The proof additionally shows the continuity of $(1 + \widetilde{T})^{-1} :  V \to V^1$ with the finer topology on $V^1$. This is a property of Friedrichs' self-adjoint extensions which is not shared by the other self-adjoint extensions of a given symmetric operator. This allows provides an important consequence used in \hyperref[sec:Delta-a-has-purely-discrete-spectrum]{$\tilde{\Delta}_a$ has purely discrete spectrum}.
\begin{corollary}\label{sec:compact-inclusion-compact-resolvent}
When the inclusion $V^1 \to V$ is compact, the resolvent $(1 + \widetilde{T})^{-1} : V \to V$ is compact.
\end{corollary}
\begin{proof}
In the notation of the proof of the theorem, $B : V \to V^1 \to V$ is the composition of this continuous map with the injection $V^1 \to V$ where $V^1$ has the finer topology. The composition of a continuous linear map with a compact operator is compact, so compactness of $V^1 \to V$ with the finer topology on $V^1$ suffices to prove compactness of the resolvent.
$\square$
\end{proof}

%
%

\subsection{Unbounded self-adjoint operators with compact resolvents}\label{sec:Unbounded-ops-with-compact-resolvent}
The following result and its proof are standard but reproduced here for reference (cf. \cite{PBG} Ch. 9). In the following, let $T$ be a possibly-unbounded operator with (dense) domain $D_T \subset V$.
\begin{lemma}\label{spectrum-T-spectrum-compact-resolvent}
For a not-necessarily-bounded self-adjoint operator $T$, if $T^{-1}$ exists and is compact, then $(T - \lambda)^{-1}$ exists and is a compact operator for $\lambda$ off a discrete set in $\mathbb{C}$, and is meromorphic in $\lambda$. Further, the spectrum of $T$ and non-zero spectrum of $T^{-1}$ are in the bijection $\lambda \leftrightarrow \lambda^{-1}$.
\end{lemma}
\begin{proof}
The set of eigenvalues or point spectrum of $T$ consists of $\lambda \in \mathbb{C}$ such that $T - \lambda$ fails to be injective. The continuous spectrum consists of $\lambda$ with $T - \lambda$ injective and with dense image, but not surjective. Further, for possibly unbounded operators, we require a bounded (and thus continuous) inverse $(T - \lambda)^{-1}$ on $(T - \lambda)D_T$ for $\lambda$ to be in the continuous spectrum. The residual spectrum consists of $\lambda$ with $T - \lambda$ injective, but $(T - \lambda)D_T$ not dense.

The description of continuous spectrum simplifies for closed $T$ , that is, for $T$ with closed graph: we claim that for $(T - \lambda)^{-1}$ densely defined and continuous, $(T - \lambda)D_T$ is the whole space, so $(T - \lambda)^{-1}$ is everywhere defined, implying $\lambda$ cannot be in the residual spectrum. In particular, the continuity gives a constant $C$ such that $|x| \leq C \cdot |(T - \lambda)x|$ for all $x \in D_T$. Then $(T - \lambda)x_i$ Cauchy implies $x_i$ Cauchy, and $T$ closed implies $T (\lim x_i ) = \lim T x_i$ . Thus, $(T - \lambda)D_T$ is closed and the density of $(T - \lambda )D_T$ implies it is the whole space.

We need to prove that for $T^{-1}$ compact, the resolvent $(T - \lambda)^{-1}$ exists and is compact for $\lambda$ off a discrete set, and is meromorphic in $\lambda$. The non-zero spectrum of the compact self-adjoint operator $T^{-1}$ is point spectrum, from basic spectral theory for compact operators (see \cite{PBG} and references therein). We claim that the spectrum of $T$ and non-zero spectrum of $T^{-1}$ are in the natural bijection $\lambda \leftrightarrow \lambda^{-1}$. Since both $T$ and (the multiplication operator) $\lambda$ are invertible (and thus injective), the algebraic identities
$$
T^{-1} - \lambda^{-1} = T^{-1}( \lambda - T)\lambda^{-1} \qquad T - \lambda = T(\lambda^{-1} - T^{-1})\lambda
$$
imply that failure of either $T - \lambda$ or $T^{-1} - \lambda^{-1}$ to be injective forces the failure of the other, so the point spectra are identical. For (non-zero) $\lambda^{-1}$ not an eigenvalue of compact $T^{-1}$, $T^{-1} - \lambda^{-1}$ is injective and has a continuous, everywhere-defined inverse. That is, by the spectral theorem for self-adjoint compact operators, if $S$ is a compact self-adjoint then $S - \lambda$ is surjective for $\lambda \neq 0$ not an eigenvalue of $S$. For such $\lambda$, inverting the relation $T - \lambda = T(\lambda^{-1}- T^{-1})\lambda$ gives
$$
(T - \lambda)^{-1} = \lambda^{-1} (\lambda^{-1} - T^{-1})^{-1}T^{-1}
$$
from which $(T - \lambda)^{-1}$ is continuous and everywhere-defined so that $\lambda$ cannot be in the spectrum of $T$. Finally, $\lambda = 0$ is not in the spectrum of $T$ , because $T^{-1}$ exists and is continuous. This establishes the bijection.

Thus, for $T^{-1}$ compact self-adjoint, the spectrum of $T$ is countable, with no accumulation point in $\mathbb{C}$. Letting $R_\lambda = (T - \lambda)^{-1}$ the resolvent relation
$$
R_\lambda = (R_\lambda - R_0) + R_0 = (\lambda - 0) R_\lambda R_0 + R_0 = (\lambda R_\lambda + 1) \circ R_0
$$
expresses $R_\lambda$ as the composition of a continuous operator with a compact operator, proving its compactness.
$\square$
\end{proof}
Continuity is then immediate from Hilbert's relation
$$
(T - \lambda)^{-1} (\lambda - \mu)(T - \mu)^{-1} = (T - \lambda)^{-1} ((T - \mu)  - (T -\lambda)) (T - \mu)^{-1} = (T - \lambda)^{-1}  - (T - \mu)^{-1}
$$
Dividing through by $\lambda - \mu$ then gives
$$
\frac{(T - \lambda)^{-1} - (T - \mu)^{-1}}{\lambda - \mu} = (T - \lambda)^{-1}(T - \mu)^{-1}
$$
proving differentiability.
%
%

%
%
%

\subsection{Distributional characterization of Friedrichs extensions}\label{sec:Distrl-Char-Friedrichs-Ext}

Describing Friedrichs extensions of restrictions of $\Delta$ in terms of distributions can facilitate a finer analysis (cf. \cite{Bombieri-Garrett} and \cite{PBG} for details). In particular, for the cases being examined, this can be done abstractly and in the same context as the construction of the Friedrichs extension (cf. \cite{PBG} \S 11.2).

Let $V$ be a Hilbert space with a complex conjugation map $v \to \bar{v}$ compatible with the hermitian inner product on $V$ (i.e., for $v, w \in V$, $\langle v , w \rangle = \overline{\langle w , v \rangle }$, etc.). This gives a complex-linear isomorphism $c : V \to V^\ast$ of $V$ to its dual $V^\ast$ via Riesz-Fr{\' e}chet composed with complex conjugation, by $c :  v \to \langle \cdot , \bar{v} \rangle$. Let $S$ be a symmetric operator on $V$ with dense domain $D$, with $\langle Sv,v \rangle \geq \langle v,v \rangle$ for $v \in D$. Suppose that $S$ commutes with the conjugation map. Put $\langle x,y \rangle_{1} = \langle Sx,y \rangle$ for $x, y \in D$,and let $V^1$ be the completion of $D$ with respect to this norm. The identity map $D \to D$ induces a continuous injection $j : V^1 \to V$ with dense image. 

Write $V^{-1}$ for the Hilbert-space dual $(V^1)^\ast$ of $V^1$, with hermitian inner product $\langle, \rangle_{-1}$. Let $j^\ast$ be the adjoint map $j^\ast : V^\ast \to (V^1)^\ast$ of $j$, so composition with complex conjugation $c$ gives
\[
\xymatrix{
V^1  \ar[r]^j & V \ar@/^1pc/[rr]^{j^\ast \circ c} \ar[r]_c & V \ar[r]_{j^\ast} & V^{-1} \\
D \ar[u] \ar[ur]}
\]
There is a continuous linear map $S^\sharp : V^1 \to V^{-1}$, with the respective topologies, given by
$$
S^\sharp(x)(y) = \langle x, \bar{y} \rangle_1 \qquad (\text{for } x, y \in V^1)
$$
By Riesz-Fr{\' e}chet, this map is a topological isomorphism.
\begin{lemma}
The restriction of $S^\sharp$ to the domain of $\widetilde{S}$ is $j^\ast \circ c \circ \widetilde{S}$ . The domain of $\widetilde{S}$ is
$$
\text{domain } \widetilde{S} = \widetilde{D}  = \{ x \in V^1 \; : \; S^\sharp x \in (j^\ast \circ c)V \}
$$
\end{lemma}
\begin{proof} 
(We recall the proof from \cite{PBG} \S 11.2 and \S 9.2) By construction of the Friedrichs extension, its domain is exactly $\widetilde{D} = \widetilde{S}^{-1}V$. Thus, for $x = \widetilde{S}^{-1}x^\prime$ with $x^\prime \in V$, for all $y \in V^1$
$$
(S^\sharp x)(y) = (S^\sharp \widetilde{S}^{-1} x^\prime)(y) = \langle \tilde{S}^{-1} x^\prime, \bar{y} \rangle_{-1} = \langle x, \bar{y} \rangle = ((j^\ast \, \circ \, c)x^\prime)(y) = ((j^\ast \circ \, c \, \circ \, \widetilde{S})x)(y)
$$
Thus, the restriction of $S^\sharp$ to the domain $\widetilde{D}$ of $\widetilde{S}$ is essentially $\widetilde{S}$, namely,
$$
S^\sharp \big|_{\widetilde{D}} = (j^\ast \circ c \circ \widetilde{S}) \big|_{\widetilde{D}}
$$
Which implies that $S^\sharp : V^1 \to V^{-1}$ extends $\widetilde{S}$. However, we also have that for $S^\sharp x = (j^\ast \circ c)y$ with $y \in V$, this then implies that for all $z \in V^1$
$$
\langle z, \bar{x} \rangle_1 = (S^\sharp x)(z) = ((j^\ast \circ c)y)(z) = (\lambda y)(jz) = \langle jz, \bar{y} \rangle = \langle z, \widetilde{S}^{-1}\bar{y} \rangle_1
$$
giving $\bar{x} = \widetilde{S}^{-1}\bar{y}$ and the domain of $\widetilde{S}$ is as claimed. $\square$
\end{proof}
In the following, the development is made in the context of application to the case of pseudo-Eisenstein series $\Psi_\phi$ with $\phi \in C^\infty_c(a, \infty)$. Let $\Theta \subset D$ be stable under conjugation, and stable under $S$. Let $V_\Theta$ be the (closure of the) orthogonal complement to $\Theta$ in $V$. Let $S_\Theta$ be the  restriction of $S$ to $D_\Theta = D \cap V_\Theta$. The $S$-stability assumption on $\Theta$ gives $S(D_\Theta) \subset V_\Theta$. While $D_\Theta  = D \cap V_\Theta \subset V^1 \cap V_\Theta$, and since both $V^1$ and $V_\Theta$ are closed in $V$ (so their intersection is closed),  we have that the $V^1$ closure of $D_\Theta$ is a subset of $V^1 \cap V_\Theta$. However, the $V^1$-density of $D_\Theta$ in $V^1 \cap V_\Theta$ is not clear in general and we must assume that $D_\Theta$ is $V^1$-dense in $V^1 \cap V_\Theta$. This is shown under the hypotheses in the cases we are interested (see \hyperref[sec:Density-of-Automorphic-Test-Functions]{density of automorphic test functions} and also \cite{PBG} \S 10.3.1).

This density assumption legitimizes the natural sequel: $S_\Theta$ with domain $D_\Theta$ is densely defined and symmetric on $V_\Theta$, so it has a Friedrichs extension $\widetilde{S}_\Theta$, with domain $\widetilde{D}_\Theta$. The extension 
$$
(S_\Theta)^\sharp  :  V^1 \cap V_\Theta \to (V^1 \cap V_\Theta)^\ast
$$
is described by 
$$
(S_\Theta)^\ast(x)(y) = \langle x,y \rangle_1 \qquad (\text{for } x,y \in V^1 \cap V_\Theta)
$$
Let
$$
i_\Theta :  V^1 \cap V_\Theta \to V^1  \qquad i^\ast_\Theta :  V^{-1} = (V^1)^\ast \to (V^1 \cap V_\Theta)^\ast 
$$
be the inclusion and its adjoint, fitting into a diagram
\[
\xymatrix{
V^1  \ar[r]^j & V \ar[r]^{j^\ast \circ c} & V^{-1} \ar[d]^{i^\ast_\Theta} \\
V^1 \cap V_\Theta \ar[u]^{i_\Theta}  \ar[r] & V_\Theta \ar[u] & (V^1 \cap V_\Theta)^\ast}
\]
As in \cite{PBG} we have:
\begin{claim}
$(S_\Theta)^\sharp = i^\ast_\Theta \circ S^\sharp \circ i_\Theta$, and the domain of $\widetilde{S}_\Theta$ is
$$
D_\Theta = \{ x \in V^1 \cap V_\Theta \; : \; (S^\sharp \circ i_\Theta)x \in (j^\ast \circ c)V \; + \; \Theta \} = \{ x \in V^1 \cap V_\Theta \; : \; S^\sharp_\Theta x \in (i^\ast_\Theta \, \circ \, j^\ast \circ \, c)V \}
$$
and $\tilde{S}_\Theta x = y$, with $x \in V^1 \cap V_\Theta$ and $y \in V$, if and only if $(S^\sharp \circ i_\Theta)x = (j^\ast \circ c)y + \theta$ for some $\theta$ in the $V^{-1}$-closure of $(j^\ast \circ c)\Theta$.
\end{claim}
\begin{proof}
The assumption of denseness of $D_\Theta$ in $V^1 \cap V_\Theta$ legitimizes formation of the Friedrichs extension as an unbounded self-adjoint operator (densely defined) on $V$. For $x, y \in V^1 \cap V_\Theta$
$$
(i^\ast_\Theta \circ S^\sharp \circ i_\Theta)(x)(y) = S^\sharp(x)(y) = \langle i_\Theta x, i_\Theta \bar{y} \rangle_1 = \langle x, \bar{y} \rangle_1 = (S_\Theta)^\sharp(x)(y)
$$
which is the first statement of the claim. From the above, the Friedrichs extension $\widetilde{S}_\Theta$ is characterized by
$$
\langle z, \widetilde{S}^{-1}_\Theta y \rangle_1 =  \langle z,y \rangle \qquad (\text{for } z \in D_\Theta \text{ and } y \in V_\Theta)
$$
so that, given $S^\sharp x = (j^\ast \circ c)y + \theta$ with $x \in V^1 \cap V_\Theta$, $y \in V$, and $\theta$ in the $V^{-1}$ closure of $(j^\ast \circ c)\Theta$, take $z \in D_\Theta$ and compute
\begin{align*}
\langle x, \bar{z} \rangle_1 =& (S^\sharp x)(z) = ((j^\ast \circ c)y + \theta)(z) = (j^\ast \bar{y})(z) + \theta(z) \\
=& \langle z, \bar{y} \rangle + 0 = \langle y, \widetilde{S}^{-1}_\Theta S \bar{z} \rangle = \langle \widetilde{S}^{-1}_\Theta y, S \bar{z} \rangle = \langle \widetilde{S}^{-1}_\Theta y, \bar{z} \rangle_1
\end{align*}
thus showing that $\widetilde{S}^{-1}_\Theta x = y$. However, $(S_\Theta)^\sharp x = (i^\ast_\Theta \circ j^\ast \circ c)y$ if and only if $(S^\sharp \circ i_\Theta)x = y + \theta$ for some $\theta \in \text{ker}i^\ast_\Theta$, and $\text{ker} i^\ast_\Theta$ is the closure of $\Theta$ in $V^{-1}$. $\square$ 
\end{proof}

%
%
%

\subsection{Re-characterizing pseudo-Laplacians in terms of distributions}\label{sec:Distl-Char-Pseudo-Laplacians}

The re-characterization of Friedrichs extensions in terms of distributions applies to pseudo-Laplacians $\tilde{\Delta}_a$ for $a \gg 1$ large enough so that the density results developed in \hyperref[sec:Density-of-Automorphic-Test-Functions]{density of automorphic test functions} legitimizes the discussion. This will be needed for meromorphic continuation of Eisenstein series beyond the critical line. Our discussion follows \cite{PBG}.

Refering to the notation of the previous section, take $V = L^2(\Gamma\backslash G/K)$, use the pointwise conjugation map $c : L^2(\Gamma\backslash G/K) \to L^2(\Gamma\backslash G/K)$, let $D = C^\infty_c(\Gamma\backslash G/K)$, put $S = (1 - \Delta) \big|_D$, and let $\Theta = \Theta_a$ be the space of pseudo-Eisenstein series $\Psi_\phi$ with $\phi \in C^\infty_c(a,+\infty)$ with $a \gg 1$ large enough so that the density discussion above 
holds. Let $V^1 = \frak{B}^1$ be the completion of $D$ with respect to the norm given by 
$$
|f|^2_{\frak{B}^1} = \int_{\Gamma\backslash G/K} (1 - \Delta)f \cdot \bar{f} = \langle (1 - \Delta)f,f \rangle
$$
Let $\frak{B}^{-1}$ be the Hilbert space dual of $\frak{B^1}$. Letting $j : \frak{B}^1 \to V$ be the inclusion and $j^\ast$ its adjoint, we have a picture
\[\xymatrix@R+1pc@C+1pc{
\frak{B}^1 \ar[r]^j & V \ar[r]^{\enspace j^\ast \, \circ \, c \quad} & \frak{B}^{-1} 
}\]
Letting $\eta_a$ be the functional on $D$ which evaluates constant terms at height a, 
 the proof of the following lemma uses the standard spectral theory on multi-tori:

\begin{lemma}
For $a \gg 1$ sufficiently large, $\eta_a \in \frak{B}^{-1}$.
\end{lemma}

\begin{proof}
As expected, take $b^\prime \gg 1$ large enough so that the standard Siegel set $\frak{S}_{b^\prime}$ meets no translate $\gamma\frak{S}_{b^\prime}$ with $\gamma \in \Gamma$ unless $\gamma \in N \cap \Gamma$, so that the cylinder $C_{b^\prime} = (P \cap \Gamma)\backslash \frak{S}_{b^\prime}$ injects to $\Gamma\backslash G/K$. Take $a > b^\prime$. Since the support of $\eta_a$ is compact and properly inside $\frak{S}_{b^\prime}$, there is a test function $\psi$ identically $1$ on the support of $\eta_a$, and supported inside $\frak{S}_{b\prime}$. Then $\psi \cdot \eta_a = \eta_a$, in the sense that $\eta_a(f) = \eta_a(\psi f)$ for all test functions $f$. Thus, it suffices to consider test functions with support in a subset $X = (N \cap \Gamma)\backslash N \times (b^\prime, b^{\prime\prime})$ of the cylinder $C_{b\prime} = (N \cap \Gamma)\backslash N \times (b^\prime, +\infty) \approx (\mathbb{Z}\backslash \mathbb{R})^r \times (b^\prime, b^{\prime\prime})$, with $b^{\prime\prime} < + \infty$. 

Identifying the endpoints of the finite interval $(b^\prime, b^{\prime\prime}) \subset [b^\prime, b^{\prime\prime}]$ identifies it with another circle, thus imbedding $X \subset \mathbb{T}^{r+1}$. 
In this case, the $\frak{B}^1$ and $L^2$ norms on $X$ are uniformly comparable to those on $\mathbb{T}^{r + 1}$ descended from the Euclidean versions. Thus, to prove $\eta_a \in \frak{B}^{-1}$, it suffices to prove that the functional $\theta$ given by integration along $\mathbb{T}^r \times \{ 0 \}$ inside $\mathbb{T}^{r + 1}$ is in the corresponding $\frak{B}^{-1}$ space there. The advantage is that we can use Fourier series, since the spectral theory of $\mathbb{T}$ and $\mathbb{T}^n$ is already available (cf. \cite{PBG} \S 10.2 and \S 11.3).
%
 That is, parametrizing $\mathbb{T}^{r + 1}$ as $\mathbb{Z}^{r + 1}\backslash\mathbb{R}^{r + 1}$, let $\psi_\xi$ be $\psi(x) = e^{2 \pi i \xi \cdot x}$ for $\xi, x \in \mathbb{R}$ and $\xi \cdot x$ the usual inner product on $\mathbb{R}^{r + 1}$. Letting $\xi = (\xi_1, \ldots , \xi_{r + 1})$, the Fourier coefficients of $\theta$ are
$$
\theta(\xi) = \theta(\psi_\xi) = \int_{\mathbb{T}^r \times \{ 0 \}} \psi_\xi (x) \; dx = 
\begin{cases} 
0 & (\text{for } (\xi_1, \ldots , \xi_r) \neq (0 \ldots 0)) \\
1 & (\text{for } (\xi_1, \ldots , \xi_r) = (0 \ldots 0))
\end{cases}
$$
Thus, the $s^\text{th}$ Sobolev norm of $\theta$ is
$$
\sum_{\xi \in \mathbb{Z}^{r+1}}  |\theta(\xi)|^2 \cdot (1+|\xi|^2)^s = \sum_{\xi_{r+1} \in \mathbb{Z}}  1 \cdot (1+|\xi_{r+1}|^2)^s
$$
which is finite for $\text{Re}(s) < -\frac{1}{2}$ . Certainly it is finite for $s = -1$, giving the desired conclusion. $\square$
\end{proof}

In the previous lemma, on $\mathbb{T}^{r+1}$, $\theta$ is certainly the suitable Sobolev space limit of its finite subsums, which are smooth. This pulls back to an assertion that $\eta_a$ is in the $\frak{B}^1$ closure of test functions. We need a stronger assertion in order to use the re-characterization of the previous section. Following \cite{PBG}, we have:
\begin{lemma}
$\eta_a$ is in the $\frak{B}^{-1}$-closure of $\Theta$.
\end{lemma}
\begin{proof}
Again, by the previous lemma, $\eta_a$ is a $\frak{B}^{-1}$-limit of a sequence $\{ f_r \}$ of test functions on $\Gamma\backslash G/K$ or on the cylinder $C_{b^\prime}$. Following the approach in \cite{PBG} \S 10.3,
%
%
we show that suitable smooth truncations of the $f_r$, to put them into $\Theta$, still converge to $\eta_a$ in $\frak{B}^{-1}$. As in the previous proof, using $a \gg 1$, we can convert the question to one on $\mathbb{T}^{r+1}$ or on $\mathbb{T}^r \times \mathbb{R}$. Further, since nothing is happening in the first $r$ coordinates, it suffices to consider prove the following claim on $\mathbb{R}$.

That is, in the standard Sobolev spaces $H^s$ on $\mathbb{R}$ (see \cite{PBG} and references therein),
%
%
we claim that the standard Dirac $\delta$ on $\mathbb{R}$ is an $H^{-1}$ limit of a sequence of test functions supported in $[0, +\infty)$. Let $u$ be a test function on $\mathbb{R}$ which is $0$ in $(-\infty, 0]$, is non-negative with integral $1$ on $[0, +\infty)$. For $n = 1, 2, 3, \ldots$, let $u_r (t) = n \cdot u(nt)$. We claim that $u_r \to \delta$ in $H^{-1}$. Taking Fourier transforms,
$$
\widehat{u_r} (\xi) = \int_\mathbb{R} e^{- 2 \pi i \xi t} n \cdot u(nt) dt = \int_\mathbb{R} e^{- 2 \pi i \xi t/n } u(t) dt = \widehat{u} (\xi /n)
$$
The Fourier transform of $\delta$ is $1$, since $\delta(t \to e^{2 \pi i \xi t}) = 1$ for all $\xi \in \mathbb{R}$. The function $\widehat{u}$ is still a Schwartz function. We want to show that, as $n \to +\infty$,
$$
\int_\mathbb{R} \big| \widehat{u} (\xi /n) - 1 \big|^2 \cdot (1 + \xi^2 )^{-1} d\xi \to 0
$$
Certainly $\widehat{u}$ is bounded, so, given $\varepsilon > 0$, there is $N \gg 1$ such that for all $n$
$$
\int_{|\xi| \geq N} \big| \widehat{u} (\xi /n) - 1 \big|^2 \cdot (1 + \xi^2 )^{-1} d\xi < \varepsilon
$$
By the differentiability of $\widehat{u}$,
$$
\hat{u} (\xi /n) = \widehat{u}(0) + (\xi /n) \cdot \widehat{u}^{\, \prime} (t_o) \qquad (\text{for some } t_o \text{ between } 0 \text{ and } \xi/n)
$$
Since the integral of $u$ is $1$, $\hat{u}(0) = 1$. The derivative $u$ is continuous, so has a bound $B$ on $[-1,1]$. For $|\xi| \leq N$, take $n$ large enough so that $|\xi /n| < \varepsilon \leq 1$. Then
$$
 \int_{|\xi| \leq N} \big| \widehat{u} (\xi /n) - 1 \big|^2 \cdot (1 + \xi^2 )^{-1} d\xi = \int_{|\xi| \geq N} \big| (\xi /n) \cdot \widehat{u}^{\, \prime} (t_o) \big|^2 \cdot (1 + \xi^2 )^{-1} d\xi \\
$$
$$
\leq \int_{|\xi| \leq N} \varepsilon^2 \cdot B^2 \cdot (1 + \xi^2 )^{-1} d\xi \leq \varepsilon^2 \cdot B^2 \int_\mathbb{R}  (1 + \xi^2 )^{-1} d\xi \ll \varepsilon
$$
So that, in the spectral-side description of the topology on $H^{-1}$, we have the desired convergence. $\square$
\end{proof}
\begin{corollary}
$\widetilde{\Delta} u = f$ for $f \in L^2_a(\Gamma\backslash G/K)$ if and only if $u \in \frak{B}^1 \cap L^2_a(\Gamma\backslash G/K)$, and $\Delta u = f + c \cdot \eta_a$ for some constant $c$.
\end{corollary}
\begin{remark}
In particular, the proof mechanisms just above show that $u \in \frak{B}^1 \cap L^2_a(\Gamma\backslash G/K)$ implies that the constant term is in the Euclidean Sobolev space $H^1(\mathbb{R})$ as a function of the coordinate $y$. By Sobolev imbedding (see \cite{PBG} chapters 9 and 12), 
this implies continuity of the constant term, so vanishing in $\eta > a$ implies $\eta_a u = 0$. Conversely, if $u \in \frak{B}^1$ and $\eta_a u = 0$, we could truncate $u$ at height $a$ without disturbing the condition $u \in \frak{B}^1$, to put 
$\wedge^a u$ in $\frak{B}^1 \cap L^2_a(\Gamma\backslash G/K)$. In fact, after we have the meromorphic continuation of Eisenstein series in hand, and once we have a spectral form of global automorphic Sobolev spaces $\frak{B}^s$, one can easily prove that the conditions $(\Delta - \lambda) u = \eta_a$, $u \in \frak{B}^1$, and $\eta_a u = 0$ imply $\eta_{b^\prime} u = 0$ for all $b^\prime \geq a$.
\end{remark}
\begin{remark}
For $\lambda_w$ not the eigenvalue of a cuspform, the homogeneous equation $(\Delta - \lambda_w)u = 0$ has no non-zero solution, so the constant $c$ must be non-zero for non-zero $u$.
\end{remark}
\begin{proof}
Using the distributional characterization above for the Friedrichs extension, 
%
the previous lemma shows that $\eta_a$ is in the $\frak{B}^{-1}$ closure $\Theta_{-1}$ of $\Theta = \Theta_a$. Then, for $a \gg 1$, we must show that the intersection of that closure with the image $\Delta\frak{B}^1$ is at most $\mathbb{C} \cdot \eta_a$.

On one hand, because $a \gg 1$, $\Theta_{-1}$ consists of distributions which, on a Siegel set $\frak{S}_{b^\prime}$ with $b^\prime$ just slightly less than $a$, have support inside $\frak{S}_a \subset \frak{S}_{b^\prime}$. On the cylinder $C_{b^\prime} = \Gamma_\infty \backslash \frak{S}_{b^\prime}$, the product of circles $(N \cap \Gamma)\backslash N \approx \mathbb{T}^r$ acts by translations, descending to the quotient from $G/K$. By reduction theory, the restrictions to $C_{a^\prime}$ of every pseudo-Eisenstein series $\Psi_\phi$ with $\phi \in C^\infty_c [a, \infty)$ are invariant under $(N \cap \Gamma)\backslash N$, so anything in the $\frak{B}^{-1}$ closure is likewise invariant.

On the other hand, consider the possible images of $\frak{B}^1 \cap L^2_a(\Gamma\backslash G/K)$ by $\Delta$. Certainly $D \cap V_\Theta$ consists of functions with constant term vanishing in $\eta \geq a$, and taking $\frak{B}^1$ completion preserves this property. Since $\Theta_{-1}$ is $(N \cap \Gamma)\backslash N$-invariant and the Laplacian commutes with the group action, it suffices to look at $(N \cap \Gamma)\backslash N$-integral averages restricted to the cylinder $C_{b^\prime}$ . Such an integral is a restriction of the constant term $c_P v$ to $C_{b^\prime}$ , and vanishes in $\eta > a$.

Thus, the intersection of possible images by $\widetilde{\Delta}_a$ with $\Theta_{-1}$ consists of $(N \cap \Gamma)\backslash N$-invariant distributions in $\frak{B}^{-1}$ supported on $Z = \{ \eta \leq a \} \cap \{ \eta \geq a \} \approx (N \cap \Gamma)\backslash N$. Distributions supported on submanifolds (cf. \cite{PBG} \S 11.A) are obtained as compositions of derivatives transverse to $Z$ composed with a distribution supported on $Z$. By uniqueness of invariant distributions (cf. \cite{PBG} \S 14.4), the only $(N \cap \Gamma)\backslash N$-invariant distribution on $Z \approx (N \cap \Gamma)\backslash N$ is (a scalar multiple of) integration on $(N \cap \Gamma)\backslash N$.

Certainly $\eta_a$ itself is among these functionals. No higher-order derivative (composed with $\eta_a$) gives a functional in $\frak{B}^{-1}$, as is visible already on $\mathbb{R}$: computing the $s^\text{th}$ Sobolev norm of the $n^\text{th}$ derivative $\delta^{(n)}$ of the Euclidean Dirac $\delta$,
$$
\big| \delta^{(n)} \big|^2_{H^s} = \int_\mathbb{R} \big| \widehat{\delta^{(n)}}(\xi) \big|^2 \cdot (1 + \xi^2)^s d\xi = \int_\mathbb{R} \big| (-2 \pi i \xi)^n \big|^2 \cdot (1 + \xi^2)^s d\xi
$$
This is finite only for $s < -(1 + n)$. $\square$
\end{proof}

%
%

\subsection{Meromorphic continuation of Eisenstein series to $\text{Re}(s) > \frac{1}{2}$}\label{sec:Mero-Contn-Eis}
%
%
Following \cite{PBG}, let $\widetilde{\Delta}$ be the Friedrichs extension of the restriction of the Laplacian $\Delta$ to $C^\infty_c(\Gamma \backslash G/K)$. The Friedrichs construction shows that the domain of $\widetilde{\Delta}$ is contained in a Sobolev space
$$
\text{domain } \widetilde{\Delta} \subset \frak{B}^1 = \text{ completion of } C^\infty_c(\Gamma \backslash G/K) \text{ under } \langle v, w \rangle_{\frak{B}^1} = \langle (1 - \Delta)v, w \rangle
$$
We will see that the quotient $\Gamma \backslash G/K$ is a union of a compact part $Y_o$ and a non-compact part $Y_\infty$. 
Conveniently, $Y_\infty$ has a simple geometric form as a product of a compact manifold and a ray (i.e., a real interval $(0,\infty)$).
$$
\Gamma \backslash G/K = Y_o \cup Y_\infty \qquad (\text{compact } Y_o, \text{ cusp neighborhood } Y_\infty)
$$
where, with $a \gg 1$ as before, the normalized height function $\eta(n_x m_y k) = y^n$ provides a characterization  
$$
Y_\infty = \text{ image of } \{g \in G/K \; : \; \eta(g) \geq a \} = \Gamma_\infty \backslash \{ g \in G/K  \; : \; \eta(g) \geq a \} \approx X \times [a, \infty)
$$
Where $X$ is a compact manifold. In the first case considered, $O(r,1)$, where the unipotent radical is abelian, $X$ will be a multi-torus $\mathbb{Z}^n \backslash \mathbb{R}^n$ so that conventional Fourier methods can be used. In the somewhat more complicated cases of $U(r,1)$ and $Sp^\ast(r,1)$, the spectral theory of the Laplacian on compact manifolds will be needed.
Define a smooth cut-off function $\phi$ (these will be used extensively in the sections on the \hyperref[sec:Density-of-Automorphic-Test-Functions]{density of automorphic test functions} and \hyperref[sec:Global-B1a-bounds-of-smooth-trunc-L2a-tail-norms]{smooth truncations of tails}): fix $a < a^{\prime\prime} < a^\prime$ large enough so that the image of $\{ (x,y) \in G/K : y > a^{\prime\prime} \}$ in the quotient lies within $Y_\infty$, and let
$$
\begin{cases} 
0 = \phi (y) & (\text{for } y \leq a^{\prime\prime}) \\
0 \leq \phi (y) \leq 1  & (\text{for } a^{\prime\prime} \leq y \leq a^\prime) \\
1 =\phi (y)  & (\text{for } a^\prime \leq y) 
\end{cases}
$$
Form a pseudo-Eisenstein series $h_s$ by winding up the smoothly cut-off function $\phi(g) \cdot \eta(g)^s$:
$$
h_s (g) = \sum_{\gamma \in \Gamma_\infty \backslash \Gamma} \phi(\gamma g) \cdot \eta(\gamma g)^s
$$
Since $\phi$ is supported on $\eta \geq a^{\prime\prime}$ for large $a^{\prime\prime}$, for any $g \in G/K$ there is at most one non-vanishing summand in the expression for $h_s$, so that convergence is not an issue. Thus, the pseudo-Eisenstein series $h_s$ is entire as a function-valued function of $s$. Following \cite{CdV-Eis} and \cite{PBG}, let
$$
\widetilde{E_s} = h_s - (\widetilde{\Delta} - \lambda_s)^{-1}(\Delta - \lambda_s)h_s \qquad (\text{where } \lambda_s = (r-1)^2 \cdot s(s-1) \text{ with } r-1 = \text{dim N} ) 
$$
Now we can meromorphically continue from $\text{Re}(s) > 1$ up to the critical line.
\begin{lemma}
$\widetilde{E_s} - h_s$ is a holomorphic $\frak{B}^1$-valued function of $s$ for $\text{Re}(s) > \frac{1}{2}$ and $\text{Im}(s) \neq 0$.
\end{lemma}
\begin{proof} From Friedrichs' construction (\cite{Friedrichs}, \cite{PBG} \S 9.2), the resolvent $(\widetilde{\Delta} -\lambda_s)^{-1}$ exists as an everywhere-defined, continuous operator for $s \in \mathbb{C}$ as long as $\lambda_s$ is not a non-positive real number, because of the non-positiveness of $\Delta$. Further, for $\lambda_s$ not a non-positive real, this resolvent is a holomorphic operator-valued function. In fact, for such $\lambda_s$, the resolvent $(\widetilde{\Delta} - \lambda_s)^{-1}$ injects from $L^2(\Gamma\backslash G/K)$ to $\frak{B}^1$. $\square$
\end{proof}
%
%
\begin{remark}
	The smooth function $(\Delta - \lambda_s)h_s$ is supported on the image of $b \leq y \leq b^\prime$ in $\Gamma\backslash G/K$, which is compact. Thus, it is in $L^2(\Gamma\backslash G/K)$. $\widetilde{E_s}$ does not vanish, since the resolvent maps to the domain of $\widetilde{\Delta}$ inside $L^2(\Gamma\backslash G/K)$, and that $h_s$ is not in $L^2(\Gamma\backslash G/K)$ for $\text{Re}(s) > \frac{1}{2}$. Specifically, since $h_s$ is not in $L^2(\Gamma\backslash G/K)$ but $(\widetilde{\Delta} - \lambda_s)^{-1}(\Delta - \lambda_s)h_s$ is in $L^2(\Gamma\backslash G/K)$, the difference cannot vanish.
\end{remark}
\begin{theorem}
With $\lambda_s = s \cdot (s - 1)$ not non-positive real, $u = \widetilde{E_s} - h_s$ is the unique element of the domain of $\widetilde{\Delta}$ such that
$$
(\widetilde{\Delta} - \lambda_s)u = -(\Delta - \lambda_s)h_s
$$
Thus, $\widetilde{E_s}$ is the usual Eisenstein series $E_s$ for $\text{Re}(s) > 1$, and gives an analytic continuation of $E_s - h_s$
as  $\frak{B}^1$-valued function to $\text{Re}(s) > \frac{1}{2}$ with $s \notin (\frac{1}{2},1]$. 
\end{theorem}
\begin{proof}
See \cite{CdV-Eis} and  also \cite{PBG} chapter 11: to obtain the complete meromorphic continuation requires more.
$\square$
\end{proof}

%
%

\section{Some useful lemmas}

\subsection{Trapping a compact operator in maps between Hilbert spaces}\label{Trapping-Compact-Operator}

The following lemma from \cite{PBG} (\S 10.7.3) will repeatedly prove useful in the sequel and we repeat its proof here.
\begin{lemma}\label{trap-compact-hilbert}
Let $A, B, C, D$ be Hilbert spaces, with a commutative diagram of continuous linear maps
\[
\xymatrix{
A \ar[d]_R \ar[r] & B \ar[d]^T \\
C \ar[r]_S & D}
\]
with $T: B \to D$ compact, and $S:C \to D$ with constant $m>0$ such that $|v|_C \leq m·|Sv|_D$ for all $v \in C$. Then $R: A \to C$ is also compact.
\end{lemma}
%
%
%
%
\begin{proof}
Let $X$ be the closed unit ball in $A$, with image $Y$ in $C$. By continuity, the image of $X$ in $B$ is inside a finite-radius ball $Z$. By compactness of $T: B \to D$, given $\varepsilon > 0$, the image of $Z$ in $D$ is covered by finitely-many $\frac{\varepsilon}{m}$-balls $V_1, \ldots ,V_r$. The condition on $S$ assures that the inverse images $S^{-1}(SY \cap V_j)$ are contained in $\varepsilon$-balls in $C$. Thus, $Y$ is covered by finitely-many $\varepsilon$-balls in $C$. This holds for every $\varepsilon > 0$, so the image $Y$ is pre-compact, and $R: A \to C$ is compact. $\square$
\end{proof}

\subsection{Projections via smooth functions}\label{Projections-via-Smooth-Functions}

\noindent
The following simple lemma is used repeatedly.
\\
\begin{lemma}\label{sec:bdedness-smooth-projections}
Let $U \subset V$ be open subsets of $\mathbb{R}^n$. Let $\tau$ be a smooth, real-valued function that takes values in $[0,1]$ and which has compact support contained in $U$. Then the smooth projection given by domain truncation
$$
f \to \tau \cdot f \; : \; \frak{B}^1(V) \to \frak{B}^1(U)
$$
is continous and has bound depending only on $\tau$ and its derivatives.
\end{lemma}
\begin{proof}
This is a purely local result so we could take $U$ and $V$ to be concentric balls. Specifically, using the symmetry of $\langle , \rangle$
$$
|\tau \cdot f |_{\frak{B}^1(U)} = \int_U |\tau \cdot f|^2 + \langle \nabla (\tau \cdot f), \nabla (\tau \cdot f)\rangle =  \int_U |\tau \cdot f|^2 + |\tau|^2| \nabla f|^2 + |f|^2|\nabla \tau|^2  + 2\tau f \langle \nabla f, \nabla \tau\rangle  
$$
and
\begin{align*}
|\tau \cdot f|^2 \; \leq \; & |f|^2 \\
|\tau|^2| \nabla f|^2 \; \leq \; & \text{constant} \times | \nabla f|^2 \\
 |f|^2|\nabla \tau|^2 \; \leq \; & \text{constant} \times | f|^2 \qquad (\text{i.e., since $\tau$ is smooth and compactly supported}) \\ 
 2\tau f \langle \nabla f, \nabla \tau\rangle \; = \; & -2\tau^2 |\Delta f|^2 \; \leq \; \text{constant} \times |\nabla f|^2
\end{align*}
so the left-hand-side is bounded by a constants times $|f|^2 + |\nabla f|^2$. $\square$
\end{proof}

%
%
%
%
%
%
\section{Three groups: $O(r,1)$, $U(r,1)$ and $Sp^\ast(r,1)$}\label{three-groups}
%
%
%
%
%

\subsection{The Casimir operator for $O(r,1)$}\label{Casimir Operator for $O(r,1)$}

\subsubsection{Iwasawa coordinates on $N$, $M$ and $P \cap K$}\label{O(r,1) Iwasawa coords on N,M, PcapK}

Let $G \approx O(r,1)$: 
$$G =  \{ g \in GL_{r+1}(\mathbb{R}) \ | \ g^\top S g = S\}$$
where the form $S$ will be one of
$$
S = 
\renewcommand\arraystretch{1.25}
\begin{pmatrix}
0 & 0 & 1 \\
0 & 1_{r-1} & 0 \\
1 & 0 & 0
\end{pmatrix}
\qquad
\renewcommand\arraystretch{1.75}
S^\prime = 
\begin{pmatrix}
1_r & 0 \\
0 & -1	
\end{pmatrix}
$$ 
A Cayley element is
$$ 
\renewcommand\arraystretch{1.25}
C = 
\begin{pmatrix} 
\sqrt{\frac{1}{2}}  &      0          & -\sqrt{\frac{1}{2}} \\ 
           0               &   1_{r-1}    & 0                           \\
\sqrt{\frac{1}{2}}   &        0        & \sqrt{\frac{1}{2}}
\end{pmatrix}
\qquad
$$ 
We have
$$
C S^\prime C^{-1} = C S^\prime C^\top = S
$$
Thus, the two forms $S$ and $S^\prime$ give isomorphic isometry groups. The $S^\prime$ version is more convenient for identifying the maximal compact subgroup of $G$ but the $S$ version more clearly reveals the minimal parabolic $P$ as the stabilizer of a maximal isotropic flag as follows.

Let $e_1, \ldots e_{r+1}$ be the standard basis of $\mathbb{R}^{r+1}$ where $e_i$ is a column vector with a $1$ in the $i^\text{th}$ position and $0$'s elsewhere. With $S$ as above, we see that $e_1^\top S e_1 = 0 = e_{r+1}^\top S e_{r+1}$ while $e_i^\top S e_i = 1$ for $i \ne 1 \text{ or } r+1$. The two one-dimensional subspaces spanned by $e_1$ and $e_{r+1}$ are maximal isotropic. Additionally, since $e_{r+1}^\top S e_1 = 1 = e_{1}^\top S e_{r+1}$, $e_1$ and $e_{r+1}$ are a hyperbolic pair. We take $P$ as the subgroup of $G$ that fixes the (maximal isotropic) space spanned by $e_1$. 

Any linear map that sends $\langle e_1 \rangle$ to itself is automatically an isometry on the subspace spanned by $e_1$, since $e_1$ is isotropic, so by Witt's theorem can be extended to an isometry of $S$, and thus lies in $G$.  Thus $P$ consists of matrices of the form
$$
\begin{pmatrix}
\ast & \ast & \ast \\
0 & \ast & \ast \\
0 & 0 & \ast
\end{pmatrix}
$$
with appropriate relations, and with an $(r-1) \times (r-1)$ block in the middle.
A maximal compact, visible with the diagonal version of the form,  is:
$$
K \enspace = O(r) \times O(1) = G \cap O(r+1)  \qquad \text{(standard }O(r+1)\text{)}
$$
An Iwasawa decomposition is $G = PK$, with $P$ a minimal parabolic (i.e., the stabilizer of a maximal isotropic flag as above).
The maximal compact subgroup of $G$ will have non-trivial intersection with $P$, contained in its Levi component, given by
$$P \cap K = P \cap O(r+1) = M \cap O(r+1) \approx O(r-1) \times O(1)$$
which is visible as
$$
\begin{pmatrix}
\pm 1 & 0 & 0 \\
0 & h & 0 \\
0 & 0 & \pm 1
\end{pmatrix}
$$
with $h \in O_{r-1}(\mathbb{R})$. Since $N \cap K = \{1 \}$, we have that $N$ is of the shape
$$
\begin{pmatrix}
1 & u & a \\
0 & 1_{r-1} & v^\top \\
0 & 0 & 1
\end{pmatrix}
$$
Where $u,v \in \mathbb{R}^{r-1}$ and $a \in \mathbb{R}$. The relation $g^\top S g= S$ establishes $u = -v$ and $a = -\frac{|v|^2 }{2} = -\frac{|u|^2 }{2}$ so that elements of the standard unipotent radical $N$ of $P$ are
$$
N \enspace = \enspace \{ n_x \enspace = \enspace 
\renewcommand\arraystretch{1.5}
\begin{pmatrix}
1 & x & -\frac{1}{2}|x|^2 \\
0 & 1_{r-1} & -x^\top \\
0 & 0 & 1
\end{pmatrix}
\enspace : \enspace x \in \mathbb{R}^{r-1} \}
$$
The standard Levi component is
$$
M =  \{ m_y = 
\renewcommand\arraystretch{1.5}
\begin{pmatrix}
\pm y & 0 & 0 \\
0 & h & 0 \\
0 & 0 & \pm\frac{1}{y}
\end{pmatrix}
\ : \text{ for } y > 0 \text{ and } h \in O_{r-1}(\mathbb{R})\}
$$
and the split component of the Levi component is
$$
A^+  =  \{ m_y = 
\renewcommand\arraystretch{1.5}
\begin{pmatrix}
y & 0 & 0 \\
0 & 1_{r-1} & 0 \\
0 & 0 & \frac{1}{y}
\end{pmatrix}
\ : \text{ for } y > 0)\}
$$
The model of real hyperbolic $n$-space as
$$G/K \approx O(r,1)/(O(r) \times O(1)) \approx NM/(M \cap K) \approx N \times A^+$$
provides coordinates $x \in N \approx  \mathbb{R}^{r-1}$, $y \in A^+ \approx  \mathbb{R}^\times$ from the Iwasawa decomposition giving
$$
G/K \ni n_x m_y K \longleftrightarrow (x,y) \in \mathbb{R}^{r-1} \times (0,\infty)
$$
\subsubsection{Casimir in terms of the Lie algebra of $O(r,1)$}\label{Casimir in terms of the Lie algebra for O(r,1)}
The Lie algebra $\frak{g} \approx \frak{o}(r,1)$ is
$$
\frak{g} = \{ \gamma \in \frak{gl}_{r+1}(\mathbb{R}) \enspace : \enspace \textrm{exp}^{t\gamma} \in O(r,1), \; \forall t \in \mathbb{R} \}  = \{ \gamma \; : \; \gamma^\top = - S \gamma S^{-1} \}
$$
Thus a necessary condition is that since $(e^{\gamma})^\top S(e^{\gamma}) = S$
$$
\frac{d}{dt}\bigg|_{t=0}(e^{t\gamma})^\top S(e^{t\gamma}) = 0 = \gamma^\top S + S \gamma
$$
This implies $\gamma^\top = -S \gamma S^{-1}$ which also is sufficient since exponentiation respects transpose and conjugation:
$$
(e^{t\gamma})^\top = e^{t\gamma^\top} = e^{-tS \gamma S^{-1}} = S e^{-t \gamma}S^{-1}
$$
so that
$$
(e^{t\gamma})^\top S e^{t\gamma} = (S e^{-t \gamma}S^{-1}) S e^{t\gamma} = S
$$
To express the Casimir operator in a fashion consonant with the Iwasawa decomposition, note that the right action of $\frak{k}$ annihilates functions on $G/K$. Let $e_1, \ldots e_{r-1}$ be the standard (row-vector) basis for $\mathbb{R}^{r-1}$ so that $e_i = (0,\dots ,1, \ldots 0)$ with a ``$1$'' in the $i^\text{th}$ place. Let
$$
H=
\begin{pmatrix}
1 & 0 & 0 \\
0 & 0_{r-1} & 0 \\
0 & 0 & -1
\end{pmatrix}
\quad X_i =
\begin{pmatrix}
0 & e_i & 0 \\
0 & 0_{r-1} & -e_i^\top \\
0 & 0 & 0
\end{pmatrix}
\quad Y_i = X_i^\top =
\begin{pmatrix}
0 & 0 & 0 \\
e_i^\top & 0_{r-1} & 0 \\
0 & -e_i & 0
\end{pmatrix}
$$
This misses exactly the elements in $\text{Lie}(P \cap O(r+1)) \approx \frak{o}(r-1)$. These can be supplied by including
$$
\{
\begin{pmatrix}
0 & 0 & 0 \\
0 & \beta & 0 \\
0 & 0 & 0
\end{pmatrix}
\text{  with } -\beta = \beta^\top \} \approx \frak{o}(r-1) \subset \frak{k}
$$
Choose an orthonormal basis $\theta_i$ for this copy of $\frak{o}(r-1)$ with respect to $B( \cdot , \cdot)$, for instance, $\theta_i = \frac{\sqrt{2}}{2}(e_{ij} - e_{ji})$ where $e_{ij}$ is the matrix with a $1$ in the $ij^\text{th}$ spot for $i < j$, $i,j \in \{1, r-1\}$ and $0$ elsewhere. Using the trace pairing $B(X,Y) = \text{tr}(X,Y)$ (a scalar multiple of Killing) we have the following relations
$$B(H,H) = 2$$
$$B(H,X_i) = \text{tr}(e_{1i}) = 0 = B(H,Y_j) = \text{tr}(e_{nj}) = 0$$
$$B(X_i,Y_j) = 
\text{tr}
\begin{pmatrix}
\delta_{ij} & 0 & 0 \\
0 & \delta_{ij} & 0 \\
0 & 0 & 0
\end{pmatrix}
= 2\delta_{ij}
$$
$$B(\theta_i,\theta_j) = -\delta_{ij}$$
$$B(H,\theta_k) = B(X_i,\theta_k) = B(Y_j,\theta_k) = 0 \qquad \forall i,j,k$$
\noindent
With respect to the trace pairing (with $\prime$ denoting dual as above):
$$
H^\prime = \frac{1}{2}\cdot H \qquad X^\prime_i = \frac{1}{2}\cdot Y_i = \frac{1}{2}\cdot X^\top_i  \qquad [X_i,Y_i] = H \qquad \theta^\prime_j = -\theta_j
$$
Notably, the skew-symmetric $X_i - Y_i$ lies in $\frak{k}$ and thus acts by $0$ on the right on functions on $G/K$. For $f$ a right $K$-invariant function on $G$, it suffices to evaluate $\Omega f$ at group elements $n_x m_y$ since $\Omega$ preserves the right $K$-invariance. Thus
$$
\Omega = H \cdot H^\prime + \sum_i X_i \cdot X^\prime_i + \sum_i Y_i \cdot Y^\prime_i + \sum_j \theta_j \theta^\prime_j 
$$
$$
\Omega = \frac{1}{2}H^2 + \frac{1}{2}\sum_{i=1}^{r-1} X_i \cdot Y_i + \frac{1}{2}\sum_{i=1}^{r-1} Y_i \cdot X_i - \sum_j \theta_j^2 
$$
$$
= 	\frac{1}{2}H^2 + \frac{1}{2}\sum_{i=1}^{r-1} (X_i \cdot Y_i +  Y_i \cdot X_i) - \sum_j \theta_j^2  
$$
$$
= 	\frac{1}{2}H^2 + \frac{1}{2}\sum_{i=1}^{r-1} (2 X_i \cdot Y_i +  [Y_i, X_i]) - \sum_j \theta_j^2 
$$
$$
= 	\frac{1}{2}H^2 -  \frac{r-1}{2}H + \sum_{i=1}^{r-1} X_i \cdot Y_i - \sum_j \theta_j^2 
$$
$$
= 	\frac{1}{2}H^2 -  \frac{r-1}{2}H + \sum_{i=1}^{r-1} (X_i^2 + \underbrace{X_i(Y_i - X_i)}_\text{acts by 0}) - \sum_j \underbrace{\theta_j^2}_\text{acts by 0}  
$$
$$
= 	\frac{1}{2}H^2 -  \frac{r-1}{2}H + \sum_{i=1}^{r-1} X_i^2 + \text{(acts by 0)}
$$
\subsubsection{Casimir on $G/K$ in Iwasawa coordinates}\label{Casimir for O(r,1) in Iwasawa coords}
Exponentiating
$$
e^{tH}  \enspace =
\sum_{k=0}^{\infty} \frac{(tH)^k}{k!} \enspace =	
\begin{pmatrix}
e^t & 0 & 0 \\
0 & 1_{r-1} & 0 \\
0 & 0 & e^{-t}
\end{pmatrix}
 \enspace = m_{e^t}
$$
$$
e^{t X_i} =
\sum_{n=0}^{\infty} \frac{(tX_i)^k}{k!} \enspace =
\renewcommand\arraystretch{1.5}
\begin{pmatrix}
1 & t \cdot e_i & -\frac{t^2}{2} \\
0 \enspace & 1_{r-1} & -t \cdot e_i^\top \\
0 & 0 & 1
\end{pmatrix}
\enspace 
= n_{t \cdot e_i}
$$
So that $n_x m_y \cdot e^{tH} = n_x m_y m_{e^t} = n_x m_{y e^t}$ since multiplication in $M$ is homomorphic to multiplication in $\mathbb{R}^\times$.
\noindent
To determine $H$ as an operator on $G/K$, let $g \in G$ with corresponding $g =n_{x_0} m_{y_o}$
$$
H \cdot f(g) = \frac{d}{dt}\bigg|_{t=0} f(g \cdot m_{e^t}) = \frac{d}{dt}\bigg|_{t=0} f(n_{x_0} m_{y_0 e^t}) = y_0 \; \frac{\partial}{\partial y} \bigg|_{(n_{x_0} m_{y_0})} f
$$
so $H \text{ acts by } y \frac{\partial}{\partial y}$.
A convenient relation is
$$
m_y \cdot e^{t X_i} = m_y \cdot n_{te_i} =  
\begin{pmatrix}
y & 0 & 0 \\
0 & 1_{r-1} & 0 \\
0 & 0 & \frac{1}{y}
\end{pmatrix}
\begin{pmatrix}
1 & t \cdot e_i & -\frac{t^2}{2} \\
0 \enspace & 1_{r-1} & -t \cdot e_i^\top \\
0 & 0 & 1
\end{pmatrix}
=
\begin{pmatrix}
y & yt \cdot e_i & -\frac{yt^2}{2} \\
0 \enspace & 1_{r-1} & -t \cdot e_i^\top \\
0 & 0 & \frac{1}{y}
\end{pmatrix}
$$
$$
n_{yte_i} \cdot m_y =  
\begin{pmatrix}
1 & yt \cdot e_i & -\frac{y^2t^2}{2} \\
0 \enspace & 1_{r-1} & -yt \cdot e_i^\top \\
0 & 0 & 1
\end{pmatrix}
\begin{pmatrix}
y & 0 & 0 \\
0 & 1_{r-1} & 0 \\
0 & 0 & \frac{1}{y}
\end{pmatrix}
=
\begin{pmatrix}
y & yt \cdot e_i & -\frac{yt^2}{2} \\
0 \enspace & 1_{r-1} & -t \cdot e_i^\top \\
0 & 0 & \frac{1}{y}
\end{pmatrix}
$$
which gives
$$n_x \cdot m_y \cdot e^{t X_i} =n_x \cdot m_y \cdot n_{te_i} = n_x \cdot n_{yte_i} \cdot m_y = n_{x + yte_i} \cdot m_y$$
Using the same coordinates for $g$ to determine the operator $X_i$ and, thus, $X_i^2$
$$
X_i \cdot f(g) = \frac{d}{dt}\bigg|_{t=0} f(g \cdot n_{te_i}) =  \frac{d}{dt}\bigg|_{t=0} f(n_{x_0} m_{y_0} \cdot n_{te_i})
 = \frac{d}{dt}\bigg|_{t=0} f(n_{x_0 + y_0 te_i} m_{y_0}) = y_0 \; \frac{\partial}{\partial x_i} \bigg|_{(n_{x_0} m_{y_0})} f
$$
so that
$$
X_i = y \frac{\partial}{\partial x_i} \text{  and  } X_i^2 = X_i \circ X_i = y \frac{\partial}{\partial x_i} \big(y \frac{\partial}{\partial x_i} \big) = \bigg(y \frac{\partial}{\partial x_i}\bigg)^2 = y^2 \frac{\partial^2}{\partial x_i^2}
$$
Substituting appropriately in $\Omega$ on $G/K$
$$
\Omega = \frac{1}{2}(y \frac{\partial}{\partial y})^2 - \frac{1}{2}(r-1)y \frac{\partial}{\partial y} +y^2 \sum_{i=1}^{r-1} \frac{\partial^2}{\partial x_i^2} 
$$
$$
= \frac{1}{2}y^2 \frac{\partial^2}{\partial y^2} - \frac{1}{2}(r-2)y \frac{\partial}{\partial y} +y^2 \sum_{i=1}^{r-1} \frac{\partial^2}{\partial x_i^2} 
$$
To get rid of the mildly annoying break in symmetry due to the factor of $\frac{1}{2}$ in the ``$H$'' term, we could with hindsight renormalize coordinates as
$$
\mathbb{R}^{r-1} \times (0, +\infty) \ni (x,y) \to n_{x\sqrt{2}}m_y =
\begin{pmatrix}
1 & x \cdot \sqrt{2} & -|x|^2 \\
0 & 1_{r-1} & -x^\top  \cdot \sqrt{2}\\
0 & 0 & 1
\end{pmatrix}
\begin{pmatrix}
y & 0 & 0 \\
0 & 1_{r-1} & 0 \\
0 & 0 & \frac{1}{y}
\end{pmatrix}
$$
(i.e., replacing $X_i$ and $Y_i$ above with $\sqrt{2} \cdot X_i$ and $\sqrt{2} \cdot Y_i$ ) giving, on $G/K$
$$
2 \cdot \Omega = y^2(\frac{\partial^2}{\partial x_1^2} + \ldots \frac{\partial^2}{\partial x_{r-1}^2} + \frac{\partial^2}{\partial y^2}) - (r-2)y\frac{\partial}{\partial y}
$$
Ignoring the factor of $2$ and collecting terms
$$
\Omega = y^2(\Delta_x + \frac{\partial^2}{\partial y^2}) - (r-2)y\frac{\partial}{\partial y}
$$
%
%


%
%

\subsection{The Casimir operator for $U(r,1)$}\label{Casimir Operator for U(r,1)}

\subsubsection{Iwasawa coordinates on $N$, $M$ and $P \cap K$}\label{U(r,1) Iwasawa coords on N,M, PcapK}

$$G =  \{ g \in GL_{r+1}(\mathbb{C}) \ | \ g^\ast S g = S\} \approx U(r,1)$$
where the form $S$ will be one of
$$
S = 
\renewcommand\arraystretch{1.25}
\begin{pmatrix}
0 & 0 & 1 \\
0 & 1_{r-1} & 0 \\
1 & 0 & 0
\end{pmatrix}
\qquad
\renewcommand\arraystretch{1.75}
S^\prime = 
\begin{pmatrix}
1_r & 0 \\
0 & -1	
\end{pmatrix}
$$ 
A Cayley element is
$$ 
\renewcommand\arraystretch{1.25}
C = 
\begin{pmatrix} 
\sqrt{\frac{1}{2}}  &      0          & -\sqrt{\frac{1}{2}} \\ 
           0               &   1_{r-1}    & 0                           \\
\sqrt{\frac{1}{2}}   &        0        & \sqrt{\frac{1}{2}}
\end{pmatrix}
\qquad
$$ 
We have
$$
C S^\prime C^{-1} = C S^\prime C^\ast = S
$$
Thus, the two forms $S$ and $S^\prime$ give isomorphic isometry groups. The $S^\prime$ version is more convenient for identifying the maximal compact subgroup of $G$ but the $S$ version more clearly reveals the minimal parabolic $P$ as the stabilizer of a maximal isotropic flag as follows.

 Let $e_1, \ldots e_{r+1}$ be the standard basis of $\mathbb{C}^{r+1}$ where $e_i$ is a column vector with a $1$ in the $i^\text{th}$ position and $0$'s elsewhere. With $S$ as above, we see that $e_1^\ast S e_1 = 0 = e_{r+1}^\ast S e_{r+1}$ while $e_i^\ast S e_i = 1$ for $i \ne 1 \text{ or } r+1$. The two one-dimensional subspaces spanned by $e_1$ and $e_{r+1}$ are maximal isotropic. Additionally, since $e_{r+1}^\ast S e_1 = 1 = e_{1}^\ast S e_{r+1}$, $e_1$ and $e_{r+1}$ are a hyperbolic pair. We take $P$ as the subgroup of $G$ that fixes the (maximal isotropic) space spanned by $e_1$. 

Any linear map that sends $\langle e_1 \rangle$ to itself is automatically an isometry on the subspace spanned by $e_1$, since $e_1$ is isotropic, so by Witt's theorem can be extended to an isometry of $S$, and thus lies in $G$.  Thus $P$ consists of matrices of the form
$$
\begin{pmatrix}
\ast & \ast & \ast \\
0 & \ast & \ast \\
0 & 0 & \ast
\end{pmatrix}
$$
with appropriate relations, and with an $(r-1) \times (r-1)$ block in the middle.
A maximal compact, visible with the diagonal version of the form,  is
$$
K = U(n) \times U(1) = G \cap U(r+1)  \qquad \text{(standard }U(r+1)\text{)}
$$
An Iwasawa decomposition is $G = PK$, with $P$ a minimal parabolic (i.e., the stabilizer of a maximal isotropic flag as above).
The maximal compact subgroup of $G$ will have non-trivial intersection with $P$, contained in its Levi component M, given by
$$P \cap K = P \cap U(r+1) = M \cap U(r+1) \approx U(r-1) \times U(1)$$
which is visible as
$$
\begin{pmatrix}
\alpha & 0 & 0 \\
0 & h & 0 \\
0 & 0 & \bar{\alpha}^{-1}
\end{pmatrix}
$$
with $h \in U(r-1)$ and $|\alpha| = 1$, $\alpha \in \mathbb{C}^\times$. $N$ is of the shape
$$
N \enspace = \enspace \{ n_{z,x} \enspace = \enspace 
\renewcommand\arraystretch{1.5}
\begin{pmatrix}
1 & z & -\frac{1}{2}|z|^2 + ix \\
0 & 1_{r-1} & -z^\ast \\
0 & 0 & 1
\end{pmatrix}
\enspace : \enspace z \in \mathbb{C}^{r-1} \enspace z = u + iv \in \mathbb{R}^{r-1} \oplus \; i \mathbb{R}^{r-1} \enspace x \in \mathbb{R}\}
$$
The standard Levi component is
$$
M =  \{ m_\alpha = 
\renewcommand\arraystretch{1.5}
\begin{pmatrix}
\alpha & 0 & 0 \\
0 & h & 0 \\
0 & 0 & \bar{\alpha}^{-1}
\end{pmatrix}
\ : \text{ for } \alpha \in \mathbb{C}^\times \text{ and } h \in U(r-1)\}
$$
and the split component of the Levi component is
$$
A^+  =  \{ m_y = 
\renewcommand\arraystretch{1.5}
\begin{pmatrix}
y & 0 & 0 \\
0 & 1_{r-1} & 0 \\
0 & 0 & \frac{1}{y}
\end{pmatrix}
\ : \text{ for } y > 0)\}
$$
\noindent
The Lie algebra $\frak{g} \approx \frak{u}(r,1)$ is
$$
\frak{g} = \{ \gamma \in \frak{gl}_{r+1}(\mathbb{C}) \enspace : \enspace \textrm{exp}^{t\gamma} \in U(r,1), \; \forall t \in \mathbb{R} \} = \{ \gamma \; : \; \gamma^\ast = - S \gamma S^{-1} \}
$$
Specify a real basis for $\frak{u}(r,1)$ as:
$$
\renewcommand\arraystretch{1.5}
H =
\begin{pmatrix}
1 & 0 & 0 \\
0 & 0_{r-1} & 0 \\
0 & 0 & -1
\end{pmatrix}
\tilde{H} =
\begin{pmatrix}
i & 0 & 0 \\
0 & 0_{r-1} & 0 \\
0 & 0 & i
\end{pmatrix}
X_\ell =
\begin{pmatrix}
0 & e_\ell & 0 \\
0 & 0_{r-1} & -e_\ell^\ast \\
0 & 0 & 0
\end{pmatrix}
\tilde{X}_\ell =
\begin{pmatrix}
0 & ie_\ell & 0 \\
0 & 0_{r-1} & ie_\ell^\ast \\
0 & 0 & 0
\end{pmatrix}
$$
$$
\renewcommand\arraystretch{1.5}
Z =
\begin{pmatrix}
0 & 0 & 0 \\
0 & 0_{r-1} & 0 \\
i & 0 & 0
\end{pmatrix}
\tilde{Z} =
\begin{pmatrix}
0 & 0 & i \\
0 & 0_{r-1} & 0 \\
0 & 0 & 0
\end{pmatrix}
Y_\ell =
\begin{pmatrix}
0 & 0 & 0 \\
e_\ell^\ast & 0_{r-1} & 0 \\
0 & -e_\ell & 0 \\
\end{pmatrix}
\tilde{Y}_\ell =
\begin{pmatrix}
0 & 0 & 0 \\
ie_\ell^\ast & 0_{r-1} & 0 \\
0 & ie_\ell & 0 \\
\end{pmatrix}
$$
$$
\renewcommand\arraystretch{1.5}
\Theta_\ell =
\begin{pmatrix}
0 & 0 & 0 \\
0 & \theta_\ell & 0 \\
0 & 0 & 0
\end{pmatrix}
\qquad
\{ \theta_\ell \} \text{ basis of } \frak{u}(r-1)
$$

%
%

%
%
\subsubsection{Dual basis for $\frak{u}(r,1)$}\label{Dual basis for u(r,1)}
To compute a dual basis, use a trace form, which is necessarily a multiple of  Killing
$$
B(x,y) = \text{Re}(\text{tr}(xy)) \qquad x,y \in \frak{g} = \frak{u}(r,1)
$$
giving, with $\prime$ denoting dual as above
$$
H^\prime = \frac{1}{2}\cdot H \qquad X^\prime_\ell = \frac{1}{2}\cdot Y_\ell = \frac{1}{2}\cdot X^\ast_\ell  \qquad \tilde{X}^\prime_\ell = -\frac{1}{2}\cdot \tilde{Y}_\ell = -\frac{1}{2}\cdot \tilde{X}^\ast_\ell 
$$
$$
Z^\prime = -\tilde{Z} \qquad \tilde{H}^\prime = -\frac{1}{2}\tilde{H} \qquad \theta^\prime_j = -\theta_j 
$$
Conveniently, we have
$$
[X_\ell,Y_\ell] = H \qquad [\tilde{X}_\ell,\tilde{Y}_\ell] =-H \qquad [Z, \tilde{Z}] = H
$$
The skew-Hermitian symmetry of elements of $\frak{u}(r+1)$ and $\frak{u}(r-1) \oplus \frak{u}(1) = G \cap \frak{u}(r+1) = \frak{k}$ is used in conjunction with the fact that $\frak{k}$ acts trivially on functions on $G/K$, which shows  that the linear combinations of basis elements conveniently fall in $\frak{k}$. 
$$
\tilde{H} \qquad X_\ell - Y_\ell \qquad \tilde{X}_\ell + \tilde{Y}_\ell \qquad Z + \tilde{Z}
$$
%
%
%

%
\subsubsection{Rearrangement of Casimir operator on $G/K$}\label{Simplifying  Casimir G/K for U(r,1)}
The Casimir operator in coordinates is
$$
\Omega_\frak{g} = \sum_{x \in \frak{g}} x \cdot x^\prime 
$$
We wish to express the Casimir operator in a fashion consonant with the Iwasawa decomposition, taking advantage of the fact that the right action of $\frak{k}$ annihilates functions on $G/K$. 
$$
\Omega_\frak{g} = H \cdot H^\prime + \tilde{H} \cdot \tilde{H}^\prime + Z \cdot Z^\prime + \tilde{Z} \cdot \tilde{Z}^\prime + \sum^{r-1}_{\ell = 1} X_\ell \cdot X_\ell^\prime + \sum^{r-1}_{\ell = 1} Y_\ell \cdot Y_\ell^\prime + \sum^{r-1}_{\ell = 1} \tilde{X}_\ell \cdot \tilde{X}_\ell^\prime + \sum^{r-1}_{\ell = 1} \tilde{Y}_\ell \cdot \tilde{Y}_\ell^\prime + \sum^{r-1}_{j = 1} \theta_j \cdot \theta_j^\prime
$$
Ignoring terms involving  $\tilde{H}$ and the $\theta_j$, which act by $0$ on $G/K$, substitution of the expressions for dual elements in the above gives
$$
\Omega_\frak{g} = \frac{1}{2} H^2 - (Z \cdot \tilde{Z} + \tilde{Z} \cdot Z) + \frac{1}{2} \sum^{r-1}_{\ell = 1} \big( X_\ell \cdot Y_\ell + Y_\ell \cdot X_\ell \big) - \frac{1}{2} \sum^{r-1}_{\ell = 1} \big( \tilde{X}_\ell \cdot \tilde{Y}_\ell + \tilde{Y}_\ell \cdot \tilde{X}_\ell \big)+ \{\text{acts by zero}\}
$$
Since $Z + \tilde{Z} \in \frak{k}$ and $[Z, \tilde{Z}] = H$
$$
Z \cdot \tilde{Z} + \tilde{Z} \cdot Z = ([Z, \tilde{Z}] + \tilde{Z} \cdot Z) + \tilde{Z} \cdot Z = H + 2 \cdot \tilde{Z} \cdot Z = H + 2 \cdot \big( \underbrace{\tilde{Z}(Z + \tilde{Z})}_\text{acts by 0} - \tilde{Z}^2 \big) = H - 2 \tilde{Z}^2
$$
so
$$
\Omega_\frak{g} = \frac{1}{2} H^2 - H + 2 \tilde{Z}^2 + \frac{1}{2} \sum^{r-1}_{\ell = 1} \big( X_\ell \cdot Y_\ell + Y_\ell \cdot X_\ell \big) - \frac{1}{2} \sum^{r-1}_{\ell = 1} \big( \tilde{X}_\ell \cdot \tilde{Y}_\ell + \tilde{Y}_\ell \cdot \tilde{X}_\ell \big)+ \{\text{acts by zero}\}
$$
Rewriting $X_\ell \cdot Y_\ell + Y_\ell \cdot X_\ell =  2 \cdot X_\ell \cdot Y_\ell + [Y_\ell, X_\ell] $
$$
\Omega_\frak{g} = \frac{1}{2} H^2 - H + 2 \tilde{Z}^2 + \frac{1}{2} \sum^{r-1}_{\ell = 1} \big( 2 \cdot X_\ell \cdot Y_\ell + [Y_\ell, X_\ell] \big) - \frac{1}{2} \sum^{r-1}_{\ell = 1} \big( 2 \cdot \tilde{X}_\ell \cdot \tilde{Y}_\ell + [\tilde{Y}_\ell , \tilde{X}_\ell] \big) + \{\text{acts by zero}\}
$$
Using $[X_\ell,Y_\ell] = H$ and $[\tilde{X}_\ell,\tilde{Y}_\ell] =-H$ gives
\begin{align*}
\Omega_\frak{g} =& \frac{1}{2} H^2 - r \cdot H + 2 \tilde{Z}^2 + \sum^{r-1}_{\ell = 1} X_\ell \cdot Y_\ell - \sum^{r-1}_{\ell = 1} \tilde{X}_\ell \cdot \tilde{Y}_\ell + \{\text{acts by zero}\} \\
=& \frac{1}{2} H^2 - r \cdot H + 2 \tilde{Z}^2 + \sum^{r-1}_{\ell = 1} \big( X_\ell^2 - \underbrace{X_\ell \cdot (X_\ell - Y_\ell)}_\text{acts by 0} \big) + \sum^{r-1}_{\ell = 1}\big( \tilde{X}_\ell^2 - \underbrace{\tilde{X}_\ell \cdot (\tilde{X}_\ell + \tilde{Y}_\ell)}_\text{acts by 0} \big) \\&+ \{\text{acts by zero}\} \\
=&\frac{1}{2} H^2 - r \cdot H + 2 \tilde{Z}^2 + \sum^{r-1}_{\ell = 1} X_\ell^2 + \sum^{r-1}_{\ell = 1} \tilde{X}_\ell^2 + \{\text{acts by zero}\} \\
\end{align*}

%
%

\subsubsection{Casimir on $G/K$ in Iwasawa coordinates}\label{Casimir in Iwasawa coords on U(r,1)}
A model for complex hyperbolic space is
$$
G/K \approx U(r,1)/(U(r) \times U(1)) \approx NM/(M \cap K) \approx N \times A^+
$$
with coordinates $z, x \in N \approx  \mathbb{C}^{r-1} \times \mathbb{R}$, $y \in A^+ \approx  \mathbb{R}^\times$ from the Iwasawa decomposition giving
$$
G/K \ni n_{z,x} m_y K \longleftrightarrow (z,x,y) \in \mathbb{C}^{r-1} \times \mathbb{R} \times (0,\infty)
$$
$$
  z \in \mathbb{C}^{r-1} = \{ z : u + iv \in \mathbb{R}^{r-1} \oplus \; i \mathbb{R}^{r-1} \} \qquad x \in \mathbb{R} \qquad y \in \mathbb{R}^+ = (0,\infty)
$$
Let $e_\ell$ be the unit real vector so that
\begin{align*}
	e^{tH} =& m_{e^t} \\
	e^{tX_\ell} =& n_{t e_\ell, 0} = 
\begin{pmatrix}
1 & t \cdot e_\ell & -\frac{t^2}{2} \\
0 & 1_{r-1} & -t \cdot e^\ast_\ell \\
0 & 0 & 1
\end{pmatrix} \\
	e^{t\tilde{X}_\ell} =& n_{t (i e_\ell), 0} = 
\begin{pmatrix}
1 & t \cdot ie_\ell & -\frac{t^2}{2} \\
0 & 1_{r-1} & -t \cdot (ie_\ell)^\ast \\
0 & 0 & 1
\end{pmatrix}
\end{align*}
To determine the operators on $G/K$, let $G \ni g =n_{z_0,x_0} m_{y_o}$
\begin{align*}
H \cdot f(g) =& \frac{d}{dt}\bigg|_{t=0} f(g \cdot m_{e^t}) =  \frac{d}{dt}\bigg|_{t=0} f(n_{z_0,x_0} m_{y_0} \cdot m_{e^t}) = \frac{d}{dt}\bigg|_{t=0} f(n_{z_0,x_0} m_{y_0 e^t}) \\
=& y_0\frac{d}{dt}\bigg|_{t=0} f(n_{z_0, x_0} m_{y_0 e^t}) = y_0 \; \frac{\partial f}{\partial y}(n_{z_0,x_0} m_{y_0})
\end{align*}
so $H$ acts by $y \frac{\partial}{\partial y}$ and $H^2 = H(H) = y \frac{\partial}{\partial y}(y \frac{\partial}{\partial y}) = y^2 \frac{\partial^2}{\partial y^2} + y \frac{\partial}{\partial y}$.
We have convenient relations, similar to those useful for $O(r,1)$
$$
m_y \cdot e^{t X_\ell} = m_y \cdot n_{te_\ell,0} =  
\begin{pmatrix}
y & 0 & 0 \\
0 & 1_{r-1} & 0 \\
0 & 0 & \frac{1}{y}
\end{pmatrix}
\begin{pmatrix}
1 & t \cdot e_\ell & -\frac{t^2}{2} \\
0 \enspace & 1_{r-1} & -t \cdot e_\ell^\ast \\
0 & 0 & 1
\end{pmatrix}
=
\begin{pmatrix}
y & yt \cdot e_\ell & -\frac{yt^2}{2} \\
0 \enspace & 1_{r-1} & -t \cdot e_\ell^\ast \\
0 & 0 & \frac{1}{y}
\end{pmatrix}
$$
$$
n_{yte_\ell, 0} \cdot m_y =  
\begin{pmatrix}
1 & yt \cdot e_\ell & -\frac{y^2t^2}{2} \\
0 \enspace & 1_{r-1} & -yt \cdot e_\ell^\ast \\
0 & 0 & 1
\end{pmatrix}
\begin{pmatrix}
y & 0 & 0 \\
0 & 1_{r-1} & 0 \\
0 & 0 & \frac{1}{y}
\end{pmatrix}
=
\begin{pmatrix}
y & yt \cdot e_\ell & -\frac{yt^2}{2} \\
0 \enspace & 1_{r-1} & -t \cdot e_\ell^\ast \\
0 & 0 & \frac{1}{y}
\end{pmatrix}
$$
%
%
%
%
%
%
$$
m_y \cdot e^{t \tilde{X}_\ell} = m_y \cdot n_{t(ie_\ell),0} =  
\begin{pmatrix}
y & 0 & 0 \\
0 & 1_{r-1} & 0 \\
0 & 0 & \frac{1}{y}
\end{pmatrix}
\begin{pmatrix}
1 & t \cdot i e_\ell & -\frac{t^2}{2} \\
0 \enspace & 1_{r-1} & -t \cdot (i e_\ell)^\ast \\
0 & 0 & 1
\end{pmatrix}
=
\setlength{\arraycolsep}{2pt}
\begin{pmatrix}
y & yt \cdot (i e_\ell) & -\frac{yt^2}{2} \\
0 \enspace & 1_{r-1} & -t \cdot (i e_\ell)^\ast \\
0 & 0 & \frac{1}{y}
\end{pmatrix}
$$
$$
n_{yt(i e_\ell), 0} \cdot m_y =  
\begin{pmatrix}
1 & yt \cdot (i e_\ell) & -\frac{y^2t^2}{2} \\
0 \enspace & 1_{r-1} & -yt \cdot (i e_\ell)^\ast \\
0 & 0 & 1
\end{pmatrix}
\begin{pmatrix}
y & 0 & 0 \\
0 & 1_{r-1} & 0 \\
0 & 0 & \frac{1}{y}
\end{pmatrix}
=
\begin{pmatrix}
y & yt \cdot (i e_\ell) & -\frac{yt^2}{2} \\
0 \enspace & 1_{r-1} & -t \cdot (i e_\ell)^\ast \\
0 & 0 & \frac{1}{y}
\end{pmatrix}
$$
so that
$$
m_y \cdot e^{t X_\ell} = m_y \cdot n_{te_\ell,0} =  n_{yte_\ell, 0} \cdot m_y
$$
and
$$
m_y \cdot e^{t \tilde{X}_\ell} = m_y \cdot n_{t (i e_\ell),0} =  n_{yt(i e_\ell), 0} \cdot m_y
$$
which give
$$
n_{x,u} \cdot m_y \cdot e^{t X_\ell} =n_{x,u} \cdot m_y \cdot n_{te_\ell, 0} = n_{x, u} \cdot n_{yt e_\ell, 0} \cdot m_y = n_{x + yt e_\ell, u} \cdot m_y
$$
and
$$
n_{x, u} \cdot m_y \cdot e^{t \tilde{X}_\ell} =n_{x,u} \cdot m_y \cdot n_{t (i e_\ell), 0} = n_{x, u} \cdot n_{yt(i e_\ell), 0} \cdot m_y = n_{x + yt( i e_\ell), u} \cdot m_y
$$
Use real coordinates $z = u + iv$ for $g$ to determine the operator $X_i$
\begin{align*}
X_\ell \cdot f(g) =& \frac{d}{dt}\bigg|_{t=0} f(g \cdot e^{tX_\ell}) = \frac{d}{dt}\bigg|_{t=0} f(g \cdot n_{te_\ell, 0}) =  \frac{d}{dt}\bigg|_{t=0} f(n_{z_0, x_0} m_{y_0} \cdot n_{te_\ell, 0}) \\
 =& \frac{d}{dt}\bigg|_{t=0} f(n_{z_0 + y_0 te_\ell, x_0} m_{y_0}) = y_0 \; \frac{\partial f}{\partial u_\ell}(n_{z_0, x_0} m_{y_0})\end{align*}
giving
$$
X_\ell = y \frac{\partial}{\partial u_\ell} \text{  and  } X_\ell^2 = X_\ell(X_\ell) = y \frac{\partial}{\partial u_\ell} \big(y \frac{\partial}{\partial u_\ell} \big) = \bigg(y \frac{\partial}{\partial u_\ell}\bigg)^2 = y^2 \frac{\partial^2}{\partial u_\ell^2}
$$
Analogously, $\tilde{X}_\ell$ is the imaginary direction $v_\ell \approx i \cdot e_\ell$
\begin{align*}
\tilde{X}_\ell \cdot f(g) =& \frac{d}{dt}\bigg|_{t=0} f(g \cdot e^{t\tilde{X}_\ell}) = \frac{d}{dt}\bigg|_{t=0} f(g \cdot n_{t i e_\ell, 0}) =  \frac{d}{dt}\bigg|_{t=0} f(n_{z_0, x_0} m_{y_0} \cdot n_{t i e_\ell, 0}) \\
 =& \frac{d}{dt}\bigg|_{t=0} f(n_{z_0 + y_0 t i e_\ell, x_0} m_{y_0}) = y_0 \; \frac{\partial f}{\partial v_\ell}(n_{z_0, x_0} m_{y_0})
\end{align*}
so correspondingly
$$
\tilde{X}_\ell = y \frac{\partial}{\partial v_\ell} \text{  and  } \tilde{X}_\ell^2 = \tilde{X}_\ell(\tilde{X}_\ell) = y \frac{\partial}{\partial v_\ell} \big( y \frac{\partial}{\partial v_\ell} \big) = \bigg( y \frac{\partial}{\partial v_\ell}\bigg)^2 = y^2 \frac{\partial^2}{\partial v_\ell^2}
$$
Last, for the $\tilde{Z}$ term:
$$
e^{t \tilde{Z}} = n_{0,t} =
\begin{pmatrix}
1 & 0 & i t \\
0 & 1_{r-1} & 0 \\
0 & 0 & 1
\end{pmatrix} \\
$$
\begin{align*}
\tilde{Z} \cdot f(g) =& \frac{d}{dt}\bigg|_{t=0} f(g \cdot e^{t\tilde{Z}}) = \frac{d}{dt}\bigg|_{t=0} f(g \cdot n_{0,t}) =  \frac{d}{dt}\bigg|_{t=0} f(n_{z_0, x_0} m_{y_0} \cdot n_{0, t}) \\
=& 
\frac{d}{dt}\bigg|_{t=0} f \Bigg(
\begin{pmatrix}
1 & z_0 & -\frac{|z_0|^2}{2} + i x_0 \\
0 \enspace & 1_{r-1} & z_0^\ast \\
0 & 0 & 1
\end{pmatrix}
\begin{pmatrix}
y_0 & 0 & 0 \\
0 & 1_{r-1} & 0 \\
0 & 0 & \frac{1}{y_0}
\end{pmatrix}
\begin{pmatrix}
1 & 0 &  i t \\
0 & 1_{r-1} & 0 \\
0 & 0 & 1
\end{pmatrix}
\Bigg) \\
=& 
\frac{d}{dt}\bigg|_{t=0} f \Bigg(
\begin{pmatrix}
1 & z_0 & -\frac{|z_0|^2}{2} + i x_0 \\
0 \enspace & 1_{r-1} & z_0^\ast \\
0 & 0 & 1
\end{pmatrix}
\begin{pmatrix}
y_0 & 0 & i y_0 t \\
0 & 1_{r-1} & 0 \\
0 & 0 & \frac{1}{y_0}
\end{pmatrix}
\Bigg) \\
=& 
\frac{d}{dt}\bigg|_{t=0} f \Bigg(
\begin{pmatrix}
y_0 & z_0 & -\frac{|z_0|^2}{2 y_0} + i (\frac{x_0}{y_0} + y_0 t) \\
0 \enspace & 1_{r-1} & \frac{1}{y_0}z_0^\ast \\
0 & 0 & \frac{1}{y_0}
\end{pmatrix}
\Bigg) \\
=&  
\frac{d}{dt}\bigg|_{t=0} f \Bigg(
\begin{pmatrix}
1 & z_0 & -\frac{|z_0|^2}{2} + i (x_0 + y_0^2 t) \\
0 \enspace & 1_{r-1} & z_0^\ast \\
0 & 0 & 1
\end{pmatrix}
\begin{pmatrix}
y_0 & 0 & 0 \\
0 & 1_{r-1} & 0 \\
0 & 0 & \frac{1}{y_0}
\end{pmatrix}
\Bigg) \\
=& 
\frac{d}{dt}\bigg|_{t=0} f (z_0, (x_0 + y_0^2 t), y_0) = y_0^2 \frac{\partial f}{\partial x}(z_0, x_0, y_0)
\end{align*}
so correspondingly
$$
\tilde{Z} = y^2 \frac{\partial}{\partial x} \text{  and  } \tilde{Z}^2 = \tilde{Z} \circ \tilde{Z} = y^2 \frac{\partial}{\partial x} \big( y^2 \frac{\partial}{\partial x} \big) = \bigg( y^2 \frac{\partial}{\partial x}\bigg)^2 = y^4 \frac{\partial^2}{\partial x^2}
$$

\noindent
Substituting appropriately in $\Omega$ on $G/K$
\begin{align*}
\Omega_\frak{g} =& \frac{1}{2} H^2 - r \cdot H + 2 \tilde{Z}^2 + \sum^{r-1}_{\ell = 1} X_\ell^2 + \sum^{r-1}_{\ell = 1} \tilde{X}_\ell^2 + \{\text{acts by zero}\}  \\
=& \frac{1}{2} \big( y^2 \frac{\partial^2}{\partial y^2} + y \frac{\partial}{\partial y} \big) - r \cdot \big(y \frac{\partial}{\partial y} \big) + 2 y^4 \frac{\partial^2}{\partial x^2} + y^2 \sum^{r-1}_{\ell = 1} \frac{\partial^2}{\partial u_\ell^2} + y^2 \sum^{r-1}_{\ell = 1} \frac{\partial^2}{\partial v_\ell^2}  + \{\text{acts by zero}\} \\
=& \frac{1}{2} y^2 \frac{\partial^2}{\partial y^2} - \frac{2r-1}{2} y \frac{\partial}{\partial y} +  2y^4 \frac{\partial^2}{\partial x^2} + y^2 \big( \sum^{r-1}_{\ell = 1} \frac{\partial^2}{\partial u_\ell^2} + \sum^{r-1}_{\ell = 1} \frac{\partial^2}{\partial v_\ell^2}  \big) + \{\text{acts by zero}\} \\
=& y^2 \bigg( \sum^{r-1}_{\ell = 1} \frac{\partial^2}{\partial u_\ell^2} + \sum^{r-1}_{\ell = 1} \frac{\partial^2}{\partial v_\ell^2} + \frac{1}{2} \frac{\partial^2}{\partial y^2}+ 3y^2 \frac{\partial^2}{\partial x^2} \bigg) - \frac{2r-1}{2} y \frac{\partial}{\partial y} + \{\text{acts by zero}\} \\
=& y^2 \bigg( \Delta_u + \Delta_v + \frac{1}{2} \frac{\partial^2}{\partial y^2}+ 2y^2 \frac{\partial^2}{\partial x^2} \bigg) - \frac{2r-1}{2} y \frac{\partial}{\partial y} + \{\text{acts by zero}\}
\end{align*}

To get remove the factor of $\frac{1}{2}$ in the ``$H$'' term, we could renormalize coordinates as
$$
G/K \approx U(r,1)/(U(r) \times U(1)) \approx NM/(M \cap K) \approx N \times A^+
$$
with coordinates $z, x \in N \approx  \mathbb{C}^{r-1} \times \mathbb{R}$, $y \in A^+ \approx  \mathbb{R}^\times$ from the Iwasawa decomposition giving
$$
G/K \ni n_{z,x} m_y K \longleftrightarrow (z,x,y) \in \mathbb{C}^{r-1} \times \mathbb{R} \times (0,\infty)
$$
$$
  z \in \mathbb{C}^{r-1} = \{ z : u + iv \in \mathbb{R}^{r-1} \oplus \; i \mathbb{R}^{r-1} \} \qquad x \in \mathbb{R} \qquad y \in \mathbb{R}^+ = (0,\infty)
$$
$$
\mathbb{C}^{r-1} \times \mathbb{R} \times (0, +\infty) \ni (x,y) \to n_{z\cdot \sqrt{2},\frac{x}{4}}m_{y} =
\renewcommand\arraystretch{1.5}
\begin{pmatrix}
1 & z\cdot \sqrt{2} & -|z|^2 + \frac{ix}{4} \\
0 & 1_{r-1} & -z^\ast \cdot \sqrt{2} \\
0 & 0 & 1
\end{pmatrix}
\begin{pmatrix}
y & 0 & 0 \\
0 & 1_{r-1} & 0 \\
0 & 0 & \frac{1}{y}
\end{pmatrix}
$$
giving, on $G/K$
$$
2 \cdot \Omega = y^2 \bigg( \Delta_u + \Delta_v + \frac{\partial^2}{\partial y^2}+ y^2 \frac{\partial^2}{\partial x^2} \bigg) - (2r-1) y \frac{\partial}{\partial y}
$$
Ignoring the factor of $2$ and collecting terms
$$
\Omega = y^2 \bigg( \Delta_u + \Delta_v + \frac{\partial^2}{\partial y^2}+ y^2 \frac{\partial^2}{\partial x^2} \bigg) - (2r-1) y \frac{\partial}{\partial y}
$$
%
%


%
%

\subsection{The Casimir operator for $Sp^\ast(r,1)$}\label{Casimir Operator for Sp*(r,1)}

%
%

\subsubsection{Iwasawa coordinates on $N$, $M$ and $P \cap K$}\label{Sp*(r,1) Iwasawa coords on N,M, PcapK}

%
%
%
%
$$G =  \{ g \in GL_{r+1}(\mathbb{H}) \ | \ g^\ast S g = S\} \approx Sp^\ast (r,1)$$
where the form $S$ will be one of
$$
S = 
\renewcommand\arraystretch{1.25}
\begin{pmatrix}
0 & 0 & 1 \\
0 & 1_{r-1} & 0 \\
1 & 0 & 0
\end{pmatrix}
\qquad
\renewcommand\arraystretch{1.75}
S^\prime = 
\begin{pmatrix}
1_r & 0 \\
0 & -1	
\end{pmatrix}
$$ 
A Cayley element is
$$ 
\renewcommand\arraystretch{1.25}
C = 
\begin{pmatrix} 
\sqrt{\frac{1}{2}}  &      0          & -\sqrt{\frac{1}{2}} \\ 
           0               &   1_{r-1}    & 0                           \\
\sqrt{\frac{1}{2}}   &        0        & \sqrt{\frac{1}{2}}
\end{pmatrix}
\qquad
$$ 
We have
$$
C S^\prime C^{-1} = C S^\prime C^\ast = S
$$
Thus, the two forms $S$ and $S^\prime$ give isomorphic isometry groups. The $S^\prime$ version is more convenient for identifying the maximal compact subgroup of $G$ but the $S$ version more clearly reveals the minimal parabolic $P$ as the stabilizer of a maximal isotropic flag as follows.

 Let $e_1, \ldots e_{r+1}$ be the standard basis of $\mathbb{H}^{r+1}$ where $e_i$ is a column vector with a $1$ in the $i^\text{th}$ position and $0$'s elsewhere. With $S$ as above, we see that $e_1^\ast S e_1 = 0 = e_{r+1}^\ast S e_{r+1}$ while $e_i^\ast S e_i = 1$ for $i \ne 1 \text{ or } r+1$. The two one-dimensional subspaces spanned by $e_1$ and $e_{r+1}$ are maximal isotropic. Additionally, since $e_{r+1}^\ast S e_1 = 1 = e_{1}^\ast S e_{r+1}$, $e_1$ and $e_{r+1}$ are a hyperbolic pair. We take $P$ as the subgroup of $G$ that fixes the (maximal isotropic) space spanned by $e_1$. 

Any linear map that sends $\langle e_1 \rangle$ to itself is automatically an isometry on the subspace spanned by $e_1$, since $e_1$ is isotropic, so by Witt's theorem can be extended to an isometry of $S$, and thus lies in $G$.  Thus $P$ consists of matrices of the form
$$
\begin{pmatrix}
\ast & \ast & \ast \\
0 & \ast & \ast \\
0 & 0 & \ast
\end{pmatrix}
$$
with appropriate relations, and with an $(r-1) \times (r-1)$ block in the middle.
A maximal compact, visible with the diagonal version of the form,  is
$$
K = Sp^\ast(r) \times Sp^\ast(1) = G \cap Sp^\ast(r+1)  \qquad \text{(standard }Sp^\ast(r+1)\text{)}
$$
An Iwasawa decomposition is $G = PK$, with $P$ a minimal parabolic (i.e., the stabilizer of a maximal isotropic flag as above).
The maximal compact subgroup of $G$ will have non-trivial intersection with $P$, contained in its Levi component M, given by
$$P \cap K = P \cap Sp^\ast(r+1) = M \cap Sp^\ast(r+1) \approx Sp^\ast(r-1) \times Sp^\ast(1)$$
which is visible as
$$
\begin{pmatrix}
\alpha & 0 & 0 \\
0 & h & 0 \\
0 & 0 & \bar{\alpha}^{-1}
\end{pmatrix}
$$
with $h \in Sp^\ast(r-1)$ and $|\alpha| = 1$, $\alpha \in \mathbb{H}^\times$. 
$N$ is of the shape
$$
\begin{pmatrix}
1 & u & a \\
0 & 1_{r-1} & v^\top \\
0 & 0 & 1
\end{pmatrix}
$$
where $u,v \in \mathbb{H}^{r-1}$ and $a \in \mathbb{H}$. To determine $u, v$ and $a$, for $T \in N$
$$
T^\ast S T= 
\renewcommand\arraystretch{1.25}
\begin{pmatrix}
1 & 0 & 0 \\
u^\ast & 1_{r-1} & 0 \\
\bar{a} & \bar{v} & 1
\end{pmatrix}
\begin{pmatrix}
0 & 0 & 1 \\
0 & 1_{r-1} & 0 \\
1 & 0 & 0
\end{pmatrix}
\begin{pmatrix}
1 & u & a \\
0 & 1_{r-1} & v^\top \\
0 & 0 & 1
\end{pmatrix}
= S =
\begin{pmatrix}
0 & 0 & 1 \\
0 & 1_{r-1} & 0 \\
1 & 0 & 0
\end{pmatrix}
$$
$$
= 
\renewcommand\arraystretch{1.25}
\begin{pmatrix}
0 & 0 & 1 \\
0 & 1_{r-1} & u^\ast \\
1 & \bar{v} & \bar{a}
\end{pmatrix}
\begin{pmatrix}
1 & u & a \\
0 & 1_{r-1} & v^\top \\
0 & 0 & 1
\end{pmatrix}
=
\begin{pmatrix}
0 & 0 & 1 \\
0 & 1_{r-1} & u^\ast + v^\top \\
1 & u + \bar{v} & |v|^2 + 2\text{Re}(a)
\end{pmatrix}
= S =
\begin{pmatrix}
0 & 0 & 1 \\
0 & 1_{r-1} & 0 \\
1 & 0 & 0
\end{pmatrix}
$$
giving relations $u = -\bar{v}$ and $\text{Re}(a) = -\frac{|v|^2 }{2} = -\frac{|u|^2 }{2}$ so that elements of the standard unipotent radical $N$ of $P$ are
$$
N \enspace = \enspace \{ n_{z,p,q,r} \enspace = \enspace 
\renewcommand\arraystretch{1.5}
\begin{pmatrix}
1 & z & -\frac{1}{2}|z|^2 + ip + jq + kr \\
0 & 1_{r-1} & -z^\ast \\
0 & 0 & 1
\end{pmatrix}
\}
$$
where
$$
z = x+ iw + ju + kv \in \mathbb{R}^{r-1} \oplus \; i \mathbb{R}^{r-1} \oplus \; j \mathbb{R}^{r-1} \oplus \; k \mathbb{R}^{r-1} \approx \mathbb{H}^{r-1} \quad  p,q,r \in \mathbb{R}
$$
The standard Levi component is
$$
M =  \{ m_\alpha = 
\renewcommand\arraystretch{1.5}
\begin{pmatrix}
\alpha & 0 & 0 \\
0 & h & 0 \\
0 & 0 & \bar{\alpha}^{-1}
\end{pmatrix}
\ : \text{ for } \alpha \in \mathbb{H}^\times \text{ and } h \in Sp^\ast(r-1)\}
$$
and the split component of the Levi component is
$$
A^+  =  \{ m_y = 
\renewcommand\arraystretch{1.5}
\begin{pmatrix}
y & 0 & 0 \\
0 & 1_{r-1} & 0 \\
0 & 0 & \frac{1}{y}
\end{pmatrix}
\ : \text{ for } y > 0)\}
$$
The Lie algebra of $G = Sp^\ast(r,1)$ is
$$
\frak{g} = \frak{sp}^\ast(r,1) = \{ \gamma \in \frak{gl}_{r+1}(\mathbb{H}) \enspace : \enspace \textrm{exp}^{t\gamma} \in Sp^\ast(r,1), \; \forall t \in \mathbb{R} \} = \{ \gamma \; : \; \gamma^\ast = - S \gamma S^{-1} \}
$$
Since $(e^{\gamma})^\ast S(e^{\gamma}) = S$
$$
\frac{d}{dt}\bigg|_{t=0}(e^{t\gamma})^\ast S(e^{t\gamma}) = 0 = \gamma^\ast S + S \gamma
$$
This implies $\gamma^\ast = -S \gamma S^{-1}$ which also is sufficient since exponentiation respects transpose and conjugation.
%
%
%
%
%
$$
\renewcommand\arraystretch{1.25}
\frak{sp}^\ast(r,1) = \Bigg\{ 
\begin{pmatrix}
a & b &  c \\
-d^\ast & e & -b^\ast \\
g & d & -\bar{a}
\end{pmatrix}
\Bigg|
\enspace a \in \mathbb{H},  \;g, c \in i\mathbb{R} \oplus j\mathbb{R} \oplus k\mathbb{R}, \; b, d \in \mathbb{H}^{r-1}, \; e \in \frak{sp}^\ast(r-1) \Bigg\}
$$
Specify a real basis for $\frak{sp}^\ast(r,1)$ as:
$$
\renewcommand\arraystretch{1.5}
H =
\begin{pmatrix}
1 & 0 & 0 \\
0 & 0_{r-1} & 0 \\
0 & 0 & -1
\end{pmatrix}
H_i =
\begin{pmatrix}
i & 0 & 0 \\
0 & 0_{r-1} & 0 \\
0 & 0 & i
\end{pmatrix}
H_j =
\begin{pmatrix}
j & 0 & 0 \\
0 & 0_{r-1} & 0 \\
0 & 0 & j
\end{pmatrix}
H_k =
\begin{pmatrix}
k & 0 & 0 \\
0 & 0_{r-1} & 0 \\
0 & 0 & k
\end{pmatrix}
$$
$$
\renewcommand\arraystretch{1.5}
X_\ell =
{\scriptstyle
\begin{pmatrix}
0 & e_\ell & 0 \\
0 & 0_{r-1} & -e_\ell^\ast \\
0 & 0 & 0
\end{pmatrix}}
X_{i,\ell} =
{\scriptstyle
\begin{pmatrix}
0 & ie_\ell & 0 \\
0 & 0_{r-1} & ie_\ell^\ast \\
0 & 0 & 0
\end{pmatrix}}
X_{j,\ell} =
{\scriptstyle
\begin{pmatrix}
0 & je_\ell & 0 \\
0 & 0_{r-1} & je_\ell^\ast \\
0 & 0 & 0
\end{pmatrix}}
X_{k,\ell} =
{\scriptstyle
\begin{pmatrix}
0 & ke_\ell & 0 \\
0 & 0_{r-1} & ke_\ell^\ast \\
0 & 0 & 0
\end{pmatrix}}
$$
$$
\renewcommand\arraystretch{1.5}
Y_\ell =
{\scriptstyle
\begin{pmatrix}
0 & 0 & 0 \\
e_\ell^\ast & 0_{r-1} & 0 \\
0 & -e_\ell & 0 \\
\end{pmatrix}}
Y_{i,\ell} =
{\scriptstyle
\begin{pmatrix}
0 & 0 & 0 \\
ie_\ell^\ast & 0_{r-1} & 0 \\
0 & ie_\ell & 0 \\
\end{pmatrix}}
Y_{j,\ell} =
{\scriptstyle
\begin{pmatrix}
0 & 0 & 0 \\
je_\ell^\ast & 0_{r-1} & 0 \\
0 & je_\ell & 0 \\
\end{pmatrix}}
Y_{k,\ell} =
{\scriptstyle
\begin{pmatrix}
0 & 0 & 0 \\
ke_\ell^\ast & 0_{r-1} & 0 \\
0 & ke_\ell & 0 \\
\end{pmatrix}}
$$
$$
\renewcommand\arraystretch{1.5}
Z_{i^-} =
{\scriptstyle
\begin{pmatrix}
0 & 0 & 0 \\
0 & 0_{r-1} & 0 \\
i & 0 & 0
\end{pmatrix}}
Z_{i^+} =
{\scriptstyle
\begin{pmatrix}
0 & 0 & i \\
0 & 0_{r-1} & 0 \\
0 & 0 & 0
\end{pmatrix}}
Z_{j^-} =
{\scriptstyle
\begin{pmatrix}
0 & 0 & 0 \\
0 & 0_{r-1} & 0 \\
j & 0 & 0
\end{pmatrix}}
Z_{j^+} =
{\scriptstyle
\begin{pmatrix}
0 & 0 & j \\
0 & 0_{r-1} & 0 \\
0 & 0 & 0
\end{pmatrix}}
$$
$$
\renewcommand\arraystretch{1.5}
Z_{k^-} =
{\scriptstyle
\begin{pmatrix}
0 & 0 & 0 \\
0 & 0_{r-1} & 0 \\
k & 0 & 0
\end{pmatrix}}
Z_{k^+} =
{\scriptstyle
\begin{pmatrix}
0 & 0 & k \\
0 & 0_{r-1} & 0 \\
0 & 0 & 0
\end{pmatrix}}
\theta_\ell =
\begin{pmatrix}
0 & 0 & 0 \\
0 & \theta_\ell & 0 \\
0 & 0 & 0
\end{pmatrix}
\qquad
\{ \theta_\ell \} \text{ basis of } \frak{sp}^\ast(r-1)
$$
Where we mildly abuse notation by using $\theta_\ell$ both as the inner and outer elements.
%
%
%

\subsubsection{Dual basis for $\frak{sp}^\ast(r,1)$}\label{Dual basis for sp*(r,1)}

Using a trace pairing, which is a multiple of Killing, association of dual elements is as follows ($\prime$ denotes dual element):

$$
H^\prime = \frac{1}{2} H \quad H_i^\prime = -\frac{1}{2} H_i \quad H_j^\prime = -\frac{1}{2} H_j \quad H_k^\prime = -\frac{1}{2} H_k
$$
$$
X^\prime_\ell = \frac{1}{2}\cdot Y_\ell = \frac{1}{2}\cdot X^\ast_\ell \qquad X^\prime_{i,\ell} = -\frac{1}{2}\cdot Y_{i,\ell} = -\frac{1}{2}\cdot X^\ast_{i,\ell}
$$
$$
X^\prime_{j,\ell} = -\frac{1}{2}\cdot Y_{j,\ell} = -\frac{1}{2}\cdot X^\ast_{j,\ell} \qquad X^\prime_{k,\ell} = -\frac{1}{2}\cdot Y_{k,\ell} = -\frac{1}{2}\cdot X^\ast_{k,\ell}
$$
$$
Z^\prime_{i^-} = -Z^\prime_{i^+} \qquad Z^\prime_{j^-} = -Z^\prime_{j^+} \qquad Z^\prime_{k^-} = -Z^\prime_{k^+} \qquad \theta^\prime_\ell = - \theta_\ell
$$
Conveniently, we have
$$
H = [X_\ell,Y_\ell] = [Z_{i^-}, Z_{i^+}] = [Z_{j^-}, Z_{j^+}] = [Z_{k^-}, Z_{k^+}]
$$
$$
-H = [X_{i,\ell},Y_{i,\ell}] = [X_{j,\ell},Y_{j,\ell}] = [X_{k,\ell},Y_{k,\ell}]
$$
The skew-Hermitian symmetry of elements of $\frak{sp}(r+1)$ and $\frak{sp}(r-1) \oplus \frak{sp}(1) = G \cap \frak{u}(r+1) = \frak{k}$ is used in conjunction with the fact that $\frak{k}$ acts trivially on functions on $G/K$, which shows  that the linear combinations of basis elements conveniently fall in $\frak{k}$. 
$$
H_i, H_j, H_k \quad X_\ell - Y_\ell \quad X_{i,\ell} + Y_{i,\ell}, X_{j,\ell} + Y_{j,\ell}, X_{k,\ell} + Y_{k,\ell} \quad Z_{i^-} + Z_{i^+}, Z_{j^-} + Z_{j^+}, Z_{k^-} + Z_{k^+}
$$
\subsubsection{Rearrangement of Casimir operator on $G/K$}\label{Simplifying  Casimir G/K for Sp*(r,1)}
\begin{align*}
\Omega_\frak{g} =& H \cdot H^\prime + H_i \cdot H_i^\prime + H_j \cdot H_j^\prime + H_k \cdot H_k^\prime \\
 &+ Z_{i^-} \cdot Z^\prime_{i^-} + Z_{i^+} \cdot Z^\prime_{i^+} + Z_{j^-} \cdot Z^\prime_{j^-} + Z_{j^+} \cdot Z^\prime_{j^+} + Z_{k^-} \cdot Z^\prime_{k^-} + Z_{k^+} \cdot Z^\prime_{k^+} \\
&+ \sum^{r-1}_{\ell = 1} X_\ell \cdot X_\ell^\prime + \sum^{r-1}_{\ell = 1} Y_\ell \cdot Y_\ell^\prime + \sum^{r-1}_{\ell = 1} X_{i,\ell} \cdot X_{i,\ell}^\prime + \sum^{r-1}_{\ell = 1} Y_{i,\ell} \cdot Y_{i,\ell}^\prime + \sum^{r-1}_{\ell = 1} X_{j,\ell} \cdot X_{j,\ell}^\prime \\
 &+ \sum^{r-1}_{\ell = 1} Y_{j,\ell} \cdot Y_{j,\ell}^\prime + \sum^{r-1}_{\ell = 1} X_{k,\ell} \cdot X_{k,\ell}^\prime + \sum^{r-1}_{\ell = 1} Y_{k,\ell} \cdot Y_{k,\ell}^\prime + \sum^{r-1}_{j = 1} \theta_j \cdot \theta_j^\prime
\end{align*}
Ignoring terms involving  $H_i$, $H_j$, $H_k$ and the $\theta_j$, which act by $0$ on $G/K$, substitution of the expressions for dual elements in the above gives
\begin{align*}
\Omega_\frak{g} =& \frac{1}{2}H^2 - (Z_{i^-} \cdot Z_{i^+} + Z_{i^+} \cdot Z_{i^-} + Z_{j^-} \cdot Z_{j^+} + Z_{j^+} \cdot Z_{j^-} + Z_{k^-} \cdot Z_{k^+} + Z_{k^+} \cdot Z_{k^-}) \\
&+ \frac{1}{2}\sum^{r-1}_{\ell = 1} (X_\ell \cdot Y_\ell +  Y_\ell \cdot X_\ell) - \frac{1}{2}\sum^{r-1}_{\ell = 1} (X_{i,\ell} \cdot Y_{i,\ell}+ Y_{i,\ell} \cdot X_{i,\ell}) \\
&- \frac{1}{2}\sum^{r-1}_{\ell = 1} (X_{j,\ell} \cdot Y_{j,\ell} + Y_{j,\ell} \cdot X_{j,\ell}) - \frac{1}{2}\sum^{r-1}_{\ell = 1} (X_{k,\ell} \cdot Y_{k,\ell} + Y_{k,\ell} \cdot X_{k,\ell}) + (\text{acts by zero})
\end{align*}
Similar to the case of $\frak{u}(r,1)$, we have the equivalences
\begin{itemize}
	\item[] $Z_{i^-} \cdot Z_{i^+} + Z_{i^+} \cdot Z_{i^-} = H - 2 Z_{i^+}^2$ 
	\item[] $Z_{j^-} \cdot Z_{j^+} + Z_{j^+} \cdot Z_{j^-} = H - 2 Z_{j^+}^2$ 
	\item[] $Z_{k^-} \cdot Z_{k^+} + Z_{k^+} \cdot Z_{k^-} = H - 2 Z_{k^+}^2$
\end{itemize}
Substituting these in the previous expression
\begin{align*}
\Omega_\frak{g} =& \frac{1}{2}H^2 - 3H + 2 Z_{i^+}^2 + 2 Z_{j^+}^2 + 2 Z_{k^+}^2 \\
&+ \frac{1}{2}\sum^{r-1}_{\ell = 1} (X_\ell \cdot Y_\ell +  Y_\ell \cdot X_\ell) \\
&- \frac{1}{2}\sum^{r-1}_{\ell = 1} (X_{i,\ell} \cdot Y_{i,\ell}+ Y_{i,\ell} \cdot X_{i,\ell}) \\
&- \frac{1}{2}\sum^{r-1}_{\ell = 1} (X_{j,\ell} \cdot Y_{j,\ell} + Y_{j,\ell} \cdot X_{j,\ell}) \\
&- \frac{1}{2}\sum^{r-1}_{\ell = 1} (X_{k,\ell} \cdot Y_{k,\ell} + Y_{k,\ell} \cdot X_{k,\ell}) + \{\text{acts by zero}\}
\end{align*}
\begin{align*}
\Omega_\frak{g} =& \frac{1}{2}H^2 - 3H + 2 Z_{i^+}^2 + 2 Z_{j^+}^2 + 2 Z_{k^+}^2  \\
&+ \frac{1}{2}\sum^{r-1}_{\ell = 1} (2 \cdot X_\ell \cdot Y_\ell +  [Y_\ell , X_\ell]) \\
&- \frac{1}{2}\sum^{r-1}_{\ell = 1} (2 \cdot X_{i,\ell} \cdot Y_{i,\ell}+ [Y_{i,\ell} , X_{i,\ell}]) \\
&- \frac{1}{2}\sum^{r-1}_{\ell = 1} (2 \cdot X_{j,\ell} \cdot Y_{j,\ell} + [Y_{j,\ell} , X_{j,\ell}]) \\
&- \frac{1}{2}\sum^{r-1}_{\ell = 1} (2 \cdot X_{k,\ell} \cdot Y_{k,\ell} + [Y_{k,\ell} , X_{k,\ell}]) + \{\text{acts by zero}\}
\end{align*}
using the bracket relationships above and continuing
\begin{align*}
\Omega_\frak{g} =& \frac{1}{2}H^2  - (2r+1) H + 2 Z_{i^+}^2 + 2 Z_{j^+}^2 + 2 Z_{k^+}^2\\
&+ \sum^{r-1}_{\ell = 1} X_\ell \cdot Y_\ell - \sum^{r-1}_{\ell = 1} X_{i,\ell} \cdot Y_{i,\ell} \\
&- \sum^{r-1}_{\ell = 1} X_{j,\ell} \cdot Y_{j,\ell} - \sum^{r-1}_{\ell = 1} X_{k,\ell} \cdot Y_{k,\ell} + \{\text{acts by zero}\}
\end{align*}
\begin{align*}
\Omega_\frak{g} =& \frac{1}{2}H^2  - (2r+1) H + 2 Z_{i^+}^2 + 2 Z_{j^+}^2 + 2 Z_{k^+}^2\\
&+ \sum^{r-1}_{\ell = 1} (X_\ell^2 - \underbrace{X_\ell(X_\ell - Y_\ell)}_\text{acts by 0} ) \\
&+ \sum^{r-1}_{\ell = 1} (X_{i,\ell}^2  - \underbrace{X_{i,\ell}(X_{i,\ell} + Y_{i,\ell})}_\text{acts by 0} ) \\
&+ \sum^{r-1}_{\ell = 1} (X_{j,\ell}^2 - \underbrace{X_{j,\ell}(X_{j,\ell} + Y_{j,\ell})}_\text{acts by 0} ) \\
&+ \sum^{r-1}_{\ell = 1} (X_{k,\ell}^2 - \underbrace{X_{k,\ell}(X_{k,\ell} + Y_{k,\ell})}_\text{acts by 0} ) + \{\text{acts by zero}\}
\end{align*}
$$
\Omega_\frak{g} = \frac{1}{2}H^2  - (2r+1) H + 2 Z_{i^+}^2 + 2 Z_{j^+}^2 + 2 Z_{k^+}^2 + \sum^{r-1}_{\ell = 1} X_\ell^2 + \sum^{r-1}_{\ell = 1} X_{i,\ell}^2 + \sum^{r-1}_{\ell = 1} X_{j,\ell}^2 + \sum^{r-1}_{\ell = 1} X_{k,\ell}^2 + \{\text{acts by zero}\}
$$
\subsubsection{Casimir for $G/K$ in Iwasawa coordinates}\label{Casimir in Iwasawa coords on Sp*(r,1)}
A model for quaternionic hyperbolic space is
$$
G/K \approx Sp^\ast (r,1)/(Sp^\ast (r) \times Sp^\ast (1)) \approx NM/(M \cap K) \approx N \times A^+
$$
with coordinates $z, p, q, r \in N \approx  \mathbb{H}^{r-1} \times \mathbb{R}^3$, $y \in A^+ \approx  \mathbb{R}^\times$ from the Iwasawa decomposition giving
$$
G/K \ni n_{z,q,q^\prime,q^{\prime\prime}} m_y K \longleftrightarrow (z,q,q^\prime,q^{\prime\prime},y) \in \mathbb{H}^{r-1} \times \mathbb{R}^3 \times (0,\infty)
$$
$$
z = x+ iw + ju + kv \in \mathbb{R}^{r-1} \oplus \; i \mathbb{R}^{r-1} \oplus \; j \mathbb{R}^{r-1} \oplus \; k \mathbb{R}^{r-1} \approx \mathbb{H}^{r-1} \quad  q,q^\prime,q^{\prime\prime} \in \mathbb{R}
$$
Let $e_\ell$ be the unit real vector so that
\begin{align*}
	e^{tH} =& m_{e^t} \\
	e^{tX_\ell} =& n_{t e_\ell, 0,0,0} = 
\begin{pmatrix}
1 & t \cdot e_\ell & -t^2 \\
0 & 1_{r-1} & -t \cdot e^\ast_\ell \\
0 & 0 & 1
\end{pmatrix} \\
	e^{t X_{i,\ell}} =& n_{t i e_\ell, 0,0,0} = 
\begin{pmatrix}
1 & t \cdot ie_\ell & -t^2 \\
0 & 1_{r-1} & -t \cdot ie_\ell^\ast \\
0 & 0 & 1
\end{pmatrix} \qquad (\text{similarly for } j \text{ and } k)\\
	e^{t Z_{i^+}} =& n_{0,t,0,0} = 
\begin{pmatrix}
1 & 0 & t i \\
0 & 1_{r-1} & 0 \\
0 & 0 & 1
\end{pmatrix} \qquad (\text{similarly for } j \text{ and } k)
\end{align*}
To determine the operators on $G/K$, let $G \ni g =n_{z_0,q_0,q^\prime_0,q^{\prime\prime}_0} m_{y_o}$
\begin{align*}
H \cdot f(g) =& \frac{d}{dt}\bigg|_{t=0} f(g \cdot m_{e^t}) =  \frac{d}{dt}\bigg|_{t=0} f(n_{z_0,q_0,q^\prime_0,q^{\prime\prime}_0} m_{y_0} \cdot m_{e^t}) = \frac{d}{dt}\bigg|_{t=0} f(n_{z_0,q_0,q^\prime_0,q^{\prime\prime}_0} m_{y_0 e^t}) \\
=& y_0\frac{d}{dt}\bigg|_{t=0} f(n_{z_0,q_0,q^\prime_0,q^{\prime\prime}_0} m_{y_0 e^t}) = y_0 \; \frac{\partial f}{\partial y}n_{z_0,q_0,q^\prime_0,q^{\prime\prime}_0} m_{y_0})
\end{align*}
so $H$ acts by $y \frac{\partial}{\partial y}$ and $H^2 = H \circ H = y \frac{\partial}{\partial y}(y \frac{\partial}{\partial y}) = y^2 \frac{\partial^2}{\partial y^2} + y \frac{\partial}{\partial y}$.
We have convenient relations, similar to those useful for $O(r,1)$
$$
m_y \cdot e^{t X_\ell} = m_y \cdot n_{te_\ell,0,0,0} =  
\begin{pmatrix}
y & 0 & 0 \\
0 & 1_{r-1} & 0 \\
0 & 0 & \frac{1}{y}
\end{pmatrix}
\begin{pmatrix}
1 & t \cdot e_\ell & -\frac{t^2}{2} \\
0 \enspace & 1_{r-1} & -t \cdot e_\ell^\ast \\
0 & 0 & 1
\end{pmatrix}
=
\begin{pmatrix}
y & yt \cdot e_\ell & -\frac{yt^2}{2} \\
0 \enspace & 1_{r-1} & -t \cdot e_\ell^\ast \\
0 & 0 & \frac{1}{y}
\end{pmatrix}
$$
$$
n_{yte_\ell, 0,0,0} \cdot m_y =  
\begin{pmatrix}
1 & yt \cdot e_\ell & -\frac{y^2t^2}{2} \\
0 \enspace & 1_{r-1} & -yt \cdot e_\ell^\ast \\
0 & 0 & 1
\end{pmatrix}
\begin{pmatrix}
y & 0 & 0 \\
0 & 1_{r-1} & 0 \\
0 & 0 & \frac{1}{y}
\end{pmatrix}
=
\begin{pmatrix}
y & yt \cdot e_\ell & -\frac{yt^2}{2} \\
0 \enspace & 1_{r-1} & -t \cdot e_\ell^\ast \\
0 & 0 & \frac{1}{y}
\end{pmatrix}
$$
%
%
%
%
$$
m_y \cdot e^{t X_{i,\ell}} = m_y \cdot n_{tie_\ell,0,0,0} =  
\begin{pmatrix}
y & 0 & 0 \\
0 & 1_{r-1} & 0 \\
0 & 0 & \frac{1}{y}
\end{pmatrix}
\begin{pmatrix}
1 & t \cdot i e_\ell & -\frac{t^2}{2} \\
0 \enspace & 1_{r-1} & -t \cdot i e_\ell^\ast \\
0 & 0 & 1
\end{pmatrix}
=
\begin{pmatrix}
y & yt \cdot i e_\ell & -\frac{yt^2}{2} \\
0 \enspace & 1_{r-1} & -t \cdot i e_\ell^\ast \\
0 & 0 & \frac{1}{y}
\end{pmatrix}
$$
$$
n_{yt(i e_\ell), 0,0,0} \cdot m_y =  
\begin{pmatrix}
1 & yt \cdot i e_\ell & -\frac{y^2t^2}{2} \\
0 \enspace & 1_{r-1} & -yt \cdot i e_\ell^\ast \\
0 & 0 & 1
\end{pmatrix}
\begin{pmatrix}
y & 0 & 0 \\
0 & 1_{r-1} & 0 \\
0 & 0 & \frac{1}{y}
\end{pmatrix}
=
\begin{pmatrix}
y & yt \cdot i e_\ell & -\frac{yt^2}{2} \\
0 \enspace & 1_{r-1} & -t \cdot i e_\ell^\ast \\
0 & 0 & \frac{1}{y}
\end{pmatrix}
$$
so that
$$
m_y \cdot e^{t X_\ell} = m_y \cdot n_{te_\ell,0,0,0} =  n_{yte_\ell, 0,0,0} \cdot m_y
$$
and
$$
m_y \cdot e^{t X_{i,\ell}} = m_y \cdot n_{t i e_\ell,0,0,0} =  n_{yti e_\ell, 0,0,0} \cdot m_y
$$
which give
$$
n_{z_0,q_0,q^\prime_0,q^{\prime\prime}_0} \cdot m_y \cdot e^{t X_\ell} =n_{z_0,q_0,q^\prime_0,q^{\prime\prime}_0} \cdot m_y \cdot n_{te_\ell, 0,0,0} = n_{z_0,q_0,q^\prime_0,q^{\prime\prime}_0} \cdot n_{yt e_\ell, 0,0,0} \cdot m_y = n_{z_0 + yt e_\ell,q_0,q^\prime_0,q^{\prime\prime}_0} \cdot m_y
$$
and (similarly for $j$ and $k$ operators)
$$
n_{z_0,q_0,q^\prime_0,q^{\prime\prime}_0} \cdot m_y \cdot e^{t X_{i,\ell}} =n_{z_0,q_0,q^\prime_0,q^{\prime\prime}_0} \cdot m_y \cdot n_{t i e_\ell, 0,0,0} = n_{z_0,q_0,q^\prime_0,q^{\prime\prime}_0} \cdot n_{yti e_\ell, 0,0,0} \cdot m_y = n_{z_0 + yt i e_\ell,q_0,q^\prime_0,q^{\prime\prime}_0} \cdot m_y
$$
Use real coordinates $z = x+ iw + ju + kv$ for $g$ to determine the operators $X_\ell$, $X_{i,\ell}$, $X_{j,\ell}$ and $X_{k,\ell}$
\begin{align*}
X_\ell \cdot f(g) =& \frac{d}{dt}\bigg|_{t=0} f(g \cdot e^{tX_\ell}) = \frac{d}{dt}\bigg|_{t=0} f(g \cdot n_{te_\ell, 0}) =  \frac{d}{dt}\bigg|_{t=0} f(n_{z_0,q_0,q^\prime_0,q^{\prime\prime}_0} m_{y_0} \cdot n_{te_\ell, 0,0,0}) \\
 =& \frac{d}{dt}\bigg|_{t=0} f(n_{z_0 + y_0 te_\ell,q_0,q^\prime_0,q^{\prime\prime}_0} m_{y_0}) = y_0 \; \frac{\partial f}{\partial x_\ell}(n_{z_0,q_0,q^\prime_0,q^{\prime\prime}_0} m_{y_0})\end{align*}
giving
$$
X_\ell = y \frac{\partial}{\partial x_\ell} \text{  and  } X_\ell^2 = X_\ell(X_\ell) = y \frac{\partial}{\partial x_\ell} \big(y \frac{\partial}{\partial x_\ell} \big) = \bigg(y \frac{\partial}{\partial x_\ell}\bigg)^2 = y^2 \frac{\partial^2}{\partial x_\ell^2}
$$
Analogously, $X_{i,\ell}$ is the imaginary direction $w_\ell \approx i \cdot e_\ell$, $X_{j,\ell}$ is the imaginary direction $u_\ell \approx j \cdot e_\ell$ and $X_{k,\ell}$ is the imaginary direction $v_\ell \approx k \cdot e_\ell$ so
$$
X_{i,\ell} \cdot f(g) =  \frac{d}{dt}\bigg|_{t=0} f(n_{z_0 + y_0 t i e_\ell,q_0,q^\prime_0,q^{\prime\prime}_0} m_{y_0}) = y_0 \; \frac{\partial f}{\partial w_\ell}(n_{z_0,q_0,q^\prime_0,q^{\prime\prime}_0} m_{y_0})
$$
$$
X_{i,\ell} = y \frac{\partial}{\partial w_\ell} \text{  and  } X_{i,\ell}^2 = X_{i,\ell}(X_{i,\ell}) = y \frac{\partial}{\partial w_\ell} \big( y \frac{\partial}{\partial w_\ell} \big) = \bigg( y \frac{\partial}{\partial w_\ell}\bigg)^2 = y^2 \frac{\partial^2}{\partial w_\ell^2}
$$
$$
X_{j,\ell} \cdot f(g) =  \frac{d}{dt}\bigg|_{t=0} f(n_{z_0 + y_0 t j e_\ell,q_0,q^\prime_0,q^{\prime\prime}_0} m_{y_0}) = y_0 \; \frac{\partial f}{\partial u_\ell}(n_{z_0,q_0,q^\prime_0,q^{\prime\prime}_0} m_{y_0})
$$
$$
X_{j,\ell} = y \frac{\partial}{\partial u_\ell} \text{  and  } X_{j,\ell}^2 = X_{j,\ell}(X_{j,\ell}) = y \frac{\partial}{\partial u_\ell} \big( y \frac{\partial}{\partial u_\ell} \big) = \bigg( y \frac{\partial}{\partial u_\ell}\bigg)^2 = y^2 \frac{\partial^2}{\partial u_\ell^2}
$$
$$
X_{k,\ell} \cdot f(g) =  \frac{d}{dt}\bigg|_{t=0} f(n_{z_0 + y_0 t k e_\ell,q_0,q^\prime_0,q^{\prime\prime}_0} m_{y_0}) = y_0 \; \frac{\partial f}{\partial v_\ell}(n_{z_0,q_0,q^\prime_0,q^{\prime\prime}_0} m_{y_0})
$$
$$
X_{k,\ell} = y \frac{\partial}{\partial v_\ell} \text{  and  } X_{k,\ell}^2 = X_{k,\ell}(X_{k,\ell}) = y \frac{\partial}{\partial v_\ell} \big( y \frac{\partial}{\partial v_\ell} \big) = \bigg( y \frac{\partial}{\partial v_\ell}\bigg)^2 = y^2 \frac{\partial^2}{\partial v_\ell^2}
$$
Last, for the $Z$ terms, worked out for $Z_{i^+}$ (the $q$ direction): $Z_{j^+}$ and $Z_{k^+}$ are analogous (i.e., the $q^\prime$ and $q^{\prime\prime}$ directions, respectively).
$$
e^{t Z_{i^+}} = n_{0,t,0,0} =
\begin{pmatrix}
1 & 0 & i t \\
0 & 1_{r-1} & 0 \\
0 & 0 & 1
\end{pmatrix} \\
$$
\begin{align*}
Z_{i^+} \cdot f(g) =& \frac{d}{dt}\bigg|_{t=0} f(g \cdot e^{tZ_{i^+}}) = \frac{d}{dt}\bigg|_{t=0} f(g \cdot n_{0,t,0,0}) =  \frac{d}{dt}\bigg|_{t=0} f(n_{z_0, q_0,q^\prime_0,q^{\prime\prime}_0} m_{y_0} \cdot n_{0,t,0,0}) \\
=& 
\frac{d}{dt}\bigg|_{t=0} f \Big(
\begin{pmatrix}
1 & z_0 & -\frac{|z_0|^2}{2} + i q_0 + j q^\prime_0 + k q^{\prime\prime}_0 \\
0 \enspace & 1_{r-1} & z_0^\ast \\
0 & 0 & 1
\end{pmatrix}
\begin{pmatrix}
y_0 & 0 & 0 \\
0 & 1_{r-1} & 0 \\
0 & 0 & \frac{1}{y_0}
\end{pmatrix}
\begin{pmatrix}
1 & 0 &  i t \\
0 & 1_{r-1} & 0 \\
0 & 0 & 1
\end{pmatrix}
\Big) \\
=& 
\frac{d}{dt}\bigg|_{t=0} f \Big(
\begin{pmatrix}
1 & z_0 & -\frac{|z_0|^2}{2} + i q_0 + j q^\prime_0 + k q^{\prime\prime}_0 \\
0 \enspace & 1_{r-1} & z_0^\ast \\
0 & 0 & 1
\end{pmatrix}
\begin{pmatrix}
y_0 & 0 & i y_0 t \\
0 & 1_{r-1} & 0 \\
0 & 0 & \frac{1}{y_0}
\end{pmatrix}
\Big) \\
=& 
\frac{d}{dt}\bigg|_{t=0} f \Big(
\begin{pmatrix}
y_0 & z_0 & -\frac{|z_0|^2}{2 y_0} + \frac{i(q_0 + y_0^2 t)  + j q^\prime_0 + k q^{\prime\prime}_0}{y_0} \\
0 \enspace & 1_{r-1} & \frac{1}{y_0}z_0^\ast \\
0 & 0 & \frac{1}{y_0}
\end{pmatrix}
\Big) \\
=& 
\frac{d}{dt}\bigg|_{t=0} f \Big(
\begin{pmatrix}
1 & z_0 & -\frac{|z_0|^2}{2} + i(q_0 + y_0^2 t)  + j q^\prime_0 + k q^{\prime\prime}_0 \\
0 \enspace & 1_{r-1} & z_0^\ast \\
0 & 0 & 1
\end{pmatrix}
\begin{pmatrix}
y_0 & 0 & 0 \\
0 & 1_{r-1} & 0 \\
0 & 0 & \frac{1}{y_0}
\end{pmatrix}
\Big) \\
=& 
\frac{d}{dt}\bigg|_{t=0} f (z_0, (q_0 + y_0^2 t),q^\prime_0,q^{\prime\prime}_0, y_0) = y_0^2 \frac{\partial f}{\partial q}(z_0, q_0, q^\prime_0, q^{\prime\prime}_0, y_0)
\end{align*}
so correspondingly
$$
Z_{i^+} = y^2 \frac{\partial}{\partial q} \text{  and  } Z_{i^+}^2 = Z_{i^+} \circ Z_{i^+} = y^2 \frac{\partial}{\partial q} \big( y^2 \frac{\partial}{\partial q} \big) = \bigg( y^2 \frac{\partial}{\partial q}\bigg)^2 = y^4 \frac{\partial^2}{\partial q^2}
$$
and
$$
Z_{j^+} = y^4 \frac{\partial^2}{\partial q^{\prime 2}} \qquad Z_{k^+} = y^4 \frac{\partial^2}{\partial q^{\prime\prime 2}}
$$
\noindent
Substituting appropriately in $\Omega$ on $G/K$
\begin{align*}
\Omega_\frak{g} = \; & \frac{1}{2}H^2  - (2r+1) H + 2 Z_{i^+}^2 + 2 Z_{j^+}^2 + 2 Z_{k^+}^2 + \sum^{r-1}_{\ell = 1} X_\ell^2 + \sum^{r-1}_{\ell = 1} X_{i,\ell}^2 + \sum^{r-1}_{\ell = 1} X_{j,\ell}^2 + \sum^{r-1}_{\ell = 1} X_{k,\ell}^2 \\
= \; & \frac{1}{2}\big( y^2 \frac{\partial^2}{\partial y^2} + y \frac{\partial}{\partial y} \big) - (2r+1) \big( y \frac{\partial}{\partial y} \big) + 2 y^4 \frac{\partial^2}{\partial q^2} + 2 y^4 \frac{\partial^2}{\partial q^{\prime 2}} + 2 y^4 \frac{\partial^2}{\partial q^{\prime\prime 2}} \\
	& \; + y^2 \sum^{r-1}_{\ell = 1} \frac{\partial^2}{\partial x_\ell^2} + y^2 \sum^{r-1}_{\ell = 1} \frac{\partial^2}{\partial w_\ell^2} + y^2 \sum^{r-1}_{\ell = 1} \frac{\partial^2}{\partial u_\ell^2} + y^2 \sum^{r-1}_{\ell = 1} \frac{\partial^2}{\partial v_\ell^2} \\
= \; & \frac{1}{2} y^2 \frac{\partial^2}{\partial y^2}  - \frac{4r+1}{2} \big( y \frac{\partial}{\partial y} \big) + 2 y^4 \frac{\partial^2}{\partial q^2} + 2 y^4 \frac{\partial^2}{\partial q^{\prime 2}} + 2 y^4 \frac{\partial^2}{\partial q^{\prime\prime 2}} \\
	& \; + y^2 \big( \sum^{r-1}_{\ell = 1} \frac{\partial^2}{\partial x_\ell^2} + \sum^{r-1}_{\ell = 1} \frac{\partial^2}{\partial w_\ell^2} +  \sum^{r-1}_{\ell = 1} \frac{\partial^2}{\partial u_\ell^2} + \sum^{r-1}_{\ell = 1} \frac{\partial^2}{\partial v_\ell^2} \big) \\
=\; &  y^2 \big(    \sum^{r-1}_{\ell = 1} \frac{\partial^2}{\partial x_\ell^2} + \sum^{r-1}_{\ell = 1} \frac{\partial^2}{\partial w_\ell^2} +  \sum^{r-1}_{\ell = 1} \frac{\partial^2}{\partial u_\ell^2} + \sum^{r-1}_{\ell = 1} \frac{\partial^2}{\partial v_\ell^2}  + \frac{1}{2}\frac{\partial^2}{\partial y^2} + 2 y^2 \frac{\partial^2}{\partial q^2} + 2 y^2 \frac{\partial^2}{\partial q^{\prime 2}} + 2 y^2 \frac{\partial^2}{\partial q^{\prime\prime 2}} ) \\
	& \; - \frac{4r+1}{2}  y \frac{\partial}{\partial y}   \\
=&  y^2 \big(  \Delta_x + \Delta_w + \Delta_u + \Delta_v + \frac{1}{2}\frac{\partial^2}{\partial y^2} + 2 y^2 \frac{\partial^2}{\partial q^2} + 2 y^2 \frac{\partial^2}{\partial q^{\prime 2}} + 2 y^2 \frac{\partial^2}{\partial q^{\prime\prime 2}} \big) - \frac{4r+1}{2}  y \frac{\partial}{\partial y} 
\end{align*}
To clear coefficients, we could renormalize the coordinates $z, p, q, r \in N \approx  \mathbb{H}^{r-1} \times \mathbb{R}^3$, $y \in A^+ \approx  \mathbb{R}^\times$ from the Iwasawa decomposition
$$
G/K \ni n_{z,q,q^\prime,q^{\prime\prime}} m_y K \longleftrightarrow (z,q,q^\prime,q^{\prime\prime},y) \in \mathbb{H}^{r-1} \times \mathbb{R}^3 \times (0,\infty)
$$
by
\begin{align*}
& \mathbb{H}^{r-1} \times \mathbb{R}^3 \times (0, +\infty) \ni  (z,q,q^\prime,q^{\prime\prime},y) \to n_{z\cdot \sqrt{2},\frac{p}{4},\frac{q}{4},\frac{r}{4}}m_{y} = \\
&
\renewcommand\arraystretch{1.5}
\begin{pmatrix}
1 & z\cdot \sqrt{2} & -|z|^2 + \frac{ip}{4} + \frac{iq}{4} + \frac{ir}{4} \\
0 & 1_{r-1} & -z^\ast \cdot \sqrt{2} \\
0 & 0 & 1
\end{pmatrix}
\begin{pmatrix}
y & 0 & 0 \\
0 & 1_{r-1} & 0 \\
0 & 0 & \frac{1}{y}
\end{pmatrix}
\end{align*}
$$
2 \Omega_\frak{g} = y^2 \big(  \Delta_x + \Delta_w + \Delta_u + \Delta_v + \frac{\partial^2}{\partial y^2} + y^2 \frac{\partial^2}{\partial q^2} + y^2 \frac{\partial^2}{\partial q^{\prime 2}} + y^2 \frac{\partial^2}{\partial q^{\prime\prime 2}} \big) - (4r+1)  y \frac{\partial}{\partial y} 
$$
Ignoring the factor of $2$
$$
\Omega_\frak{g} = y^2 \big(  \Delta_x + \Delta_w + \Delta_u + \Delta_v + \frac{\partial^2}{\partial y^2} + y^2 \frac{\partial^2}{\partial q^2} + y^2 \frac{\partial^2}{\partial q^{\prime 2}} + y^2 \frac{\partial^2}{\partial q^{\prime\prime 2}} \big) - (4r+1)  y \frac{\partial}{\partial y} 
$$
%
%
%
%
%
%
%
%

\section{Positivity of fragments of $-\Delta$}\label{Postivity-of-fragments-of-Delta}

%
%
%
%

\subsection{$O(r,1)$}\label{Postivity-of-fragments-of-Delta-O(r,1)}

The Lax-Phillips argument requires not only that $-\Delta$ itself be non-negative but that the two natural summands of $-\Delta$ in Iwasawa coordinates are both non-negative. For instance, the first-order term $(r-2)y\frac{\partial}{\partial y}$ will be essential to cancel terms occuring in integration by parts.

As before, $G \approx O(r,1)$ with maximal compact $K$ and minimal parabolic subgroup $P = NA^+ (P \cap K)$, $N \approx \mathbb{R}^{r-1}$ and $A^+ \approx \mathbb{R}^+ = (0, \infty)$. We have an Iwasawa decomposition $G = NA^+ K$.
Let $\Gamma$ be an arithmetic subgroup of $G$ where we grant that $\Gamma$ has finitely-many cusps by reduction theory \cite{Borel-1965-66b} \cite{Borel-HarishChandra-1962}. Assume $\Gamma \cap N$ acts on $\mathbb{R}^{r-1}$ (i.e., $N$) by translation by $\mathbb{Z}^{r-1}$.

Examining the possible components of $-\Delta$, the fragment $-y^2\Delta_x$ is non-negative since derivatives in $x$ do not interact with the coefficient $y^2$ or the $y^{-r}$ in the measure.
Thus we must show that, for $f \in C^\infty_c(\mathbb{R}^{r-1} \times \mathbb{R}^\times)$
$$
\int_{(N \cap \Gamma) \backslash N} \int_{y > c} -(y^2\frac{\partial^2}{\partial y^2} - (r-2)y\frac{\partial}{\partial y})f \cdot \bar{f} \ \frac{dx\, dy}{y^r} \geq 0
$$
Integrating by parts once in $y$ on the $\frac{\partial^2}{\partial y^2}$ term gives
\begin{align*}
& \int_{(N \cap \Gamma) \backslash N} \int_{y > c} -y^2\frac{\partial^2}{\partial y^2} f \cdot \bar{f} \ \frac{dx\, dy}{y^r} \\
=& \int_{(N \cap \Gamma) \backslash N} \int_{y > c} -\frac{\partial^2}{\partial y^2}f \cdot y^{2-r}\bar{f} \ dx \, dy 
= \int_{(N \cap \Gamma) \backslash N} \int_{y > c} y^{2-r}\bar{f} \cdot (-\frac{\partial^2}{\partial y^2}f)\ dx\, dy \\
=& \enspace \cancelto{0}{(y^{2-r} \bar{f})(-\frac{\partial}{\partial y}f) \bigg|_\partial} \quad - \int_{(N \cap \Gamma) \backslash N} \int_{y > c} -\frac{\partial}{\partial y}f \cdot \frac{\partial}{\partial y}(y^{2-r} \bar{f}) \ dx \, dy \\
=& \int_{(N \cap \Gamma) \backslash N} \int_{y > c} \frac{\partial}{\partial y}f \cdot \frac{\partial}{\partial y}(y^{2-r} \bar{f}) \ dx \, dy 
= \int_{(N \cap \Gamma) \backslash N} \int_{y > c}	\frac{\partial}{\partial y}f \cdot ((2-r)y^{1-r} + y^{2-r}\frac{\partial}{\partial y})\bar{f} \ dx \, dy
\end{align*}
The $\frac{\partial}{\partial y}f \cdot (2-r)y^{1-r}\bar{f}$ cancels the corresponding term in the original expression so that
\begin{align*}
& \int_{(N \cap \Gamma) \backslash N} \int_{y > c}  -(y^2\frac{\partial^2}{\partial y^2} - (r-2)y\frac{\partial}{\partial y})f \cdot \bar{f} \ \frac{dx \, dy}{y^r} \\
=& \int_{(N \cap \Gamma) \backslash N} \int_{y > c}  y\frac{\partial}{\partial y}f \cdot y\frac{\partial}{\partial y}\bar{f} \ \frac{dx \, dy}{y^r} = \int_{(N \cap \Gamma) \backslash N} \int_{y > c}  y^2 \bigg| \frac{\partial f}{\partial y} \bigg|^2 \ \frac{dx \, dy}{y^r} \geq 0
\end{align*}
So that, with the invariant Laplacian $\Delta$
$$
\int_{(N \cap \Gamma) \backslash N} \int_{y > c}   -\Delta f \cdot \bar{f} \frac{dx \, dy}{y^r} = \int_{(N \cap \Gamma) \backslash N} \int_{y > c}   (y\nabla_x f)^2 + y^2\bigg|\frac{\partial f}{\partial y}\bigg|^2 \frac{dx \, dy}{y^r} \geq 0
$$
and the invariant Laplacian $-\Delta$, along with its components corresponding to the Iwasawa decomposition, are non-negative.

%
%

\subsection{$U(r,1)$}\label{Postivity-of-fragments-of-Delta-U(r,1)}

\begin{align*}
\Omega =& y^2 \bigg( \Delta_u + \Delta_v + \frac{\partial^2}{\partial y^2}+ y^2 \frac{\partial^2}{\partial x^2} \bigg) - (2r-1) y \frac{\partial}{\partial y} \\
-\Delta =& -y^2 \bigg( \Delta_u + \Delta_v + y^2 \frac{\partial^2}{\partial x^2} \bigg) - \bigg(y^2 \frac{\partial^2}{\partial y^2} - (2r-1) y \frac{\partial}{\partial y} \bigg)
\end{align*}
Examining the possible components of $-\Delta$, the fragment $-(\Delta_u + \Delta_v + y^2 \frac{\partial^2}{\partial x^2})$ is seen to be non-negative since derivatives in $u$, $v$ and $x$ do not interact with the coefficient $y^2$ or the $y^{-(2r+1)}$ in the measure, though the coefficient $y^2$ of $\frac{\partial^2}{\partial x^2}$ calls for further examination. For $f \in C^\infty_c((N \cap \Gamma) \backslash N \times (a, \infty)$
\begin{align*}
& \int_{(N \cap \Gamma) \backslash N} \int_{y > c}  -(\Delta_u + \Delta_v + y^2 \frac{\partial^2}{\partial x^2})f(n_{z,x} m_y) \cdot \overline{f(n_{z,x} m_y)} \ \frac{dz\, dx\, dy}{y^{2r+1}} \\
=& \underbrace{\int_{(N \cap \Gamma) \backslash N} \int_{y > c}  -\Delta_u f(n_{z,x} m_y) \cdot \overline{f(n_{z,x} m_y)} \ \frac{dz\, dx\, dy}{y^{2r+1}}}_\text{does not interact with y - so positive} \\
&+ \underbrace{\int_{(N \cap \Gamma) \backslash N} \int_{y > c}  -\Delta_v f(n_{z,x} m_y) \cdot \overline{f(n_{z,x} m_y)} \ \frac{dz\, dx\, dy}{y^{2r+1}} }_\text{does not interact with y - so positive} \\
&+ \int_{(N \cap \Gamma) \backslash N} \int_{y > c} \bigg(-y^2 \frac{\partial^2}{\partial x^2} \bigg)f(n_{z,x} m_y) \cdot \overline{f(n_{z,x} m_y)} \ \frac{dz\, dx\, dy}{y^{2r+1}} \\
\end{align*}
The third term is shown positive via integration by parts
\begin{align*}
& \int_{(N \cap \Gamma) \backslash N} \int_{y > c} \bigg(-y^2 \frac{\partial^2}{\partial x^2} \bigg)f(n_{z,x} m_y) \cdot \overline{f(n_{z,x} m_y)} \ \frac{dz\, dx\, dy}{y^{2r+1}} \\
=& \int_{(N \cap \Gamma) \backslash N} \int_{y > c} - \frac{\partial^2}{\partial x^2} f(n_{z,x} m_y) \cdot y^{1 - 2n} \cdot \overline{f(n_{z,x} m_y)} \ dz\, dx\, dy \\
=& \int_{(N \cap \Gamma) \backslash N} \int_{y > c}  y^{1 - 2n} \cdot \overline{f(n_{z,x} m_y)} cdot \bigg(  - \frac{\partial^2}{\partial x^2} f(n_{z,x} m_y) \bigg)  \ dz\, dx\, dy \\
=& \enspace \cancelto{0}{\bigg( y^{1 - 2n} \overline{f} \bigg) \cdot \bigg(  - \frac{\partial}{\partial x} f \bigg) \bigg|_\partial} -  \int_{(N \cap \Gamma) \backslash N} \int_{y > c} \bigg(  - \frac{\partial}{\partial x} f \bigg) \cdot \underbrace{\bigg(  y^{1 - 2n} \cdot \frac{\partial}{\partial x} \overline{f} \bigg) }_\text{ $\frac{\partial}{\partial x}$ ignores $y$}   \ dz\, dx\, dy \\
=&  \int_{(N \cap \Gamma) \backslash N} \int_{y > c} y^2 \cdot \bigg| \frac{\partial f}{\partial x} \bigg|^2 \frac{ \ dz\, dx\, dy}{y^{2r+1}} \geq 0
\end{align*}
For later use, rewrite this fragment of $\Delta$ as
$$
-(\Delta_u + \Delta_v + y^2 \frac{\partial^2}{\partial x^2}) = -(\Delta_u + \Delta_v + a^2 \frac{\partial^2}{\partial x^2}) -((y^2 - a^2) \frac{\partial^2}{\partial x^2})
$$
Note that, for $c \gg a$, the first term is clearly positive as there is no dependence on $y$ and the second, perturbed expression is also non-negative
\begin{align*}
& \int_{(N \cap \Gamma) \backslash N} \int_{y > c} \bigg(-(y^2 - a^2) \frac{\partial^2}{\partial x^2} \bigg)f(n_{z,x} m_y) \cdot \overline{f(n_{z,x} m_y)} \ \frac{dz\, dx\, dy}{y^{2r+1}} \\
=&  \int_{(N \cap \Gamma) \backslash N} \int_{y > c} (y^2 - a^2) \cdot \bigg| \frac{\partial f}{\partial x} \bigg|^2 \frac{ \ dz\, dx\, dy}{y^{2r+1}} \geq 0
\end{align*}
Last, we must show that, for $f \in C^\infty_c(\mathbb{C}^{r-1} \times \mathbb{R} \times \mathbb{R}^\times)$
$$
\int_{(N \cap \Gamma) \backslash N} \int_{y > c}  -(y^2 \frac{\partial^2}{\partial y^2} - (2r-1) y \frac{\partial}{\partial y})f(n_{z,x} m_y) \cdot \overline{f(n_{z,x} m_y)} \ \frac{dz\, dx\, dy}{y^{2r+1}} \geq 0
$$
Integrating by parts once in the $\frac{\partial^2}{\partial y^2}$ term gives
\begin{align*}
& \int_{(N \cap \Gamma) \backslash N} \int_{y > c}  -y^2 \frac{\partial^2}{\partial y^2} f(n_{z,x} m_y) \cdot \overline{f(n_{z,x} m_y)} \ \frac{dz\, dx\, dy}{y^{2r+1}} \\
=& \int_{(N \cap \Gamma) \backslash N} \int_{y > c}  - \frac{\partial^2}{\partial y^2} f(n_{z,x} m_y) \cdot y^{1-2n} \overline{f(n_{z,x} m_y)} \ dz\, dx\, dy \\
=& \int_{(N \cap \Gamma) \backslash N} \int_{y > c} y^{1-2n} \overline{f(n_{z,x} m_y)} \cdot \bigg( - \frac{\partial^2}{\partial y^2} f(n_{z,x} m_y) \bigg)\ dz\, dx\, dy \\
=& \enspace \cancelto{0}{(y^{1-2n} \bar{f})(-\frac{\partial}{\partial y}f) \bigg|_\partial} \enspace - \int_{(N \cap \Gamma) \backslash N} \int_{y > c} (-\frac{\partial}{\partial y}f) \cdot \frac{\partial}{\partial y}(y^{1-2n} \bar{f}) \; dz\, dx\, dy \\
=& \int_{(N \cap \Gamma) \backslash N} \int_{y > c} \frac{\partial}{\partial y}f \cdot \bigg( (1 -2n)y^{-2n} \bar{f} + y^{1-2n} \frac{\partial}{\partial y}\bar{f} \bigg)  \; dz\, dx\, dy\\
=& \int_{(N \cap \Gamma) \backslash N} \int_{y > c} (1 -2n)y^{-2n} \frac{\partial}{\partial y}f \bar{f} \; dz\, dx\, dy \\
& + \int_{(N \cap \Gamma) \backslash N} \int_{y > c} \frac{\partial}{\partial y}f y^{1-2n} \frac{\partial}{\partial y}\bar{f}   \; dz\, dx\, dy\\
=& \int_{(N \cap \Gamma) \backslash N} \int_{y > c} (1 -2n) y \frac{\partial}{\partial y}f \bar{f} \; \frac{dz\, dx\, dy}{y^{2r+1}} \\
& + \int_{(N \cap \Gamma) \backslash N} \int_{y > c} y^2 \frac{\partial}{\partial y}f  \frac{\partial}{\partial y}\bar{f}   \; \frac{dz\, dx\, dy}{y^{2r+1}}\\
\end{align*}
The first term cancels the corresponding term involving $y\frac{\partial}{\partial y}$ in the original expression, giving
\begin{align*}
& \int_{(N \cap \Gamma) \backslash N} \int_{y > c}  -(y^2 \frac{\partial^2}{\partial y^2} - (2r-1) y \frac{\partial}{\partial y})f \cdot \overline{f} \ \frac{dz\, dx\, dy}{y^{2r+1}} \\
=& \int_{(N \cap \Gamma) \backslash N} \int_{y > c} y^2 \frac{\partial}{\partial y}f  \frac{\partial}{\partial y}\bar{f}   \; \frac{dz\, dx\, dy}{y^{2r+1}} \\
=& \int_{(N \cap \Gamma) \backslash N} \int_{y > c} y^2 \bigg|\frac{\partial f}{\partial y} \bigg|^2   \; \frac{dz\, dx\, dy}{y^{2r+1}} \geq 0 \qquad \square
\end{align*}
%
%

%
%

\subsection{$Sp^\ast(r,1)$}\label{Postivity-of-fragments-of-Delta-Sp*(r,1)}
The Casimir operator for $Sp^\ast(r,1)$ is
\begin{align*}
\Omega_\frak{g} = \; & y^2 \big(  \Delta_x + \Delta_w + \Delta_u + \Delta_v + \frac{\partial^2}{\partial y^2} + y^2 \frac{\partial^2}{\partial p^2} + y^2 \frac{\partial^2}{\partial q^2} + y^2 \frac{\partial^2}{\partial r^2} ) - (4r+1)  y \frac{\partial}{\partial y}  \\
-\Delta = \; & - y^2 \big(  \underbrace{\Delta_x + \Delta_w + \Delta_u + \Delta_v}_\text{does not interact with y - so positive} +  y^2 \frac{\partial^2}{\partial p^2} + y^2 \frac{\partial^2}{\partial q^2} + y^2 \frac{\partial^2}{\partial r^2} ) \\
&+ \big(- y^2\frac{\partial^2}{\partial y^2} + (4r+1)  y \frac{\partial}{\partial y} \big)
\end{align*}
Similar to $U(r,1)$, in the first sum, only the terms with a coefficient of $y^2$ need be shown positive; the calculation is done for  the $p$ variable as the other two are similar
\begin{align*}
& \int_{(N \cap \Gamma) \backslash N} \int_{y > c} \bigg(-y^2 \frac{\partial^2}{\partial p^2} \bigg)f(n_{z,q,q^\prime,q^{\prime\prime}} m_y) \cdot \overline{f(n_{z,q,q^\prime,q^{\prime\prime}} m_y)} \ \frac{dz\, dq\, dq^\prime\, dq^{\prime\prime}\, dy}{y^{4r+3}} \\
=& \int_{(N \cap \Gamma) \backslash N} \int_{y > c} - \frac{\partial^2}{\partial p^2} f(n_{z,q,q^\prime,q^{\prime\prime}} m_y) \cdot y^{-1 - 4r} \cdot \overline{f(n_{z,q,q^\prime,q^{\prime\prime}} m_y)} \ dz\, dq\, dq^\prime\, dq^{\prime\prime}\, dy \\
=& \int_{(N \cap \Gamma) \backslash N} \int_{y > c}  y^{-1 - 4r} \cdot \overline{f(n_{z,q,q^\prime,q^{\prime\prime}} m_y)} cdot \bigg(  - \frac{\partial^2}{\partial p^2} f(n_{z,q,q^\prime,q^{\prime\prime}} m_y) \bigg)  \ dz\, dq\, dq^\prime\, dq^{\prime\prime}\, dy \\
=& \cancelto{0}{\bigg( y^{-1 - 4r} \overline{f} \bigg) \cdot \bigg(  - \frac{\partial}{\partial p} f \bigg) \bigg|_\partial} -  \int_{(N \cap \Gamma) \backslash N} \int_{y > c} \bigg(  - \frac{\partial}{\partial p} f \bigg) \cdot \underbrace{\bigg(  y^{-1 - 4r} \cdot \frac{\partial}{\partial p} \overline{f} \bigg) }_\text{ $\frac{\partial}{\partial p}$ ignores $y$}   \ dz\, dq\, dq^\prime\, dq^{\prime\prime}\, dy \\
=&  \int_{(N \cap \Gamma) \backslash N} \int_{y > c} y^2 \bigg| \frac{\partial f}{\partial p} \bigg|^2 \frac{ \ dz\, dq\, dq^\prime\, dq^{\prime\prime}\, dy}{y^{4r+3}} \geq 0
\end{align*}
We also have the variation similar to the $U(r,1)$ case that will be useful:
\begin{align*}
& \int_{(N \cap \Gamma) \backslash N} \int_{y > c} -(y^2 -a^2) \frac{\partial^2}{\partial p^2} f \; \cdot \; \bar{f} \;  \frac{dz\, dq\, dq^\prime\, dq^{\prime\prime}\, dy}{y^{4r+3}} \\
&= \int_{(N \cap \Gamma) \backslash N} \int_{y > c} (y^2 - a^2)  \bigg| \frac{\partial f}{\partial p} \bigg|^2 \frac{ \ dz\, dq\, dq^\prime\, dq^{\prime\prime}\, dy}{y^{4r+3}} \geq 0
\end{align*}

This leaves the positivity of the derivatives in $y$. The pattern occurring previously is again present: the coefficient of $\frac{\partial}{\partial y}$ is equal to the exponent of $y$ in the denominator of the measure minus $2$. We must show that, for $f \in C^\infty_c(\mathbb{C}^{r-1} \times \mathbb{R} \times \mathbb{R}^\times)$
$$
\int_{(N \cap \Gamma) \backslash N} \int_{y > c}  -(y^2 \frac{\partial^2}{\partial y^2} - (4r+1) y \frac{\partial}{\partial y})f(n_{z,q,q^\prime,q^{\prime\prime}} m_y) \cdot \overline{f(n_{z,q,q^\prime,q^{\prime\prime}} m_y)} \ \frac{dz\, dq\, dq^\prime\, dq^{\prime\prime}\, dy}{y^{4r+3}} \geq 0
$$
Integrating by parts once in the $\frac{\partial^2}{\partial y^2}$ term gives
\begin{align*}
& \int_{(N \cap \Gamma) \backslash N} \int_{y > c}  -y^2 \frac{\partial^2}{\partial y^2} f(n_{z,q,q^\prime,q^{\prime\prime}} m_y) \cdot \overline{f(n_{z,q,q^\prime,q^{\prime\prime}} m_y)} \ \frac{dz\, dq\, dq^\prime\, dq^{\prime\prime}\, dy}{y^{4r+3}} \\
=& \int_{(N \cap \Gamma) \backslash N} \int_{y > c}  - \frac{\partial^2}{\partial y^2} f(n_{z,q,q^\prime,q^{\prime\prime}} m_y) \cdot y^{-1-4r} \overline{f(n_{z,q,q^\prime,q^{\prime\prime}} m_y)} \ dz\, dq\, dq^\prime\, dq^{\prime\prime}\, dy \\
=& \int_{(N \cap \Gamma) \backslash N} \int_{y > c} y^{-1-4r} \overline{f(n_{z,q,q^\prime,q^{\prime\prime}} m_y)} \cdot \bigg( - \frac{\partial^2}{\partial y^2} f(n_{z,q,q^\prime,q^{\prime\prime}} m_y) \bigg)\ dz\, dq\, dq^\prime\, dq^{\prime\prime}\, dy \\
=& \enspace \cancelto{0}{(y^{1-2n} \bar{f})(-\frac{\partial}{\partial y}f) \bigg|_\partial} \enspace - \int_{(N \cap \Gamma) \backslash N} \int_{y > c} (-\frac{\partial}{\partial y}f) \cdot \frac{\partial}{\partial y}(y^{-1-4r} \bar{f}) \; dz\, dq\, dq^\prime\, dq^{\prime\prime}\, dy \\
=& \int_{(N \cap \Gamma) \backslash N} \int_{y > c} \frac{\partial}{\partial y}f \cdot \bigg( (-1-4r)y^{-2-4r} \bar{f} + y^{-1-4r} \frac{\partial}{\partial y}\bar{f} \bigg)  \; dz\, dq\, dq^\prime\, dq^{\prime\prime}\, dy\\
=& \int_{(N \cap \Gamma) \backslash N} \int_{y > c} (-1-4r)y^{-2-4r} \frac{\partial}{\partial y}f \bar{f} \; dz\, dq\, dq^\prime\, dq^{\prime\prime}\, dy \\
& + \int_{(N \cap \Gamma) \backslash N} \int_{y > c} y^{-1-4r} \frac{\partial}{\partial y}f  \frac{\partial}{\partial y}\bar{f}   \; dz\, dq\, dq^\prime\, dq^{\prime\prime}\, dy\\
=& \int_{(N \cap \Gamma) \backslash N} \int_{y > c} -(4r+1) y \frac{\partial}{\partial y}f \bar{f} \; \frac{dz\, dq\, dq^\prime\, dq^{\prime\prime}\, dy}{y^{4r+3}} \\
& + \int_{(N \cap \Gamma) \backslash N} \int_{y > c} y^2 \frac{\partial}{\partial y}f  \frac{\partial}{\partial y}\bar{f}   \; \frac{dz\, dq\, dq^\prime\, dq^{\prime\prime}\, dy}{y^{4r+3}}\\
\end{align*}
The first term cancels the corresponding term involving $y\frac{\partial}{\partial y}$ in the original expression, giving
\begin{align*}
&\int_{(N \cap \Gamma) \backslash N} \int_{y > c}  -(y^2 \frac{\partial^2}{\partial y^2} - (4r+1) y \frac{\partial}{\partial y})f \cdot \overline{f} \ \frac{dz\, dq\, dq^\prime\, dq^{\prime\prime}\, dy}{y^{4r+3}} \\
=& \int_{(N \cap \Gamma) \backslash N} \int_{y > c} y^2 \frac{\partial}{\partial y}f  \frac{\partial}{\partial y}\bar{f}   \; \frac{dz\, dq\, dq^\prime\, dq^{\prime\prime}\, dy}{y^{4r+3}} \\
=& \int_{(N \cap \Gamma) \backslash N} \int_{y > c} y^2 \bigg|\frac{\partial f}{\partial y} \bigg|^2   \; \frac{dz\, dq\, dq^\prime\, dq^{\prime\prime}\, dy}{y^{4r+3}} \geq 0 \qquad \square
\end{align*}

%
%
%
%
%
%
%
%

\section{For $a \gg 1 \; \mathscr{D}_a \text{ is dense in } L^2_a(\Gamma \backslash G /K)$ }\label{sec:Density-of-Automorphic-Test-Functions}

To establish notation:
\begin{align*}
C^\infty_c(\Gamma \backslash G /K) =&  \text{ right } K \text{-invariant functions in } C^\infty_c(\Gamma \backslash G) = C^\infty_c(\Gamma \backslash G)^K \\
L^2(\Gamma \backslash G /K) =& \text{ right } K \text{-invariant functions in }L^2(\Gamma \backslash G) = L^2(\Gamma \backslash G)^K \\
|f|^2_{\frak{B}^1} = & \langle (1 - \Delta)f , f \rangle = \langle f ,f \rangle + \langle (-\Delta)f,f \rangle \\
\frak{B}^1 = & \text{ completion of } C^\infty_c(\Gamma \backslash G / K) \text{ with respect to the } \frak{B}_1 \text{ norm} \\
\eta(n m_y k) = & \; y^r (\text{In some circumstances, which will be clear, this may be } r-1)\\
L^2_a(\Gamma \backslash G /K) = & \; \{ f \in L^2(\Gamma \backslash G /K) \; : \; c_P f (g) = 0 \text{ for } \eta(g) \geq a \} \\
 =& \text{ {\it pseudo-cuspforms} with cut-off height } a \\
\mathscr{D}_a = & \; C^\infty_c(\Gamma \backslash G /K) \cap L^2_a(\Gamma \backslash G /K) \\
\Delta_a = &\; \Delta \text{ restricted to } \mathscr{D}_a \\
\frak{B}^1_a = & \text{ closure of } \mathscr{D}_a \text{ in } L^2_a(\Gamma \backslash G /K) \\
\end{align*}
%
%

%
%

\subsection{$O(r,1)$}\label{Density-of-Automorphic-Test-Functions-O(r,1)}

%
%
\begin{lemma}\label{density-of-auto-test-fns-on1}
	For $a \gg 1 \; \mathscr{D}_a \text{ is dense in } L^2_a(\Gamma \backslash G /K)$
\end{lemma}
\begin{proof}
Recall that a (standard) Siegel set is a subset of $G$ given by a compact set $C \subset N$ and a (height) parameter $t$
$$
\frak{S}_{t,C} = \{ na_y k \; : \; n \in C, k \in K, y \geq t \}
$$
Let $N$ be the unipotent radical of the standard minimal parabolic $P$. For $g \in G$, let $g = n_g m_g k_g$ be $g$'s Iwasawa coordinates, with $m_g \in A^+$. 
 By reduction theory (cf. \S [3.3] in \cite{PBG}) there is a sufficiently small $t_o > 0$ and compact $C \subset N$ such that the standard Siegel set
$$
\frak{S} = \frak{S}_{t_o, C} = \{G \ni g = nmk \; : \; n \in C \subset N, m \in A^+, k \in K, \eta(g) = \eta(m) \geq t_o \}
$$
satisfies $\Gamma \cdot \frak{S} = G$. Fix such a $\frak{S}_{t_o, C}$ and let $\frak{S}_a$ be the subset of $\frak{S}_{t_o, C}$ given by
$$
\frak{S}_a = \big\{ g \in \frak{S}_{t_o, C} \big| \enspace \eta(g) = \eta(n_g a_g) = \eta(a_y) > a \big\}
$$
By reduction theory, there is a height $a \gg 1$ so that $\frak{S}_a$ satisfies $\frak{S}_a \cap \gamma \frak{S}_a \ne \emptyset$ implies $\gamma \in \Gamma \cap P = \Gamma_\infty$. In the following, $f \in L^2_a(\Gamma \backslash G /K)$ is first approximated by test functions $f_r \in C^\infty_c(\Gamma \backslash G /K)$ by general methods, and then the condition $a \gg 1$ is used in conjunction with a family of smooth cut-off functions of the constant near height $a$, with the width of cut-off region shrinking to $0$.

Per the above, take $a \gg 1$ so that $\frak{S}_{a - \frac{1}{t}}$ meets its translates $\gamma \frak{S}_{a - \frac{1}{t}}$ only for $\gamma \in \Gamma_\infty$, for all sufficiently large $t$. This allows separation of variables in $\frak{S}_{a - \frac{1}{t}}$ in the sense that the cylinder $C_{a - \frac{1}{t}} =  (\Gamma \cap P) \backslash \frak{S}_{a - \frac{1}{t}}$ injects to $\Gamma \backslash G /K$. Let
$$
|f|^2_{a - \frac{1}{t}} = \int_{C_{a - \frac{1}{t}}} |f(n_x m_y)|^2 \frac{dx \; dy}{y^r} \leq  \int_{\Gamma \backslash G /K} |f(n_x m_y)|^2 \frac{dx \; dy}{y^r} = |f|^2_{L^2}
$$
Let $f_r \in C^\infty_c(\Gamma \backslash G / K)$ with $f_r \to f$ in $L^2(\Gamma \backslash G / K)$. However, while $f \in L^2_a(\Gamma \backslash G / K)$, one would suspect that the constant terms of the $f_r$ are not too far from that of $f$ (and the difference must be going to zero, even in $L^2(\Gamma \backslash G / K)$), so that a smooth truncation of the constant terms of the $f_r$ should produce functions also approaching $f$. Namely, our strategy is to start with a general, and generic, approximating sequence in $L^2$ and remove the part that is keeping this sequence from being in $L^2_a$.

Using Iwasawa coordinates $n_x$ and $m_y$ with $x \in \mathbb{R}^{r-1}$ and $y \in (0, \infty)$, the height is $\eta(x,y) = y^r$. Let $\beta$ be a smooth function on $\mathbb{R}$ such that
$$
\begin{cases} 
0 = \beta (y) & (\text{for } 0 \leq y < -1) \\
0 \leq \beta (y) \leq 1  & (\text{for } -1 \leq y \leq 0) \\
1 =\beta (y)  & (\text{for } 0 \leq y) 
\end{cases}
$$
For $t > 1$, put $\beta_t(y) = \beta(t(y - a))$, and define a smooth function on $N \backslash G /K$ by
$$
\phi_{r,t}(x,y) = 
\begin{cases} 
\beta_t(y^r)\cdot c_P f_r(y) = \beta_t(\eta(x,y))\cdot c_P f_r(y) & (\text{for } y^r \geq a - \frac{1}{t}) \\
0   & (\text{for } y^r < a - \frac{1}{t})
\end{cases}
$$
For $t > 0$ large enough so that $\frak{S}_{a - \frac{1}{t}}$ meets its translates $\gamma \frak{S}_{a - \frac{1}{t}}$ only for $\gamma \in \Gamma_\infty$, let $\Psi_{r,t} = \Psi_{\phi_{r,t}}$ be the pseudo-Eisenstein series made from $\phi_{r,t}$:
$$
\Psi_{r,t}(x,y) = \sum_{\gamma \in \Gamma_\infty \backslash \Gamma} \phi_{r,t}(\gamma \cdot n_x m_y)
$$
The assumption on $t$ assures that in the region $y^r > a - \frac{1}{t}$ we have $\Psi_{r,t} = c_P \Psi_{r,t} = \phi_{n.t}$. Thus $c_P(f_r - \Psi_{r,t})$ vanishes in $y \geq a$, so $f_r - \Psi_{r,t} \in L^2_a(\Gamma \backslash G /K)$, as desired.
By the triangle inequality
$$
|f - (f_r -\Psi_{r,t})|_{L^2} \leq |f - f_r|_{L^2}  + |\Psi_{r,t}|_{L^2} 
$$
where by assumption $|f -f_r|_{L^2} \to 0$. Thus it suffices to show that the $L^2$ norm of $\Psi_{r,t}$ goes to $0$ for large $n$ and $t$. Since $a \gg 1$,
$$
|\Psi_{r,t}|_{L^2}  = |\Psi_{r,t}|_{C_{a - \frac{1}{t}}} = |\phi_{r,t}|_{C_{a - \frac{1}{t}}} = |\beta(t(y-a)) \cdot c_P f_r|_{C_{a - \frac{1}{t}}} \leq |c_P f_r|_{C_{a - \frac{1}{t}}}
$$
The cylinder $C_{a - \frac{1}{t}}$ admits a natural translation action of the product of circle groups (i.e., $\mathbb{T}^{r-1} \approx (N \cap \Gamma) \backslash N$), inherited from the translation of the $x$-component in Iwasawa coordinates $x,y$. This induces an action of $\mathbb{T}^{r-1}$ on $L^2(C_{a - \frac{1}{t}})$ with the norm $|\cdot |_{C_{a - \frac{1}{t}}}$. Thus, the map $f \to c_P f$ is given by a continuous, compactly-supported, $L^2(C_{a - \frac{1}{t}})$-valued integrand, which exists as a Gelfand-Pettis integral (cf. \S [14.1] in \cite{PBG}). This implies that the restriction of $c_P f_r$ to $C_{a - \frac{1}{t}}$ goes to $c_P f$ in $L^2(C_{a - \frac{1}{t}})$. As $c_P f$ is supported in the range $\eta(g) \leq a$ and the measure of $C_a - C_{a - \frac{1}{t}}$ goes to $0$ as $t \to +\infty$, the $C_{a - \frac{1}{t}}$-norm of $c_P f$ also goes to $0$ as $t \to +\infty$, since $c_P f$ is locally integrable. In particular, this implies the diagonal terms $\Psi_{n.n}$ go to $0$ in $L^2$ norm, so that $f_r - \Psi_{r,n}$ go to $f$ in $L^2$ norm, proving the density of $\mathscr{D}_a$ in $L^2_a$. $\square$
\end{proof}

%
%
\newpage
\subsection{$U(r,1)$}\label{Density-of-Automorphic-Test-Functions-U(r,1)}

\begin{lemma}\label{density-of-auto-test-fns-un1}
	For $a \gg 1 \; \mathscr{D}_a \text{ is dense in } L^2_a(\Gamma \backslash G /K)$
\end{lemma}
\begin{proof}
Recall that a (standard) Siegel set is a subset of $G$ given by a compact set $C \subset N$ and a (height) parameter $\ell$:
$$
\frak{S}_{\ell,C} = \{ na_y k \; : \; n \in C, k \in K, y \geq \ell \}
$$
Where $N$ is the unipotent radical of the standard minimal parabolic $P$. By reduction theory, take $C$ large enough and $\ell$ small enough so that $\frak{S} = \frak{S}_{\ell,C}$ surjects to $\Gamma\backslash G$. For $g \in G$, let $g = n_g m_g k_g$ be $g$'s Iwasawa coordinates, with $m_g \in A^+$. 
 By reduction theory (cf. \S [3.3] in \cite{PBG}), given $C$, there is a sufficiently large height $a \gg 1$ so that $\frak{S}_{a,C}$ has the property that $\frak{S}_{\ell,C} \cap \gamma \frak{S}_{a,C} \ne \emptyset$ implies $\gamma \in \Gamma \cap P = \Gamma_\infty$. In the following, $f \in L^2_a(\Gamma \backslash G /K)$ is first approximated by test functions $f_r \in C^\infty_c(\Gamma \backslash G /K)$ by general methods, and then the condition $a \gg 1$ is used in conjunction with a family of smooth cut-off functions of the constant term near height $a$, with the width of the cut-off region shrinking to $0$. Since $K$ has no role in the following, the Siegel sets can be construed as lying in $G/K$ and have Iwasawa coordinates $n_g a_g$:
$$G/K \approx U(r,1)/(U(n) \times U(1)) \approx NM/(M \cap K) \approx N \times A^+$$
\\
\\
By reduction theory, take $a \gg 1$ sufficiently large so that $\frak{S}_{a - \frac{1}{t}}$ meets its translates $\gamma \frak{S}_{a - \frac{1}{t}}$ only for $\gamma \in \Gamma \cap P$, for all sufficiently large $t$. Then the cylinder $C_{a - \frac{1}{t}} =  (\Gamma \cap P) \backslash \frak{S}_{a - \frac{1}{t}}$ injects to $\Gamma \backslash G /K$ which allows separation of variables in $\frak{S}_{a - \frac{1}{t}}$ in the sense that the cylinder $C_{a - \frac{1}{t}} =  (\Gamma \cap P) \backslash \frak{S}_{a - \frac{1}{t}}$ injects to $\Gamma \backslash G /K$:
$$
C_{a - \frac{1}{t}} =  (\Gamma \cap P) \backslash \frak{S}_{a - \frac{1}{t}} \; \approx \; (\Gamma \cap N) \backslash N \times (a - \frac{1}{t}, \infty)
$$
Let
$$
|f|^2_{a - \frac{1}{t}} = \int_{C_{a - \frac{1}{t}}} |f(n_{z,x} m_y)|^2 \frac{dz \; dx \; dy}{y^{2r+1}} \leq  \int_{\Gamma \backslash G /K} |f(n_{z,x} m_y)|^2 \frac{dz \; dx \; dy}{y^{2r+1}} = |f|^2_{L^2}
$$
Let $f_r \in C^\infty_c(\Gamma \backslash G / K)$ with $f_r \to f$ in $L^2(\Gamma \backslash G / K)$. However, while $f \in L^2_a(\Gamma \backslash G / K)$, it would seem likely that the constant terms of the $f_i$ are not too far from that of $f$. That is, by assumption,
\begin{align*}
|f - f_i|^2_{L^2} =& \int_{\Gamma \backslash G /K} |f(n_{z,x} m_y) - f_i(n_{z,x} m_y)|^2 \; \frac{dz \; dx \; dy}{y^{2r+1}}  \\
=& \int_{(\Gamma \cap N) \backslash N} \int_{A^+} |f(n_{z,x} m_y) - f_i(n_{z,x} m_y)|^2 \; \frac{dz \; dx \; dy}{y^{2r+1}} \\
=& \int_{A^+} \int_{(\Gamma \cap N) \backslash N}  |f(n_{z,x} m_y) - f_i(n_{z,x} m_y)|^2 \; \frac{dz \; dx \; dy}{y^{2r+1}} \\
=& \int_0^\infty \int_{(\Gamma \cap N) \backslash N}  |f(n_{z,x} m_y) - f_i(n_{z,x} m_y)|^2 \; \frac{dz \; dx \; dy}{y^{2r+1}} \\
\geq & \int_{y > a - \frac{1}{t}} \int_{(\Gamma \cap N) \backslash N}  |f(n_{z,x} m_y) - f_i(n_{z,x} m_y)|^2 \; \frac{dz \; dx \; dy}{y^{2r+1}} \\
= & \int_{y > a - \frac{1}{t}} \bigg( \int_{(\Gamma \cap N) \backslash N}  |f(n_{z,x} m_y) - f_i(n_{z,x} m_y)|^2 \; dz \; dx \bigg) \; \frac{dy}{y^{2r+1}} \\
= & \int_{y > a - \frac{1}{t}} |f - f_i|^2_{L^2((\Gamma \cap N) \backslash N)}(m_y) \; \frac{dy}{y^{2r+1}}
\end{align*}
That is, the integral 
$|f - f_i|^2$ on the compact manifold $(N \cap \Gamma)\backslash N$, as a function of the ray (i.e., $A^+$) variable  must be going to zero, implying that the integrand
$$
|f - f_i|^2_{L^2((\Gamma \cap N) \backslash N)}(m_y)
$$
considered as a function of the coordinate $y$ (and correspondingly constant on the transverse compact manifold {\it leaves}), must be going to zero in $L^2(\Gamma \backslash G / K)$. Thus, a smooth truncation of the constant terms of the $f_r$ should produce functions also approaching $f$. We start with a general, and generic, approximating sequence in $L^2(\Gamma \backslash G / K)$ and remove the part that is keeping this sequence from being in $L^2_a(\Gamma \backslash G / K)$.

Using Iwasawa coordinates $n_{z,x}$ and $m_y$ with $z \in \mathbb{C}^{r-1}$, $x \in \mathbb{R}$ and $y \in (0, \infty)$, the height is $\eta(n_{z,x} m_y) = \eta(z,x,y) = y^{2r+1}$. Let $\beta$ be a smooth function on $\mathbb{R}$ such that
$$
\begin{cases} 
0 = \beta (y) & (\text{for } 0 \leq y < -1) \\
0 \leq \beta (y) \leq 1  & (\text{for } -1 \leq y \leq 0) \\
1 =\beta (y)  & (\text{for } 1 \leq y) 
\end{cases}
$$
For $t > 1$, put $\beta_t(h) = \beta(t(h - a))$, and define a smooth function on $N \backslash G /K$ by
$$
\phi_{i,t}(n_{z,x} m_y) = 
\begin{cases} 
\beta_t(\eta(n_{z,x} m_y))\cdot c_P f_i(y) = \beta_t(y^{2r+1})\cdot c_P f_i(y)  & (\text{for } y^{2r+1} \geq a - \frac{1}{t}) \\
0   & (\text{for } y^{2r+1} < a - \frac{1}{t})
\end{cases}
$$
For $t > 0$ large enough so that $\frak{S}_{a - \frac{1}{t}}$ meets its translates $\gamma \frak{S}_{a - \frac{1}{t}}$ only for $\gamma \in \Gamma \cap P$, let $\Psi_{i,t} = \Psi_{\phi_{i,t}}$ be the pseudo-Eisenstein series:
$$
\Psi_{i,t}(n_{z,x} m_y) = \sum_{\gamma \in \Gamma_\infty \backslash \Gamma} \phi_{i,t}(\gamma \cdot n_{z,x} m_y)
$$
The assumption on $t$ assures that in the region $y^{2r+1} > a - \frac{1}{t}$ we have 
$$\Psi_{i,t} = c_P \Psi_{i,t} = \phi_{i,t}$$
Thus $c_P(f_i - \Psi_{i,t})$ vanishes in $y \geq a$, so $f_i - \Psi_{i,t} \in L^2_a(\Gamma \backslash G /K)$, as desired. By the triangle inequality
$$
|f - (f_i -\Psi_{i,t})|_{L^2} \leq |f - f_i|_{L^2}  + |\Psi_{i,t}|_{L^2} 
$$
where by assumption $|f -f_i|_{L^2} \to 0$. Thus it suffices to show that the $L^2$ norm of $\Psi_{i,t}$ goes to $0$ for large $i$ and $t$. Since $a \gg 1$,
$$
|\Psi_{i,t}|_{L^2}  = |\Psi_{i,t}|_{C_{a - \frac{1}{t}}} = |\phi_{i,t}|_{C_{a - \frac{1}{t}}} = |\beta(t(y-a)) \cdot c_P f_i|_{C_{a - \frac{1}{t}}} \leq |c_P f_i|_{C_{a - \frac{1}{t}}}
$$

The cylinder $C_{a - \frac{1}{t}} \approx (N \cap \Gamma) \backslash N \times (a - \frac{1}{t}, \infty)$ is isomorphic to the product of a compact manifold $(N \cap \Gamma) \backslash N$ times the ray $(a - \frac{1}{t}, \infty)$. In $L^2(C_{a - \frac{1}{t}})$ with the norm $|\cdot |_{C_{a - \frac{1}{t}}}$, since the integral of the $L^2((N \cap \Gamma)\backslash N)$ norm-squared goes to zero by the above, the map $f \to c_P f$ is given by a continuous, compactly-supported, $L^2(C_{a - \frac{1}{t}})$-valued integrand, which exists as a Gelfand-Pettis integral (cf. \S [14.1] in \cite{PBG}). This implies that the restriction of $c_P f_r$ to $C_{a - \frac{1}{t}}$ goes to $c_P f$ in $L^2(C_{a - \frac{1}{t}})$. As $c_P f$ is supported in the range $\eta(g) \leq a$ and the measure of $C_a - C_{a - \frac{1}{t}}$ goes to $0$ as $t \to +\infty$, the $C_{a - \frac{1}{t}}$-norm of $c_P f$ also goes to $0$ as $t \to +\infty$, since $c_P f$ is locally integrable. In particular, this implies the diagonal terms $\Psi_{i,i}$ go to $0$ in $L^2$ norm, so that $f_r - \Psi_{i,i}$ go to $f$ in $L^2$ norm, proving the density of $\mathscr{D}_a$ in $L^2_a$. $\square$
\end{proof}

%
%

\subsection{$Sp^\ast(r,1)$}\label{Density-of-Automorphic-Test-Functions-Sp*(r,1)}

\begin{lemma}\label{density-of-auto-test-fns-un1}
For $a \gg 1 \; \mathscr{D}_a \text{ is dense in } L^2_a(\Gamma \backslash G /K)$
\end{lemma}
\begin{proof}
Recall that a (standard) Siegel set is a subset of $G$ given by a compact set $C \subset N$ and a (height) parameter $\ell$:
$$
\frak{S}_{\ell,C} = \{ na_y k \; : \; n \in C, k \in K, y \geq \ell \}
$$
Where $N$ is the unipotent radical of the standard minimal parabolic $P$. By reduction theory, take $C$ large enough and $\ell$ small enough so that $\frak{S} = \frak{S}_{\ell,C}$ surjects to $\Gamma\backslash G$. For $g \in G$, let $g = n_g m_g k_g$ be $g$'s Iwasawa coordinates, with $m_g \in A^+$. 
By reduction theory (cf. \S [3.3] in \cite{PBG}), given $C$, there is a sufficiently large height $a \gg 1$ so that $\frak{S}_{a,C}$ has the property that $\frak{S}_{\ell,C} \cap \gamma \frak{S}_{a,C} \ne \emptyset$ implies $\gamma \in \Gamma \cap P = \Gamma_\infty$. In the following, $f \in L^2_a(\Gamma \backslash G /K)$ is first approximated by test functions $f_r \in C^\infty_c(\Gamma \backslash G /K)$ by general methods, and then the condition $a \gg 1$ is used in conjunction with a family of smooth cut-off functions of the constant term near height $a$, with the width of the cut-off region shrinking to $0$. Since $K$ has no role in the following, the Siegel sets can be construed as lying in $G/K$ and have Iwasawa coordinates $n_g a_g$:
$$G/K \approx Sp^\ast(r,1)/(Sp^\ast(r) \times Sp^\ast(1)) \approx NM/(M \cap K) \approx N \times A^+$$

By reduction theory, take $a \gg 1$ sufficiently large so that $\frak{S}_{a - \frac{1}{t}}$ meets its translates $\gamma \frak{S}_{a - \frac{1}{t}}$ only for $\gamma \in \Gamma \cap P$, for all sufficiently large $t$. Then the cylinder $C_{a - \frac{1}{t}} =  (\Gamma \cap P) \backslash \frak{S}_{a - \frac{1}{t}}$ injects to $\Gamma \backslash G /K$ which allows separation of variables in $\frak{S}_{a - \frac{1}{t}}$ in the sense that the cylinder $C_{a - \frac{1}{t}} =  (\Gamma \cap P) \backslash \frak{S}_{a - \frac{1}{t}}$ injects to $\Gamma \backslash G /K$:
$$
C_{a - \frac{1}{t}} =  (\Gamma \cap P) \backslash \frak{S}_{a - \frac{1}{t}} \; \approx \; (\Gamma \cap N) \backslash N \times (a - \frac{1}{t}, \infty)
$$
Let
$$
|f|^2_{a - \frac{1}{t}} = \int_{C_{a - \frac{1}{t}}} |f(n_{z,q,q^\prime,q^{\prime\prime}} m_y)|^2 \frac{dx \; dy}{y^{4r+3}} \leq  \int_{\Gamma \backslash G /K} |f(n_{z,q,q^\prime,q^{\prime\prime}} m_y)|^2 \frac{dx \; dy}{y^{4r+3}} = |f|^2_{L^2}
$$
Let $f_r \in C^\infty_c(\Gamma \backslash G / K)$ with $f_r \to f$ in $L^2(\Gamma \backslash G / K)$. However, while $f \in L^2_a(\Gamma \backslash G / K)$, it would seem likely that the constant terms of the $f_i$ are not too far from that of $f$. That is, by assumption,
\begin{align*}
|f - f_i|^2_{L^2} =& \int_{\Gamma \backslash G /K} |f(n_{z,q,q^\prime,q^{\prime\prime}} m_y) - f_i(n_{z,q,q^\prime,q^{\prime\prime}} m_y)|^2 \; \frac{dz\, dq\, dq^\prime\, dq^{\prime\prime}\, dy}{y^{4r+3}}  \\
=& \int_{(\Gamma \cap N) \backslash N} \int_{A^+} |f(n_{z,q,q^\prime,q^{\prime\prime}} m_y) - f_i(n_{z,q,q^\prime,q^{\prime\prime}} m_y)|^2 \; \frac{dz\, dq\, dq^\prime\, dq^{\prime\prime}\, dy}{y^{4r+3}} \\
=& \int_{A^+} \int_{(\Gamma \cap N) \backslash N}  |f(n_{z,q,q^\prime,q^{\prime\prime}} m_y) - f_i(n_{z,q,q^\prime,q^{\prime\prime}} m_y)|^2 \; \frac{dz\, dq\, dq^\prime\, dq^{\prime\prime}\, dy}{y^{4r+3}} \\
=& \int_0^\infty \int_{(\Gamma \cap N) \backslash N}  |f(n_{z,q,q^\prime,q^{\prime\prime}} m_y) - f_i(n_{z,q,q^\prime,q^{\prime\prime}} m_y)|^2 \; \frac{dz\, dq\, dq^\prime\, dq^{\prime\prime}\, dy}{y^{4r+3}} \\
\geq & \int_{y > a - \frac{1}{t}} \int_{(\Gamma \cap N) \backslash N}  |f(n_{z,q,q^\prime,q^{\prime\prime}} m_y) - f_i(n_{z,q,q^\prime,q^{\prime\prime}} m_y)|^2 \; \frac{dz\, dq\, dq^\prime\, dq^{\prime\prime}\, dy}{y^{4r+3}} \\
= & \int_{y > a - \frac{1}{t}} \bigg( \int_{(\Gamma \cap N) \backslash N}  |f(n_{z,q,q^\prime,q^{\prime\prime}} m_y) - f_i(n_{z,q,q^\prime,q^{\prime\prime}} m_y)|^2 \; dz\, dq\, dq^\prime\, dq^{\prime\prime}\bigg) \; \frac{dy}{y^{4r+3}} \\
= & \int_{y > a - \frac{1}{t}} |f - f_i|^2_{L^2((\Gamma \cap N) \backslash N)}(m_y) \; \frac{dy}{y^{4r+3}}
\end{align*}
That is, the integral $|f - f_i|^2$ on the compact manifold $(N \cap \Gamma)\backslash N$, as a function of the ray (i.e., $A^+$) variable  must be going to zero, implying that the integrand
$$
|f - f_i|^2_{L^2((\Gamma \cap N) \backslash N)}(m_y)
$$
considered as a function of the 
coordinate $y$ (and correspondingly constant on the transverse compact manifold {\it leaves}), must be going to zero in $L^2(\Gamma \backslash G / K)$. Thus, a smooth truncation of the constant terms of the $f_r$ should produce functions also approaching $f$. We start with a general, and generic, approximating sequence in $L^2(\Gamma \backslash G / K)$ and remove the part that is keeping this sequence from being in $L^2_a(\Gamma \backslash G / K)$.

Using Iwasawa coordinates with coordinates $z,q,q^\prime,q^{\prime\prime} \in N \approx  \mathbb{H}^{r-1} \times \mathbb{R}^3$, $y \in A^+ \approx  \mathbb{R}^\times \approx (0, \infty)$, the height is $\eta(n_{z,q,q^\prime,q^{\prime\prime}} m_y) = \eta(z,q,q^\prime,q^{\prime\prime},y) = y^{4r+3}$. Let $\beta$ be a smooth function on $\mathbb{R}$ such that
$$
\begin{cases} 
0 = \beta (y) & (\text{for } 0 \leq y < -1) \\
0 \leq \beta (y) \leq 1  & (\text{for } -1 \leq y \leq 0) \\
1 =\beta (y)  & (\text{for } 1 \leq y) 
\end{cases}
$$
For $t > 1$, put $\beta_t(h) = \beta(t(h - a))$, and define a smooth function on $N \backslash G /K$ by
$$
\phi_{i,t}(n_{z,q,q^\prime,q^{\prime\prime}} m_y) = 
\begin{cases} 
\beta_t(\eta(n_{z,q,q^\prime,q^{\prime\prime}} m_y))\cdot c_P f_i(y) = \beta_t(y^{4r+3})\cdot c_P f_i(y)  & (\text{for } y^{4r+3} \geq a - \frac{1}{t}) \\
0   & (\text{for } y^{4r+3} < a - \frac{1}{t})
\end{cases}
$$
For $t > 0$ large enough so that $\frak{S}_{a - \frac{1}{t}}$ meets its translates $\gamma \frak{S}_{a - \frac{1}{t}}$ only for $\gamma \in \Gamma \cap P$, let $\Psi_{i,t} = \Psi_{\phi_{i,t}}$ be the pseudo-Eisenstein series:
$$
\Psi_{i,t}(n_{z,x} m_y) = \sum_{\gamma \in \Gamma_\infty \backslash \Gamma} \phi_{i,t}(\gamma \cdot n_{z,x} m_y)
$$
The assumption on $t$ assures that in the region $y^{2r+1} > a - \frac{1}{t}$ we have 
$$\Psi_{i,t} = c_P \Psi_{i,t} = \phi_{i,t}$$
Thus $c_P(f_i - \Psi_{i,t})$ vanishes in $y \geq a$, so $f_i - \Psi_{i,t} \in L^2_a(\Gamma \backslash G /K)$, as desired. By the triangle inequality
$$
|f - (f_i -\Psi_{i,t})|_{L^2} \leq |f - f_i|_{L^2}  + |\Psi_{i,t}|_{L^2} 
$$
where by assumption $|f -f_i|_{L^2} \to 0$. Thus it suffices to show that the $L^2$ norm of $\Psi_{i,t}$ goes to $0$ for large $i$ and $t$. Since $a \gg 1$,
$$
|\Psi_{i,t}|_{L^2}  = |\Psi_{i,t}|_{C_{a - \frac{1}{t}}} = |\phi_{i,t}|_{C_{a - \frac{1}{t}}} = |\beta(t(y-a)) \cdot c_P f_i|_{C_{a - \frac{1}{t}}} \leq |c_P f_i|_{C_{a - \frac{1}{t}}}
$$
The cylinder $C_{a - \frac{1}{t}} \approx (N \cap \Gamma) \backslash N \times (a - \frac{1}{t}, \infty)$ is isomorphic to the product of a compact manifold $(N \cap \Gamma) \backslash N$ times the ray $(a - \frac{1}{t}, \infty)$. In $L^2(C_{a - \frac{1}{t}})$ with the norm $|\cdot |_{C_{a - \frac{1}{t}}}$, since the integral of the $L^2((N \cap \Gamma)\backslash N)$ norm-squared goes to zero by the above, the map $f \to c_P f$ is given by a continuous, compactly-supported, $L^2(C_{a - \frac{1}{t}})$-valued integrand, which exists as a Gelfand-Pettis integral (cf. \S [14.1] in \cite{PBG}). This implies that the restriction of $c_P f_r$ to $C_{a - \frac{1}{t}}$ goes to $c_P f$ in $L^2(C_{a - \frac{1}{t}})$. As $c_P f$ is supported in the range $\eta(g) \leq a$ and the measure of $C_a - C_{a - \frac{1}{t}}$ goes to $0$ as $t \to +\infty$, the $C_{a - \frac{1}{t}}$-norm of $c_P f$ also goes to $0$ as $t \to +\infty$, since $c_P f$ is locally integrable. In particular, this implies the diagonal terms $\Psi_{i,i}$ go to $0$ in $L^2$ norm, so that $f_r - \Psi_{i,i}$ go to $f$ in $L^2$ norm, proving the density of $\mathscr{D}_a$ in $L^2_a$. $\square$
\end{proof}

%
%

\section{$L^2$ norms of truncated tails go to $0$ strongly}\label{sec:L2-norms-zero-strongly}
Recall that
\begin{align*}
|f|^2_{\frak{B}^1} = & \langle (1 - \Delta)f , f \rangle = \langle f ,f \rangle + \langle (-\Delta)f,f \rangle \\
\frak{B}^1 \equiv & \text{ completion of } C^\infty_c(\Gamma \backslash G / K) \text{ with respect to the } \frak{B}_1 \text{ norm} \\
L^2_a(\Gamma \backslash G /K) = & \; \{ f \in L^2(\Gamma \backslash G /K) \; : \; c_P f (g) = 0 \text{ for } \eta(g) \geq a \} \\
\mathscr{D}_a = & \; C^\infty_c(\Gamma \backslash G /K) \cap L^2_a(\Gamma \backslash G /K) 
\end{align*}
%

%
%

\subsection{$O(r,1)$}\label{L2-norms-zero-strongly-O(r,1)}

\begin{lemma}\label{sec:L2a-tail-norm-bded-by-B1a-norm}
Let $B$ be the unit ball in $\frak{B}^1_a(\Gamma \backslash G /K)$ then, given $\varepsilon > 0$, a cutoff $c \gg a$ can be made sufficiently large so that the image of $B$ in $L^2_a(\Gamma \backslash G /K)$ lies in a single $\varepsilon$-ball in $L^2_a(\Gamma \backslash G /K)$. That is, for $f \in \frak{B}^1_a(\Gamma \backslash G /K)$
$$
\lim_{c \to \infty}  \int_{N_\mathbb{Z} \backslash N_\mathbb{R}} \int_{y > c} |f(n_x m_y)|^2 \; \frac{dx \; dy}{y^r} \longrightarrow 0 \enspace (\text{uniformly for } |f|_{\frak{B}^1_a} \leq 1)
$$
The following stronger estimate will also be shown: that for suitably large $y > c \gg 1$
$$
\int_{(N \cap \Gamma) \backslash N} \int_{y > c} |f(n_x m_y)|^2 \; \frac{dx \; dy}{y^r} \ll \frac{1}{c^2} \cdot |f|^2_{\frak{B}^1_a}
$$
\end{lemma}
\begin{proof}
%
%
Let $\xi$ run over characters of $N_\mathbb{Z} \backslash N_\mathbb{R} \approx \mathbb{T}^{r-1}$ and take height $c \geq c_0 \geq a \gg 1$. With Iwasawa coordinates $x,y$, write the Fourier expansion in $x$ as 
$$
f(x,y) = \sum_\xi \hat{f}(\xi, y) \xi(x) \enspace \bigg(\; = \sum_{\xi \in \mathbb{Z}^{r-1}} \hat{f}(\xi, y) e^{2 \pi i \xi \cdot x} \bigg)
$$
Since $f \in \frak{B}^1_a$, in Iwasawa coordinates $n_x , m_y$
$$
c_P f(x,y) = \int_{(N \cap \Gamma)\backslash N} f(x,y) dn_x \; = \int_{((N \cap \Gamma))\backslash N} f(x,y) e^{2 \pi i \; 0 \cdot x}dn_x \; = \;\hat{f}(0,y) \;\; ( 0 \in \mathbb{Z}^{r-1} )
$$
so that $\hat{f}(0, y) = 0$ when $y \geq c \gg a$. By Plancherel in $x$ 
$$
\int_{N_\mathbb{Z} \backslash N_\mathbb{R}} \int_{y \geq c} |f|^2 \ \frac{dx \; dy}{y^r} = \int_{(N \cap \Gamma) \backslash N} \int_{y > c} |f(n_x m_y)|^2 \; \frac{dx \; dy}{y^r} = \sum_{\xi \in \mathbb{Z}^{r-1}} \int_{y \geq c} |\hat{f}(\xi, y)|^2 \ \frac{dy}{y^r}
$$
Since $\hat{f}(0, y) = 0$ when $y \geq c$, the sum is over $\xi \neq 0$, so that $|\xi| \geq 1$ and
$$
\sum_\xi \int_{y \geq c} |\hat{f}(\xi, y)|^2 \ \frac{dy}{y^r} \; \ll \; \sum_\xi \int_{y \geq c} |\xi|^2 \cdot|\hat{f}(\xi, y)|^2 \ \frac{dy}{y^r}
$$
With $\Delta_x$ the Euclidean Laplacian in $x$,
$$
|\xi|^2 \cdot \hat{f}(\xi,y) = \frac{1}{4 \pi^2} (- \Delta_x f)\widehat{\;\;}(\xi,y) \ll (- \Delta_x f)\widehat{\;\;}(\xi,y)
$$
Substituting this back and applying Plancherel
\begin{align*}
\sum_\xi \int_{y \geq c} |\xi|^2 \cdot|\hat{f}(\xi, y)|^2 \ \frac{dy}{y^r} \; \ll & \; \sum_\xi \int_{y \geq c} (- \Delta_x f)\widehat{\;\;}(\xi,y) \overline{\widehat{f}}(\xi, y)  \frac{dy}{y^r} \\
=& \; \int_{(N \cap \Gamma) \backslash N} \int_{y > c} - \Delta_x f \cdot \bar{f} \;\; \frac{dx \; dy}{y^r}
\end{align*}
Again using that $y > c \gg a \gg 1$
$$
\int_{(N \cap \Gamma) \backslash N} \int_{y > c} - \Delta_x f \cdot \bar{f} \;\; \frac{dx \; dy}{y^r} \; \leq \; \frac{1}{c^2} \int_{(N \cap \Gamma) \backslash N} \int_{y > c} - y^2 \Delta_x f \cdot \bar{f} \;\; \frac{dx \; dy}{y^r}
$$
Recall the positivity result
$$
0 \leq \int -\Big(y^2\frac{\partial^2}{\partial y^2} - (r-2)y\frac{\partial}{\partial y} \Big)f \cdot \bar{f} \ \frac{dx \, dy}{y^r} 
$$
so that
\begin{align*}
& \frac{1}{c^2} \int_{(N \cap \Gamma) \backslash N} \int_{y > c} - y^2 \Delta_x f \cdot \bar{f} \;\; \frac{dx \; dy}{y^r} \\
& \leq \frac{1}{c^2} \int_{(N \cap \Gamma) \backslash N} \int_{y > c} \bigg( - y^2 \Delta_x f - y^2\frac{\partial^2}{\partial y^2} + (r-2)y\frac{\partial}{\partial y})f \bigg) \cdot \bar{f} \;\; \frac{dx \; dy}{y^r}
\end{align*}
Substituting back, for smooth $f$ with support in $y \geq c \gg a$ 
$$
\int_{(N \cap \Gamma) \backslash N} \int_{y > c} |f(n_x m_y)|^2 \; \frac{dx \; dy}{y^r} \;\; \ll \;\;  \frac{1}{c^2} \int_{(N \cap \Gamma) \backslash N} \int_{y > c} -\Delta f \cdot \bar{f} \;\; \frac{dx \; dy}{y^r} 
$$
Of course, also
$$
0 \leq \frac{1}{c^2} \int_{(N \cap \Gamma) \backslash N} \int_{y > c} |f(x,y)|^2 \; \frac{dx \; dy}{y^r}
$$
So adding this to the right side above gives
$$
\int_{(N \cap \Gamma) \backslash N} \int_{y > c} |f(n_x m_y)|^2 \; \frac{dx \; dy}{y^r} \ll  \frac{1}{c^2} \int_{(N \cap \Gamma) \backslash N} \int_{y > c} (1 -\Delta) f \cdot \bar{f} \;\; \frac{dx \; dy}{y^r}  \leq \frac{1}{c^2} \cdot |f|^2_{\frak{B}^1_a}
$$
as claimed. $\square$
\end{proof}

%
%

\subsection{$U(r,1)$}\label{L2-norms-zero-strongly-U(r,1)}

\begin{lemma}\label{sec:L2a-tail-norm-bded-by-B1a-norm-U(r,1)}
Let $B$ be the unit ball in $\frak{B}^1_a(\Gamma \backslash G /K)$ then, given $\varepsilon > 0$, a cutoff $c \gg a$ can be made sufficiently large so that the image of $B$ in $L^2_a(\Gamma \backslash G /K)$ lies in a single $\varepsilon$-ball in $L^2_a(\Gamma \backslash G /K)$. That is, for $f \in \frak{B}^1_a(\Gamma \backslash G /K)$
$$
\lim_{c \to \infty}  \int_{(N \cap \Gamma)\backslash N} \int_{y > c} |f(n_{z,x} m_y)|^2 \; \frac{dz \; dx \; dy}{y^{2r+1}} \longrightarrow 0 \enspace (\text{uniformly for } |f|_{\frak{B}^1_a} \leq 1)
$$
The following stronger estimate will also be shown: that for suitably large $y > c \gg 1$
$$
\int_{(N \cap \Gamma) \backslash N} \int_{y > c} |f(n_{z,x} m_y)|^2 \; \frac{dz \; dx \; dy}{y^{2r+1}} \ll \frac{1}{c^2} \cdot |f|^2_{\frak{B}^1_a}
$$
\end{lemma}
\begin{proof}
Since $N \subset U(r,1)$ is no longer commutative, we cannot follow the approach as in $O(r,1)$ where a literal Fourier series on $(N \cap \Gamma)\backslash N) \approx \mathbb{Z}^{r-1}\backslash \mathbb{R}^{r-1} \approx \mathbb{T}^{r-1}$ was used. However, $(N \cap \Gamma)\backslash N$ is still a (smooth) compact manifold. Recall the coordinates used on $N$:
$$
N \enspace = \enspace \{ n_{z,x} \enspace = \enspace 
\renewcommand\arraystretch{1.5}
\begin{pmatrix}
1 & z & -\frac{1}{2}|z|^2 + ix \\
0 & 1_{r-1} & -z^\ast \\
0 & 0 & 1
\end{pmatrix}
\enspace : \enspace z \in \mathbb{C}^{r-1} \enspace z = u + iv \in \mathbb{R}^{r-1} \oplus \; i \mathbb{R}^{r-1} \enspace x \in \mathbb{R}\}
$$
and that the Laplacian of $G/K$ in these coordinates is
$$
\Delta = y^2 \bigg( \Delta_u + \Delta_v + \frac{\partial^2}{\partial y^2}+ y^2 \frac{\partial^2}{\partial x^2} \bigg) - (2r-1) y \frac{\partial}{\partial y}
$$
The non-compact component of the quotient corresponding to the height is
$$
 Y_\infty =  \big\{ \Gamma\backslash U(r,1) / K \; : \; \eta(g) = \eta(n_{z,x}m_y) = y^{2r+1} > a \big\} \approx (N \cap \Gamma) \backslash N \times (a,\infty)
$$
While $Y_\infty$ is homeomorphic to the product $(N \cap \Gamma) \backslash N \times (a,\infty)$, the implicit geometry of the ``leaves'' (i.e., corresponding to $(N \cap \Gamma) \backslash N$) includes a dependency on the $y$ parameter from the complementary ray $(a, \infty) \approx \mathbb{R}^+$. 
Denote the compact manifold $(N \cap \Gamma) \backslash N$ by $X$.  $\Delta$ decomposes into a sum of components tangential to the factors so that by re-arranging terms in $\Delta$ we can express it in terms of derivatives tangential and transverse to $N$ (and thus also to $(N \cap \Gamma)\backslash N = X$), giving
$$
\Delta = y^2 \underbrace{\bigg( \Delta_u + \Delta_v + y^2 \frac{\partial^2}{\partial x^2} \bigg)}_\text{$S_y$ tangential to $N$} + \underbrace{\bigg(y^2\frac{\partial^2}{\partial y^2} - (2r-1) y \frac{\partial}{\partial y} \bigg)}_\text{$\partial^2_y$ transverse to $N$}
$$
For each $y \in (a,\infty)$, $S_y$ is a symmetric semi-bounded operator and has compact resolvent. The dependence of $S_y$ on the coordinate $y$ from $M$ requires some attention to make sure eigenvalues of $S_y$ are uniformly bounded away from zero for all $y \in (a, \infty)$. 
To this end, let $\lambda$ be the greatest number such that $\langle -S_y v,v\rangle \ge \lambda \langle v,v \rangle$:
$$
\lambda = \sup_c \enspace \langle -S_y v,v\rangle \ge c \langle v,v \rangle \text{ for } v \in C^\infty_c(X) \text{ and } v \perp 1
$$
Since $-S_y$ has positive discrete spectrum, $\lambda$ is well-defined and we claim that $\lambda$ is a lower bound for the non-zero eigenvalues of $-S_y$.
On the collar $X \times (a, \infty)$, rewrite $S_y$ as the sum of a symmetric operator independent of $y$ and a non-negative symmetric operator
$$
S_y = \underbrace{(\Delta_u + \Delta_v + a^2 \frac{\partial^2}{\partial x^2})}_\text{$S_a$ independent of $y$} + \underbrace{(y^2 - a^2)\frac{\partial^2}{\partial x^2}}_\text{$T$} = S_a + T
$$ 
Substituting $S_a$ and $T$
$$
\langle -S_y v,v\rangle = \langle -(S_a + T) v, v\rangle = \langle -S_a v, v\rangle + \langle -Tv, v\rangle \geq \langle -S_a v, v \rangle
$$
by the non-negativity of $-T$.
The operator $-S_a$ is also a positive, symmetric operator and thus $L^2_a$ decomposes purely discretely for $-S_a$. Label the eigenfunctions and eigenvalues of $-S_a$ as $\widetilde{\lambda}_j$ and $\widetilde{\phi}_j$. The Gelfand condition for $y \gg a$ implies the $0^\text{th}$ coefficient $\widetilde{c}_0$ in the expansion relative to $S_a$ vanishes uniformly (also true for $S_y$). That is, the compactness of $X$ implies $1 \in L^2(X)$, and since $S_a$ has no constant term,  the constant function $1$ is an eigenfunction of $S_a$ corresponding to the first eigenvalue $\lambda_0 = 0$ since the inner-product of the restriction of $f$ to $(N \cap \Gamma)\backslash N = X$ with the constant function $1 = \widetilde{\phi_0}$
\begin{align*}
c_P f(m_y) =& \; 0 = \int_{(N \cap \Gamma) \backslash N} f(n_{w,t} \cdot n_{z,x} m_y) \; dn_{w,t}  = \int_{(N \cap \Gamma) \backslash N} f(n_{z,x} m_y) \; dn_{z,x} \\
=& \int_{(N \cap \Gamma) \backslash N} f(n_{z,x} m_y) \cdot \bar{1} \; dn_{z,x} = \widetilde{c_0}(m_y) \quad (\text{i.e., for } y \gg a)
\end{align*}
 Similarly define $\widetilde{\lambda}$ by
$$
\widetilde{\lambda} = \sup_c \enspace \langle -S_a v,v\rangle \ge c \langle v,v \rangle \text{ for } v \in C^\infty_c(X) \text{ and } v \perp 1
$$
In particular, $\widetilde{\lambda}$ is independent of $y$ since $S_a$ is. We then have for all $v \in C^\infty_c(X)$ and $v \perp 1$:
$$
\langle -S_y v,v\rangle = \langle -(S_a + T) v, v\rangle = \langle -S_a v, v\rangle + \langle -Tv, v\rangle \geq \langle -S_a v, v \rangle \geq \widetilde{\lambda} \langle v, v\rangle
$$
but $\lambda$ was the greatest value satisfying this bound for $S_y$ so we must have 
$$
\lambda \geq \widetilde{\lambda} > 0
$$
independent of the $y$ coordinate. For $f \in C^\infty_c(\Gamma\backslash G/K) \cap L^2_a(\Gamma\backslash G/K)$, on the collar $X \times [a, \infty)$:
\begin{align*}	
\int_{(N \cap \Gamma) \backslash N} \int_{y \geq c} |f|^2 & \ \frac{dz \; dx \; dy}{y^{2r+1}} = \int_{y \geq c} \bigg( \int_{(N \cap \Gamma) \backslash N} f \cdot \bar{f} \, dz \; dx \bigg) \;\frac{ dy}{y^{2r+1}}  = \int_{y \geq c} \langle f, f \rangle_X \; \frac{dy}{y^{2r+1}} \\
\leq &   \int_{y \geq c}  \frac{1}{\lambda} \langle -S_y f, f \rangle_X \; \frac{dy}{y^{2r+1}} \leq \int_{y \geq c}  \frac{1}{\widetilde{\lambda}} \langle -S_y f, f \rangle_X \; \frac{dy}{y^{2r+1}}\\
=& \frac{1}{\widetilde{\lambda} } \int_{(N \cap \Gamma) \backslash N} \int_{y \geq c} -S_y f \cdot \bar{f} \; \; \frac{dz \; dx \; dy}{y^{2r+1}} \\
 \leq & \; \frac{1}{c^2} \frac{1}{\widetilde{\lambda} }  \int_{(N \cap \Gamma) \backslash N} \int_{y \geq c} - y^2 S_y f \cdot \bar{f} \; \; \frac{dz \; dx \; dy}{y^{2r+1}} \qquad (y > c )\\
\leq & \; \frac{1}{c^2} \frac{1}{\widetilde{\lambda}}  \int_{(N \cap \Gamma) \backslash N} \int_{y \geq c} - y^2 S_y f \cdot \bar{f} - \partial_y^2 f \cdot \bar{f} \; \; \frac{dz \; dx \; dy}{y^{2r+1}} \qquad ({\scriptstyle\text{positivity of fragments}})\\
=& \; \frac{1}{c^2} \frac{1}{\widetilde{\lambda} } \int_{(N \cap \Gamma) \backslash N} \int_{y \geq c} - (y^2 S_y + \partial_y^2) f \cdot \bar{f}  \; \; \frac{dz \; dx \; dy}{y^{2r+1}}   \\
=& \; \frac{1}{c^2} \frac{1}{\widetilde{\lambda}}  \int_{(N \cap \Gamma) \backslash N} \int_{y \geq c} - \Delta f \cdot \bar{f}  \; \; \frac{dz \; dx \; dy}{y^{2r+1}} \\
\leq & \; \frac{1}{c^2} \frac{1}{\widetilde{\lambda}} \int_{(N \cap \Gamma) \backslash N} \int_{y \geq c} - \Delta f \cdot \bar{f} + f\cdot \bar{f}  \; \; \frac{dz \; dx \; dy}{y^{2r+1}} \qquad (|f|^2 \geq 0)\\
=& \; \frac{1}{c^2} \frac{1}{\widetilde{\lambda}} \int_{(N \cap \Gamma) \backslash N} \int_{y \geq c} (1 - \Delta) f \cdot \bar{f}  \; \; \frac{dz \; dx \; dy}{y^{2r+1}} = \frac{1}{c^2} \big( \frac{1}{\widetilde{\lambda}} |f|^2_{\frak{B}^1_a} \big) \ll \frac{1}{c^2} |f|^2_{\frak{B}^1_a}
\end{align*}
Which is the claimed bound on the $L^2$ tails. $\square$
\end{proof}

%
%

\subsection{$Sp^\ast(r,1)$}\label{L2-norms-zero-strongly-Sp*(r,1)}

\begin{lemma}\label{sec:L2a-tail-norm-bded-by-B1a-norm}
Let $B$ be the unit ball in $\frak{B}^1_a(\Gamma \backslash G /K)$ then, given $\varepsilon > 0$, a cutoff $c \gg a$ can be made sufficiently large so that the image of $B$ in $L^2_a(\Gamma \backslash G /K)$ lies in a single $\varepsilon$-ball in $L^2_a(\Gamma \backslash G /K)$. That is, for $f \in \frak{B}^1_a(\Gamma \backslash G /K)$
$$
\lim_{c \to \infty}  \int_{N_\mathbb{Z} \backslash N_\mathbb{R}} \int_{y > c} |f(n_{z,q,q^\prime,q^{\prime\prime}} m_y)|^2 \; \frac{dz \; dq\, dq^\prime\, dq^{\prime\prime}\, dy}{y^{4r+3}} \longrightarrow 0 \enspace (\text{uniformly for } |f|_{\frak{B}^1_a} \leq 1)
$$
The following stronger estimate will also be shown: that for suitably large $y > c \gg 1$
$$
\int_{(N \cap \Gamma) \backslash N} \int_{y > c} |f(n_{z,q,q^\prime,q^{\prime\prime}} m_y)|^2 \; \frac{dz \; dq\, dq^\prime\, dq^{\prime\prime}\, dy}{y^{4r+3}} \ll \frac{1}{c^2} \cdot |f|^2_{\frak{B}^1_a}
$$
\end{lemma}
\begin{proof}
%
%
%
%
As with $U(r,1)$, since $N \subset Sp^\ast(r,1)$ is no longer commutative, we cannot follow the approach as in $O(r,1)$ where a literal Fourier series on $(N \cap \Gamma)\backslash N) \approx \mathbb{Z}^{r-1}\backslash \mathbb{R}^{r-1} \approx \mathbb{T}^{r-1}$ was used. However, $(N \cap \Gamma)\backslash N)$ is again a compact Riemannian manifold; recall the coordinates used on $N$:
$$
N \enspace = \enspace \{ n_{z,q,q^\prime,q^{\prime\prime}} \enspace = \enspace 
\renewcommand\arraystretch{1.5}
\begin{pmatrix}
1 & z & -\frac{1}{2}|z|^2 + iq + jq^\prime + kq^{\prime\prime} \\
0 & 1_{r-1} & -z^\ast \\
0 & 0 & 1
\end{pmatrix}
\}
$$
where
$$
z = x+ iw + ju + kv \in \mathbb{R}^{r-1} \oplus \; i \mathbb{R}^{r-1} \oplus \; j \mathbb{R}^{r-1} \oplus \; k \mathbb{R}^{r-1} \approx \mathbb{H}^{r-1} \quad  q,q^\prime,q^{\prime\prime} \in \mathbb{R}
$$
and that the Laplacian in these coordinates is
$$
\Delta = \Omega_\frak{g} =  y^2 \big(  \Delta_x + \Delta_w + \Delta_u + \Delta_v + \frac{\partial^2}{\partial y^2} + y^2 \frac{\partial^2}{\partial q^2} + y^2 \frac{\partial^2}{\partial q^{\prime \, 2}} + y^2 \frac{\partial^2}{\partial q^{\prime\prime \, 2}} ) - (4r+1)  y \frac{\partial}{\partial y} 
$$
The quotient $\Gamma\backslash U(r,1) / K \approx (N \cap \Gamma) \backslash N \times (0,\infty)$ and the product implies that $\Delta$ decomposes into a sum of components tangential to the factors. By re-arranging terms in $\Delta$ and expressing it in terms of derivatives tangential and transverse to $N$ (and thus also to $(N \cap \Gamma)\backslash N$), we have an expression for the Laplacian:
$$
\Delta =  y^2 \underbrace{\bigg(\Delta_x + \Delta_w + \Delta_u + \Delta_v + y^2 \frac{\partial^2}{\partial q^2} + y^2 \frac{\partial^2}{\partial q^{\prime \, 2}} + y^2 \frac{\partial^2}{\partial q^{\prime\prime \, 2}} \bigg)}_\text{$S_y$ tangential to $N$} + \underbrace{\bigg(\frac{\partial^2}{\partial y^2}  - (4r+1) y \frac{\partial}{\partial y} \bigg)}_\text{$\partial^2_y$ transverse to $N$}
$$

For each $y \in (a,\infty)$, $S_y$ is a symmetric semi-bounded operator and has compact resolvent. The dependence of $S_y$ on the coordinate $y$ from $M$ requires some attention to make sure eigenvalues of $S_y$ are uniformly bounded away from zero for all $y \in (a, \infty)$. 
To this end, let $\lambda$ be the greatest number such that $\langle -S_y v,v\rangle \ge \lambda \langle v,v \rangle$:
$$
\lambda = \sup_c \enspace \langle -S_y v,v\rangle \ge c \langle v,v \rangle \text{ for } v \in C^\infty_c(X) \text{ and } v \perp 1
$$
Since $-S_y$ has positive discrete spectrum, $\lambda$ is well-defined and we claim that $\lambda$ is a lower bound for the non-zero eigenvalues of $-S_y$.
On the collar $X \times (a, \infty)$, rewrite $S_y$ as the sum of a symmetric operator independent of $y$ and a non-negative symmetric operator
\begin{align*}
S_y =& \underbrace{(\Delta_x + \Delta_w + \Delta_u + \Delta_v +a^2 \frac{\partial^2}{\partial q^2} + a^2 \frac{\partial^2}{\partial q^{\prime \, 2}} + a^2 \frac{\partial^2}{\partial q^{\prime\prime \, 2}} )}_\text{$S_a$ independent of $y$} + \underbrace{(y^2 - a^2)( \frac{\partial^2}{\partial q^2} + \frac{\partial^2}{\partial q^{\prime \, 2}} + \frac{\partial^2}{\partial q^{\prime\prime \, 2}} )}_\text{$T$} \\
=& S_a + T
\end{align*}
Substituting $S_a$ and $T$
$$
\langle -S_y v,v\rangle = \langle -(S_a + T) v, v\rangle = \langle -S_a v, v\rangle + \langle -Tv, v\rangle \geq \langle -S_a v, v \rangle
$$
by the non-negativity of $-T$.
The operator $-S_a$ is also a positive, symmetric operator and thus $L^2_a$ decomposes purely discretely for $-S_a$. Label the eigenfunctions and eigenvalues of $-S_a$ as $\widetilde{\lambda}_j$ and $\widetilde{\phi}_j$. The Gelfand condition for $y \gg a$ implies the $0^\text{th}$ coefficient $\widetilde{c}_0$ in the expansion relative to $S_a$ vanishes uniformly (also true for $S_y$). That is, the compactness of $X$ implies $1 \in L^2(X)$, and since $S_a$ has no constant term,  the constant function $1$ is an eigenfunction of $S_a$ corresponding to the first eigenvalue $\lambda_0 = 0$ since the inner-product of the restriction of $f$ to $(N \cap \Gamma)\backslash N = X$ with the constant function $1 = \widetilde{\phi_0}$
\begin{align*}
c_P f(m_y) =& \; 0 = \int_{(N \cap \Gamma) \backslash N} f(n_{w,f,g,h,t} \cdot n_{z,q,q^\prime,q^{\prime\prime}} m_y) \; dn_{w,f,g,h,t}  = \int_{(N \cap \Gamma) \backslash N} f(n_{z,q,q^\prime,q^{\prime\prime}} m_y) \; dn_{z,q,q^\prime,q^{\prime\prime}} \\
=& \int_{(N \cap \Gamma) \backslash N} f(n_{z,q,q^\prime,q^{\prime\prime}} m_y) \cdot \bar{1} \; dn_{z,q,q^\prime,q^{\prime\prime}} = \widetilde{c_0}(m_y) \quad (\text{i.e., for } y \gg a)
\end{align*}
 Similarly define $\widetilde{\lambda}$ by
$$
\widetilde{\lambda} = \sup_c \enspace \langle -S_a v,v\rangle \ge c \langle v,v \rangle \text{ for } v \in C^\infty_c(X) \text{ and } v \perp 1
$$
In particular, $\widetilde{\lambda}$ is independent of $y$ since $S_a$ is. We then have for all $v \in C^\infty_c(X)$ and $v \perp 1$:
$$
\langle -S_y v,v\rangle = \langle -(S_a + T) v, v\rangle = \langle -S_a v, v\rangle + \langle -Tv, v\rangle \geq \langle -S_a v, v \rangle \geq \widetilde{\lambda} \langle v, v\rangle
$$
but $\lambda$ was the greatest value satisfying this bound for $S_y$ so we must have 
$$
\lambda \geq \widetilde{\lambda} > 0
$$
independent of the $y$ coordinate. For $f \in C^\infty_c(\Gamma\backslash G/K) \cap L^2_a(\Gamma\backslash G/K)$, on the collar $X \times [a, \infty)$:
\begin{align*}	
\int_{(N \cap \Gamma) \backslash N} & \int_{y \geq c} |f|^2 \ \frac{dz \; dq\, dq^\prime\, dq^{\prime\prime}\, dy}{y^{4r+3}} = \int_{y \geq c} \bigg( \int_{(N \cap \Gamma) \backslash N} f \cdot \bar{f} \, dz \; dq\, dq^\prime\, dq^{\prime\prime} \bigg) \;\frac{ dy}{y^{4r+3}}  \\
=& \int_{y \geq c} \langle f, f \rangle_X \; \frac{dy}{y^{4r+3}} \\
\leq &   \int_{y \geq c}  \frac{1}{\lambda} \langle -S_y f, f \rangle_X \; \frac{dy}{y^{4r+3}} \leq \int_{y \geq c}  \frac{1}{\widetilde{\lambda}} \langle -S_y f, f \rangle_X \; \frac{dy}{y^{4r+3}}\\
=& \frac{1}{\widetilde{\lambda} } \int_{(N \cap \Gamma) \backslash N} \int_{y \geq c} -S_y f \cdot \bar{f} \; \; \frac{dz \; dq\, dq^\prime\, dq^{\prime\prime}\, dy}{y^{4r+3}} \\
 \leq & \; \frac{1}{c^2} \frac{1}{\widetilde{\lambda} }  \int_{(N \cap \Gamma) \backslash N} \int_{y \geq c} - y^2 S_y f \cdot \bar{f} \; \; \frac{dz \; dq\, dq^\prime\, dq^{\prime\prime}\, dy}{y^{4r+3}} \qquad (y > c )\\
\leq & \; \frac{1}{c^2} \frac{1}{\widetilde{\lambda}}  \int_{(N \cap \Gamma) \backslash N} \int_{y \geq c} - y^2 S_y f \cdot \bar{f} - \partial_y^2 f \cdot \bar{f} \; \; \frac{dz \; dq\, dq^\prime\, dq^{\prime\prime}\, dy}{y^{4r+3}} \enspace ({\scriptstyle\text{positivity of fragments}})\\
=& \; \frac{1}{c^2} \frac{1}{\widetilde{\lambda} } \int_{(N \cap \Gamma) \backslash N} \int_{y \geq c} - (y^2 S_y + \partial_y^2) f \cdot \bar{f}  \; \; \frac{dz \; dq\, dq^\prime\, dq^{\prime\prime}\, dy}{y^{4r+3}}   \\
=& \; \frac{1}{c^2} \frac{1}{\widetilde{\lambda}}  \int_{(N \cap \Gamma) \backslash N} \int_{y \geq c} - \Delta f \cdot \bar{f}  \; \; \frac{dz \; dq\, dq^\prime\, dq^{\prime\prime}\, dy}{y^{4r+3}} \\
\leq & \; \frac{1}{c^2} \frac{1}{\widetilde{\lambda}} \int_{(N \cap \Gamma) \backslash N} \int_{y \geq c} - \Delta f \cdot \bar{f} + f\cdot \bar{f}  \; \; \frac{dz \; dq\, dq^\prime\, dq^{\prime\prime}\, dy}{y^{4r+3}} \qquad (|f|^2 \geq 0)\\
=& \; \frac{1}{c^2} \frac{1}{\widetilde{\lambda}} \int_{(N \cap \Gamma) \backslash N} \int_{y \geq c} (1 - \Delta) f \cdot \bar{f}  \; \; \frac{dz \; dq\, dq^\prime\, dq^{\prime\prime}\, dy}{y^{4r+3}} = \frac{1}{c^2} \big( \frac{1}{\widetilde{\lambda}} |f|^2_{\frak{B}^1_a} \big) \ll \frac{1}{c^2} |f|^2_{\frak{B}^1_a}
\end{align*}
Which is the claimed bound on the $L^2$ tails. $\square$
\end{proof}

%
%

\section{$\frak{B}^1$ norms of tails are bounded by global $\frak{B}^1$ norms}\label{sec:Global-B1a-bounds-of-smooth-trunc-L2a-tail-norms}
The previous inequality did not directly apply to smooth truncations of $f$ in $\frak{B}^1_a$ near height $c > a$, nor establish that a collection of smooth truncations $\phi_\infty \cdot f$ over all heights $c > a$ can be chosen with $\frak{B}^1$-norms uniformly bounded for $f \in B$.

The following conventions will be used in this section: for fixed height $\eta$, for $t \geq 1$, the smoothly cut-off tail $f^{[t]}$ is described as follows. Let $\phi$ be a smooth function such that $0 \leq \phi \leq 1$ on $(0,\infty)$
$$
\begin{cases} 
0 = \phi (y) & (\text{for } 0 \leq y \leq 1) \\
0 \leq \phi (y) \leq 1  & (\text{for } 1 \leq y \leq 2) \\
1 =\phi (y)  & (\text{for } 2 \leq y) 
\end{cases}
$$
Since $\phi$ is smooth and constant outside the compact interval $[1,2]$, there is a common pointwise bound $C_\phi < \infty$ for $|\phi|$, $|\phi^\prime|$ and $|\phi^{\prime\prime}|$. For $t > 0$, define a smooth cut-off function by
$$
\phi_t(y) = \phi(y/t)
$$
so that $\phi_t(y) \to 0 \enspace \forall y \text{ as } t \to \infty$.
%
%

%
%

\subsection{$O(r,1)$}\label{Global-B1a-bounds-of-smooth-trunc-L2a-tail-norms-O(r,1)}

For $f \in H^1 = H^1(\Gamma \backslash G/K)$, let $f^{[t]}(n_x m_y) = \phi_t(y)\cdot f(n_x m_y)$. Use Iwasawa coordinates $n_x m_y K \longleftrightarrow (x,y) \in \mathbb{R}^{r-1} \times (0,\infty)$.
\begin{lemma}
$|f^{[t]}|_{H^1} \ll |f|_{H^1} \; (\text{implied constant independent of } f \text{ and } t \geq 1)$
\end{lemma}
\begin{proof}
We have
$$
\Omega |_{G/K} = \Delta = y^2(\Delta_x + \frac{\partial^2}{\partial y^2}) - (r-2)y\frac{\partial}{\partial y}
$$

Since $\Delta$ has real coefficients, it suffices to treat real-valued $f$. Since $0 \leq \phi_t \leq 1$, clearly $|f^{[t]}|_{L^2} = |\phi_t \cdot f|_{L^2} \leq |f|_{L^2}$. For the other part of the $H^1$ norm
\begin{align*}
	\langle -\Delta f^{[t]}, f^{[t]} \rangle =& \int_{(N \cap \Gamma) \backslash N} \int_{y \geq c} -( y^2(\Delta_x + \frac{\partial^2}{\partial y^2}) - (r-2)y\frac{\partial}{\partial y}) f^{[t]} f^{[t]} \frac{dx \, dy}{y^r} \\
	=& \int_{(N \cap \Gamma) \backslash N} \int_{y \geq c} -(\phi_t(y)^2)( y^2  \Delta_x f(x, y)) f(x, y) \frac{dx \, dy}{y^r} \\
	-&  \int_{(N \cap \Gamma) \backslash N} \int_{y \geq c} (y^2\frac{\partial^2}{\partial y^2} - (r-2)y\frac{\partial}{\partial y}) f^{[t]} f^{[t]} d\frac{dx \, dy}{y^r}
\end{align*}
Since the $\Delta_x$ factor treats $y$ as a constant, temporarily ignore the first term in $\Delta_x$ to expand the term with the derivatives in $y$
\begin{align*}
	=&  -\int_{(N \cap \Gamma) \backslash N} \int_{y \geq c} (y^2\frac{\partial^2}{\partial y^2} - (r-2)y\frac{\partial}{\partial y}) f^{[t]} f^{[t]} \frac{dx \, dy}{y^r}\\
	=&  -\int_{(N \cap \Gamma) \backslash N} \int_{y \geq c} \Bigg(\bigg(y^2\frac{\partial^2}{\partial y^2} - (r-2)y\frac{\partial}{\partial y}\bigg)\big(\phi(\frac{y}{t})f(x,y)\big)\Bigg) (\phi(\frac{y}{t})f(x,y)) \frac{dx \, dy}{y^r}\\
	=&  -\int_{(N \cap \Gamma) \backslash N} \int_{y \geq c} y^2\Bigg( \bigg(\frac{1}{t^2}\frac{\partial^2 \phi}{\partial y^2}(\frac{y}{t})\bigg)f(x,y) + \frac{2}{t}\bigg(\frac{\partial \phi}{\partial y}(\frac{y}{t})\bigg)\bigg(\frac{\partial f}{\partial y}(x,y)\bigg) \\
	& + \phi(\frac{y}{t})\frac{\partial^2 f}{\partial y^2}(x,y) \Bigg) \phi(\frac{y}{t})f(x,y) \frac{dx \, dy}{y^r}\\	
	+& (r-2)\int_{(N \cap \Gamma) \backslash N} \int_{y \geq c}  y \Bigg(\frac{1}{t}\frac{\partial \phi}{\partial y}(\frac{y}{t})f(x,y) + \phi(\frac{y}{t})\frac{\partial f}{\partial y}(x,y) \Bigg)\phi(\frac{y}{t})f(x,y) \frac{dx \, dy}{y^r}
\end{align*}
Collecting, and again including the $\Delta_x$ term
\begin{flalign*}
    =& \int_{(N \cap \Gamma) \backslash N} \int_{y \geq c} -(\phi_t(y)^2)( y^2  \Delta_x f(x, y)) f(x, y) \frac{dx \, dy}{y^r} && \text{(A 1)}\\
	-& \int_{(N \cap \Gamma) \backslash N} \int_{y \geq c} \frac{1}{t^2} \Bigg( y^2 \frac{\partial^2 \phi}{\partial y^2}(\frac{y}{t}) \phi(\frac{y}{t}) \Bigg)\bigg( f(x,y) \bigg)^2\frac{dx \, dy}{y^r} && \text{(B 1)}\\
	-&  \int_{(N \cap \Gamma) \backslash N} \int_{y \geq c}  \frac{2}{t}\Bigg(y^2 \frac{\partial \phi}{\partial y}(\frac{y}{t})  \phi(\frac{y}{t}) \Bigg)\Bigg(\frac{\partial f}{\partial y}(y) f(x,y) \Bigg)   \frac{dx \, dy}{y^r} && \text{(C 1)}\\
	-&  \int_{(N \cap \Gamma) \backslash N} \int_{y \geq c} y^2  \bigg( \phi(\frac{y}{t}) \bigg)^2 \bigg( \frac{\partial^2 f}{\partial y^2}(y)  f(x,y) \bigg) \frac{dx \, dy}{y^r} && \text{(D 1)}\\
	+&  (r-2)\int_{(N \cap \Gamma) \backslash N} \int_{y \geq c}  \frac{1}{t} \bigg( y \frac{\partial \phi}{\partial y}(\frac{y}{t})  \phi(\frac{y}{t}) \bigg) \bigg( f(x,y) \bigg)^2 \frac{dx \, dy}{y^r} && \text{(E 1)}\\
	+&  (r-2)\int_{(N \cap \Gamma) \backslash N} \int_{y \geq c}  y \bigg( \phi(\frac{y}{t}) \bigg)^2 \bigg(  \frac{\partial f}{\partial y}(y) f(x,y) \bigg) \frac{dx \, dy}{y^r} && \text{(F 1)}
\end{flalign*}
Combining (A 1), (D 1) and (F 1) into (A2)
\begin{align*}
    =& \int_{(N \cap \Gamma) \backslash N} \int_{y \geq c} -(\phi_t(y)^2)\Bigg( y^2 \bigg( \Delta_x  + \frac{\partial^2}{\partial y^2}\bigg)  -  (r-2) y \frac{\partial}{\partial y}\Bigg) f(x,y) f(x, y)  \frac{dx \, dy}{y^r} && \text{(A 2)}\\
	-& \int_{(N \cap \Gamma) \backslash N} \int_{y \geq c} \frac{1}{t^2} \Bigg( y^2 \frac{\partial^2 \phi}{\partial y^2}(\frac{y}{t}) \phi(\frac{y}{t}) \Bigg)\bigg( f(x,y) \bigg)^2\frac{dx \, dy}{y^r}  && \text{(B 2)}\\
	-&  \int_{(N \cap \Gamma) \backslash N} \int_{y \geq c}  \frac{2}{t}\Bigg(y^2 \frac{\partial \phi}{\partial y}(\frac{y}{t})  \phi(\frac{y}{t}) \Bigg)\Bigg(\frac{\partial f}{\partial y}(y) f(x,y) \Bigg)   \frac{dx \, dy}{y^r}  && \text{(C 2)}\\
	+&  (r-2)\int_{(N \cap \Gamma) \backslash N} \int_{y \geq c}  \frac{1}{t} \bigg( y \frac{\partial \phi}{\partial y}(\frac{y}{t})  \phi(\frac{y}{t}) \bigg) \bigg( f(x,y) \bigg)^2 \frac{dx \, dy}{y^r} && \text{(D 2)}
\end{align*}
Note that $\frac{\partial f}{\partial y} f = \frac{1}{2}\frac{\partial}{\partial y}(f)^2$ and use integration by parts on (C 2) to move the derivative from $f$ (changing the sign of the term) giving (C3), noting also that (A 2) is $\Delta$ to get
\begin{align*}
    =& \int_{(N \cap \Gamma) \backslash N} \int_{y \geq c} \phi_t(y)^2\bigg( -\Delta f(x,y) f(x, y) \bigg)  \frac{dx \, dy}{y^r}  && \text{(A 3)}\\
	-& \int_{(N \cap \Gamma) \backslash N} \int_{y \geq c} \frac{1}{t^2} \Bigg( y^2 \frac{\partial^2 \phi}{\partial y^2}(\frac{y}{t}) \phi(\frac{y}{t}) \Bigg)\bigg( f(x,y) \bigg)^2\frac{dx \, dy}{y^r} && \text{(B 3)}\\
	+&  \int_{(N \cap \Gamma) \backslash N} \int_{y \geq c}  \frac{2}{t}\Bigg(2y  \frac{\partial \phi}{\partial y}(\frac{y}{t})  \phi(\frac{y}{t}) \\
	+& y^2 \frac{1}{t} \frac{\partial^2 \phi}{\partial y^2}(\frac{y}{t}) \phi(\frac{y}{t}) 
+	y^2 \frac{1}{t} \bigg( \frac{\partial \phi}  {\partial y}(\frac{y}{t})\bigg)^2 \Bigg) \bigg(f(x,y)\bigg)^2    \frac{dx \, dy}{y^r} && \text{(C 3)}\\
	+&  (r-2)\int_{(N \cap \Gamma) \backslash N} \int_{y \geq c}  \frac{1}{t} \bigg( y \frac{\partial \phi}{\partial y}(\frac{y}{t})  \phi(\frac{y}{t}) \bigg) \bigg( f(x,y) \bigg)^2 \frac{dx \, dy}{y^r} && \text{(D 3)}
\end{align*}
Combining (B 3), (C 3) and (D 3) on the common term of $\big( f(x,y) \big)^2$
\begin{align*}
    =& \int_{(N \cap \Gamma) \backslash N} \int_{y \geq c} \phi_t(y)^2\bigg( -\Delta f(x,y) f(x, y) \bigg)   \frac{dx \, dy}{y^r} \\
+& \int_{(N \cap \Gamma) \backslash N} \int_{y \geq c} \Bigg( \frac{y^2}{t^2} \frac{\partial^2 \phi}{\partial y^2}(\frac{y}{t}) \phi(\frac{y}{t})  +  \frac{2y^2}{t^2}\bigg( \frac{\partial \phi}  {\partial y}(\frac{y}{t})\bigg)^2  + (r+2) \frac{y}{t} \frac{\partial \phi}{\partial y}(\frac{y}{t})  \phi(\frac{y}{t}) \Bigg)\bigg( f(x,y) \bigg)^2 \frac{dx \, dy}{y^r}
\end{align*}
By assumption $0 \leq \phi \leq 1$, which applies also to $\phi_t$, so the first integral is bounded by $\int -\Delta f \cdot f$. The second expression is bounded in terms of $|f|_{L^2}^2$ as follows: $\phi^\prime$ and $\phi^{\prime\prime}$ are supported in $[1,2]$, so $\phi_t^\prime$ and $\phi_t^{\prime\prime}$ are supported in $[t,2t]$. Using the earlier pointwise bound $C_\phi < \infty$ on $\phi^\prime$ and $\phi^{\prime\prime}$, gives the estimate the second integral (no loss in generality assuming $c \leq t$):
\begin{align*}
& \int_{(N \cap \Gamma) \backslash N} \int_{y \geq c} \Bigg( \frac{y^2}{t^2} \frac{\partial^2 \phi}{\partial y^2}(\frac{y}{t}) \phi(\frac{y}{t})  +  \frac{2y^2}{t^2}\bigg( \frac{\partial \phi}  {\partial y}(\frac{y}{t})\bigg)^2  + (r+2) \frac{y}{t} \frac{\partial \phi}{\partial y}(\frac{y}{t})  \phi(\frac{y}{t}) \Bigg) ( f(x,y) )^2 \frac{dx \, dy}{y^r} \\
=& \int_{(N \cap \Gamma) \backslash N} \int_t^{2t} \Bigg( \frac{y^2}{t^2} \frac{\partial^2 \phi}{\partial y^2}(\frac{y}{t}) \phi(\frac{y}{t})  +  \frac{2y^2}{t^2}\bigg( \frac{\partial \phi}  {\partial y}(\frac{y}{t})\bigg)^2  + (r+2) \frac{y}{t} \frac{\partial \phi}{\partial y}(\frac{y}{t})  \phi(\frac{y}{t}) \Bigg) ( f(x,y) )^2 \frac{dx \, dy}{y^r} \\ \\
\leq& \int_{(N \cap \Gamma) \backslash N} \int_t^{2t} \Bigg| \frac{y^2}{t^2} \frac{\partial^2 \phi}{\partial y^2}(\frac{y}{t}) \phi(\frac{y}{t})  +  \frac{2y^2}{t^2}\bigg( \frac{\partial \phi}  {\partial y}(\frac{y}{t})\bigg)^2  + (r+2) \frac{y}{t} \frac{\partial \phi}{\partial y}(\frac{y}{t})  \phi(\frac{y}{t}) \Bigg| ( f(x,y) )^2 \frac{dx \, dy}{y^r} \\ \\
\end{align*}
since $y < 2t$
\begin{align*}
\leq & \int_{(N \cap \Gamma) \backslash N} \int_t^{2t} \Bigg| 4 \frac{\partial^2 \phi}{\partial y^2}(\frac{y}{t}) \phi(\frac{y}{t})  +  8 \bigg( \frac{\partial \phi}  {\partial y}(\frac{y}{t})\bigg)^2  + 2(r+2) \frac{\partial \phi}{\partial y}(\frac{y}{t})  \phi(\frac{y}{t}) \Bigg| ( f(x,y) )^2 \frac{dx \, dy}{y^r} \\ \\
\end{align*}
using $|\phi|, |\phi^\prime|, |\phi^{\prime\prime}| < C_\phi$
\begin{align*}
\leq & \int_{(N \cap \Gamma) \backslash N} \int_t^{2t} \Bigg( 4 C_\phi^2  +  8 C_\phi^2  + 2(r+2) C_\phi^2 \Bigg) ( f(x,y) )^2 \; \frac{dx \, dy}{y^r} \\
\ll_\phi & \enspace |f|_{L^2}^2 = c_\phi |f|_{L^2}^2
\end{align*}
Where $c_\phi$ depends only on $\phi$, $\phi^\prime$ and $\phi^{\prime\prime}$. Combining the estimates gives
$$
|f^{[t]}|_{H^1}^2  \ll_\phi |f|_{H^1}^2
$$
as claimed. $\square$
\end{proof}

%
%

\subsection{$U(r,1)$}\label{Global-B1a-bounds-of-smooth-trunc-L2a-tail-norms-U(r,1)}

Use Iwasawa coordinates 
$$
G/K \ni n_{z,x} m_y K \longleftrightarrow (z,x,y) \in \mathbb{C}^{r-1} \times \mathbb{R} \times (0,\infty)
$$
$$
  z \in \mathbb{C}^{r-1} = \{ z : u + iv \in \mathbb{R}^{r-1} \oplus \; i \mathbb{R}^{r-1} \} \qquad x \in \mathbb{R} \qquad y \in \mathbb{R}^+ = (0,\infty)
$$
For fixed height $\eta$, for $t \geq 1$, let $f^{[t]}$ be the smoothly cut-off tail of $f$ as follows. Let $\phi$ be a smooth function such that $0 \leq \phi \leq 1$ on $(0,\infty)$
$$
\begin{cases} 
0 = \phi (y) & (\text{for } 0 \leq y \leq 1) \\
0 \leq \phi (y) \leq 1  & (\text{for } 1 \leq y \leq 2) \\
1 =\phi (y)  & (\text{for } 2 \leq y) 
\end{cases}
$$
Since $\phi$ is smooth and constant outside the compact interval $[1,2]$, there is a common pointwise bound $C_\phi < \infty$ for $\phi$, $\phi^\prime$ and $\phi^{\prime\prime}$. For $t > 0$, define a smooth cut-off function by
$$
\phi_t(y) = \phi(y/t)
$$
so that $\phi_t(y) \to 0 \enspace \forall y \text{ as } t \to \infty$.

For $f \in H^1 = H^1(\Gamma \backslash G/K)$, let $f^{[t]}(n_{z,x} m_y) = \phi_t(y)\cdot f(n_{z,x} m_y)$.

\begin{lemma}
$|f^{[t]}|_{H^1} \ll |f|_{H^1} \; (\text{implied constant independent of } f \text{ and } t \geq 1)$
\end{lemma}
\begin{proof}
We have
$$
\Omega |_{G/K} = \Delta = y^2 \bigg( \Delta_u + \Delta_v + \frac{\partial^2}{\partial y^2}+ y^2 \frac{\partial^2}{\partial x^2} \bigg) - (2r-1) y \frac{\partial}{\partial y}
$$
Since $\Delta$ has real coefficients, it suffices to treat real-valued $f$. Since $0 \leq \phi_t \leq 1$, clearly $|f^{[t]}|_{L^2} = |\phi_t \cdot f|_{L^2} \leq |f|_{L^2}$. For the other part of the $H^1$ norm
\begin{align*}
	\langle -\Delta f^{[t]}, f^{[t]} \rangle =& \int_{(N \cap \Gamma) \backslash N} \int_{y \geq c} -(y^2 \bigg( \Delta_u + \Delta_v + \frac{\partial^2}{\partial y^2}+ y^2 \frac{\partial^2}{\partial x^2} \bigg) - (2r-1) y \frac{\partial}{\partial y}) f^{[t]} f^{[t]} \frac{dz \, dx \, dy}{y^{2r+1}} \\
	=& \int_{(N \cap \Gamma) \backslash N} \int_{y \geq c} -\phi_t(y)^2 y^2 \bigg( \big( \Delta_u + \Delta_v + y^2 \frac{\partial^2}{\partial x^2} \big)  f(n_{z,x} m_y) \bigg) f(n_{z,x} m_y) \frac{dz \, dx \, dy}{y^{2r+1}} \\
	-& \int_{(N \cap \Gamma) \backslash N} \int_{y \geq c} (y^2 \frac{\partial^2}{\partial y^2} - (2r-1) y \frac{\partial}{\partial y}) f^{[t]} f^{[t]} \frac{dz \, dx \, dy}{y^{2r+1}}
\end{align*}
\noindent
Since the variable $y$ is treated a constant by the operators in the first integral, temporarily ignore the first terms and expand the second term containing derivatives in $y$, writing $f(z,x,y)$ for $f(n_{z,x},m_y)$:
\begin{align*}
	& - \int_{(N \cap \Gamma) \backslash N} \int_{y \geq c} \bigg(y^2 \frac{\partial^2}{\partial y^2} - (2r-1) y \frac{\partial}{\partial y} \bigg) f^{[t]} f^{[t]} \frac{dz \, dx \, dy}{y^{2r+1}} \\
	=&  -\int_{(N \cap \Gamma) \backslash N} \int_{y \geq c} \Bigg(\bigg(y^2\frac{\partial^2}{\partial y^2} - (2r-1)y\frac{\partial}{\partial y}\bigg)\big(\phi(\frac{y}{t})f(z,x,y)\big)\Bigg) (\phi(\frac{y}{t})f(z,x,y)) \frac{dz \, dx \, dy}{y^{2r+1}}\\
	=&  -\int_{(N \cap \Gamma) \backslash N} \int_{y \geq c} y^2\Bigg( \bigg(\frac{1}{t^2}\frac{\partial^2 \phi}{\partial y^2}(\frac{y}{t})\bigg)f(z,x,y) + \frac{2}{t}\bigg(\frac{\partial \phi}{\partial y}(\frac{y}{t})\bigg)\bigg(\frac{\partial f}{\partial y}(z,x,y)\bigg) \\
	+& \; \phi(\frac{y}{t})\frac{\partial^2 f}{\partial y^2}(z,x,y) \Bigg) \phi(\frac{y}{t})f(z,x,y) \frac{dz \, dx \, dy}{y^{2r+1}}\\	
	+&  \; (2r-1)\int_{(N \cap \Gamma) \backslash N} \int_{y \geq c}  y \Bigg(\frac{1}{t}\frac{\partial \phi}{\partial y}(\frac{y}{t})f(z,x,y) + \phi(\frac{y}{t})\frac{\partial f}{\partial y}(z,x,y) \Bigg)\phi(\frac{y}{t})f(z,x,y) \frac{dz \, dx \, dy}{y^{2r+1}} \\
	=&  -\int_{(N \cap \Gamma) \backslash N} \int_{y \geq c} \Bigg( \frac{y^2}{t^2} \phi(\frac{y}{t}) \frac{\partial^2 \phi}{\partial y^2}(\frac{y}{t}) \big( f(z,x,y) \big)^2 + \frac{2 y^2}{t} \phi(\frac{y}{t}) \frac{\partial \phi}{\partial y}(\frac{y}{t})\frac{\partial f}{\partial y}(z,x,y) f(z,x,y) \\
	& + y^2 \big(\phi(\frac{y}{t}) \big)^2\frac{\partial^2 f}{\partial y^2}(z,x,y) f(z,x,y) \Bigg) \frac{dz \, dx \, dy}{y^{2r+1}}\\	
	+& (2r-1)\int_{(N \cap \Gamma) \backslash N} \int_{y \geq c}   \Bigg(\frac{y}{t}\phi(\frac{y}{t})\frac{\partial \phi}{\partial y}(\frac{y}{t}) \big( f(z,x,y) \big)^2 + y \big( \phi(\frac{y}{t}) \big)^2 \frac{\partial f}{\partial y}(z,x,y) f(z,x,y) \Bigg) \frac{dz \, dx \, dy}{y^{2r+1}}
\end{align*}
\noindent
Expanding, collecting like terms and re-adding the terms including other derivatives
\begin{flalign*}
    =&  \int_{(N \cap \Gamma) \backslash N} \int_{y \geq c} -\phi_t(y)^2 y^2 \bigg( \big( \Delta_u + \Delta_v + y^2 \frac{\partial^2}{\partial x^2} \big)  f(n_{z,x} m_y) \bigg) f(z,x,y) \frac{dz \, dx \, dy}{y^{2r+1}} && \text{(A 1)}\\
    -& \int_{(N \cap \Gamma) \backslash N} \int_{y \geq c} \frac{y^2}{t^2} \phi(\frac{y}{t}) \frac{\partial^2 \phi}{\partial y^2}(\frac{y}{t}) \big( f(z,x,y) \big)^2 \frac{dz \, dx \, dy}{y^{2r+1}} && \text{(B 1)}\\
    -& \int_{(N \cap \Gamma) \backslash N} \int_{y \geq c} \frac{2 y^2}{t} \phi(\frac{y}{t}) \frac{\partial \phi}{\partial y}(\frac{y}{t}) \frac{\partial f}{\partial y}(z,x,y) f(z,x,y) \frac{dz \, dx \, dy}{y^{2r+1}} && \text{(C 1)}\\
    -& \int_{(N \cap \Gamma) \backslash N} \int_{y \geq c} y^2 \big(\phi(\frac{y}{t}) \big)^2\frac{\partial^2 f}{\partial y^2}(z,x,y) f(z,x,y) \frac{dz \, dx \, dy}{y^{2r+1}} && \text{(D 1)}\\
    +& (2r-1) \int_{(N \cap \Gamma) \backslash N} \int_{y \geq c} \frac{y}{t}\phi(\frac{y}{t})\frac{\partial \phi}{\partial y}(\frac{y}{t}) \big( f(z,x,y) \big)^2 \frac{dz \, dx \, dy}{y^{2r+1}} && \text{(E 1)}\\
    +& (2r-1) \int_{(N \cap \Gamma) \backslash N} \int_{y \geq c} y \big( \phi(\frac{y}{t}) \big)^2 \frac{\partial f}{\partial y}(z,x,y) f(z,x,y) \frac{dz \, dx \, dy}{y^{2r+1}} &&  \text{(F 1)}
\end{flalign*}
(A 1), (D 1) and (F 1) have a common factor of $\phi_t^2$ and combine into (A2)
\begin{flalign*}
    =&  \int_{(N \cap \Gamma) \backslash N} \int_{y \geq c} -\phi_t(y)^2 y^2 \bigg( \big( \Delta_u + \Delta_v + y^2 \frac{\partial^2}{\partial x^2} + \frac{\partial^2}{\partial y^2} \\
    -& (2r-1)y \frac{\partial }{\partial y}\big)  f(z,x,y) \bigg) f(z,x,y) \frac{dz \,dx \, dy}{y^{2r+1}} && \text{(A 2)}\\
    -& \int_{(N \cap \Gamma) \backslash N} \int_{y \geq c} \frac{y^2}{t^2} \phi(\frac{y}{t}) \frac{\partial^2 \phi}{\partial y^2}(\frac{y}{t}) \big( f(z,x,y) \big)^2 \frac{dz \, dx \, dy}{y^{2r+1}} && \text{(B 2)}\\
    -& \int_{(N \cap \Gamma) \backslash N} \int_{y \geq c} \frac{2 y^2}{t} \phi(\frac{y}{t}) \bigg(\frac{\partial \phi}{\partial y}(\frac{y}{t})\bigg)\bigg(\frac{\partial f}{\partial y}(z,x,y)\bigg) f(z,x,y) \frac{dz \, dx \, dy}{y^{2r+1}} && \text{(C 2)}\\
    +& (2r-1) \int_{(N \cap \Gamma) \backslash N} \int_{y \geq c} \frac{y}{t}\phi(\frac{y}{t})\frac{\partial \phi}{\partial y}(\frac{y}{t}) \big( f(z,x,y) \big)^2 \frac{dz \, dx \, dy}{y^{2r+1}} && \text{(D 2)}
\end{flalign*}
Note that $\frac{\partial f}{\partial y} f = \frac{1}{2}\frac{\partial}{\partial y}(f)^2$ and use integration by parts (changing the sign of the term) on (C 2) to move the derivative from $f$ giving (C3), noting also that (A 2) is $\Delta$ to get
\begin{flalign*}
    =&  \int_{(N \cap \Gamma) \backslash N} \int_{y \geq c} -\phi_t(y)^2 \Delta  f(z,x,y) f(z,x,y) \frac{dz \,dx \, dy}{y^{2r+1}} && \text{(A 3)}\\
    -& \int_{(N \cap \Gamma) \backslash N} \int_{y \geq c} \frac{y^2}{t^2} \phi(\frac{y}{t}) \frac{\partial^2 \phi}{\partial y^2}(\frac{y}{t}) \big( f(z,x,y) \big)^2 \; \frac{dz \, dx \, dy}{y^{2r+1}} && \text{(B 3)}\\
    +& \int_{(N \cap \Gamma) \backslash N} \int_{y \geq c} \Bigg( \frac{4y}{t} \phi(\frac{y}{t}) \frac{\partial \phi}{\partial y}(\frac{y}{t}) + \frac{2 y^2}{t^2} \bigg(\frac{\partial \phi}{\partial y}(\frac{y}{t})\bigg)^2 \\
    +& \frac{2y^2}{t^2} \phi(\frac{y}{t}) \frac{\partial^2 \phi}{\partial y^2}(\frac{y}{t}) \Bigg) \big( f(z,x,y) \big)^2 \; \frac{dz \, dx \, dy}{y^{2r+1}} && \text{(C 3)}\\
    +& (2r-1) \int_{(N \cap \Gamma) \backslash N} \int_{y \geq c} \frac{y}{t}\phi(\frac{y}{t})\frac{\partial \phi}{\partial y}(\frac{y}{t}) \big( f(z,x,y) \big)^2 \; \frac{dz \, dx \, dy}{y^{2r+1}} && \text{(D 3)}
\end{flalign*}
Combining (B3), (C3) and (D3) on the common factor of $\big( f(z,x,y) \big)^2$
\begin{flalign*}
    =&  \int_{(N \cap \Gamma) \backslash N} \int_{y \geq c} -\phi_t(y)^2 \Delta  f(z,x,y) f(z,x,y) \frac{dz \,dx \, dy}{y^{2r+1}} && \text{(A 4)}\\
    +& \int_{(N \cap \Gamma) \backslash N} \int_{y \geq c} \Bigg( -\frac{y^2}{t^2} \phi(\frac{y}{t}) \frac{\partial^2 \phi}{\partial y^2}(\frac{y}{t}) 
    +  \frac{4y}{t} \phi(\frac{y}{t}) \frac{\partial \phi}{\partial y}(\frac{y}{t}) + \frac{2 y^2}{t^2} \bigg(\frac{\partial \phi}{\partial y}(\frac{y}{t})\bigg)^2 \\
    +& \frac{2y^2}{t^2} \phi(\frac{y}{t}) \frac{\partial^2 \phi}{\partial y^2}(\frac{y}{t}) 
    + (2r-1) \frac{y}{t}\phi(\frac{y}{t})\frac{\partial \phi}{\partial y}(\frac{y}{t}) \Bigg) \big( f(z,x,y) \big)^2 \; \frac{dz \, dx \, dy}{y^{2r+1}} && \text{(B 4)}
\end{flalign*}
Consolidating, suppressing the argument $(\frac{y}{t})$ for clarity, and writing derivatives in $\phi$ as primed ($\prime$)
\begin{flalign*}
    =&  \int_{(N \cap \Gamma) \backslash N} \int_{y \geq c} -\phi_t(y)^2 \Delta  f(z,x,y) f(z,x,y) \frac{dz \,dx \, dy}{y^{2r+1}} && \text{(A 4)}\\
    +& \int_{(N \cap \Gamma) \backslash N} \int_{y \geq c} \Bigg( \frac{y^2}{t^2} \phi \, \phi^{\prime\prime} + \frac{2 y^2}{t^2} \big(\phi^\prime\big)^2 + (2r+3) \frac{y}{t}\phi \, \phi^\prime \Bigg) \big( f(z,x,y) \big)^2 \; \frac{dz \, dx \, dy}{y^{2r+1}} && \text{(B 4)}
\end{flalign*}
By assumption $0 \leq \phi \leq 1$, which applies also to $\phi_t$, so the first integral is bounded by $\int -\Delta f \cdot f$. The second expression is bounded in terms of $|f|_{L^2}^2$ as follows: $\phi^\prime$ and $\phi^{\prime\prime}$ are supported in $[1,2]$, so $\phi_t^\prime$ and $\phi_t^{\prime\prime}$ are supported in $[t,2t]$ and thus $y < 2t$ so $\frac{y}{t} < 2$.
\begin{align*}
& \int_{(N \cap \Gamma) \backslash N} \int_{y \geq c} \Bigg( \frac{y^2}{t^2} \phi \, \phi^{\prime\prime} + \frac{2 y^2}{t^2} \big(\phi^\prime\big)^2 + (2r+3) \frac{y}{t}\phi \, \phi^\prime \Bigg) \big( f(z,x,y) \big)^2 \; \frac{dz \, dx \, dy}{y^{2r+1}}  \\
&\leq \Bigg| \int_{(N \cap \Gamma) \backslash N} \int_{y \geq c} \Bigg( \frac{y^2}{t^2} \phi \, \phi^{\prime\prime} + \frac{2 y^2}{t^2} \big(\phi^\prime\big)^2 + (2r+3) \frac{y}{t}\phi \, \phi^\prime \Bigg) \big( f(z,x,y) \big)^2 \; \frac{dz \, dx \, dy}{y^{2r+1}} \Bigg|  \\
&\leq \int_{(N \cap \Gamma) \backslash N} \int_{y \geq c} \Bigg|  \frac{y^2}{t^2} \phi \, \phi^{\prime\prime} + \frac{2 y^2}{t^2} \big(\phi^\prime\big)^2 + (2r+3) \frac{y}{t}\phi \, \phi^\prime \Bigg| \big( f(z,x,y) \big)^2 \; \frac{dz \, dx \, dy}{y^{2r+1}}   \\
&\leq \int_{(N \cap \Gamma) \backslash N} \int_{y \geq c} \Bigg( \bigg| \frac{y^2}{t^2} \phi \, \phi^{\prime\prime} \bigg| + \bigg|\frac{2 y^2}{t^2} \big(\phi^\prime\big)^2\bigg| + (2r+3) \bigg|\frac{y}{t}\phi \, \phi^\prime \bigg| \Bigg) \big( f(z,x,y) \big)^2 \; \frac{dz \, dx \, dy}{y^{2r+1}} 
\end{align*}
Using the earlier pointwise bound $C_\phi < \infty$ on $\phi$, $\phi^\prime$ and $\phi^{\prime\prime}$:
\begin{align*}
&\leq \int_{(N \cap \Gamma) \backslash N} \int_{y \geq c} \bigg( 4 C_\phi^2 + 8  C_\phi^2 + (2r+3) C_\phi \bigg) \big( f(z,x,y) \big)^2 \; \frac{dz \, dx \, dy}{y^{2r+1}} 
\end{align*}
So that, combining (A 4) and (B 4) above, for $C_\phi^\prime = \text{ polynomial in } C_\phi$
$$
\ll_\phi C_\phi^\prime \int_{(N \cap \Gamma) \backslash N} \int_{y \geq c} \big( 1 - \Delta \big)f(z,x,y) \cdot f(z,x,y) \; \frac{dz \, dx \, dy}{y^{2r+1}} \ll_\phi |f|_{H^1}^2
$$
as claimed. $\square$
\end{proof}

%
%

\subsection{$Sp^\ast(r,1)$}\label{Global-B1a-bounds-of-smooth-trunc-L2a-tail-norms-Sp*(r,1)}

A model for quaternionic hyperbolic space is
$$
G/K \approx Sp^\ast (r,1)/(Sp^\ast (r) \times Sp^\ast (1)) \approx NM/(M \cap K) \approx N \times A^+
$$
with coordinates $z, p, q, r \in N \approx  \mathbb{H}^{r-1} \times \mathbb{R}^3$, $y \in A^+ \approx  \mathbb{R}^\times$ from the Iwasawa decomposition giving
$$
G/K \ni n_{z,q,q^\prime,q^{\prime\prime}} m_y K \longleftrightarrow (z,q,q^\prime,q^{\prime\prime},y) \in \mathbb{H}^{r-1} \times \mathbb{R}^3 \times (0,\infty)
$$
$$
z = x+ iw + ju + kv \in \mathbb{R}^{r-1} \oplus \; i \mathbb{R}^{r-1} \oplus \; j \mathbb{R}^{r-1} \oplus \; k \mathbb{R}^{r-1} \approx \mathbb{H}^{r-1} \quad  q,q^\prime,q^{\prime\prime} \in \mathbb{R}
$$
For fixed height $\eta$, for $t \geq 1$, let $f^{[t]}$ be the smoothly cut-off tail of $f$ as follows. Let $\phi$ be a smooth function such that $0 \leq \phi \leq 1$ on $(0,\infty)$
$$
\begin{cases} 
0 = \phi (y) & (\text{for } 0 \leq y \leq 1) \\
0 \leq \phi (y) \leq 1  & (\text{for } 1 \leq y \leq 2) \\
1 =\phi (y)  & (\text{for } 2 \leq y) 
\end{cases}
$$
Since $\phi$ is smooth and constant outside the compact interval $[1,2]$, there is a common pointwise bound $C_\phi < \infty$ for $\phi$, $\phi^\prime$ and $\phi^{\prime\prime}$. For $t > 0$, define a smooth cut-off function by
$$
\phi_t(y) = \phi(y/t)
$$
so that $\phi_t(y) \to 0 \enspace \forall y \text{ as } t \to \infty$. For $f \in H^1 = H^1(\Gamma \backslash G/K)$, set 
$$
f^{[t]}(n_{z,q,q^\prime,q^{\prime\prime}} \, m_y) = \phi_t(y)\cdot f(n_{z,q,q^\prime,q^{\prime\prime}} \, m_y) = \phi(\frac{y}{t})\cdot f(z,q,q^\prime,q^{\prime\prime},y)
$$

\begin{lemma}
$|f^{[t]}|_{H^1} \ll |f|_{H^1} \; (\text{implied constant independent of } f \text{ and } t \geq 1)$
\end{lemma}
\begin{proof}
We have
$$
\Omega |_{G/K} = \Delta = y^2 \big(  \Delta_x + \Delta_w + \Delta_u + \Delta_v + \frac{\partial^2}{\partial y^2} + y^2 \frac{\partial^2}{\partial q^2} + y^2 \frac{\partial^2}{\partial q^{\prime 2}} + y^2 \frac{\partial^2}{\partial q^{\prime\prime 2}} \big) - (4r+1)  y \frac{\partial}{\partial y} 
$$
Since $\Delta$ has real coefficients, it suffices to treat real-valued $f$. Since $0 \leq \phi_t \leq 1$, clearly $|f^{[t]}|_{L^2} = |\phi_t \cdot f|_{L^2} \leq |f|_{L^2}$. For the other part of the $H^1$ norm,  first gather terms using that the variable $y$ is treated as a constant by the non-$y$ derivatives:
\begin{flalign*}
	\langle -\Delta f^{[t]}, f^{[t]} \rangle =& \int_{(N \cap \Gamma) \backslash N} \int_{y \geq c} -\Bigg( y^2 \big(  \Delta_x + \Delta_w + \Delta_u + \Delta_v + \frac{\partial^2}{\partial y^2} + y^2 \frac{\partial^2}{\partial q^2} + y^2 \frac{\partial^2}{\partial q^{\prime 2}} + y^2 \frac{\partial^2}{\partial q^{\prime\prime 2}} \big) && \\
	&- (4r+1)  y \frac{\partial}{\partial y} \Bigg) f^{[t]} f^{[t]} \; \frac{dz \; dq\, dq^\prime\, dq^{\prime\prime}\, dy}{y^{4r+3}}  &&\text{(1)}\\
	=& \int_{(N \cap \Gamma) \backslash N} \int_{y \geq c} -\phi_t(y)^2 y^2 \Bigg( \Delta_x + \Delta_w + \Delta_u + \Delta_v &&\\
	&+ y^2 \frac{\partial^2}{\partial q^2} + y^2 \frac{\partial^2}{\partial q^{\prime 2}} + y^2 \frac{\partial^2}{\partial q^{\prime\prime 2}} \Bigg)  f(n_{z,q,q^\prime,q^{\prime\prime}} \, m_y) f(n_{z,q,q^\prime,q^{\prime\prime}} \, m_y) \; \frac{dz \; dq\, dq^\prime\, dq^{\prime\prime}\, dy}{y^{4r+3}}  && \text{(2)}\\
	-& \int_{(N \cap \Gamma) \backslash N} \int_{y \geq c} \big( y^2 \frac{\partial^2}{\partial y^2} - (4r+1) y \frac{\partial}{\partial y} \big) f^{[t]} f^{[t]} \; \frac{dz \; dq\, dq^\prime\, dq^{\prime\prime}\, dy}{y^{4r+3}} && \text{(3)}
\end{flalign*}
 Temporarily ignore the first terms and expand the second term containing derivatives in $y$ (i.e., term (3)). To lighten notation, we suppress the arguments of $f$ as only the $y$-derivative will appear.
\begin{align*}
	& - \int_{(N \cap \Gamma) \backslash N} \int_{y \geq c} \big( y^2 \frac{\partial^2}{\partial y^2} - (4r+1) y \frac{\partial}{\partial y} \big) f^{[t]} f^{[t]} \; \frac{dz \; dq\, dq^\prime\, dq^{\prime\prime}\, dy}{y^{4r+3}}  \\
	=& - \int_{(N \cap \Gamma) \backslash N} \int_{y \geq c} \Bigg( \big( y^2 \frac{\partial^2}{\partial y^2} - (4r+1) y \frac{\partial}{\partial y} \big) \big( \phi(\frac{y}{t})\cdot f \big) \Bigg) \; \cdot  \phi(\frac{y}{t}) \cdot  f \; \frac{dz \; dq\, dq^\prime\, dq^{\prime\prime}\, dy}{y^{4r+3}}  
\end{align*}
\begin{align*}
	=& - \int_{(N \cap \Gamma) \backslash N} \int_{y \geq c} \Bigg( \big( y^2 \frac{\partial^2}{\partial y^2} - (4r+1) y \frac{\partial}{\partial y} \big) \big( \phi(\frac{y}{t})\cdot f \big) \Bigg) \; \phi(\frac{y}{t}) \cdot  f \; \frac{dz \; dq\, dq^\prime\, dq^{\prime\prime}\, dy}{y^{4r+3}}  \\
	=& - \int_{(N \cap \Gamma) \backslash N} \int_{y \geq c} \Bigg( y^2 \frac{\partial^2}{\partial y^2} \big( \phi(\frac{y}{t})\cdot f \big)  - (4r+1) y \frac{\partial}{\partial y} \big( \phi(\frac{y}{t})\cdot f \big) \Bigg) \; \phi(\frac{y}{t}) \cdot  f \; \frac{dz \; dq\, dq^\prime\, dq^{\prime\prime}\, dy}{y^{4r+3}}  \\
	=&  -\int_{(N \cap \Gamma) \backslash N} \int_{y \geq c} y^2\Bigg( \bigg(\frac{1}{t^2}\frac{\partial^2 \phi}{\partial y^2}(\frac{y}{t})\bigg)f + \frac{2}{t}\bigg(\frac{\partial \phi}{\partial y}(\frac{y}{t})\bigg)\bigg(\frac{\partial f}{\partial y}\bigg) 
	+ \; \phi(\frac{y}{t})\frac{\partial^2 f}{\partial y^2} \Bigg) \cdot \\
	& \phi(\frac{y}{t}) \cdot f  \; \frac{dz \; dq\, dq^\prime\, dq^{\prime\prime}\, dy}{y^{4r+3}} \\	
	+&  \; (4r+1)\int_{(N \cap \Gamma) \backslash N} \int_{y \geq c}  y \Bigg(\frac{1}{t}\frac{\partial \phi}{\partial y}(\frac{y}{t})f + \phi(\frac{y}{t})\frac{\partial f}{\partial y} \Bigg) \cdot \phi(\frac{y}{t}) \cdot f \; \frac{dz \; dq\, dq^\prime\, dq^{\prime\prime}\, dy}{y^{4r+3}}
\end{align*}
Writing derivatives as primes ($\prime$)
\begin{align*}
	=&  -\int_{(N \cap \Gamma) \backslash N} \int_{y \geq c} \Bigg( \frac{y^2}{t^2} \phi^{\prime\prime} f + \frac{2y^2}{t}\phi^\prime f^\prime + y^2\phi f^{\prime\prime} \Bigg) \cdot \phi \cdot f  \; \frac{dz \; dq\, dq^\prime\, dq^{\prime\prime}\, dy}{y^{4r+3}} \\	
	+&  \; (4r+1)\int_{(N \cap \Gamma) \backslash N} \int_{y \geq c}  \Bigg(\frac{y}{t} \phi^\prime f + y \,\phi f^\prime \Bigg) \cdot \phi \cdot f \; \frac{dz \; dq\, dq^\prime\, dq^{\prime\prime}\, dy}{y^{4r+3}} \\
	=&  -\int_{(N \cap \Gamma) \backslash N} \int_{y \geq c} \Bigg( y^2\phi_t^{\prime\prime} f + 2y^2\phi_t^\prime f^\prime + y^2\phi_t f^{\prime\prime} \Bigg) \cdot \phi_t \cdot f  \; \frac{dz \; dq\, dq^\prime\, dq^{\prime\prime}\, dy}{y^{4r+3}} \\
	+&  \; (4r+1)\int_{(N \cap \Gamma) \backslash N} \int_{y \geq c}  \Bigg(y \phi_t^\prime f + y \,\phi_t f^\prime \Bigg) \cdot \phi_t \cdot f \; \frac{dz \; dq\, dq^\prime\, dq^{\prime\prime}\, dy}{y^{4r+3}} \\
\end{align*}
Multiplying out and collecting like terms in $\phi$ 
\begin{align*}
	=&  \int_{(N \cap \Gamma) \backslash N} \int_{y \geq c} \Bigg( - \phi_t^2 \big( y^2 \cdot f^{\prime\prime} - (4r+1)y \cdot f^\prime \big) \cdot f \\
	 &+ \big((4r+1)y \phi_t^\prime\cdot \phi_t - y^2 \phi_t^{\prime\prime} \cdot \phi_t \big) \cdot f^2 - 2y^2 \phi_t^\prime \cdot \phi_t \cdot (f^\prime \cdot f) \Bigg)  \; \frac{dz \; dq\, dq^\prime\, dq^{\prime\prime}\, dy}{y^{4r+3}} 
\end{align*}
Returning the earlier terms in other derivatives
\begin{flalign*}	
=& \int_{(N \cap \Gamma) \backslash N} \int_{y \geq c} -\phi_t(y)^2 y^2 \bigg( \Delta_x + \Delta_w + \Delta_u + \Delta_v 
	+ \; y^2 \frac{\partial^2}{\partial q^2} \\
	&+ y^2 \frac{\partial^2}{\partial q^{\prime 2}} + y^2 \frac{\partial^2}{\partial q^{\prime\prime 2}} \bigg)  f \cdot f \; \frac{dz \; dq\, dq^\prime\, dq^{\prime\prime}\, dy}{y^{4r+3}}  && \text{(A 1)}\\
    &- \int_{(N \cap \Gamma) \backslash N} \int_{y \geq c} \phi_t^2 \cdot \bigg( y^2 \cdot \frac{\partial^2 f}{\partial y^2} \cdot f - (4r+1) y \,\frac{\partial f}{\partial y} \cdot f\bigg)  \; \frac{dz \; dq\, dq^\prime\, dq^{\prime\prime}\, dy}{y^{4r+3}}  && \text{(B 1)}\\
    &+ \int_{(N \cap \Gamma) \backslash N} \int_{y \geq c} \bigg(  \bigg( (4r+1) y \cdot \frac{\partial \phi_t}{\partial y} \cdot \phi_t - \; y^2 \cdot \frac{\partial^2 \phi_t}{\partial y^2} \cdot \phi_t \bigg)  \cdot f^2 \\
&- 2y^2 \frac{\partial f}{\partial y} \cdot f  \cdot \frac{\partial \phi_t}{\partial y} \cdot \phi_t \bigg) \; \frac{dz \; dq\, dq^\prime\, dq^{\prime\prime}\, dy}{y^{4r+3}} && \text{(C 1)}\\
\end{flalign*}
(A 1) and (B 1) combine to $\Delta$ so the above equals
\begin{flalign*}	
    =& \int_{(N \cap \Gamma) \backslash N} \int_{y \geq c} - \Delta f \cdot f \cdot \phi_t(y)^2 \; \frac{dz \; dq\, dq^\prime\, dq^{\prime\prime}\, dy}{y^{4r+3}}  && \text{(A 2)}\\
    +& \int_{(N \cap \Gamma) \backslash N} \int_{y \geq c} \bigg( \bigg( (4r+1) y \cdot \frac{\partial \phi_t}{\partial y} \cdot \phi_t - \; y^2 \cdot \frac{\partial^2 \phi_t}{\partial y^2} \cdot \phi_t \bigg)  \cdot f^2 \bigg) \; \frac{dz \; dq\, dq^\prime\, dq^{\prime\prime}\, dy}{y^{4r+3}} && \text{(B 2)}\\
    +& \int_{(N \cap \Gamma) \backslash N} \int_{y \geq c} \bigg( - 2y^2 \cdot \frac{\partial \phi_t}{\partial y} \cdot \phi_t \cdot \frac{\partial f}{\partial y} \cdot f \bigg) \; \frac{dz \; dq\, dq^\prime\, dq^{\prime\prime}\, dy}{y^{4r+3}} && \text{(C 2)}\\
\end{flalign*}
Note that $\frac{\partial f}{\partial y} f = \frac{1}{2}\frac{\partial}{\partial y}(f)^2$ and use integration by parts (changing the sign of the term) on (C 2) to move the derivative from $f$ giving (C3), factor out the common term of $|f|^2$, and expand terms in $\phi_t(y) = \phi(\frac{y}{t})$
\begin{flalign*}	
    =& \int_{(N \cap \Gamma) \backslash N} \int_{y \geq c} - \Delta f \cdot f \cdot \phi(\frac{y}{t})^2 \; \frac{dz \; dq\, dq^\prime\, dq^{\prime\prime}\, dy}{y^{4r+3}}  && \text{(A 3)}\\
    +& \int_{(N \cap \Gamma) \backslash N} \int_{y \geq c}  \Bigg( (4r-7) \frac{y}{t} \frac{\partial \phi}{\partial y}(\frac{y}{t}) \cdot \phi(\frac{y}{t}) \enspace - \enspace \frac{5y^2}{t^2} \cdot \frac{\partial^2 \phi}{\partial y^2}(\frac{y}{t}) \cdot \phi(\frac{y}{t})  \\
    & \qquad  - \frac{y^2}{t^2} \big( \frac{\partial \phi}{\partial y}(\frac{y}{t}) \big)^2 \Bigg) \cdot f^2 \enspace \frac{dz \; dq\, dq^\prime\, dq^{\prime\prime}\, dy}{y^{4r+3}} && \text{(B 3)}\\
\end{flalign*}

Since $0 \leq \phi \leq 1$, (A 2) is positive and clearly bounded by $\langle -\Delta f, f \rangle$. By assumption $0 \leq \phi \leq 1$, which applies also to $\phi_t$, so the first integral is bounded by $\int -\Delta f \cdot f$. The second expression is bounded in terms of $|f|_{L^2}^2$ as follows: $\phi^\prime$ and $\phi^{\prime\prime}$ are supported in $[1,2]$, so $\phi_t^\prime$ and $\phi_t^{\prime\prime}$ are supported in $[t,2t]$ and thus $y < 2t$ so $\frac{y}{t} < 2$. This along with the pointwise bound $C_\phi < \infty$ on $|\phi|$, $|\phi^\prime|$ and $|\phi^{\prime\prime}|$ estimate the the second integral (B3):
\begin{align*}
    \leq & \; \int_{(N \cap \Gamma) \backslash N} \int_{y \geq c} \Bigg( \; \bigg|  (4r-7) \frac{y}{t} \frac{\partial \phi}{\partial y}(\frac{y}{t}) \cdot \phi(\frac{y}{t})\, \bigg|\;  \;  \\
    & \enspace \quad +  \bigg| \; \frac{5y^2}{t^2} \cdot \frac{\partial^2 \phi}{\partial y^2}(\frac{y}{t}) \cdot \phi(\frac{y}{t}) \; \bigg| \;  \\
    & \qquad \quad + \bigg| \; \frac{y^2}{t^2} \bigg( \frac{\partial \phi}{\partial y}(\frac{y}{t}) \bigg)^2  \; \bigg| \; \Bigg) \cdot |f|^2 \;\;\frac{dz \; dq\, dq^\prime\, dq^{\prime\prime}\, dy}{y^{4r+3}} \\
    &\leq \int_{(N \cap \Gamma) \backslash N} \int_{y \geq c} \bigg( \text{ multiple of } C_\phi^2 \bigg) \cdot  |f|^2 \; \frac{dz \; dq\, dq^\prime\, dq^{\prime\prime}\, dy}{y^{4r+3}} = \text{constant} \times |f|^2_{L^2}
\end{align*}
as claimed $\square$
\end{proof}


%
%

\section{A gradient operator $\nabla$ on functions on $\Gamma \backslash G /K$}\label{Gradient-operator}
%
%
%
%
Following \cite{PBG}, to facilitate estimates using local Iwasawa coordinates, we define a type of gradient operator $\nabla$ on functions on $\Gamma \backslash G/K$ with a form of integration by parts property

$$
\int_{\Gamma \backslash G/K} -\Delta f \cdot \bar{f} =  \int_{\Gamma \backslash G/K} \langle \nabla f, \nabla f \rangle_\frak{s}
$$
where the inner product is taken on the vector space $\frak{s}$ in which $\nabla$ takes its values.

Let $\gamma \to \gamma^\theta$ be an involutive automorphism (e.g., a Cartan involution) on the Lie algebra of $G$ as $\frak{g} = \frak{s} + \frak{k}$ such that $\frak{k}$, the Lie algebra of the maximal compact subgroup $K$, is the $+1$ eigenspace, letting $\frak{s}$ be the $-1$ eigenspace. Here, for $O(r,1)$, we can take $\gamma^\theta = - \gamma^\top$. Let $\langle , \rangle_\frak{s}$ be a positive-definite real-valued inner product on $\frak{s}$, invariant under the action of $K$:
$$
\langle k \alpha k^{-1} , \; k \beta k^{-1} \rangle_\frak{s} = \langle  \alpha , \beta \rangle_\frak{s} \qquad (\text{for all } \alpha, \beta \in \frak{s} \text{ and } k \in K)
$$
The form $\langle, \rangle_\frak{s}$ can be given by restricing the trace to a {\it reduced} trace form which is proportional to the Cartan-Killing form
$$
\langle \alpha , \beta \rangle_\frak{s} = \langle \alpha , \beta \rangle_\text{trace} = \text{tr}(\alpha \cdot \beta) \qquad (\text{for } \alpha, \beta \in \frak{g})
$$
where tr is matrix trace. Extend scalars on $\frak{s}$ (i.e., by abuse of terminology, call $\frak{s} = \frak{s} \otimes_\mathbb{R} \mathbb{C}$) so that $\langle , \rangle_\frak{s}$ is positive-definite hermitian (i.e., take the second argument as complex-conjugate). Since $\Gamma$ is discrete in $G$ and differentiation is local, differentiating of right translation $X_x$ for $x \in \frak{g}$ acts on functions on $G$ or $\Gamma \backslash G$.
$$
(X_x f)(g) = \frac{d}{dt}\bigg|_{t=0} f(g \cdot e^{tx}) \qquad (\text{for } f \in C^1 \text{ and } g \in G \text{ or } \Gamma \backslash G)
$$
Name this map $\rho(x): x \to X_x$; $\rho$ is a $K$-equivariant $\frak{g}$-valued function since for $f \in C^1$ (i.e., considered either on $G$ or $\Gamma \backslash G$)
\begin{align*}
\big( k \cdot \rho(x) \cdot k^{-1} \cdot f\big)(g) =& \; k \cdot X_x \cdot k^{-1} \cdot f(g) = (k ( X_x (k^{-1} f)))(g)  = ( X_x (k^{-1} f))(gk) \\
=&  \frac{d}{dt}\bigg|_{t=0} ( (k^{-1} f)(g \; k \; e^{tx}) = \frac{d}{dt}\bigg|_{t=0}  f(g \; k \; e^{tx} \; k^{-1}) = \frac{d}{dt}\bigg|_{t=0} f (g e^{tk xk^{-1}}) \\
=& X_{kxk^{-1}} f(g)
\end{align*}
 The operator $\nabla$ is an $\frak{s}$-valued operator on functions on either $G$ or $\Gamma \backslash G$. Similar to the characterization of the Casimir operator, the commutativity of the diagram
$$
\xymatrix{
\mathrm{End}_\mathbb{R}(\frak{s})  \ar@{.>}[d] \ar[r]^\approx & \frak{s} \otimes_\mathbb{R} \frak{s}^\ast \ar[r]^{\langle, \rangle_\frak{s}} & \frak{s} \otimes_\mathbb{R} \frak{s} \ar[r]^{\rho \otimes_\mathbb{R} 1_\frak{s}} & \rho(\frak{s}) \otimes_\mathbb{R} \frak{s} \ar@{.>}[d] \\
1_\frak{s} \ar[rrr]                          &&&                      \nabla 
}
$$
exhibits the $K$-equivariance of $\nabla$.

Identify $\frak{s}^\ast$ with $\frak{s}$ via $x \to \langle -, x \rangle_\frak{s}$ to express $\nabla$ in coordinates in terms of a $\langle , \rangle_\frak{s}$-orthonormal basis of $\frak{s}$
$$
\nabla = \sum_j X_{x_j} \cdot x_j \qquad \nabla f = \sum_j X_{x_j}f \cdot x_j \qquad (\text{for } f \in C^\infty)
$$
the $K$-equivariance of the diagram shows this expession of $\nabla$ is independent of the basis of $\frak{s}$ 
%
%
\begin{lemma}\label{grad-int-by-parts}
	For $f \in C^\infty_c(\Gamma \backslash G/K) = C^\infty_c(\Gamma \backslash G)^K$,
$$
\int_{\Gamma \backslash G} - \Delta f \cdot \bar{f} = \int_{\Gamma \backslash G} \langle \nabla f, \nabla f \rangle_\frak{s}
$$	
\end{lemma}
\begin{proof}
Since $\langle , \rangle_\frak{s}$ is proportional to the trace form, write $x_i$ for $X_{x_i} = \rho(x_i)$ for $x_i \in \frak{s}$, and let $\theta_j$ be a basis for $\frak{k}$ such that $\langle \theta_i, \theta_j\rangle_\text{trace} = -\delta_{ij}$ with Kronecker and the trace pairing. As before, the Casimir operator is then the image of $\sum_j x_j^2 - \sum_i \theta_i^2$ in the universal enveloping algebra. On right $K$-invariant functions, the $\theta$ terms vanish and the Casimir operator reduces to $\sum_j x_j^2$. Integration by parts gives
\begin{align*}
\int_{\Gamma \backslash G} -\Delta f \cdot \bar{f} =& \int_{\Gamma \backslash G} -\sum_j x_j^2 f \cdot \bar{f} = \int_{\Gamma \backslash G} \sum_j x_j f \cdot \overline{x_j f} \\
=& \int_{\Gamma \backslash G} \sum_j \langle x_j f \cdot x_j, x_j f \cdot x_j \rangle_\frak{s} = \int_{\Gamma \backslash G} \langle \nabla f , \nabla f \rangle_\frak{s} 
\end{align*}
$\square$
\end{proof}

%
%
%
%
%
%

\section{With $a \gg 1$, the inclusion $\frak{B}^1_a(\Gamma \backslash G/K) \to L^2_a(\Gamma \backslash G/K)$ is compact.}\label{Rellich-compact-Inclusion-FrakBa-to-L2a}

%
%

We follow the approach of \cite{PBG} which elaborates on the developement treated in \cite{LP}. The idea used in all three cases, for $O(r,1)$, $U(r,1)$ and $Sp^\ast(r,1)$, is that the usual Rellich compactness lemma, asserting compactness of proper inclusions of Sobolev spaces on compact Riemannian symmetric spaces, as proven by localizing to multi-tori $\mathbb{T}^r$ (cf. \S\S 9.5.12 and 9.5.15 in \cite{PBG}, and also the appendix), reduces the issue to the estimates established above on $L^2$ and smooth cut-off truncations of the tails.

The total boundedness criterion for pre-compactness 
requires that, given $\varepsilon > 0$, the image in $L^2_a(\Gamma \backslash G/K)$ of the unit ball $B$ in $\frak{B}^1_a(\Gamma \backslash G/K)$ can be covered by finitely-many balls of radius $\varepsilon$ (cf. \S 14.7 in \cite{PBG}). We show this by using the estimates on the tails of smooth cutoffs developed previously. 

\subsection{$O(r,1)$}\label{Rellich-compact-Inclusion-FrakBa-to-L2a-O(r,1)}

\begin{theorem}
	For $G = O(r,1)$, $K = O(r) \times O(1)$ and arithmetic subgroup $\Gamma \subset G$, then with $a \gg 1$ the inclusion $\frak{B}^1_a(\Gamma \backslash G/K) \to L^2_a(\Gamma \backslash G/K)$ is compact
\end{theorem}

\begin{proof}
Let $(x,y)$ be coordinates from the Iwasawa decomposition: $x \in N \approx  \mathbb{R}^{r-1}$, $y \in A^+ \approx  \mathbb{R}^\times$, by reduction theory, let $\frak{S}$ be a fixed Siegel set that surjects to the quotient $\Gamma \backslash G$. Then, given $c \geq a$, let $Y_o$ and $Y_\infty$ be the respective images of $\{g \in \frak{S} \; : \; \eta(g) = \eta(x,y) = y^r  \leq c+1\}$ and $\{g \in \frak{S} \; : \; \eta(g) = y^r \geq c \}$ in $\Gamma \backslash G/K$. By construction, the interiors of $Y_o$ and $Y_\infty$ cover $\Gamma \backslash G/K$. 

From general considerations we can take a smooth partition of unity $\{ \phi_i \}$ 
so 	 $\sum_i \phi_i = 1$ on $Y_o$. Let $\phi_\infty$ be a smooth function that is identically 1 for $\eta \geq c$ and that $\text{support}(\phi_\infty) \subset U_\infty$. The smooth function $\phi_\infty$ should be chosen in a fashion similar to the smooth cut-offs used previously: i.e., $\phi_\infty$ is $0$ for $0 \leq \eta \leq c-1$, $\phi_\infty$ is $1$ for $\eta \geq c$ and all values of $\phi_\infty$ lie between $0$ and $1$. In particular, $\phi_\infty$ is zero on a non-empty open contained in $Y_o$.
It will also be useful later that $\phi_\infty$ and $1 - \phi_\infty$ form a two-element partition of unity.

Iwasawa coordinates provide an explicit way of separating $\Gamma \backslash G/K$ into a cusp part $Y_\infty$ and a compact body $Y_o$.
Let $B$ be the unit ball in $\frak{B}^1_a(\Gamma\backslash G/K)$. Given $\varepsilon > 0$, take $c - 1 > c^\prime \gg a$ sufficiently large (with associated $Y_o$, $Y_\infty$ and smooth cut-off $\phi_\infty$ determined by $c$) so that $\phi_\infty \cdot B$ lies in a single $\varepsilon \backslash 2$-ball in $L^2_a(\Gamma \backslash G /K)$. That is, using \cref{sec:L2a-tail-norm-bded-by-B1a-norm}, choose $c^\prime \gg a$ so that
$$
\int_{(N \cap \Gamma) \backslash N} \int_{y > c^\prime} |f|^2 \; \frac{dx \; dy}{y^r} \ll \frac{1}{(c^\prime)^2} \cdot |f|^2_{\frak{B}^1_a(\Gamma \backslash G /K)}
$$
Choose $\phi_\infty$, such that $0 \leq \phi_\infty \leq 1$ and $\phi_\infty = 0$ for $y \leq c-1$:
$$
 \int_{(N \cap \Gamma) \backslash N} \int_{y \geq a} |\phi_\infty \cdot f|^2 \; \frac{dx \; dy}{y^r} \leq \int_{(N \cap \Gamma) \backslash N} \int_{y > c^\prime} |f|^2 \; \frac{dx \; dy}{y^r} \ll \frac{1}{(c^\prime)^2} \cdot |f|^2_{\frak{B}^1_a(\Gamma \backslash G /K)}
$$
That is,
$$
|\phi_\infty \cdot f|^2_{L^2_a(\Gamma \backslash G /K)} \ll \frac{1}{(c^\prime)^2} \cdot |f|^2_{\frak{B}^1_a(\Gamma \backslash G /K)}
$$
Specifically, choose large $c^\prime$ so that
$$
|\phi_\infty \cdot f|^2_{L^2_a(\Gamma \backslash G /K)}  < \frac{\varepsilon}{2} <  \frac{1}{(c^\prime)^2} \cdot |f|^2_{\frak{B}^1_a(\Gamma \backslash G /K)}
$$
Next, note that $(1 - \phi_\infty) \cdot f$ has support in $Y_o$ for any $f \in \frak{B}^1_a(\Gamma \backslash G /K)$ and the map
$$
(1 - \phi_\infty) \cdot \; : \: \frak{B}^1_a(\Gamma \backslash G /K) \to \frak{B}^1_a(Y_o)
$$
is continuous and therefore bounded by \cref{sec:bdedness-smooth-projections}.

Thus, the image of $B$ under the restriction of this map
$$
B \to (1 -\phi_\infty) \cdot B \to \frak{B}^1_a(Y_o) \to L^2(Y_o) \subset L^2(\Gamma \backslash G /K)
$$
is pre-compact and can be covered by finitely many $\varepsilon/2$ balls in $L^2(\Gamma \backslash G /K)$. Finally, $B = \phi_\infty \cdot B + (1 - \phi_\infty) \cdot B$ is therefore covered by finitely-many $\varepsilon/2$ balls in $L^2_a(\Gamma \backslash G /K)$ so the image of $B$ in $L^2(\Gamma \backslash G/K)$ is pre-compact. $\square$
\end{proof}
%
%
%
%
%

\subsection{$U(r,1)$}\label{Rellich-compact-Inclusion-FrakBa-to-L2a-U(r,1)}

\begin{theorem}
	For $G = U(r,1)$, $K = U(r) \times U(1)$ and arithmetic subgroup $\Gamma \subset G$, then with $a \gg 1$ the inclusion $\frak{B}^1_a(\Gamma \backslash G/K) \to L^2_a(\Gamma \backslash G/K)$ is compact
\end{theorem}
\begin{proof}
Let  $(z,x,y)$ be coordinates from the Iwasawa decomposition: $z,x \in N \approx  \mathbb{C}^{r-1} \times \mathbb{R}$, $y \in A^+ \approx  \mathbb{R}^\times$, by reduction theory, let $\frak{S}$ be a fixed Siegel set that surjects to the quotient $\Gamma \backslash G$. Then, given $c \geq a$, let $Y_o$ and $Y_\infty$ be the respective images of $\{g \in \frak{S} \; : \; \eta(g) = \eta(z,x,y) = y^{2r+1}  \leq c+1\}$ and $\{g \in \frak{S} \; : \; \eta(g) = y^{2r+1} \geq c \}$ in $\Gamma \backslash G/K$. By construction, the interiors of $Y_o$ and $Y_\infty$ cover $\Gamma \backslash G/K$. 

From general considerations we can take a smooth partition of unity $\{ \phi_i \}$ 
so 	 $\sum_i \phi_i = 1$ on $Y_o$. Let $\phi_\infty$ be a smooth function that is identically 1 for $\eta \geq c$ and that $\text{support}(\phi_\infty) \subset U_\infty$. The smooth function $\phi_\infty$ should be chosen in a fashion similar to the smooth cut-offs used previously: i.e., $\phi_\infty$ is $0$ for $0 \leq \eta \leq c-1$, $\phi_\infty$ is $1$ for $\eta \geq c$ and all values of $\phi_\infty$ lie between $0$ and $1$. In particular, $\phi_\infty$ is zero on a non-empty open contained in $Y_o$.
It will also be useful later that $\phi_\infty$ and $1 - \phi_\infty$ form a two-element partition of unity.

Iwasawa coordinates provide an explicit way of separating $\Gamma \backslash G/K$ into a cusp part $Y_\infty$ and a compact body $Y_o$.
Let $B$ be the unit ball in $\frak{B}^1_a(\Gamma\backslash G/K)$. Given $\varepsilon > 0$, take $c - 1 > c^\prime \gg a$ sufficiently large (with associated $Y_o$, $Y_\infty$ and smooth cut-off $\phi_\infty$ determined by $c$) so that $\phi_\infty \cdot B$ lies in a single $\varepsilon \backslash 2$-ball in $L^2_a(\Gamma \backslash G /K)$. That is, using \cref{sec:L2a-tail-norm-bded-by-B1a-norm}, choose $c^\prime \gg a$ so that
$$
\int_{(N \cap \Gamma) \backslash N} \int_{y > c^\prime} |f|^2 \; \frac{dz \; dx \; dy}{y^{2r+1}} \ll \frac{1}{(c^\prime)^2} \cdot |f|^2_{\frak{B}^1_a(\Gamma \backslash G /K)}
$$
Choose $\phi_\infty$, such that $0 \leq \phi_\infty \leq 1$ and $\phi_\infty = 0$ for $y \leq c-1$:
$$
 \int_{(N \cap \Gamma) \backslash N} \int_{y \geq a} |\phi_\infty \cdot f|^2 \; \frac{dz \; dx \; dy}{y^{2r+1}} \leq \int_{(N \cap \Gamma) \backslash N} \int_{y > c^\prime} |f|^2 \; \frac{dz \; dx \; dy}{y^{2r+1}} \ll \frac{1}{(c^\prime)^2} \cdot |f|^2_{\frak{B}^1_a(\Gamma \backslash G /K)}
$$
That is,
$$
|\phi_\infty \cdot f|^2_{L^2_a(\Gamma \backslash G /K)} \ll \frac{1}{(c^\prime)^2} \cdot |f|^2_{\frak{B}^1_a(\Gamma \backslash G /K)}
$$
Specifically, choose large $c^\prime$ so that
$$
|\phi_\infty \cdot f|^2_{L^2_a(\Gamma \backslash G /K)}  < \frac{\varepsilon}{2} <  \frac{1}{(c^\prime)^2} \cdot |f|^2_{\frak{B}^1_a(\Gamma \backslash G /K)}
$$
Next, note that $(1 - \phi_\infty) \cdot f$ has support in $Y_o$ for any $f \in \frak{B}^1_a(\Gamma \backslash G /K)$ and the map
$$
(1 - \phi_\infty) \cdot \; : \: \frak{B}^1_a(\Gamma \backslash G /K) \to \frak{B}^1_a(Y_o)
$$
is continuous and therefore bounded by \cref{sec:bdedness-smooth-projections}.

Thus, the image of $B$ under the restriction of this map
$$
B \to (1 -\phi_\infty) \cdot B \to \frak{B}^1_a(Y_o) \to L^2(Y_o) \subset L^2(\Gamma \backslash G /K)
$$
is pre-compact and can be covered by finitely many $\varepsilon/2$ balls in $L^2(\Gamma \backslash G /K)$. Finally, $B = \phi_\infty \cdot B + (1 - \phi_\infty) \cdot B$ is therefore covered by finitely-many $\varepsilon/2$ balls in $L^2_a(\Gamma \backslash G /K)$ so the image of $B$ in $L^2(\Gamma \backslash G/K)$ is pre-compact. $\square$
\end{proof}
%
%
%

%
%

\subsection{$Sp^\ast(r,1)$}\label{Rellich-compact-Inclusion-FrakBa-to-L2a-Sp*(r,1)}

\begin{theorem}
	For $G = O(r,1)$, $K = O(r) \times O(1)$ and arithmetic subgroup $\Gamma \subset G$, then with $a \gg 1$ the inclusion $\frak{B}^1_a(\Gamma \backslash G/K) \to L^2_a(\Gamma \backslash G/K)$ is compact
\end{theorem}
\begin{proof}
	Let $z,q,q^\prime,q^{\prime\prime}$ be coordinates from the Iwasawa decomposition: $z,q,q^\prime,q^{\prime\prime} \in N \approx  \mathbb{H}^{r-1} \times \mathbb{R}^3$, $y \in A^+ \approx  \mathbb{R}^\times$, by reduction theory, let $\frak{S}$ be a fixed Siegel set that surjects to the quotient $\Gamma \backslash G$. Then, given $c \geq a$, let $Y_o$ and $Y_\infty$ be the respective images of $\{g \in \frak{S} \; : \; \eta(g) = \eta(z,q,q^\prime,q^{\prime\prime},y) = y^{4r+3}  \leq c+1\}$ and $\{g \in \frak{S} \; : \; \eta(g) = y^{4r+3} \geq c \}$ in $\Gamma \backslash G/K$. By construction, the interiors of $Y_o$ and $Y_\infty$ cover $\Gamma \backslash G/K$. 

	From general considerations we can take a smooth partition of unity $\{ \phi_i \}$ 
	so 	 $\sum_i \phi_i = 1$ on $Y_o$. Let $\phi_\infty$ be a smooth function that is identically 1 for $\eta \geq c$ and that $\text{support}(\phi_\infty) \subset U_\infty$. The smooth function $\phi_\infty$ should be chosen in a fashion similar to the smooth cut-offs used previously: i.e., $\phi_\infty$ is $0$ for $0 \leq \eta \leq c-1$, $\phi_\infty$ is $1$ for $\eta \geq c$ and all values of $\phi_\infty$ lie between $0$ and $1$. In particular, $\phi_\infty$ is zero on a non-empty open contained in $Y_o$.
	It will also be useful later that $\phi_\infty$ and $1 - \phi_\infty$ form a two-element partition of unity.

	Iwasawa coordinates provide an explicit way of separating $\Gamma \backslash G/K$ into a cusp part $Y_\infty$ and a compact body $Y_o$.
	Let $B$ be the unit ball in $\frak{B}^1_a(\Gamma\backslash G/K)$. Given $\varepsilon > 0$, take $c - 1 > c^\prime \gg a$ sufficiently large (with associated $Y_o$, $Y_\infty$ and smooth cut-off $\phi_\infty$ determined by $c$) so that $\phi_\infty \cdot B$ lies in a single $\varepsilon \backslash 2$-ball in $L^2_a(\Gamma \backslash G /K)$. That is, using \cref{sec:L2a-tail-norm-bded-by-B1a-norm}, choose $c^\prime \gg a$ so that
	$$
	\int_{(N \cap \Gamma) \backslash N} \int_{y > c^\prime} |f|^2 \; \frac{dz \, dq\, dq^\prime\, dq^{\prime\prime}\, dy}{y^{4r+3}} \ll \frac{1}{(c^\prime)^2} \cdot |f|^2_{\frak{B}^1_a(\Gamma \backslash G /K)}
	$$
	Choose $\phi_\infty$, such that $0 \leq \phi_\infty \leq 1$ and $\phi_\infty = 0$ for $y \leq c-1$:
	$$
	 \int_{(N \cap \Gamma) \backslash N} \int_{y \geq a} |\phi_\infty \cdot f|^2 \; \frac{dz \, dq\, dq^\prime\, dq^{\prime\prime}\, dy}{y^{4r+3}} \leq \int_{(N \cap \Gamma) \backslash N} \int_{y > c^\prime} |f|^2 \; \frac{dz \, dq\, dq^\prime\, dq^{\prime\prime}\, dy}{y^{4r+3}} \ll \frac{1}{(c^\prime)^2} \cdot |f|^2_{\frak{B}^1_a(\Gamma \backslash G /K)}
	$$
	That is,
	$$
	|\phi_\infty \cdot f|^2_{L^2_a(\Gamma \backslash G /K)} \ll \frac{1}{(c^\prime)^2} \cdot |f|^2_{\frak{B}^1_a(\Gamma \backslash G /K)}
	$$
	Specifically, choose large $c^\prime$ so that
	$$
	|\phi_\infty \cdot f|^2_{L^2_a(\Gamma \backslash G /K)}  < \frac{\varepsilon}{2} <  \frac{1}{(c^\prime)^2} \cdot |f|^2_{\frak{B}^1_a(\Gamma \backslash G /K)}
	$$
	Next, note that $(1 - \phi_\infty) \cdot f$ has support in $Y_o$ for any $f \in \frak{B}^1_a(\Gamma \backslash G /K)$ and the map
	$$
	(1 - \phi_\infty) \cdot \; : \: \frak{B}^1_a(\Gamma \backslash G /K) \to \frak{B}^1_a(Y_o)
	$$
	is continuous and therefore bounded by \cref{sec:bdedness-smooth-projections}.

	Thus, the image of $B$ under the restriction of this map
	$$
	B \to (1 -\phi_\infty) \cdot B \to \frak{B}^1_a(Y_o) \to L^2(Y_o) \subset L^2(\Gamma \backslash G /K)
	$$
	is pre-compact and can be covered by finitely many $\varepsilon/2$ balls in $L^2(\Gamma \backslash G /K)$. Finally, $B = \phi_\infty \cdot B + (1 - \phi_\infty) \cdot B$ is therefore covered by finitely-many $\varepsilon/2$ balls in $L^2_a(\Gamma \backslash G /K)$ so the image of $B$ in $L^2(\Gamma \backslash G/K)$ is pre-compact. $\square$
	\end{proof}
	%
	%
	%

%
%
%
%
%
%
%
%

\section{$\tilde{\Delta}_a$ has purely discrete spectrum}\label{Delta-a-has-purely-discrete-spectrum}

The Rellich-like result showing the compactness of the inclusion $\frak{B}^1_a(\Gamma\backslash G/K) \to L^2_a(\Gamma\backslash G/K)$ for $a \gg 1$ established for $O(r,1)$, $U(r,1)$ and $Sp^\ast(r,1)$, is used to show that the Friedrichs self-adjoint extension $\widetilde{\Delta}_a$ of the restriction $\Delta_a$ of $\Delta$ to test functions $\mathscr{D}_a$ in $L^2_a$ has compact resolvent, thus establishing that $\widetilde{\Delta}_a$ has purely discrete spectrum.

We briefly recap the exposition in the section \hyperref[sec:Friedrichs-extension-defn]{Friedrichs' canonical self-adjoint extensions}. Let $T: V \to V$ be a positive, semi-bounded operator on a Hilbert space $V$ with dense domain $D$ and define a hermitian form $\langle ,\rangle_1$ and corresponding norm $| \cdot |_1$ by
$$
\langle v,w \rangle_1 = \langle v,w\rangle + \langle Tv,w\rangle = \langle v,(1+T)w \rangle = \langle (1+T)v,w \rangle \quad (\text{for } v,w \in D)
$$
The symmetry and positivity of $T$ make $\langle , \rangle_1$ positive-definite hermitian on $D$, and $\langle v, w \rangle_1$ is defined if at least one of $v$, $w$ is in $D$. Let Let $V^1$ be the Sobolev-like Hilbert-space defined by the completion of $D$ with respect to the metric induced by the norm $| \cdot |_1$ on $D$ (the definition of $V^1$ is analogous to standard definitions of $L^2$ Sobolev spaces via $\Delta$ or the Fourier transform). Since the norm $| \cdot |_1$ dominates the norm on $V$ (by positivity of $T$), the completion $V^1$ maps continuously to $V$. Since these are Hilbert spaces, the map is injective. Friedrichs' Theorem then tells us there is a positive self-adjoint extension $\widetilde{T}$ with domain $\widetilde{D} \subset V^1$. In particular, there is an important corollary: When the inclusion $V^1 \to V$ is compact, the resolvent $(1 + \widetilde{T})^{-1} : V \to V$ is compact. Substituting $L^2_a(\Gamma\backslash G/K)$ for $V$, $\mathscr{D}_a$ for $D$ and $\frak{B}^1_a$ for $V^1$ in the above, we have that the resolvent $(1 - \widetilde{\Delta}_a)^{-1}$ is compact.

Then, using the result \hyperref[sec:spectrum-T-spectrum-compact-resolvent]{recovering the spectrum of an operator from its resolvent} described earlier, we have that $\widetilde{\Delta}_a$ has discrete spectrum and the spectrum of $\widetilde{\Delta}_a$ matches that of its resovlent $(\lambda - \widetilde{\Delta}_a)^{-1}$.  We follow the development in \cite{PBG} \S\S 10.9 and 9.4: for $\lambda$ off a discrete set $X$ in $\mathbb{C}$, the inverse $(\widetilde{\Delta}_a - \lambda)^{-1}$ exists, is a compact operator and the operator-valued function
$$
\lambda \to  (\tilde{\Delta}_a - \lambda)^{-1} \enspace \text{as a map} \enspace L^2_a(\Gamma \backslash G /K) \to L^2_a(\Gamma \backslash G /K)
$$
is meromorphic in $\lambda \in \mathbb{C} - X$. The decomposition of $L^2_a(\Gamma \backslash G /K)$ is discrete: there is an orthogonal basis of $L^2_a(\Gamma \backslash G /K)$ consisting of $\widetilde{\Delta}_a$-eigenvectors. The eigenvectors of $\widetilde{\Delta}_a$ are eigenvectors of $(\widetilde{\Delta}_a - \lambda)^{-1}$ for every $\lambda$ not in the spectrum of $\widetilde{\Delta}_a$, and eigenvalues $\lambda$ of $\widetilde{\Delta}_a$ are in bijection with non-zero eigenvalues of $(\widetilde{\Delta}_a - \lambda)^{-1}$ by $\lambda \leftrightarrow (1-\lambda)^{-1}$.

%
%

\section{Discrete decompostion of pseudo-cuspforms}\label{Discrete-Decomp-Pseudo-Cuspforms}

For $a \gg 1$, the space of pseudo-cuspforms, namely functions in $L^2_a(\Gamma \backslash G/K)$ whose constant terms vanish above height $\eta(g) = a$ (the case a = 0 is the usual space of $L^2$ cuspforms) decomposes discretely for $\widetilde{\Delta}_a$.

%
%

\section{Full meromorphic continuation of Eisenstein series for $O(r,1)$, $U(r,1)$ and $Sp^\ast(r,1)$}\label{Mero-Contn-Eis}

 The estimates and compact inclusion results established the discreteness of the spectrum of the Friedrichs extension $\widetilde{\Delta}_a$ to the restriction of the Laplacian $\Delta$ to the space $\mathscr{D}_a$ automorphic test functions.  We reproduce the approach explicated in \cite{PBG} to show the meromorphic continuation of the Eisenstein series $E_s$.
\begin{theorem}
$E_s$ has a meromorphic continuation in $s \in \mathbb{C}$, as a smooth function of moderate growth on $\Gamma\backslash G$. As a function of $s$, $E_s(g)$ is of at most polynomial growth vertically, which is uniform in bounded strips and for $g$ in compact subsets of $G$.	(proof below)
\end{theorem}
Some consequences of the meromorphic continuation can be inferred quickly.
\begin{corollary}
The eigenfunction property 
$$
\Delta E_s = \lambda_s \cdot E_s \text{ where } \lambda_s = c^2 \cdot s(s - 1) \text{ for suitable } c \in \mathbb{R}
$$
persists under meromorphic continuation.	
\end{corollary}
\begin{proof}
Both $\Delta E_s$ and $\lambda_s \cdot E_s$ are holomorphic function-valued functions of $s$, taking values in the topological vector space of smooth functions. They agree in the region of convergence $\text{Re}(s) > 1$, so by the vector-valued form of the identity principle (\cite{PBG} \S 15.2) they agree on their mutual domain of convergence. $\square$
\end{proof}
\begin{corollary}
The meromorphic continuation of $E_s$ implies the meromorphic continuation of the constant term $c_P E_s = \eta^s + c_s \eta^{1-s}$, in particular, of the function $c_s$.
\end{corollary}
\begin{proof}
Since $E_s$ meromorphically continues at least as a smooth function, the integral over the compact set $(N \cap \Gamma)\backslash N$ giving a pointwise value $c_P E_s(g)$ of the constant term certainly converges absolutely. That is, the function-valued function 
$$
n \longrightarrow (g \to E_s(ng))
$$
is a continuous, smooth-function-valued function and has a smooth-function-valued Gelfand-Pettis integral 
$$
g \longrightarrow c_P E_s(g) 
$$
(see \cite{PBG} \S 14.1). Thus, the constant term $c_P E_s$ of the continuation of $E_s$ must still be of the form $A_s \eta^s + B_s \eta^{1-s}$ for some smooth functions $A_s$ and $B_s$, since (at least for $s \neq 1$ ) $\eta^s$ and $\eta^{1-s}$ are the two linearly independent solutions of
$$
\Delta f = \lambda_s \cdot f
$$ 
for functions $f$ on $N\backslash G/K \approx A^+$. Thus, in the region of convergence $\text{Re}(s) > 1$, the linear independence of $\eta^s$  and $\eta^{1-s}$ gives $A_s = 1$ and $B_s = c_s$. The vector-valued form of analytic continuation implies that $A_s = 1$ throughout, and that $B_s = c_s$ throughout. In particular, this establishes the meromorphic continuation of $c_s$. $\square$
\end{proof}

Fix $a \gg 1$. Let $\widetilde{\Delta}_a$  be the Friedrichs extension of the restriction of the Laplacian $\Delta$ to $\mathscr{D}_a = C^\infty_c(\Gamma\backslash G/K) \cap L^2_a(\Gamma\backslash G/K)$. The Friedrichs construction shows that the domain of $\widetilde{\Delta}_a$  is contained in a Sobolev space:
$$
\text{domain } \widetilde{\Delta}_a \subset \frak{B}^1 =  \text{ completion of } \mathscr{D}_a \text{ relative to } \langle v, w \rangle_{\frak{B}^1} = \langle (1 - \Delta ) v, w \rangle 
$$
The domain of $\widetilde{\Delta}_a$  contains the smaller Sobolev space
$$
\frak{B}^2 =  \text{ completion of } \mathscr{D}_a \text{ relative to } \langle v, w\rangle_{\frak{B}^2} = \langle (1 - \Delta)^2 v, w \rangle 
$$
As before, use conditions on the height $\eta$ to decompose the quotient $\Gamma\backslash G/K$ as a union of a compact part $Y_{cpt} = Y_o$, whose geometry does not matter, and a geometrically simple non-compact part $Y_\infty$:
$$
\Gamma\backslash G/K = Y_o \cup Y_\infty \qquad \text{(compact  } Y_o, \text{ ``tubular'' cusp neighborhood } Y_\infty)
$$
relative to a condition on the normalized height function $\eta$ with $\eta(n \cdot m_y \cdot k) = a \gg 1$.
$$
Y_\infty = \text{ image of } \{ g \in G/K : \eta(g) \geq a \} = \Gamma_\infty \backslash \{g \in G/K : \eta (g) \geq a \}
$$

Define a smooth cut-off function $\tau$ as usual: fix $a^{\prime\prime} < a^\prime$ large enough so that the image of $\{ (x, y) \in G/K : y > a^{\prime\prime} \}$ in the quotient is in $Y_\infty$, and let
$$
\tau(g) = 
\begin{cases} 
1 & (\text{for } \eta(g) > a^\prime \\
0   & (\text{for } \eta(g) < a^{\prime\prime}
\end{cases}
$$
Form a pseudo-Eisenstein series $h_s$ by winding up the smoothly cut-off function $\tau(g) \cdot \eta(g)^s$:
$$
h_s(g) =  \sum_{\gamma \in \Gamma_\infty \backslash \Gamma} \tau(\gamma g) \cdot \eta(\gamma g)^s
$$
Since $\tau$ is supported on $\eta \geq a^{\prime\prime}$ for large $a^{\prime\prime}$, for any $g \in G/K$ there is at most one non-vanishing summand in the expression for $h_s$, and convergence is not an issue. Thus, the pseudo-Eisenstein series $h_s$ is entire as a function-valued function of $s$. Let 
$$
\widetilde{E}_s = h_s - (\widetilde{\Delta}_a - \lambda_s)^{-1} (\Delta - \lambda_s)h_s  \qquad (\text{where } \lambda_s = c \cdot s(s-1)) 
$$
As earlier, we have
\begin{claim}
$\widetilde{E}_s - h_s$ is a holomorphic $\frak{B}^1$-valued function of $s$ for $\text{Re}(s) > 1$ and $\text{Im}(s) \neq 0$.	
\end{claim} 
\begin{proof}
From Friedrichs' construction, the resolvent $(\widetilde{\Delta}_a - \lambda_s)^{-1}$ exists as an everywhere-defined,
continuous operator for $s \in \mathbb{C}$ for $\lambda_s$ not a non-positive real number, because of the non-positive-ness of $\Delta$. Further, for $\lambda_s$ not a non-positive real, the resolvent is a holomorphic operator-valued function. In fact, for such $\lambda_s$, the resolvent $(\widetilde{\Delta}_a - \lambda_s)^{-1}$ injects from $L^2(\Gamma\backslash G/K)$ to $\frak{B}^1$. $\square$	
\end{proof}
\begin{remark}
The smooth function $(\Delta - \lambda_s) h_s$ is supported on the image of $a^{\prime\prime} \leq y \leq a^\prime$ in $\Gamma\backslash G/K$, which is compact. Thus, it is in $L^2(\Gamma\backslash G/K)$. Note that $\widetilde{E}_s$ does not vanish since the resolvent maps to the domain of $\Delta$ inside $L^2(\Gamma\backslash G/K)$, and that $h_s$ is not in $L^2(\Gamma\backslash G/K)$ for $\text{Re}(s) > \frac{1}{2}$. Thus, since $h_s$ is not in $L^2(\Gamma\backslash G/K)$ and $(\widetilde{\Delta}_a - \lambda_s)^{-1}(\Delta - \lambda_s) h_s$ is in $L^2(\Gamma\backslash G/K)$, the difference cannot vanish.
\end{remark}
And as before,
\begin{theorem}
If $\lambda_s = c \cdot s(s - 1)$ is not non-positive real, then $u = \widetilde{E}_s- h_s$ is the unique element of the domain of $\widetilde{\Delta}_a$  such that 
$$
(\widetilde{\Delta}_a - \lambda_s) u = -(\Delta - \lambda_s) h_s
$$
Thus, $\widetilde{E}_s$ is the usual Eisenstein series $E_s$ for $\text{Re}(s) > 1$. (proof as in the earlier section \hyperref[sec:Mero-Contn-Eis]{meromorphic continuation of Eisenstein series to $\text{Re}(s) > \frac{1}{2}$})
\end{theorem}

\textbf{\emph{Proof of full meromorphic continuation}}: since the resolvent $(\widetilde{\Delta}_a - \lambda_s)^{-1}$ is a compact operator, $\widetilde{\Delta}_a$ has purely discrete spectrum. Thus, the resolvent $(\widetilde{\Delta}_a - \lambda_s)^{-1}$ is meromorphic in $s$ in $\mathbb{C}$, and thus $E_s = h_s + (\widetilde{\Delta}_a - \lambda_s)^{-1}(\Delta - \lambda_s) h_s$ is meromorphic in $s$ in $\mathbb{C}$. $\square$


%
%

\section{More general $\mathbb{Q}$-rank one orthogonal groups}\label{more-general-rational-rank-one-orthogonal}

We next address a general class of $\mathbb{Q}$-rank one groups. In this case the underlying group $G = O(S)$ preserving a bilinear form $S$ is a product of simple groups. While we will describe this product and the groups that arise as factors, we will treat $G$ directly. 

%
%
%
%
\subsection{Unified treatment of $G$ at archimedean places}\label{unified-treatment-of-O(S)-at-arch-places}

Let $k$ be a number field (finite extension of $\mathbb{Q}$) of degree $[k:\mathbb{Q}]=n$, and $S$ a $k$-valued $k$-bilinear form on an $(r+1)$-$k$-dimensional $k$-vector space $V$ with $r+1 \geq 5$. Let $G$ be the orthogonal group of the form $S$ and assume  that the $k$-dimension of the maximal totally $S$-isotropic $k$-subspace of $V$ is one. The $\mathbb{Q}$-rational points $G(\mathbb{Q})$ of the associated orthogonal subgroup $G = O(S)$ consist of the $k$-linear maps of $V$ preserving $S$. 

Let $k_\infty$ be the {\'e}tale algebra $k_\infty  = k \otimes_\mathbb{Q} \mathbb{R} \approx \bigoplus_{\nu | \infty}  k_\nu$ with canonical diagonal copy of $k$ given by the sum of the local archimedean embeddings:
$$
\sigma_\infty : k \hookrightarrow k_\infty \qquad \sigma_\infty = \bigoplus_{\nu | \infty} \sigma_\nu \qquad \sigma_\nu : k \hookrightarrow k_\nu
$$

The group $G(\mathbb{R})$ of real points of this algebraic group is a literal orthogonal group over $k_\infty$
$$
G = G(\mathbb{R}) \approx \prod_{\nu | \infty} O(\sigma_\nu(S), k_\nu) 
$$
where $\sigma_\nu(S)$ is the image of $S$ at the place $\nu$ corresponding to the embedding $\sigma_\nu: k \to k_\nu$.
Since the $k$-dimension of the maximal totally isotropic subspace over $k$ is $1$, the form $S$ is perhaps best globally written
$$
S = 
\renewcommand\arraystretch{1.25}
\begin{pmatrix}
0 & 0 & 1 \\
0 & S^\prime & 0 \\
1 & 0 & 0
\end{pmatrix}
$$
with $S^\prime$ anisotropic over $k$.

The condition $r+1 \geq 5$ implies that there will be an isotropic vector at every finite place. Since the dimension of the maximal isotropic subspace is one, and at every real archimedean place the rank is $(p,q)$ with $p+q=r+1$ (where $p$ and $q$ are both at least one), Hasse-Minkowski implies there is at least one real place where the rank is exactly $(r,1)$. In particular, there are no real anisotropic places: there are no real completions where the signature is $(r+1,0)$.

%
%
\subsubsection{Linear algebra over the commutative ring $k_\infty$}\label{linear-algebra-over-k-infty}

Since the characteristic is $0$, $k$ is separable. Even though the extension is not typically Galois, the {\it Galois trace} from $k$ to $\mathbb{Q}$ has an intrinsic sense: it is a non-zero $\mathbb{Q}$-linear map from $k$ to $\mathbb{Q}$ and is the sum of all Galois conjugates, whether or not they lie in $k$. Additionally,
$$
\langle x,y \rangle =\text{tr}_{k/\mathbb{Q}}(xy)
$$
is a non-degenerate $\mathbb{Q}$-bilinear $\mathbb{Q}$-valued pairing on $k \times k$. The $\mathbb{R}$-linear extension of trace, denoted similarly
$$
\langle x,y \rangle =\text{tr}_{k_\infty/\mathbb{R}}(xy)
$$
is a non-zero $\mathbb{R}$-linear map $k_\infty \to \mathbb{R}$ via the tensor product characterization of $k_\infty$ and is the sum of local traces 
$$
\text{tr}_{k_\infty/\mathbb{R}}(\alpha) = \sum_{\nu | \infty} \text{tr}_{k_\nu/\mathbb{R}}(\alpha)
$$

Let $S_\infty$ be the image of the bilinear form under the map $\sigma_\infty : k  \to k_\infty$ so that
$$
S_\infty = \sigma_\infty(S) = 
\renewcommand\arraystretch{1.25}
\begin{pmatrix}
0 & 0 & 1 \\
0 & \sigma_\infty(S^\prime) & 0 \\
1 & 0 & 0
\end{pmatrix}
=
\renewcommand\arraystretch{1.25}
\begin{pmatrix}
0 & 0 & 1 \\
0 & S_\infty^\prime & 0 \\
1 & 0 & 0
\end{pmatrix}
$$
is a symmetric $(r+1)\times(r+1)$ matrix with entries in $k_\infty$ and so that the local components $S_\nu^\prime$ of $S_\infty^\prime$ are symmetric non-degenerate matrices with entries in $k_\nu$. $(S_\infty^\prime)^{-1}$ will mean the $(r-1)\times (r-1)$ matrix over $k_\infty$ whose $\nu^\text{th}$ component is $(S_\nu^\prime)^{-1}$.

At each archimedean place $\nu$, there is a coordinate change matrix $B_\nu$ so that $S_\nu^\prime = (B_\nu^\prime)^\top \cdot Q_\nu^\prime \cdot B_\nu^\prime =$ where 
$$
Q_\nu^\prime =
\begin{cases} 
\begin{pmatrix}
0 & 0 & 1 \\
0 & 1_{p^+ q^-} & 0 \\
1 & 0 & 0
\end{pmatrix}
& \text{for } \nu \text{ real (via Inertia Theorem)} \\
& \\
\begin{pmatrix}
0 & 0 & 1 \\
0 & 1_{r-1} & 0 \\
1 & 0 & 0
\end{pmatrix}
& \text{for } \nu \text{ complex}
\end{cases}
$$
Where $1_{p^+ q^-}$ = diagonal matrix with $p$ entries of $+1$ and $q$ entries of $-1$ and $p+q=r-1$. If $p$ or $q$ are $0$, we set $1_{p^+ q^-} = 1_{r-1}$.

Let $B_\infty^\prime$ be the $(r-1)\times(r-1)$ matrix over $k_\infty$ whose $\nu^\text{th}$ factor is $B_\nu^\prime$ and $Q_\infty^\prime$ be the $(r-1)\times(r-1)$ matrix over $k_\infty$ whose $\nu^\text{th}$ factor is $Q_\nu^\prime$  and set
$$
B_\infty = 
\begin{pmatrix}
1 & 0 & 0 \\
0 & B_\infty^\prime & 0 \\
0 & 0 & 1
\end{pmatrix}
\qquad
Q_\infty = 
\begin{pmatrix}
0 & 0 & 1 \\
0 & Q_\infty^\prime & 0 \\
1 & 0 & 0
\end{pmatrix}
$$
so that we have
$$
S_\infty = B_\infty^\top \cdot Q_\infty \cdot B_\infty 
$$
%
%

%
%
\subsubsection{The group  $G=G(\mathbb{R}) \subset GL(r+1, k_\infty)$}\label{the-group-G-G(R)}

Let $G=G(\mathbb{R})$ be the real Lie group acting on the real vector space $V_\infty = V_\mathbb{R} = V\otimes_\mathbb{Q} k_\infty$ preserving $S_\infty$:
$$
G=G(\mathbb{R}) = \big\{ g \in GL(r+1, k_\infty) : g^\top \cdot S_\infty \cdot g = S_\infty \big\}
$$

Let $e_1$ be a basis of an isotropic line in $V_\infty$ and $P$ the parabolic subgroup of $G$ fixing the isotropic line spanned by $e_1$. Thus $P$ has the shape
$$
P = P_\infty =
\begin{pmatrix}
\ast & \ast & \ast \\
0 & \ast & \ast \\
0 & 0 & \ast
\end{pmatrix}
$$	
and decomposes as $P_\infty = P = NM = N_\infty M_\infty$ with unipotent radical $N$ and a complementary Levi component $M$. We will typically suppress the $\infty$ subscript but occassionally use it to emphasize a distinction between more {\it global} (i.e., defined over $k_\infty$) and local (i.e., corresponding to individual completions $\nu | \infty$) perspectives.
There is a maximal compact $K$ of $G$ given by 
$$
K  = K_\infty = \prod_{\nu | \infty} K_\nu \subset  \prod_{\nu | \infty} O(\sigma_\nu(S), k_\nu) = G(\mathbb{R})
$$

The condition $g^\top S_\infty \, g= S_\infty$ implies that elements of the unipotent radical $N$ of $P$ are
$$
N \enspace = N_\infty \enspace = \enspace \bigg\{ n_x \enspace = \enspace 
\renewcommand\arraystretch{1.5}
\begin{pmatrix}
1 & x & -\frac{1}{2} x \cdot (S_\infty^\prime)^{-1} \cdot x^\top \\
0 & 1_{r-1} & -(S_\infty^\prime)^{-1} x^\top \\
0 & 0 & 1
\end{pmatrix}
\enspace : \enspace x \in k_{\infty}^{(r-1)} \bigg\}
$$
This expression for elements of $N$ is true for all $S_\nu^\prime$. For real places, the canonical $Q$ coordinates for the case $p=r-1$ and $q=0$ were derived previously and for places $\nu$ where $p, q > 1$ (calculated below):
\begin{align*}
n_{e_i} \enspace = \enspace 
\renewcommand\arraystretch{1.5}
\begin{pmatrix}
1 & t \cdot e_i & -\frac{t^2}{2} \\
0 & 1_{r-1} & - t \cdot e_i^\top \\
0 & 0 & 1
\end{pmatrix}
&\quad
n_{t \cdot e_i} \enspace = \enspace 
\renewcommand\arraystretch{1.5}
\begin{pmatrix}
1 & t \cdot e_i & \frac{t^2}{2} \\
0 & 1_{r-1} & t \cdot e_i^\top \\
0 & 0 & 1
\end{pmatrix} \\
1 \leq i \leq p \qquad &\qquad p+1 \leq i \leq p+q = r-1
\end{align*}
At complex places, we will use coordinates for $N$ of $z = u + iv$ so that $z \in \mathbb{C}^{r-1}$ and $u,v \in \mathbb{R}^{r-1}$.
The standard Levi component is
$$
M = M_\infty = 
\begin{pmatrix}
u & 0 & 0 \\
0 & h & 0 \\
0 & 0 &  u^{-1}
\end{pmatrix}
$$
with $h \in O(S^\prime)$ and $u \in GL(1, k_\infty)$.
The Levi component subsequently decomposes into the product of its split component:
$$
A^+  =  A_\infty^+  =  \bigg\{ m_y = 
\renewcommand\arraystretch{1.5}
\begin{pmatrix}
\ell & 0 & 0 \\
0 & 1_{r-1} & 0 \\
0 & 0 & \ell^{-1}
\end{pmatrix}
\ : \; \ell \in (0,\infty) \bigg\}
$$
and the complement of the split component in the Levi component:
$$
M^1  = M_\infty^1 = \bigg\{ 
\renewcommand\arraystretch{1.5}
\begin{pmatrix}
u & 0 & 0 \\
0 & m^\prime & 0 \\
0 & 0 &  u^{-1}
\end{pmatrix}
\ : \; \text{where } u \in GL(1,k_\infty) \text{, } N_{k_\infty/\mathbb{R}}(u) = 1 \text{ and } m^\prime \in O(S^\prime) \bigg\}
$$
The invocation of Fujisaki's lemma, the Units Theorem, and the Compactness of Anisotropic Quotients, will use that $M^1$ further factors as $M^1 = M^1_1 \cdot M^1_2$ where
\begin{align*}
M^1_1  =&  \bigg\{ m_1(u) :
\renewcommand\arraystretch{1.5}
\begin{pmatrix}
u & 0 & 0 \\
0 & 1_{r-1} & 0 \\
0 & 0 &  u^{-1}
\end{pmatrix}
\ : \; \text{where } u \in GL(1,k_\infty)  \text{ and } N_{k_\infty/\mathbb{R}}(u) = 1 \bigg\} \\
M^1_2  =&  \bigg\{ m_2(h) :
\renewcommand\arraystretch{1.5}
\begin{pmatrix}
1 \quad & 0 \qquad & 0 \quad \\
0 & h & 0 \\
0 & 0 &  1
\end{pmatrix}
\ : \; \text{where } h \in O(S^\prime) \bigg\}
\end{align*}
We note that $A^+$ and $M^1_1$ are in the center of $M$, though $M^1_2$ is typically nonabelian, and that $M^1_1 \cap K$ will include
$$
O(1,k_\infty) = \prod_{\nu | \infty} O(1,k_\nu)
$$
namely, elements of $O(1,\mathbb{R}) \approx (\pm 1)$ at real places and $U(1) \approx S^1$ at complex places. Thus the Iwasawa-Levi-Mal'cev decomposition is
$$
G = PK = NMK \approx N \times A^+ \times M^1_1 \times M^1_2
$$
The model of $G/K$ is
$$
G/K \approx N \times A^+ \times M^1_1/(M^1_1 \cap K) \times M^1_2/(M^1_2 \cap K)
$$

Let $\frak{o}$ be the algebraic integers in $k$ and let $\Gamma = G \cap GL(r+1,\frak{o})$, and $\Gamma_\infty = \Gamma \cap P$.

%
%
\subsubsection{Lie algebra of $\frak{g}_\infty$}\label{lie-algebra-frak-g-infty}
To express the local Casimir operator in a fashion consonant with the Iwasawa decomposition, note that the right action of $\frak{k}$ annihilates functions on $G/K$. The Lie algebra $\frak{g}_\infty$ of $G_\infty$ is determined by the infinitesimal version of the isometry condition:
$$
\frak{g}_\infty \approx \bigg\{ X \in \frak{gl}(r+1,k_\infty) \, : \, X^\top \cdot S_\infty + S_\infty \cdot X = 0  \bigg\} 
$$
where
$$
V_\infty = V \otimes_\mathbb{Q} \mathbb{R} \approx \bigoplus_{\nu | \infty} V_\nu = \bigoplus_{\nu | \infty} V \otimes_k k_\nu = V \otimes_k k_\infty
$$
Substitute the change of coordinates expression
$$
S_\infty = B_\infty^\top \cdot Q_\infty \cdot B_\infty
$$
$$
Q_\infty = (B_\infty^\top)^{-1} \cdot S_\infty \cdot (B_\infty)^{-1}
$$
in the expression for $\frak{g}_\infty$
\begin{align*}
0 =& \; X^\top \cdot S_\infty + S_\infty \cdot X \\
   =& \; X^\top \cdot (B_\infty^\top \cdot Q_\infty \cdot B_\infty) + (B_\infty^\top \cdot Q_\infty \cdot B_\infty) \cdot X \\
   & \text{multiply on the left by } (B_\infty^\top)^{-1 }\text{ and on the right by } B_\infty^{-1} \\
   =& \; ((B_\infty^\top)^{-1 }\cdot X^\top \cdot B_\infty^\top) \cdot Q_\infty +  Q_\infty \cdot (B_\infty \cdot X \cdot B_\infty^{-1})\\   
   =& \; (B_\infty \cdot X \cdot B_\infty^{-1 })^\top \cdot Q_\infty +  Q_\infty \cdot (B_\infty \cdot X \cdot B_\infty^{-1})\\   
\end{align*}
Since $Q_\infty$ is symmetric and $Q_\infty^2=1$, for $X \in \frak{g}_\infty$, define a Cartan involution $X^\theta$ by mapping $X$ to the $Q$-coordinates and using the Cartan involution in $Q$-coordinates given by the negative transpose at real places and negative transpose-conjugate at complex places:
$$
X^\theta_\nu =
\begin{cases} 
- B_\nu^{-1 } \cdot (B_\nu \cdot X_\nu \cdot B_\nu^{-1 })^\top \cdot B_\nu = - (B_\nu^\top \cdot B_\nu)^{-1 } \cdot X_\nu^\top \cdot  (B_\nu^\top \cdot B_\nu) & \text{for } \nu \text{ real }\\
- B_\nu^{-1 } \cdot (B_\nu \cdot X_\nu \cdot B_\nu^{-1 })^\ast \cdot B_\nu = - (B_\nu^\ast \cdot B_\nu)^{-1 } \cdot X_\nu^\ast \cdot  (B_\nu^\ast \cdot B_\nu) & \text{for } \nu \text{ complex}\\
\end{cases}
$$
To simplify notation, set $\mathcal{B}_\nu = B_\nu^\top \cdot B_\nu$ for $\nu$ a real place and $\mathcal{B}_\nu = B_\nu^\ast \cdot B_\nu$ for $\nu$ complex and extend $\theta$ by factors to define $X^\theta$ on $\frak{g}_\infty$.

Since $Q_\nu^2 =$ identity for real and complex $\nu$, the $\pm 1$ eigenspaces for $\frak{o}(Q_\nu)$ are more easily identified.

Additionally, the change-of-coordinates defined by $\mathcal{B}$ provide an isomorphism between expressions for $\frak{g}_\infty$ in $S$-coordinates and the corresponding expressions in $Q$-coordinates. Since operations in the $Q$-coordinates are much simpler we will identify the relevant subalgebras in the $Q$-coordinates. All of the groups are isomorphic to products of subgroups of matrix groups and since the map of Lie algebras is by $x \to gxg^{-1}$, the group map is $h \to ghg^{-1}$, and vice-versa, because $\text{exp}(gxg^{-1})=g \cdot \text{exp}(x) \cdot g^{-1}$. In the sequel, the computations will be done in the $Q$-coordinates since all operations can be mapped back to $S$-coordinates via $\mathcal{B}$ without loss of generality.

%
%
%
\subsection{Casimir operator at real $\nu$ on $G/K$}\label{casimir-nu-real}
%
%
%
%
\subsubsection{Lie algebra of $\frak{g}_\nu$ for real $\nu$}\label{lie-algebra-frak-g-nu-real}
At a real place $\nu$, we proceed with local $Q$ coordinates given by
$$
Q_\nu
=
\renewcommand\arraystretch{1.25}
\begin{pmatrix}
0 & 0 & 1 \\
0 & Q_\nu^\prime & 0 \\
1 & 0 & 0
\end{pmatrix}
\quad
Q_\nu^\prime =
\begin{pmatrix}
0 & 0 & 1 \\
0 & 1_{p^+ q^-} & 0 \\
1 & 0 & 0
\end{pmatrix}
\quad
1_{p^+ q^-} = \text{ diagonal matrix with } p \; 1\text{'s and } q \; -1\text{'s}
$$
To lighten notation, we will suppress the subscript $\nu$ for the elements of the Lie algebra unless needed. The Lie algebra for $O(Q_\nu)$ is characterized by
$$
X^\top \cdot Q_\nu + Q_\nu \cdot X= 0
$$
$$
X^\top \cdot Q_\nu = -  Q_\nu \cdot X
$$
 At real $\frak{g}_\nu$ let  
$$
X = 
\renewcommand\arraystretch{1.25}
\begin{pmatrix}
a & b & c \\
d^\top & e & f^\top \\
g & h & i
\end{pmatrix}
\qquad b,d,f,h \in \mathbb{R}^{r-1} \enspace a,c,g,i \in \mathbb{R}
$$
which implies $X$ is of the form:
$$
X = 
\renewcommand\arraystretch{1.25}
\begin{pmatrix}
a & b & 0 \\
d^\top & e & -1_{p^+ q^-} \cdot b^\top \\
0 & -d \cdot 1_{p^+ q^-} & -a
\end{pmatrix}
\qquad b,d \in \mathbb{R}^{r-1} \enspace a \in \mathbb{R} \enspace e \in \frak{o}(Q_\nu^\prime)
$$
Where $1_{p^+ q^-}$ is as defined in \S 14.1.1 so that $e \in \frak{o}(Q_\nu^\prime) = \frak{o}(p,q)$.
Let $e_i$ be the standard basis of $\mathbb{R}^{r-1}$ and define a basis for $\frak{o}(Q_\nu)$ by
$$
H=
\begin{pmatrix}
1 & 0 & 0 \\
0 & 0_{r-1} & 0 \\
0 & 0 & -1
\end{pmatrix}
$$
$$
X_i =
\begin{pmatrix}
0 & e_i & 0 \\
0 & 0_{r-1} & -1_{p^+ q^-} \cdot e_i^\top \\
0 & 0 & 0
\end{pmatrix}
\quad Y_i = X_i^\top =
\begin{pmatrix}
0 & 0 & 0 \\
e_i^\top & 0_{r-1} & 0 \\
0 & -e_i \cdot 1_{p^+ q^-} & 0
\end{pmatrix}
$$
This misses exactly the elements in $\text{Lie}(P \cap O(Q_\nu^\prime))$. These can be supplied by including a basis for
$$
\frak{o}(Q_\nu^\prime) \approx 
\{
\begin{pmatrix}
0 & 0 & 0 \\
0 & \theta & 0 \\
0 & 0 & 0
\end{pmatrix}
\text{  with }  \theta^\top \cdot 1_{p^+ q^-} = -1_{p^+ q^-} \cdot \theta \}  
$$
Since $Q_\nu^2 = 1$, and $\nu$ is real, negative-transpose is a Cartan involution on $\frak{g}_\nu$. With $\frak{k}$ the $+1$-eigenspace for this involution, we have that $X_i - Y_i \in \frak{k}$. We have the following bracket relation:
$$
[X_i,Y_i] = H
$$
Using the trace pairing $B(X,Y) = \text{tr}(XY)$ (a scalar multiple of Killing) we have the following relations
$$B(H,H) = 2$$
$$B(H,X_i) = \text{tr}(e_{1i}) = 0 = B(H,Y_j) = \text{tr}(e_{nj}) = 0$$
$$B(X_i,Y_j) = 
\text{tr}
\begin{pmatrix}
\delta_{ij} & 0 & 0 \\
0 & \delta_{ij} & 0 \\
0 & 0 & 0
\end{pmatrix}
= 2\delta_{ij}
$$
$$B(\theta_i,\theta_j) = \pm\delta_{ij}$$
$$B(H,\theta_k) = B(X_i,\theta_k) = B(Y_j,\theta_k) = 0 \qquad \forall i,j,k$$
With respect to the trace pairing (with $\prime$ denoting dual as above):
$$
H^\prime = \frac{1}{2}\cdot H \qquad X^\prime_i = \frac{1}{2}\cdot Y_i = \frac{1}{2}\cdot X^\top_i  \qquad [X_i,Y_i] = H \qquad \theta^\prime_j = \pm \theta_j
$$
To express the Casimir operator in a fashion consonant with the Iwasawa decomposition, note that the right action of $\frak{k}$ annihilates functions on $G/K$. Since, the skew-symmetric $X_i - Y_i$ lies in $\frak{k}$, it acts by $0$ on the right on functions on $G/K$. For $f$ a right $K$-invariant function on $G$, it suffices to evaluate $\Omega_\nu f$ at group elements $n \cdot m \in N \cdot M \approx P$ since $\Omega$ preserves the right $K$-invariance. Thus
$$
\Omega_\nu = H \cdot H^\prime + \sum_i X_i \cdot X^\prime_i + \sum_i Y_i \cdot Y^\prime_i + \sum_j \theta_j \theta^\prime_j 
$$
Note that $\sum_j \theta_j \theta^\prime_j$ is the Casimir operator $\Omega_\nu^\prime$ on $M^1_2$. While $M^1_2$ contributes to the coefficients of the Laplacian corresponding to derivatives parallel to $N$, application of compactness of anisotropic quotients will make analysis this operator unnecessary and derivatives parallel to $M^1_2$ and there will not be a need to compute these derivatives. $\Omega_\nu^\prime$ is negative definite and while this term will remain, it will be simply denoted $\Omega_\nu^\prime$.
$$
\Omega_\nu = \frac{1}{2}H^2 + \frac{1}{2}\sum_{i=1}^{r-1} X_i \cdot Y_i + \frac{1}{2}\sum_{i=1}^{r-1} Y_i \cdot X_i + \Omega_\nu^\prime
$$
$$
= 	\frac{1}{2}H^2 + \frac{1}{2}\sum_{i=1}^{r-1} (X_i \cdot Y_i +  Y_i \cdot X_i) + \Omega_\nu^\prime  
$$
$$
= 	\frac{1}{2}H^2 + \frac{1}{2}\sum_{i=1}^{r-1} (2 X_i \cdot Y_i +  [Y_i, X_i]) + \Omega_\nu^\prime
$$
$$
= 	\frac{1}{2}H^2 -  \frac{r-1}{2}H + \sum_{i=1}^{r-1} X_i \cdot Y_i + \Omega_\nu^\prime
$$
$$
= 	\frac{1}{2}H^2 -  \frac{r-1}{2}H + \sum_{i=1}^{r-1} (X_i^2 + \underbrace{X_i(Y_i - X_i)}_\text{acts by 0}) + \Omega_\nu^\prime
$$
$$
= 	\frac{1}{2}H^2 \enspace -  \enspace \frac{r-1}{2}H \enspace + \enspace \sum_{i=1}^{r-1} X_i^2 \enspace + \Omega_\nu^\prime
$$
Note that while $\Omega_\nu^\prime$ is not invariant under $G$, it does descend to the image in $\Gamma\backslash G$ of a sufficiently high Siegel set (i.e., corresponding to a sufficiently large value of the height parameter), allowing separation of variables. In particular, while the $\theta_i$ contribute to the coefficients of the Laplacian for the $N$, by the compactness of anisotopic quotients, in the quotient of the Siegel set, these coefficients are bounded since they are continuous coordinates on a compact manifold (i.e., corresponding to the bounded (compact) region $D$ in $M$). 

%
%
\subsubsection{Casimir at real $\nu$ in Iwasawa coordinates}\label{Casimir at real nu in Iwasawa coords}
To lighten notation, we will continue to suppress the subscript $\nu$ unless needed. The local Iwasawa-Levi-Mal'cev decomposition (analogous the more global expression in \S 14.1.2) is
$$
G = PK = NMK \approx N \cdot A^+ \cdot M^1_1 \cdot M^1_2
$$
where $N$ is the {\it local} factor at $\nu$ of the unipotent radical $N_\infty$, $A^+$ is the {\it local} split component, and $M^1_1$ and $M^1_2$ are the factors of $M^1$, the  {\it local} complement to the  {\it local} $A^+$. Then, write elements of  {\it local} $G$ as 
$$
g = n_x \cdot a_y \cdot m_1(u) \cdot m_2(h) \qquad h \in O(Q_\nu^\prime)
$$
Exponentiating 
$$
e^{tH}  \enspace =
\sum_{k=0}^{\infty} \frac{(tH)^k}{k!} \enspace =	
\begin{pmatrix}
e^t & 0 & 0 \\
0 & 1_{r-1} & 0 \\
0 & 0 & e^{-t}
\end{pmatrix}
 \enspace = m_1(e^t) \in A^+
$$
\begin{align*}
e^{t X_i} =
\sum_{n=0}^{\infty} \frac{(tX_i)^k}{k!} \enspace = \qquad \qquad & \\
n_{t \cdot e_i} \enspace = \enspace 
\renewcommand\arraystretch{1.5}
\begin{pmatrix}
1 & t \cdot e_i & -\frac{t^2}{2} \\
0 & 1_{r-1} & - t \cdot e_i^\top \\
0 & 0 & 1
\end{pmatrix}
&\quad
n_{t \cdot e_i} \enspace = \enspace 
\renewcommand\arraystretch{1.5}
\begin{pmatrix}
1 & t \cdot e_i & \frac{t^2}{2} \\
0 & 1_{r-1} & t \cdot e_i^\top \\
0 & 0 & 1
\end{pmatrix} \\
1 \leq i \leq p \qquad &\qquad p+1 \leq i \leq p+q = r-1
\end{align*}
Abusing notation and writing $e^\theta$ both for the exponentiated elements of $\frak{g}$ and of the corresponding central block (i.e., $\frak{o}(Q_\nu^\prime)$)
$$
e^{t \theta_i}  \enspace =
\sum_{k=0}^{\infty} \frac{(t \theta_i)^k}{k!} \enspace =	
\begin{pmatrix}
1 & 0 & 0 \\
0 & m_{2,i}^\prime(t) & 0 \\
0 & 0 & 1
\end{pmatrix}
 \enspace = m_{2,i}(t) = m_2(e^{t\theta_i}) \in M^1_2
$$
so that $n_x m_1(y) \cdot e^{tH} = n_x m_1(y) m_1(e^t) = n_x m_1(y e^t)$, since multiplication in $M^1_1$ is homomorphic to multiplication in $\mathbb{R}^\times$.
\noindent
To determine $H$ as an operator on $G/K$, let $g \in G$ with corresponding $g = n_{x_0} \cdot a_{y_0} \cdot m_1(u_0) \cdot m_2(o_0)$
\begin{align*}
H \cdot f(g) =& \frac{d}{dt}\bigg|_{t=0} f(g \cdot m_{e^t}) = \frac{d}{dt}\bigg|_{t=0} f(n_{x_0} a_{y_0} m_1(u_0) m_2(o_0) \cdot m_{e^t}) \\
=& \frac{d}{dt}\bigg|_{t=0} f(n_{x_0} a_{y_0 e^t} m_1(u_0) m_2(o_0)) = y_0 \; \frac{\partial}{\partial y} \bigg|_{(n_{x_0} a_{y_0} m_1(u_0) m_2(o_0))} f
\end{align*}
since $A^+$ is in the center of $M$ and so $H \text{ acts by } y \frac{\partial}{\partial y}$.

$M$ normalizes $N$ but does not commute with $N$; however, the normalization of $N$ by $M$ means $N$ can ``move past'' $M$ by a linear change in the $N$-coordinate $x$. We make repeated use of the elementary identity
$$
m \cdot n \cdot m^{-1} = n^\prime \Longrightarrow m \cdot n = n' \cdot m
$$
A convenient relation is (where the coefficient is $+1$ for $1 \leq i \leq p$ and $-1$ for $p+1 \leq i \leq p+q = r-1$)
$$
a_y \cdot e^{t X_i} = a_y \cdot n_{te_i} =  
\begin{pmatrix}
y & 0 & 0 \\
0 & 1_{r-1} & 0 \\
0 & 0 & \frac{1}{y}
\end{pmatrix}
\begin{pmatrix}
1 & t \cdot e_i & \pm\frac{t^2}{2} \\
0 \enspace & 1_{r-1} & \pm t \cdot e_i^\top \\
0 & 0 & 1
\end{pmatrix}
=
\begin{pmatrix}
y & yt \cdot e_i & \pm\frac{yt^2}{2} \\
0 \enspace & 1_{r-1} & \pm t \cdot e_i^\top \\
0 & 0 & \frac{1}{y}
\end{pmatrix}
$$
$$
n_{yte_i} \cdot a_y =  
\begin{pmatrix}
1 & yt \cdot e_i & \pm\frac{y^2t^2}{2} \\
0 \enspace & 1_{r-1} & \pm yt \cdot e_i^\top \\
0 & 0 & 1
\end{pmatrix}
\begin{pmatrix}
y & 0 & 0 \\
0 & 1_{r-1} & 0 \\
0 & 0 & \frac{1}{y}
\end{pmatrix}
=
\begin{pmatrix}
y & yt \cdot e_i & \pm \frac{yt^2}{2} \\
0 \enspace & 1_{r-1} & \pm t \cdot e_i^\top \\
0 & 0 & \frac{1}{y}
\end{pmatrix}
$$
which gives
$$n_x \cdot m_y \cdot e^{t X_i} =n_x \cdot m_y \cdot n_{te_i} = n_x \cdot n_{yte_i} \cdot m_y = n_{x + yte_i} \cdot m_y$$
Similarly, for real $\nu$, $m_1(u_0)$ will be $\pm 1$ in the upper-left and in the lower-right corners (i.e., corresponding to an element of $O(1,\mathbb{R}) = \{\pm\}$) so that, abusing notation somewhat and letting $u_0$ stand for the value of $O(1,\mathbb{R})$ (i.e., $\pm 1$), $m_1(u_0) \cdot n_x = n_{u_0 \cdot x} \cdot m_1(u_0)$. 

$M^1_2$ will make a more substantive-appearing, though ultimately innocuous, contribution as the coefficients of $\Delta_x$ (the Laplacian on $N$) are also functions of the $M^1_2$-coordinate $h$.
$$
m_2(h) \cdot e^{tX_i} \cdot m_2(h)^{-1} = m_2(h) \cdot n_{te_i}  \cdot m_2(h)^{-1} =
\begin{pmatrix}
1 & t (h \cdot e_i) & \pm\frac{t^2}{2} \\
0  & 1_{r-1} & \pm t \cdot (h \cdot e_i)^\top \\
0 & 0 & 1
\end{pmatrix}
=n_{h e_i}
$$
so that
$$
m_2(h) \cdot n_{te_i} = n_{t \cdot h e_i} \cdot m_2(h)
$$
Note also that the behavior of functions on $M^1_2$ will ultimately be controlled by application of the compactness of anisotropic quotients, and that the map $h \to h^{-1}$ is a smooth involution. Summarizing, we have
\begin{align*}
n_{x_0} & a_{y_0} m_1(u_0) m_2(h_0) \cdot n_{te_i} = n_{x_0} a_{y_0} m_1(u_0) \cdot n_{t\cdot h_0 e_i} \cdot m_2(h_0) \\
=& n_{x_0} a_{y_0} \cdot n_{t \cdot u_0(h_0 e_i)} \cdot m_1(u_0)  m_2(h_0)  = n_{x_0} \cdot n_{ t \cdot y_0 \cdot u_0(h_0 e_i)} \cdot a_{y_0} m_1(u_0)  m_2(h_0)  \\
=& n_{x_0 + t \cdot y_0 \cdot u_0(h_0e_i)} \cdot a_{y_0} m_1(u_0)  m_2(h_0)
\end{align*}
In $M^1_1$ (i.e., $u$) and $M^1_2$ (i.e., $h$) coordinates
$$
u_0(h_0e_i) = u_0 \sum_j h^{ij}_0 \cdot e_j
$$
where we note that the coefficients only depend on coordinates in $M^1_1$ and $M^1_2$ and are independent of $N$.
\begin{align*}
X_i \cdot f(g) =& \frac{d}{dt}\bigg|_{t=0} f(g \cdot n_{te_i}) =  \frac{d}{dt}\bigg|_{t=0} f(n_{x_0} a_{y_0} m_1(u_0) m_2(h_0) \cdot n_{te_i} ) \\
 =& \frac{d}{dt}\bigg|_{t=0} f( \, n_{x_0 + t \cdot y_0 \cdot u_0 h_0 e_i} \cdot a_{y_0} m_1(u_0)  m_2(h_0))  \\
 =&  \frac{d}{dt}\bigg|_{t=0} f( \, n_{(x_0 + t \cdot y_0 \cdot u_0 \sum_j h^{ij}_0 \cdot e_j)} \cdot a_{y_0} m_1(u_0)  m_2(h_0)) \\
 =& y_0 \; \cdot \bigg( u_0 \sum_j h^{ij}_0 \cdot \frac{\partial}{\partial x_j} \bigg)  \bigg|_{(n_{x_0} a_{y_0} m_1(u_0) m_2(h_0))} f 
\end{align*}
so that
$$
X_i \bigg|_{(n_x a_y m_1(u) m_2(h))} = y \, u \sum_j h^{ij} \cdot \frac{\partial}{\partial x_j}
$$
and
\begin{align*}
	X_i^2 =& X_i \circ X_i \\
	=& \bigg(y \, u \sum_j h^{ij} \cdot \frac{\partial}{\partial x_j} \bigg)^2 \\
	=& y^2 u^2 \sum_{j,k=1}^{r-1} h^{ij} h^{ik} \frac{\partial^2}{\partial x_j x_k}
\end{align*}
since the coefficients are independent of $x$. Note also that, locally at real $\nu$, $u \in M^1_1 \cap K = O(1,\mathbb{R}) \approx (\pm 1)$ so that $u^2=1$ and can be dropped. Substituting appropriately in $\Omega$ on $G/K$
\begin{align*}
\Omega =& 	\frac{1}{2}H^2 \enspace -  \enspace \frac{r-1}{2}H \enspace + \enspace \sum_{i=1}^{r-1} X_i^2 \enspace + \underbrace{\sum_j \theta_j \theta_j^\prime}_\text{$\Omega_\nu^\prime$ negative definite} \\
=& \frac{1}{2}(y \frac{\partial}{\partial y})^2 - \frac{(r-1)}{2}y \frac{\partial}{\partial y} + \sum_{i=1}^{r-1}  \Bigg( y^2 \sum_{j,k=1}^{r-1} h^{ij} h^{ik} \frac{\partial^2}{\partial x_j \partial x_k} \Bigg)+ \Omega_\nu^\prime \\
=& \frac{1}{2}y \frac{\partial^2}{\partial y^2} - \frac{(r-2)}{2}y \frac{\partial}{\partial y} +y^2 \sum_{i,j,k=1}^{r-1} h^{ij} h^{ik} \frac{\partial^2}{\partial x_j \partial x_k} + \Omega_\nu^\prime
\end{align*}

%
%
\subsection{Casimir operator at complex $\nu$ on $G/K$}\label{casimir-nu-complex}
%
%
%
%
\subsubsection{Lie algebra of $\frak{g}_\nu$ for complex $\nu$}\label{lie-algebra-frak-g-nu-complex}
To lighten notation, we will suppress the subscript $\nu$ for the elements of the Lie algebra unless needed. The Lie algebra for $O(Q_\nu)$ is characterized by
$$
X^\top \cdot Q_\nu + Q_\nu \cdot X = 0
$$
$$
X^\top \cdot Q_\nu = -Q_\nu \cdot X
$$
where
$$
Q_\nu = 
\begin{pmatrix}
0 & 0 & 1 \\
0 & 1_{r-1} & 0 \\
1 & 0 & 0
\end{pmatrix}
$$
Letting  
$$
X = 
\renewcommand\arraystretch{1.25}
\begin{pmatrix}
a & b & c \\
d^\top & e & f^\top \\
g & h & i
\end{pmatrix}
\qquad b,d,f,h \in \mathbb{C}^{r-1} \enspace a,c,g,i \in \mathbb{C}
$$
Implies the Lie algebra of $\frak{o}(Q_\nu)$ consists of elements of the form:
$$
X = 
\renewcommand\arraystretch{1.25}
\begin{pmatrix}
a & b & 0 \\
d^\top & e & -b^\top \\
0 & -d & -a
\end{pmatrix}
\qquad b,d \in \mathbb{C}^{r-1} \quad a \in \mathbb{C} \quad c,g \in i\mathbb{R} \quad e \in \frak{o}(r-1, \mathbb{C})
$$
Specify a real basis for $\frak{o}(r+1,\mathbb{C})$ as:
$$
H =
\begin{pmatrix}
1 & 0 & 0 \\
0 & 0_{r-1} & 0 \\
0 & 0 & -1
\end{pmatrix}
\tilde{H} =
\begin{pmatrix}
i & 0 & 0 \\
0 & 0_{r-1} & 0 \\
0 & 0 & -i
\end{pmatrix}
X_\ell =
\begin{pmatrix}
0 & e_\ell & 0 \\
0 & 0_{r-1} & -e_\ell^\top \\
0 & 0 & 0
\end{pmatrix}
\tilde{X}_\ell =
\begin{pmatrix}
0 & ie_\ell & 0 \\
0 & 0_{r-1} & -ie_\ell^\top \\
0 & 0 & 0
\end{pmatrix}
$$
$$
Y_\ell =
\begin{pmatrix}
0 & 0 & 0 \\
e_\ell^\top & 0_{r-1} & 0 \\
0 & -e_\ell & 0 \\
\end{pmatrix}
\tilde{Y}_\ell =
\begin{pmatrix}
0 & 0 & 0 \\
ie_\ell^\top & 0_{r-1} & 0 \\
0 & -ie_\ell & 0 \\
\end{pmatrix}
$$
$$
\Theta_\ell =
\begin{pmatrix}
0 & 0 & 0 \\
0 & \theta_\ell & 0 \\
0 & 0 & 0
\end{pmatrix}
\quad
\{ \theta_\ell \} \text{ orthogonal real basis of } \frak{o}(r-1, \mathbb{C})
$$
Define a bilinear form $B$ by
$$
B(X,Y) = (\text{tr}_{\mathbb{C}/\mathbb{R}} \circ \text{tr})(XY) = 2 \cdot \text{Re}(\text{tr}(XY))
$$
we have the following relationships:
\begin{align*}
B(H,H) = 2 \quad & \quad B(\tilde{H},\tilde{H}) = -2 \\
B(X_j,Y_k) = 2\delta_{jk} \quad & \quad B(\tilde{X}_j,\tilde{Y}_k) = -2\delta_{jk} \\
B(H,X_j) = B(H,\tilde{X}_j) =& 0 = B(H,Y_j) = B(H,\tilde{Y}_j) =  0 \\
B(\theta_j,\theta_k) =& \pm\delta_{jk} \\
B(H,\theta_k) = B(X_j,\theta_k) =& B(Y_j,\theta_k) = B(\tilde{X}_j,\theta_k) = B(\tilde{Y}_j,\theta_k) = 0 \qquad \forall j,k \\
\end{align*}
With respect to the trace pairing (with $\prime$ denoting dual as above):
$$
H^\prime = \frac{1}{2}\cdot H \qquad \tilde{H}^\prime = -\frac{1}{2}\cdot \tilde{H} \qquad \theta^\prime_j = \pm \theta_j 
$$
$$
X^\prime_j = \frac{1}{2}\cdot Y_j = \frac{1}{2}\cdot X_j^\top \qquad \tilde{X}_j^\prime = -\frac{1}{2}\cdot \tilde{Y}_j = -\frac{1}{2}\cdot \tilde{X}_j^\top
$$
We also have the following bracket relations:
$$
[X_\ell,Y_\ell] = H = -[\tilde{X}_\ell,\tilde{Y}_\ell] 
$$
A Cartan involution on $O(n,\mathbb{C})$ is given by
$$
\gamma^\theta = -\overline{\gamma}^\top \quad \text{(negative conjugate transpose)}
$$
The Lie algebra of the maximal compact $\frak{k}$ is the $+1$ eigenspace the Cartan involution and corresponds to operators that act trivially on functions on $G/K$. Since $\tilde{H}^\theta = \tilde{H}$, $\tilde{H}$ is ignored in determining the Casimir operator $\Omega$ on $G/K$.

To express the Casimir operator in a fashion consonant with the Iwasawa decomposition, note that the right action of $\frak{k}$ annihilates functions on $G/K$. The conjugate-skew-symmetric $X_j - Y_j$ and $\tilde{X}_j + \tilde{Y}_j$ lie in $\frak{k}$ and act by $0$ on the right on functions on $G/K$. For $f$ a right $K$-invariant function on $G$, it suffices to evaluate $\Omega f$ at group elements $n \cdot m$ since $\Omega$ preserves the right $K$-invariance. Thus
\begin{align*}
\Omega_\nu =& H \cdot H^\prime + \sum_\ell X_\ell \cdot X^\prime_\ell + \sum_\ell Y_\ell \cdot Y^\prime_\ell + \sum_\ell \tilde{X}_\ell \cdot \tilde{X}^\prime_\ell + \sum_\ell \tilde{Y}_\ell \cdot \tilde{Y}^\prime_\ell + \sum_j \theta_j \theta^\prime_j \\
=& \frac{1}{2}H^2 + \frac{1}{2}\sum_{\ell=1}^{r-1} X_\ell \cdot Y_\ell + \frac{1}{2}\sum_{\ell=1}^{r-1} Y_\ell \cdot X_\ell - \frac{1}{2}\sum_{\ell=1}^{r-1} \tilde{X}_\ell \cdot \tilde{Y}_\ell - \frac{1}{2}\sum_{\ell=1}^{r-1} \tilde{Y}_\ell \cdot \tilde{X}_\ell + \sum_j \theta_j \theta^\prime_j \\
=& \frac{1}{2}H^2 + \frac{1}{2}\sum_{\ell=1}^{r-1} \big(X_\ell \cdot Y_\ell + Y_\ell \cdot X_\ell\big) - \frac{1}{2}\sum_{\ell=1}^{r-1} \big(\tilde{X}_\ell \cdot \tilde{Y}_\ell + \tilde{Y}_\ell \cdot \tilde{X}_\ell\big) + \sum_j \theta_j \theta^\prime_j
\end{align*}
Rewriting $X_\ell \cdot Y_\ell + Y_\ell \cdot X_\ell =  2 \cdot X_\ell \cdot Y_\ell + [Y_\ell, X_\ell]$
\begin{align*}
=& \frac{1}{2}H^2  + \frac{1}{2} \sum^{r-1}_{\ell = 1} \big( 2 \cdot X_\ell \cdot Y_\ell + [Y_\ell, X_\ell] \big) - \frac{1}{2} \sum^{r-1}_{\ell = 1} \big( 2 \cdot \tilde{X}_\ell \cdot \tilde{Y}_\ell + [\tilde{Y}_\ell , \tilde{X}_\ell] \big)+ \sum_j \theta_j \theta^\prime_j 
\end{align*}
Using $[X_\ell,Y_\ell] = H \text{ and } [\tilde{X}_\ell,\tilde{Y}_\ell] =-H$ gives
\begin{align*}
=& \frac{1}{2}H^2 - (r-1) H  + \sum^{r-1}_{\ell = 1}  X_\ell \cdot Y_\ell - \sum^{r-1}_{\ell = 1}  \tilde{X}_\ell \cdot \tilde{Y}_\ell 
+ \sum_j \theta_j \theta^\prime_j \\
=& \frac{1}{2}H^2 - (r-1) H  + \sum^{r-1}_{\ell = 1} \big( X_\ell^2 - \underbrace{X_\ell \cdot (X_\ell - Y_\ell)}_\text{acts by 0} \big) \\
&+ \sum^{r-1}_{\ell = 1} \big( \tilde{X}_\ell^2 - \underbrace{\tilde{X}_\ell \cdot (\tilde{X}_\ell + \tilde{Y}_\ell)}_\text{acts by 0} \big) + \sum_j \theta_j \theta^\prime_j \\
=& \frac{1}{2}H^2 - (r-1) H + \sum^{r-1}_{\ell = 1} X_\ell^2 + \sum^{r-1}_{\ell = 1} \tilde{X}_\ell^2 + \sum_j \theta_j \theta^\prime_j \\
=& \frac{1}{2}H^2 - (r-1) H + \sum^{r-1}_{\ell = 1} X_\ell^2 + \sum^{r-1}_{\ell = 1} \tilde{X}_\ell^2 + \Omega_\nu^\prime
\end{align*}
Where $\Omega_\nu^\prime$ is the Casimir operator on $M^1_2 \approx O(r-1,\mathbb{C})$. As in the case at real places $\nu$, for complex $\nu$, while $M^1$ makes a contribution to $\Omega_\nu$ in the coefficients of derivatives parallel to $N$, this will not play a substantive role due to application of the units theorem and compactness of anisotropic quotients.

%
%
\subsubsection{Casimir at complex $\nu$ in Iwasawa coordinates}\label{Casimir at complex nu in Iwasawa coords}
$$
g = n_z \cdot a_y \cdot m_1(u) \cdot m_2(h) \qquad h \in O(Q_\nu^\prime) = M^1_2 = O(r-1,\mathbb{C}) \quad z = x + iv \in \mathbb{C}^{r-1}
$$
Exponentiating 
$$
e^{tH}  \enspace =
\sum_{k=0}^{\infty} \frac{(tH)^k}{k!} \enspace =	
\begin{pmatrix}
e^t & 0 & 0 \\
0 & 1_{r-1} & 0 \\
0 & 0 & e^{-t}
\end{pmatrix}
 \enspace = m_1(e^t) \in A^+
$$
$$
e^{t X_j} =
\sum_{k=0}^{\infty} \frac{(tX_j)^k}{k!} \enspace = n_{t \cdot e_j} \enspace = \enspace 
\renewcommand\arraystretch{1.5}
\begin{pmatrix}
1 & t \cdot e_i & -\frac{t^2}{2} \\
0 & 1_{r-1} & - t \cdot e_i^\top \\
0 & 0 & 1
\end{pmatrix} 
$$
$$
e^{t \tilde{X}_j} =
\sum_{n=0}^{\infty} \frac{(t\tilde{X}_j)^k}{k!} \enspace = n_{t \cdot i \cdot e_j} \enspace = \enspace 
\renewcommand\arraystretch{1.5}
\begin{pmatrix}
1 & t \cdot i \cdot e_j & \frac{t^2}{2} \\
0 & 1_{r-1} & - t \cdot i \cdot e_j^\top \\
0 & 0 & 1
\end{pmatrix}
$$
Abusing notation and writing $e^\theta$ both for the exponentiated elements of $\frak{g}$ and of the corresponding central block (i.e., $\frak{o}(Q_\nu^\prime)$)
$$
e^{t \theta_i}  \enspace =
\sum_{k=0}^{\infty} \frac{(t \theta_i)^k}{k!} \enspace =	
\begin{pmatrix}
1 & 0 & 0 \\
0 & m_{2,i}^\prime(t) & 0 \\
0 & 0 & 1
\end{pmatrix}
 \enspace = m_{2,i}(t) = m_2(e^{t \theta_i}) \in M^1_2 = O(r-1,\mathbb{C})
$$
so that $n_x m_1(y) \cdot e^{tH} = n_x m_1(y) m_1(e^t) = n_x m_1(y e^t)$, since multiplication in $M^1_1$ is homomorphic to multiplication in $\mathbb{C}^\times$.
\noindent
To determine $H$ as an operator on $G/K$, let $g \in G$ with corresponding $g = n_{x_0} \cdot a_{y_0} \cdot m_1(u_0) \cdot m_2(h_0)$
\begin{align*}
H \cdot f(g) =& \frac{d}{dt}\bigg|_{t=0} f(g \cdot m_{e^t}) = \frac{d}{dt}\bigg|_{t=0} f(n_{x_0} a_{y_0} m_1(u_0) m_2(h_0) \cdot m_{e^t}) \\
=& \frac{d}{dt}\bigg|_{t=0} f(n_{x_0} a_{y_0 e^t} m_1(u_0) m_2(o_0)) = y_0 \; \frac{\partial}{\partial y} \bigg|_{(n_{x_0} a_{y_0} m_1(u_0) m_2(h_0))} f \\
\end{align*}
since $A^+$ is in the center of $M$ and so $H \text{ acts by } y \frac{\partial}{\partial y}$.

$M$ normalizes $N$ but does not commute with $N$; however, the normalization of $N$ by $M$ means $N$ can ``move past'' $M$ by a linear change in the $N$-coordinate $x$. We make repeated use of the elementary identity:
$$
m \cdot n \cdot m^{-1} = n^\prime \Longrightarrow m \cdot n = n' \cdot m
$$
Analogous to the real case, for complex $\nu$, $m_1(u_0)$ will be unit complex number in the upper-left and its inverse in the lower-right corners (i.e., corresponding to an element of $O(1,\mathbb{C}) = S^1$) and $m_2(h)$ will act linearly on the $x$ coordinate in N. The operations for $O(r+1,\mathbb{C})$ are parallel to those for $O(p,q)$ by linearity and thus:
\begin{align*}
n_{z_0} \cdot a_{y_0} \cdot m_1(u_0) \cdot m_2(h_0) \cdot e^{t X_\ell} =& \; n_{z_0 + t \cdot y_0 \cdot u_0 h_0 e_\ell} \cdot a_{y_0} \cdot m_1(u_0) \cdot m_2(h_0) \\
n_{z_0} \cdot a_{y_0} \cdot m_1(u_0) \cdot m_2(h_0) \cdot e^{t \tilde{X}_\ell} =& \; n_{z_0 + t y_0 u_0 h_0 i e_\ell} \cdot a_{y_0} \cdot m_1(u_0) \cdot m_2(h_0) \\
z_0 \in \mathbb{C}^{r-1} \quad a_0 \in (0,\infty) \quad & u_0 \in O(1,\mathbb{C}) \quad h_0 \in O(r-1,\mathbb{C})
\end{align*}

$M$ makes an ultimately innocuous contribution as the coefficients of $\Delta_x$ (the Laplacian on $N$) are functions of the $M^1_1$-coordinate $u$ and $M^1_2$-coordinate $h$. In $M^1_1$ ($u$ coordinates) and $M^1_2$ ($h$ coordinates)
$$
u_0(h_0e_\ell) = u_0 \sum_{j=1}^{r-1} h^{\ell j}_0 \cdot e_j
$$
analogous to real places, the coefficients depend on coordinates in $M^1_1$ and $M^1_2$ and are independent of $N$. However, both $u$ and $h$ may contain complex numbers and thus will {\it mix} real and complex directions. That is, the $x_j$ ({\it real}) variable is associated with the $e_j$ basis vector direction while the $v_j$ ({\it real}) variable is associated with the $i e_j$ basis vector direction. Using the complex linearity, let 
$$
h = \text{Re}(h) + i \text{Im}(h) \qquad u = \text{Re}(u) + i \text{Im}(u)
$$
so that
\begin{align*}
u_0(h_0e_\ell) =& u_0 \sum_{j=1}^{r-1} h^{\ell j}_0 \cdot e_j = (\text{Re}(u_0) + i \text{Im}(u_0)) \sum_{j=1}^{r-1} (\text{Re}(h^{\ell j}_0) + i \text{Im}(h^{\ell j}_0))\cdot e_j \\
=&  \sum_{j=1}^{r-1} \bigg(\text{Re}(u_0) \text{Re}(h^{\ell j}_0) - \text{Im}(u_0) \text{Im}(h^{\ell j}_0) \bigg) \cdot e_j \\
&+ \sum_{j=1}^{r-1} \bigg( \text{Im}(u_0) \text{Re}(h^{\ell j}_0)+ \text{Re}(u_0) \text{Im}(h^{\ell j}_0) \bigg) \cdot i \, e_j \\
=&  \sum_{j=1}^{r-1} A_\ell^j(u_0,h_0) \cdot e_j +  \sum_{j=1}^{r-1} B_\ell^j(u_0,h_0) \cdot i \, e_j
\end{align*}
and similarly
\begin{align*}
u_0(h_0 i e_\ell) =& u_0 \sum_{j=1}^{r-1} h^{\ell j}_0 \cdot i e_j = u_0 \sum_{j=1}^{r-1} i (\text{Re}(h^{\ell j}_0) + i \text{Im}(h^{\ell j}_0))\cdot e_j \\
=& - \sum_{j=1}^{r-1} \bigg( \text{Im}(u_0) \text{Re}(h^{\ell j}_0)+ \text{Re}(u_0) \text{Im}(h^{\ell j}_0) \bigg) \cdot e_j \\
&+  \sum_{j=1}^{r-1} \bigg(\text{Re}(u_0) \text{Re}(h^{\ell j}_0) - \text{Im}(u_0) \text{Im}(h^{\ell j}_0) \bigg) \cdot i \, e_j \\
=& - \sum_{j=1}^{r-1} B_\ell^j(u_0,h_0) \cdot e_j +  \sum_{j=1}^{r-1} A_\ell^j(u_0,h_0) \cdot i \, e_j
\end{align*}
\begin{align*}
X_\ell \cdot f(g) =& \frac{d}{dt}\bigg|_{t=0} f(g \cdot n_{te_\ell}) =  \frac{d}{dt}\bigg|_{t=0} f(n_{z_0} a_{y_0} m_1(u_0) m_2(h_0) \cdot n_{te_\ell} ) \\
 =& \frac{d}{dt}\bigg|_{t=0} f( \, n_{z_0 + t \cdot y_0 \cdot u_0 (h_0 e_\ell)} \cdot a_{y_0} m_1(u_0)  m_2(h_0))  \\
 =&  \frac{d}{dt}\bigg|_{t=0} f( \, n_{z_0 + t \cdot y_0 \cdot (u_0 \sum_{j=1}^{r-1} h^{\ell j}_0 \cdot e_j)} \cdot a_{y_0} m_1(u_0)  m_2(h_0)) \\
 =& y_0 \; \cdot \bigg( \sum_{j=1}^{r-1} A_\ell^j(u_0,h_0) \cdot \frac{\partial}{\partial x_j} +  \sum_{j=1}^{r-1} B_\ell^j(u_0,h_0) \cdot \frac{\partial}{\partial v_j} \bigg)  \bigg|_{(n_{z_0} a_{y_0} m_1(u_0) m_2(h_0))} f 
\end{align*}
\begin{align*}
\tilde{X}_\ell \cdot f(g) =& \frac{d}{dt}\bigg|_{t=0} f(g \cdot n_{ite_\ell}) =  \frac{d}{dt}\bigg|_{t=0} f(n_{z_0} a_{y_0} m_1(u_0) m_2(o_0) \cdot n_{t i e_\ell} ) \\
 =& \frac{d}{dt}\bigg|_{t=0} f( \, n_{z_0 +  t \cdot y_0 \cdot u_0 (h_0 i e_\ell)} \cdot a_{y_0} m_1(u_0)  m_2(h_0))  \\
 =&  \frac{d}{dt}\bigg|_{t=0} f( \, n_{z_0 + t \cdot y_0 \cdot (u_0 \sum_{j=1}^{r-1} h^{\ell j}_0 \cdot i e_\ell)} \cdot a_{y_0} m_1(u_0)  m_2(h_0)) \\
 =& y_0 \; \cdot \bigg( - \sum_{j=1}^{r-1} B_\ell^j(u_0,h_0) \cdot \frac{\partial}{\partial x_j} +  \sum_{j=1}^{r-1} A_\ell^j(u_0,h_0) \cdot \frac{\partial}{\partial v_j} \bigg)  \bigg|_{(n_{z_0} a_{y_0} m_1(u_0) m_2(h_0))} f 
\end{align*}
so that
\begin{align*}
X_\ell =& y \, \bigg( \sum_{j=1}^{r-1} A_\ell^j(u,h) \cdot \frac{\partial}{\partial x_j} +  \sum_{j=1}^{r-1} B_\ell^j(u,h) \cdot \frac{\partial}{\partial v_j} \bigg) \\
\tilde{X}_\ell =& y \, \bigg( - \sum_{j=1}^{r-1} B_\ell^j(u,h) \cdot \frac{\partial}{\partial x_j} +  \sum_{j=1}^{r-1} A_\ell^j(u,h) \cdot \frac{\partial}{\partial v_j} \bigg)
\end{align*}
and
\begin{align*}
	X_\ell^2 =& X_\ell \circ X_\ell \\
	=& \Bigg( y \, \bigg( \sum_{j=1}^{r-1} A_\ell^j(u,h) \cdot \frac{\partial}{\partial x_j} +  \sum_{j=1}^{r-1} B_\ell^j(u,h) \cdot \frac{\partial}{\partial v_j} \bigg) \Bigg)^2 \\
	=& y^2 \, \bigg( \sum_{j,k=1}^{r-1} A_\ell^j(u,h)A_\ell^k(u,h) \cdot \frac{\partial^2}{\partial x_j \partial x_k} \\
	&+ 2\sum_{j,k=1}^{r-1} A_\ell^j(u,h)B_\ell^k(u,h) \cdot \frac{\partial^2}{\partial x_j \partial v_k} + \sum_{j,k=1}^{r-1} B_\ell^j(u,h)B_\ell^k(u,h) \cdot \frac{\partial^2}{\partial v_j \partial v_k}  \bigg)
\end{align*}
\begin{align*}
	\tilde{X}_\ell^2 =& \tilde{X}_\ell \circ \tilde{X}_\ell \\
	=& \Bigg( y \, \bigg( - \sum_{j=1}^{r-1} B_\ell^j(u,h) \cdot \frac{\partial}{\partial x_j} +  \sum_{j=1}^{r-1} A_\ell^j(u,h) \cdot \frac{\partial}{\partial v_j} \bigg) \Bigg)^2 \\
	=& y^2 \, \bigg( \sum_{j,k=1}^{r-1} B_\ell^j(u,h)B_\ell^k(u,h) \cdot \frac{\partial^2}{\partial x_j \partial x_k} \\
	&- 2\sum_{j,k=1}^{r-1} A_\ell^j(u,h)B_\ell^k(u,h) \cdot \frac{\partial^2}{\partial x_j \partial v_k} + \sum_{j,k=1}^{r-1} A_\ell^j(u,h)A_\ell^k(u,h) \cdot \frac{\partial^2}{\partial v_j \partial v_k}  \bigg)
\end{align*}
since the coordinates $y$, $u$ and $h$ on $M$ (and thus the functions $A$ and $B$) are independent of the $z$ coordinates $x$ and $v$ on $N$. Substituting into Casimir for complex $\nu$:
\begin{align*}
\Omega_\nu =& \frac{1}{2}H^2 - (r-1) H + \sum^{r-1}_{\ell = 1} X_\ell^2 + \sum^{r-1}_{\ell = 1} \tilde{X}_\ell^2 + \sum_j \theta_j \theta^\prime_j 
\end{align*}
\begin{align*}
	=& \frac{1}{2}(y \frac{\partial}{\partial y})^2 - (r-1) (y \frac{\partial}{\partial y}) \\
	&+ y^2 \, \sum^{r-1}_{\ell = 1}  \bigg( \sum_{j,k=1}^{r-1} A_\ell^j(u,h)A_\ell^k(u,h) \cdot \frac{\partial^2}{\partial x_j \partial x_k} \\
	&+ 2\sum_{j,k=1}^{r-1} A_\ell^j(u,h)B_\ell^k(u,h) \cdot \frac{\partial^2}{\partial x_j \partial v_k} + \sum_{j,k=1}^{r-1} B_\ell^j(u,h)B_\ell^k(u,h) \cdot \frac{\partial^2}{\partial v_j \partial v_k}  \bigg) \\
	&+ y^2 \, \sum^{r-1}_{\ell = 1}  \bigg( \sum_{j,k=1}^{r-1} B_\ell^j(u,h)B_\ell^k(u,h) \cdot \frac{\partial^2}{\partial x_j \partial x_k} \\
	&- 2\sum_{j,k=1}^{r-1} A_\ell^j(u,h)B_\ell^k(u,h) \cdot \frac{\partial^2}{\partial x_j \partial v_k} + \sum_{j,k=1}^{r-1} A_\ell^j(u,h)A_\ell^k(u,h) \cdot \frac{\partial^2}{\partial v_j \partial v_k}  \bigg) \\
	& + \sum_j \theta_j \theta^\prime_j
\end{align*}
\begin{align*}
	=& \frac{1}{2}y^2 \frac{\partial^2}{\partial y^2} - \frac{2r-3}{2}y \frac{\partial}{\partial y} \\
	& + y^2 \sum_{j,k=1}^{r-1} \Bigg( \sum^{r-1}_{\ell = 1}  \bigg(A_\ell^j(u,h)A_\ell^k(u,h) + B_\ell^j(u,h)B_\ell^k(u,h) \bigg) \Bigg) \cdot \bigg( \frac{\partial^2}{\partial x_j \partial x_k}  + \frac{\partial^2}{\partial v_j \partial v_k} \bigg)  \\
	& + \sum_j \theta_j \theta^\prime_j \\
	=&  \frac{1}{2}y^2 \frac{\partial^2}{\partial y^2} - \frac{2r-3}{2}y \frac{\partial}{\partial y} 
	+ y^2 \sum_{j,k=1}^{r-1} M_{jk}(u,h) \bigg( \frac{\partial^2}{\partial x_j \partial x_k}  + \frac{\partial^2}{\partial v_j \partial v_k} \bigg) + \Omega_\nu^\prime
\end{align*}
where $\Omega_\nu^\prime = \sum_j \theta_j \theta^\prime_j$ is the local Casimir on $M^1_2$.
Further simplification is possible by expanding the expression for the second order coefficients. A simple, if brutish, calculation shows
$$
A_\ell^j(u,h)A_\ell^k(u,h) + B_\ell^j(u,h)B_\ell^k(u,h) = \text{Re}(h^{\ell j}h^{\ell k})
$$
so that
\begin{align*}
M_{jk}(u,h) =& \sum^{r-1}_{\ell = 1}  \bigg(A_\ell^j(u,h)A_\ell^k(u,h) + B_\ell^j(u,h)B_\ell^k(u,h) \bigg) \\
&= \sum^{r-1}_{\ell = 1} \text{Re}(h^{\ell j}h^{\ell k}) = \sum^{r-1}_{\ell = 1} \text{Re}((h^\top)^{j \ell}h^{\ell k}) \\
&=  \text{Re}\bigg( \sum^{r-1}_{\ell = 1} (h^\top)^{j \ell}h^{\ell k} \bigg) = \text{Re}\bigg( \big( h^\top \cdot h \big)^{j k} \bigg) = \delta_{jk}
\end{align*}
since $h$ is in the local $M^1_2 = O(r-1,\mathbb{C})$. Thus, at complex $\nu$, we have
\begin{align*}
\Omega_\nu =& \frac{1}{2}y^2 \frac{\partial^2}{\partial y^2} - \frac{2r-3}{2}y \frac{\partial}{\partial y} 
	+ y^2 \sum_{j=1}^{r-1} \bigg( \frac{\partial^2}{\partial x_j^2}  + \frac{\partial^2}{\partial v_j^2} \bigg) + \Omega_\nu^\prime \\
&= \frac{1}{2}y^2 \frac{\partial^2}{\partial y^2} - \frac{2r-3}{2}y \frac{\partial}{\partial y} + y^2 \sum_{j=1}^{r-1} \bigg(\frac{\partial}{\partial x_j} + i\frac{\partial}{\partial v_j}\bigg)\bigg(\frac{\partial}{\partial x_j} - i\frac{\partial}{\partial v_j}\bigg) + \Omega_\nu^\prime
\end{align*}

%
%
\subsection{Global Casimir summed across archimedean places}\label{Casimir at complex nu in Iwasawa coords}
%
%
%
%

Recall that $k$ is a finite algebraic extension of $\mathbb{Q}$ and $S$ is a $k$-valued $k$-bilinear form on an $r+1$ $k$-dimensional vector space such that the dimension of the maximal totally isotropic subspace is 1. $G$ is the group of real points of the algebraic isometry group attached to $S$, attached to the {\' e}tale ring $k_\infty$ of archimedean completions of $k$: 
$$
G = G_\infty = G(\mathbb{R}) = G_\mathbb{R} = O(S_\infty) \subset GL(r+1, k_\infty)
$$
$G$ is a product of the orthogonal groups at the archimedean places of $k$:
$$
G = \prod_{\nu | \infty} G_\nu 
$$
To simplify notation, subscripts will typically only be used when discussing a particular place; the absence of a subscript means the context is the (more) {\it global} setting over all archimedean places (i.e., corresponding to the $\infty$ subscript, which we will largely suppress).

$G$ has an Iwasawa decomposition
$$
G = PK = NMK
$$
where $P$ is a parabolic subgroup, $K$ a (maximal) compact subgroup, $N$ the unipotent radical of $P$ and $M$ a Levi-Malcev complement to $N$. This gives
The Levi component of P is $M = k_\infty^\times \times O(S')$ and extending the Galois norm $N_{k/\mathbb{Q}}$ to be suitably multilinear, $M^1$ is
$$
\{ b \in k_\infty^\times : N_{k/\mathbb{Q}}(b) = N_{k_\infty/\mathbb{R}}(b)  =1  \} \times O(S^\prime)
$$
so that $M$ decomposes as a product $M = A^+ \cdot M^1 = A^+ \cdot M^1_1 \cdot M^1_2$. Up to choice of normalization and coordinates, the split component $A^+$ is the diagonal in the local split components. 
$$
P/N M^1 \approx A^+ \approx \mathbb{R}^+
$$
Consequently, being the complement of $A^+$, $M^1_1$ is not a product. $M^1_2$ is a product of the corresponding local factors. Note that $M^1$ may not be compact but $(M^1 \cap \Gamma)\backslash M^1$ will be compact as described below. Thus, as a product
$$
G = \prod_{\nu | \infty} G_\nu = \prod_{\nu | \infty} P_\nu K_\nu  = \prod_{\nu | \infty} N_\nu M_\nu K_\nu  
$$
and by taking a quotient by $K$ on the right
$$
G/K \approx N \cdot M/(M \cap K) 
$$
Let $\Gamma$ be an arithmetic subgroup of $G$ and, per the discussion in the sections on \hyperref[sec:Reduction-Theory-Siegel-Sets]{reduction theory} and \hyperref[sec:Remarks-Number-of-Cusps]{analysis on cusps}, we are treating the quotient $\Gamma\backslash G/K$ as though there were a single cusp. Specifically, for $c \in \mathbb{R}$, let $\eta(g) = \eta(n_x m^\prime \ell_y) = \eta(\ell_y)$ be the height of $g$. Let $Y_o$ and $Y_\infty$ be the respective images of $\{g \in \frak{S}  :  \eta(g) \leq c+1\}$ and $\{g \in \frak{S} : \eta(g) \geq c \}$ in $\Gamma\backslash G/K$. By construction, the interiors of $Y_o$ and $Y_\infty$ cover $\Gamma\backslash G/K$. $Y_o$ is the inverse image of a compact set under a continuous function and is thus a compact set, though with possibly complicated geometry. For sufficiently large $c$, $Y_\infty$ conveniently decomposes as a product of a compact manifold and a ray:
\begin{align*}	
		Y_\infty \approx & \; (N \cap \Gamma)\backslash N \times (M^1 \cap \Gamma)\backslash M^1/(M^1 \cap K) \times (c, \infty) \\
		\approx & (N \cap \Gamma)\backslash N \times (M^1_1 \cap \Gamma)\backslash M^1_1/(M^1_1 \cap K) \times (M^1_2 \cap \Gamma)\backslash M^1_2/(M^1_2 \cap K) \times (c, \infty) \\
		\approx & \text{ compact } \times \text{ compact } \times \text{ compact } \times \text{ ray }  
\end{align*}	
The Casimir operator on $G$ is
$$
\Omega_G = \Omega_\infty = \sum_{\nu | \infty} \Omega_\nu
$$
where each $\Omega_\nu$ contains a summand $\Omega_\nu^\prime$ corresponding to the $\nu^{\text{th}}$ factor of $M^1_2$. While $\Omega_\nu^\prime$ is not invariant under $G$, it does descend to the image in $\Gamma\backslash G$ of a sufficiently high Siegel set (i.e., corresponding to a sufficiently large value of the height parameter), allowing separation of variables. In particular, while $M^1$ contributes to the coefficients of the Laplacian tangential to $N$, by application of the units theorem and compactness of anisotopic quotients, in the quotient of the Siegel set these coefficients are bounded since they are continuous coordinates on a compact set (i.e., corresponding to the compact region $D$ in $M$ in the Siegel set).

Summed across archimedean places, since in coordinates $\Omega = \Delta$
\begin{align*}
\Delta =& \; L^y + \Delta_{M^1_1} + y^2\Delta^{\scriptscriptstyle M^1}_N + \Omega^\prime = C_2 \, y^2 \frac{\partial^2}{\partial y^2} + C_1 \, y \frac{\partial}{\partial y} + y^2\Delta^{\scriptscriptstyle M^1}_N + \Delta_{M^1}
\end{align*}
for explicit constants $C_2$ and $C_1$, where $\Delta_{M^1}$ combines $\Omega^\prime$ (Casimir on $M^1_2$) and $\Delta_{M^1_1}$, and where $\Delta^{\scriptscriptstyle M^1}_N$ includes the partial derivatives on $N$ whose coefficients depend on the $M^1$ coordinates $u$ and $h$. The coefficients in the more global expressions are tuples, namely elements of $k_\infty$.

The dependence of $\Delta$ on $M^1$, specifically that of $\Delta^{\scriptscriptstyle M^1}_N$, manifests at real places in the coefficients for the second order terms involving derivatives in directions tangent to $N$. We recall an elementary result to illustrate that this dependence is uniform. Let $X$ be the open cone of (strictly) positive-definite real $n$-by-$n$ symmetric matrices. Let $G=GL(n,\mathbb{R})$ act on $X$ by $g(x)=g^\top \cdot  x \cdot g$. Then, given a compact subset $C$ of $G$, the image $C \cdot 1=\{ g(1_n): g \in C \}$ is clearly compact. 
\begin{lemma}\label{lemma-bded-elliptic-laplacian}
Given a compact subset $Y$ of $X$, there are positive real numbers $a$ and $b$ such that $a \cdot 1_n$ is a uniform lower bound for $Y$ and $b \cdot 1_n$ is a uniform upper bound for $Y$ in $X$, in the sense that $A - a \cdot 1_n$ is positive definite, and $b \cdot 1_n - A$ is positive definite, for all $A \in Y$.	
\end{lemma}
\begin{proof}
For all $x \in \mathbb{R}^n$ and all $A \in Y \subset X$, $x^\top \cdot A \cdot x > 0$ so clearly
$$
0 < \inf_{|x|= 1} x^\top \cdot A \cdot x \; < \; \sup_{|x|= 1} x^\top \cdot A \cdot x 	< +\infty
$$
Given $a$, $b$ such that $0 < a < \inf_{|x|= 1} x^\top \cdot A \cdot x$ and $\sup_{|x|= 1} x^\top \cdot A \cdot x < b$ then:
\begin{align*}
&	x^\top \cdot (b \cdot 1_n - A) \cdot x = b \cdot |x|^2 - x^\top \cdot A \cdot x > 0 \\
&	x^\top \cdot (A - a \cdot 1_n) \cdot x =  x^\top \cdot A \cdot x - a \cdot |x|^2 > 0
\end{align*}$\square$
\end{proof}
\noindent
An analogous argument shows the result equally applies to hermitian matrices under similar assumptions. 

The Laplacian $\Delta_N$ of $N$ considered as a subgroup of $G$ corresponds to the operator $\Delta^{\scriptscriptstyle M^1}_N$ at $u = \text{identity} \in M^1_1$ and $h = \text{identity} \in M^1_2$ so we can meaningfully speak of the {\it pure} Laplacian $\Delta_N$ of $N$. We next show that the {\it perturbed} Laplacian $\Delta^{\scriptscriptstyle M^1}_N$, which is roughly the Laplacian of the coset of $N$ corresponding to a value in $M^1$, can be bounded above and below in terms of scalar multiples of the pure Laplacian $\Delta_N$. Specifically, the uniform boundedness in the quotient $\Gamma\backslash G/K$ of the coefficient $u \in M^1_1$ is a consequence of the units theorem and that of elements $h \in M^1_2$ is a consequence of the compactness of anisotropic quotients.

Temporarily ignore the dependence on $y^2$. By the compactness of anisotropic quotients and the units theorem, there are representatives for $M^1 = M^1_1 \cdot M^1_2$ that are contained in a compact subset of a high-enough Siegel set. Thus, at all archimedean places $\nu$, the coefficients for the second order terms parallel to $N$ are restricted to a compact set and thus meet the positivity, compactness and symmetry criteria of ~\cref{lemma-bded-elliptic-laplacian}. 

Now we establish the uniform estimate on $\Delta^{\scriptscriptstyle M^1}_N$. Let $\partial_{\nu, j}$ be either the corresponding real or complex linear operator with corresponding conjugate $\overline{\partial_{\nu, j}}$ and coefficient matrix $Z_{\nu, jk}$ computed above. Letting $f \in C^\infty_c(\Gamma\backslash G/K)$, so that
\begin{align*}
& \int_{(N \cap \Gamma)\backslash N} \int_{(M^1_2 \cap \Gamma)\backslash M^1_2/(M^1_2 \cap K)}  \Bigg( \sum_{\nu | \infty} \; \bigg( \sum_{j,k=1}^{r-1} Z_{\nu, jk} \, \overline{\partial_{\nu,j}} \, \partial_{\nu,k} \bigg) f \cdot \bar{f} \Bigg) \, dg \\
&= \int_{(N \cap \Gamma)\backslash N} \int_{(M^1_2 \cap \Gamma)\backslash M^1_2/(M^1_2 \cap K)}  \; \Bigg( \sum_{\nu | \infty} \; \bigg(\sum_{j=1}^{r-1} \, \overline{\partial_{\nu,j}} \, \sum_{k=1}^{r-1} Z_{\nu, jk} \, \partial_{\nu,k} \bigg) f \cdot \bar{f} \Bigg) \, dg 
\end{align*}
Using integration by parts
\begin{align*}
&= - \int_{(N \cap \Gamma)\backslash N} \int_{(M^1_2 \cap \Gamma)\backslash M^1_2/(M^1_2 \cap K)}  \sum_{\nu | \infty} \Bigg( \sum_{j=1}^{r-1}  \big( \sum_{k=1}^{r-1} Z_{\nu, jk} \, \partial_{\nu,k} f \big)  \cdot \overline{\partial_{\nu,j} f}  \Bigg) \, dg \\
&= - \int_{(N \cap \Gamma)\backslash N} \int_{(M^1_2 \cap \Gamma)\backslash M^1_2/(M^1_2 \cap K)}  \sum_{\nu | \infty} \Bigg( \sum_{j, k=1}^{r-1} Z_{\nu, jk} \, \partial_{\nu,k} f \; \overline{\partial_{\nu,j} f} \Bigg) \, dg 
\end{align*}
By the units theorem and compactness of anisotropic quotients, the $(Z_{\nu, jk})$ are covered by a compact preimage in $M^1$ so that  a corresponding form of the lemma holds with constants $a$ and $b$ for the {\it pure} Laplacian $\Delta_\nu$ of $N_\nu$. Reverse the integration by parts, which changes the signs back and thus there are constants $a$ and $b$ so that
\begin{align*}
a \cdot \int_{(N \cap \Gamma)\backslash N} & \int_{(M^1_2 \cap \Gamma)\backslash M^1_2/(M^1_2 \cap K)} \bigg( \sum_{\nu | \infty}  \Delta_{N_\nu} f \cdot \bar{f} \bigg)  \, dg \\
< & \; \int_{(N \cap \Gamma)\backslash N}  \int_{(M^1_2 \cap \Gamma)\backslash M^1_2/(M^1_2 \cap K)} \Bigg( \sum_{\nu | \infty}  \bigg( \sum_{j,k=1}^{r-1} Z_{\nu, jk} \overline{\partial_{\nu,j}}\partial_{\nu,k} \bigg) f \cdot \bar{f} \Bigg) \, dg \\
& \quad < \; b \cdot \int_{(N \cap \Gamma)\backslash N} \int_{(M^1_2 \cap \Gamma)\backslash M^1_2/(M^1_2 \cap K)} \bigg( \sum_{\nu | \infty}  \Delta_{N_\nu} f \cdot \bar{f} \bigg) \, dg 
\end{align*}
which implies that we can use the pure Laplacian $\Delta_N$ to analyze expressions involving $\Delta^{\scriptscriptstyle M^1}_N$ since the latter is bounded above and below by the former.

Combining these bounds implies the only non-compact coefficient is the {\it ray} variable $y$ in $A^+$ and we can select values $a$ and $b$ so that
\begin{align*}
&  a \cdot \int_{(N \cap \Gamma)\backslash N}  \int_{(M^1_1 \cap \Gamma)\backslash M^1_1/(M^1_1 \cap K)}  \int_{(M^1_2 \cap \Gamma)\backslash M^1_2/(M^1_2 \cap K)}  \int_{(c, \infty)}  y^2\Delta_N f \cdot \bar{f} \, dg \\
&<	\int_{(N \cap \Gamma)\backslash N}  \int_{(M^1_1 \cap \Gamma)\backslash M^1_1/(M^1_1 \cap K)}  \int_{(M^1_2 \cap \Gamma)\backslash M^1_2/(M^1_2 \cap K)}  \int_{(c, \infty)} y^2\Delta^{\scriptscriptstyle M^1}_N f \cdot \bar{f} \, dg \\
& \quad < b \cdot \int_{(N \cap \Gamma)\backslash N}  \int_{(M^1_1 \cap \Gamma)\backslash M^1_1/(M^1_1 \cap K)}  \int_{(M^1_2 \cap \Gamma)\backslash M^1_2/(M^1_2 \cap K)}  \int_{(c, \infty)}  y^2\Delta_N f \cdot \bar{f} \, dg 
\end{align*}
We have established both that $\Delta^{\scriptscriptstyle M^1}_{N_\nu}$ is negative definite at all places (so that, more globally, both $\Delta^{\scriptscriptstyle M^1}_N$ and $\Omega^\prime$ are negative-definite), and also that we may use the more convenient operator $\Delta_N$ in place of $\Delta^{\scriptscriptstyle M^1}_{N_\nu}$ in our analysis. Thus, we will next focus on exhibiting estimates for $L_y$. We then combine all of these estimates to demonstrate a Rellich-type compactness result which is used to demonstrate the discrete spectrum of $\Delta$, the discrete decomposition of pseudo-cuspforms and the meromorphic continuation of certain Eisenstein series.

%
%

\subsection{Positivity of fragments of $-\Delta$ on $G/K$}\label{Postivity-of-fragments-of-Delta-G(R)}

%
%
The Lax-Phillips argument requires not only that $-\Delta$ itself be non-negative but that the natural summands of $-\Delta$ in Iwasawa coordinates are also non-negative. The Laplacian on the quotient $\Gamma\backslash G/K$ is
$$
\Delta = L^y + y^2\Delta^{\scriptscriptstyle M^1}_N + \Omega^\prime
$$
where the right hand side is summed over infinite places. The middle term $\Delta^{\scriptscriptstyle M^1}_N$ was previously shown to be bounded above and below by the Laplacian for $N$. Thus, up to scalar multiples, we can write
$$
\Delta = L^y + y^2\Delta_N + \Omega^\prime
$$
The global $-\Delta$ is positive since the Laplacian of $G$ descends to the quotient. The middle and right-hand terms, $-y^2\Delta_N$ and $-\Omega^\prime$, are positive since the Laplacian of $N$ can be mapped to the Laplacian on $\mathbb{R}^{r-1}$ and $-\Omega^\prime$ corresponds to the Laplacian on a semisimple group.

Note that the terms $y^2\Delta_N$ and $\Omega^\prime$ do not contain derivatives parallel to $A^+$. Thus since the global $-\Delta$ is positive, the sum of the operators defined by the right hand side is positive. Since two of the three terms are positive, this implies the operator $-L_y$ is positive as restricting $-\Delta$ to functions on $A^+$ (and constant on the other directions) will isolate derivatives in the $y$-direction so $-L_y$ must be a positive operator.

%
%

\subsection{For $a \gg 1 \; \mathscr{D}_a \text{ is dense in } L^2_a(\Gamma \backslash G /K)$ }\label{sec:Density-of-Automorphic-Test-Functions-G}

To establish notation:
\begin{align*}
C^\infty_c(\Gamma \backslash G /K) =&  \text{ right } K \text{-invariant functions in } C^\infty_c(\Gamma \backslash G) = C^\infty_c(\Gamma \backslash G)^K \\
L^2(\Gamma \backslash G /K) =& \text{ right } K \text{-invariant functions in }L^2(\Gamma \backslash G) = L^2(\Gamma \backslash G)^K \\
|f|^2_{\frak{B}^1} = & \langle (1 - \Delta)f , f \rangle = \langle f ,f \rangle + \langle (-\Delta)f,f \rangle \\
\frak{B}^1 = & \text{ completion of } C^\infty_c(\Gamma \backslash G / K) \text{ with respect to the } \frak{B}_1 \text{ norm} \\
\eta(n m_y k) = & \; y^r (\text{In some circumstances, which will be clear, this may be } r-1)\\
L^2_a(\Gamma \backslash G /K) = & \; \{ f \in L^2(\Gamma \backslash G /K) \; : \; c_P f (g) = 0 \text{ for } \eta(g) \geq a \} \\
 =& \text{ {\it pseudo-cuspforms} with cut-off height } a \\
\mathscr{D}_a = & \; C^\infty_c(\Gamma \backslash G /K) \cap L^2_a(\Gamma \backslash G /K) \\
\Delta_a = &\; \Delta \text{ restricted to } \mathscr{D}_a \\
\frak{B}^1_a = & \text{ closure of } \mathscr{D}_a \text{ in } L^2_a(\Gamma \backslash G /K) \\
\end{align*}
%
%

Let $G = G(\mathbb{R})$ have a maximal compact $K$ and minimal parabolic subgroup and Iwasawa-Levi-Mal'cev decomposition
$$
G = PK = NMK \approx N \cdot A^+ \cdot M^1 \cdot K \approx N \cdot A^+ \cdot M^1_1 \cdot M^1_2 \cdot K
$$
The model of $G/K$ is
$$
G/K \approx N \cdot A^+ \cdot M^1_1/(M^1_1 \cap K) \cdot M^1_2/(M^1_2 \cap K)
$$
Use Iwasawa coordinates on $G/K$:
$$
g= n_x a_y m_1(u) m_2(h) \quad n_x \in N \enspace a_y \in A^+ \enspace m_1(u) \in M^1_1 \enspace m_2(h) \in M^1_2
$$
with
$$
m = m_1(u) m_2(h)
$$
coordinates on $M^1$. For convenience and to simplify expressions, we will write Iwasawa coordinates as $(x,y,u,h)$ unless indicated.
\begin{lemma}\label{density-of-auto-test-fns-on1}
	For $a \gg 1 \; \mathscr{D}_a \text{ is dense in } L^2_a(\Gamma \backslash G /K)$
\end{lemma}
\begin{proof}
Recall that a (standard) Siegel set is a subset of $G$ given by a compact set $C \subset N$, a compact subset $D \subset M^1$ and a (height) parameter $t$
$$
\frak{S}_{t,C,d} = \{ nm a_y k \; : \; n \in C, m \in D, k \in K, a_y \in A^+ \text{ where } y \geq t \}
$$
Reduction theory implies that there is a sufficiently small $t$ and compact sets $C_i \subset N$ and $D_i \subset M^1$, along with a finite number of elements $g_i \in G(\mathbb{Q})$ so that $\Gamma$-translates of the union of the $g_i$-translates of the standard Siegel sets $\frak{S}_i = \frak{S}_{t,C_i,D_i}$ cover $G$:
	$$
	\Gamma \cdot \bigg( \bigcup_i g_i \cdot \frak{S}_i \bigg) = G
	$$
	Such a collection of Siegel sets clearly surjects to the quotients $\Gamma\backslash G$ and $\Gamma\backslash G/K$. Given such a collection $\frak{S}$ of Siegel sets, reduction theory then guarantees the existence of a height $t_o \gg t$ so that the subsets $\frak{S}^i_{t_o}$ of $\frak{S}^i_{t,C,D}$ given by the height condition
	$$
	\frak{S}^i_{t_o} = \big\{ g  \in \frak{S}^i_{t,C,D} \big| \enspace g = n \cdot m \cdot k \enspace \text{such that } m = m^\prime a_y \text{ and the height of } a_y = \eta(a_y) > t_o \big\}
	$$
	satisfy, {\it for all $i$ and $j$}, $\frak{S}^i_{t_o} \cap \gamma \frak{S}^j_{t_o} \ne \emptyset$ implies $\gamma \in \Gamma \cap P = \Gamma_\infty$. This implies the existence of a sufficiently large height so that high-enough portions of the Siegel sets in $\frak{S}$ do not interact in the quotient: the associated cusps can be treated separately. Since there is a finite collection of Siegel sets, common bounds can be taken to work for all the Siegel sets. In particular and without loss of generality, this allows us to simplify our notation by examining a single Siegel set as the results apply to any finite collection of cusps.
Fix such a $\frak{S}_{t_o}$ and let $\frak{S}_a$ be the subset of $\frak{S}_{t_o}$ given by
$$
\frak{S}_a = \big\{ g \in \frak{S}_{t_o} \big| \enspace \eta(g) = \eta(n_g a_g) = \eta(a_y) > a \big\}
$$
By reduction theory, there is a height $a \gg 1$ so that $\frak{S}_a$ satisfies $\frak{S}_a \cap \gamma \frak{S}_a \ne \emptyset$ implies $\gamma \in \Gamma \cap P = \Gamma_\infty$. In the following, $f \in L^2_a(\Gamma \backslash G /K)$ is first approximated by test functions $f_r \in C^\infty_c(\Gamma \backslash G /K)$ by general methods, and then the condition $a \gg 1$ is used in conjunction with a family of smooth cut-off functions of the constant near height $a$, with the width of cut-off region shrinking to $0$.

Per the above, take $a \gg 1$ so that $\frak{S}_{a - \frac{1}{t}}$ meets its translates $\gamma \frak{S}_{a - \frac{1}{t}}$ only for $\gamma \in \Gamma_\infty$, for all sufficiently large $t$. This allows separation of variables in $\frak{S}_{a - \frac{1}{t}}$ in the sense that the cylinder $C_{a - \frac{1}{t}} =  (\Gamma \cap P) \backslash \frak{S}_{a - \frac{1}{t}}$ injects to $\Gamma \backslash G /K$. Let
$$
|f|^2_{a - \frac{1}{t}} = \int_{C_{a - \frac{1}{t}}} |f(x,y,u,h)|^2 \, dg \leq  \int_{\Gamma \backslash G /K} |f(x,y,u,h)|^2 \, dg = |f|^2_{L^2}
$$
Let $f_r \in C^\infty_c(\Gamma \backslash G / K)$ with $f_r \to f$ in $L^2(\Gamma \backslash G / K)$. However, while $f \in L^2_a(\Gamma \backslash G / K)$, one would suspect that the constant terms of the $f_r$ are not too far from that of $f$ (and the difference must be going to zero, even in $L^2(\Gamma \backslash G / K)$), so that a smooth truncation of the constant terms of the $f_r$ should produce functions also approaching $f$. Namely, our strategy is to start with a general, and generic, approximating sequence in $L^2$ and remove the part that is keeping this sequence from being in $L^2_a$.

Using Iwasawa coordinates $(x,y,u,h)$ with $x \in \mathbb{R}^{r-1}$ and $y \in (0, \infty)$, the height is $\eta(x,y,u,h) = \eta(y)$. Let $\beta$ be a smooth function on $\mathbb{R}$ such that
$$
\begin{cases} 
0 = \beta (y) & (\text{for } 0 \leq y < -1) \\
0 \leq \beta (y) \leq 1  & (\text{for } -1 \leq y \leq 0) \\
1 =\beta (y)  & (\text{for } 0 \leq y) 
\end{cases}
$$
For $t > 1$, put $\beta_t(y) = \beta(t(y - a))$, and define a smooth function on $N \backslash G /K$ by
$$
\phi_{r,t}(a_y m_1(u) m_2(h)) = 
\begin{cases} 
\beta_t(\eta(y))\cdot c_P f_r(y) & (\text{for } \eta(y) \geq a - \frac{1}{t}) \\
0   & (\text{for } \eta(y) < a - \frac{1}{t})
\end{cases}
$$
For $t > 0$ large enough so that $\frak{S}_{a - \frac{1}{t}}$ meets its translates $\gamma \frak{S}_{a - \frac{1}{t}}$ only for $\gamma \in \Gamma_\infty$, let $\Psi_{r,t} = \Psi_{\phi_{r,t}}$ be the pseudo-Eisenstein series made from $\phi_{r,t}$:
$$
\Psi_{r,t}(x,y,u,h) = \sum_{\gamma \in \Gamma_\infty \backslash \Gamma} \phi_{r,t}(\gamma \cdot n_x a_y m_1(u) m_2(h))
$$
The assumption on $t$ assures that in the region $y^r > a - \frac{1}{t}$ we have $\Psi_{r,t} = c_P \Psi_{r,t} = \phi_{n.t}$. Thus $c_P(f_r - \Psi_{r,t})$ vanishes in $y \geq a$, so $f_r - \Psi_{r,t} \in L^2_a(\Gamma \backslash G /K)$, as desired.

By the triangle inequality
$$
|f - (f_r -\Psi_{r,t})|_{L^2} \leq |f - f_r|_{L^2}  + |\Psi_{r,t}|_{L^2} 
$$
where by assumption $|f -f_r|_{L^2} \to 0$. Thus it suffices to show that the $L^2$ norm of $\Psi_{r,t}$ goes to $0$ for large $n$ and $t$. Since $a \gg 1$,
$$
|\Psi_{r,t}|_{L^2}  = |\Psi_{r,t}|_{C_{a - \frac{1}{t}}} = |\phi_{r,t}|_{C_{a - \frac{1}{t}}} = |\beta(t(y-a)) \cdot c_P f_r|_{C_{a - \frac{1}{t}}} \leq |c_P f_r|_{C_{a - \frac{1}{t}}}
$$
The cylinder $C_{a - \frac{1}{t}}$ admits a natural translation action of the product of circle groups (i.e., $\mathbb{T}^{r-1} \approx (N \cap \Gamma) \backslash N$), inherited from the translation of the $x$-component in Iwasawa coordinates $n_x a_y m_1(u) m_2(h)$. This induces an action of $\mathbb{T}^{r-1}$ on $L^2(C_{a - \frac{1}{t}})$ with the norm $|\cdot |_{C_{a - \frac{1}{t}}}$. Thus, the map $f \to c_P f$ is given by a continuous, compactly-supported, $L^2(C_{a - \frac{1}{t}})$-valued integrand, which exists as a Gelfand-Pettis integral (cf. \S [14.1] in \cite{PBG}). This implies that the restriction of $c_P f_r$ to $C_{a - \frac{1}{t}}$ goes to $c_P f$ in $L^2(C_{a - \frac{1}{t}})$. As $c_P f$ is supported in the range $\eta(g) \leq a$ and the measure of $C_a - C_{a - \frac{1}{t}}$ goes to $0$ as $t \to +\infty$, the $C_{a - \frac{1}{t}}$-norm of $c_P f$ also goes to $0$ as $t \to +\infty$, since $c_P f$ is locally integrable.

In particular, this implies the diagonal terms $\Psi_{n.n}$ go to $0$ in $L^2$ norm, so that $f_r - \Psi_{r,n}$ go to $f$ in $L^2$ norm, proving the density of $\mathscr{D}_a$ in $L^2_a$. $\square$
\end{proof}

%
%

\subsection{$L^2(\Gamma\backslash G(\mathbb{R})/K)$ norms of truncated tails go to $0$ strongly}\label{L2-norms-zero-strongly-G}
Recall that
\begin{align*}
|f|^2_{\frak{B}^1} = & \langle (1 - \Delta)f , f \rangle = \langle f ,f \rangle + \langle (-\Delta)f,f \rangle \\
L^2_a(\Gamma \backslash G /K) = & \; \{ f \in L^2(\Gamma \backslash G /K) \; : \; c_P f (g) = 0 \text{ for } \eta(g) \geq a \} \\
\mathscr{D}_a = & \; C^\infty_c(\Gamma \backslash G /K) \cap L^2_a(\Gamma \backslash G /K) \\
\frak{B}^1_a = & \text{ completion of } \mathscr{D}_a \text{ with respect to the } \frak{B}^1 \text{ norm} 
\end{align*}
\begin{lemma}\label{sec:L2a-tail-norm-bded-by-B1a-norm}
Let $B$ be the unit ball in $\frak{B}^1_a(\Gamma \backslash G /K)$ then, given $\varepsilon > 0$, a cutoff $c \gg a$ can be made sufficiently large so that the image of $B$ in $L^2_a(\Gamma \backslash G /K)$ lies in a single $\varepsilon$-ball in $L^2_a(\Gamma \backslash G /K)$. That is, for $f \in \frak{B}^1_a(\Gamma \backslash G /K)$, the integral of the {\it cutoff tail}
\begin{align*}
\lim_{c \to \infty}  \int_{(N \cap \Gamma)\backslash N}  & \int_{(M^1_1 \cap \Gamma)\backslash M^1_1/(M^1_1 \cap K)}  \int_{(M^1_2 \cap \Gamma)\backslash M^1_2/(M^1_2 \cap K)}  \int_c^\infty |f|^2 \, dg = 0
\end{align*}
uniformly for $|f|_{\frak{B}^1_a} \leq 1$. This will follow from stronger estimate we prove: that for suitably large $y > c \gg 1$
$$
\int_{(N \cap \Gamma)\backslash N}  \int_{(M^1_1 \cap \Gamma)\backslash M^1_1/(M^1_1 \cap K)}  \int_{(M^1_2 \cap \Gamma)\backslash M^1_2/(M^1_2 \cap K)}  \int_c^\infty |f|^2 \, dg \ll \frac{1}{c^2} \cdot |f|^2_{\frak{B}^1_a}
$$
\end{lemma}
\begin{proof}
%
%
%
%
Use Iwasawa coordinates on $G/K$:
$$
g= n_x a_y m_1(u) m_2(h) \quad n_x \in N \enspace a_y \in A^+ \enspace m_1(u) \in M^1_1 \enspace m_2(h) \in M^1_2
$$
which we write as $(x,y,u,h)$ for convenience. Recall that the Laplacian $\Delta$ on the quotient can be written as
$$
\Delta = L^y + y^2\Delta_N + \Omega^\prime
$$
since the $M^1$-dependent Laplacian $\Delta^{\scriptscriptstyle M^1}_N$ can be dominated by scalar multiples of the pure Laplacian $\Delta_N$ on $N$. Conveniently, the unipotent radical $N$ is commutative so that we may use Fourier methods in conjunction with $\Delta_N$. Set $\ell = (r-1)[k:\mathbb{Q}]$, let $\xi$ run over characters of $(N \cap \Gamma)\backslash N \approx \mathbb{T}^\ell$, and take height parameters $c \geq c_0 \geq a \gg 1$. With Iwasawa coordinates $(x,y,u,h)$, write the Fourier expansion in $x$ as 
$$
f(x,y,u,h) = \sum_\xi \hat{f}(\xi, y,u,h) \xi(x) \enspace \bigg(\; = \sum_{\xi \in \mathbb{Z}^\ell} \hat{f}(\xi, y,u,h) e^{2 \pi i \xi \cdot x} \bigg)
$$
Since $f \in \frak{B}^1_a$, in Iwasawa coordinates $(x,y,u,h)$
$$
c_P f(w) = \int_{(N \cap \Gamma)\backslash N} f(n_x w) dn_x \; = \int_{(N \cap \Gamma)\backslash N} f(x,y,u,h) e^{2 \pi i \; 0 \cdot x}dn_x \; = \;\hat{f}(0,y) \;\; ( 0 \in \mathbb{Z}^\ell )
$$
so that $\hat{f}(0, y) = 0$ when $y \geq c \gg a$. By Plancherel in $x$ 
\begin{align*}
& \int_{N_\mathbb{Z} \backslash N_\mathbb{R}} \int_{(M^1_1 \cap \Gamma)\backslash M^1_1/(M^1_1 \cap K)}  \int_{(M^1_2 \cap \Gamma)\backslash M^1_2/(M^1_2 \cap K)} \int_{y \geq c} |f|^2 \, \, dg \\
&= \int_{(N \cap \Gamma) \backslash N} \int_{(M^1_1 \cap \Gamma)\backslash M^1_1/(M^1_1 \cap K)}  \int_{(M^1_2 \cap \Gamma)\backslash M^1_2/(M^1_2 \cap K)} \int_{y > c} |f(x,y,u,h)|^2 \, \, dg \\
&= \sum_{\xi \in \mathbb{Z}^{r-1}} \int_{(M^1_1 \cap \Gamma)\backslash M^1_1/(M^1_1 \cap K)} \int_{(M^1_2 \cap \Gamma)\backslash M^1_2/(M^1_2 \cap K)} \int_{y \geq c} |\hat{f}(\xi,y,u,h)|^2 \, \, dg
\end{align*}
Since $\hat{f}(0, y) = 0$ when $y \geq c$, the sum is over $\xi \neq 0$, so that $|\xi| \geq 1$ and
\begin{align*}
&\sum_\xi \int_{(M^1_1 \cap \Gamma)\backslash M^1_1/(M^1_1 \cap K)}  \int_{(M^1_2 \cap \Gamma)\backslash M^1_2/(M^1_2 \cap K)} \int_{y \geq c} |\hat{f}(\xi, y,u,h)|^2 \, \, dg \\
& \ll \; \sum_\xi \int_{(M^1_1 \cap \Gamma)\backslash M^1_1/(M^1_1 \cap K)} \int_{(M^1_2 \cap \Gamma)\backslash M^1_2/(M^1_2 \cap K)}\int_{y \geq c} |\xi|^2 \cdot|\hat{f}(\xi, y,u,h)|^2 \, \, dg
\end{align*}
With $\Delta_N$ the Euclidean Laplacian in the $x$ coordinate on $N$,
$$
|\xi|^2 \cdot \hat{f}(\xi,y) = \frac{1}{4 \pi^2} (- \Delta_N f)\widehat{\;\;}(\xi,y) \ll (- \Delta_N f)\widehat{\;\;}(\xi,y,u,h)
$$
Substituting this back and applying Plancherel
\begin{align*}
& \sum_\xi \int_{(M^1_1 \cap \Gamma)\backslash M^1_1/(M^1_1 \cap K)}  \int_{(M^1_2 \cap \Gamma)\backslash M^1_2/(M^1_2 \cap K)} \int_{y \geq c} |\xi|^2 \cdot|\hat{f}(\xi, y,u,h)|^2 \, \, dg \\
\ll & \; \sum_\xi \int_{y \geq c} (- \Delta_N f)\widehat{\;\;}(\xi,y) \overline{\widehat{f}}(\xi, y,u,h) \, \, dg \\
=& \; \int_{(N \cap \Gamma) \backslash N} \int_{(M^1_1 \cap \Gamma)\backslash M^1_1/(M^1_1 \cap K)}  \int_{(M^1_2 \cap \Gamma)\backslash M^1_2/(M^1_2 \cap K)} \int_{y > c} - \Delta_N f \cdot \bar{f} \;\; \, dg
\end{align*}
Again using that $y > c \gg a \gg 1$
\begin{align*}
&\int_{(N \cap \Gamma) \backslash N} \int_{(M^1_1 \cap \Gamma)\backslash M^1_1/(M^1_1 \cap K)} \int_{(M^1_2 \cap \Gamma)\backslash M^1_2/(M^1_2 \cap K)} \int_{y > c} - \Delta_N f \cdot \bar{f} \;\, dg \\
& \leq \; \frac{1}{c^2} \int_{(N \cap \Gamma) \backslash N} \int_{(M^1_1 \cap \Gamma)\backslash M^1_1/(M^1_1 \cap K)} \int_{(M^1_2 \cap \Gamma)\backslash M^1_2/(M^1_2 \cap K)} \int_{y > c} - y^2 \Delta_N f \cdot \bar{f} \; \, dg
\end{align*}
Recall the positivity result
$$
0 \leq \int_{(N \cap \Gamma) \backslash N} \int_{(M^1_1 \cap \Gamma)\backslash M^1_1/(M^1_1 \cap K)} \int_{(M^1_2 \cap \Gamma)\backslash M^1_2/(M^1_2 \cap K)} \int_{y > c}  -L_y f \cdot \bar{f} \, \, dg 
$$
and 
$$
0 \leq \int_{(N \cap \Gamma) \backslash N} \int_{(M^1_1 \cap \Gamma)\backslash M^1_1/(M^1_1 \cap K)} \int_{(M^1_2 \cap \Gamma)\backslash M^1_2/(M^1_2 \cap K)} \int_{y > c}  -\Omega^\prime f \cdot \bar{f} \, \, dg 
$$
so that
\begin{align*}
& \frac{1}{c^2} \int_{(N \cap \Gamma) \backslash N} \int_{(M^1_1 \cap \Gamma)\backslash M^1_1/(M^1_1 \cap K)} \int_{(M^1_2 \cap \Gamma)\backslash M^1_2/(M^1_2 \cap K)} \int_{y > c} -y^2\Delta_N f \cdot \bar{f} \;\, dg \\
& \leq \frac{1}{c^2} \int_{(N \cap \Gamma) \backslash N} \int_{y > c} -\bigg(L^y + y^2\Delta_N + \Omega^\prime  \bigg) f \cdot \bar{f} \;\, dg
\end{align*}
Substituting back, for smooth $f$ with support in $y \geq c \gg a$ and $\Delta = L^y + y^2\Delta_N + \Omega^\prime$
\begin{align*}
& \int_{(N \cap \Gamma) \backslash N} \int_{(M^1_1 \cap \Gamma)\backslash M^1_1/(M^1_1 \cap K)} \int_{(M^1_2 \cap \Gamma)\backslash M^1_2/(M^1_2 \cap K)} \int_{y > c} |f(x,y,u,h)|^2 \; \, dg \\
&\;\; \ll \;\;  \frac{1}{c^2} \int_{(N \cap \Gamma) \backslash N} \int_{(M^1_1 \cap \Gamma)\backslash M^1_1/(M^1_1 \cap K)} \int_{(M^1_2 \cap \Gamma)\backslash M^1_2/(M^1_2 \cap K)} \int_{y > c} -\Delta f \cdot \bar{f} \;\; \, dg
\end{align*}
Of course, also
$$
0 \leq \frac{1}{c^2} \int_{(N \cap \Gamma) \backslash N} \int_{(M^1_1 \cap \Gamma)\backslash M^1_1/(M^1_1 \cap K)} \int_{(M^1_2 \cap \Gamma)\backslash M^1_2/(M^1_2 \cap K)} \int_{y > c} |f(x,y,u,h)|^2 \; \, dg
$$
So adding this to the right side above gives
\begin{align*}
& \int_{(N \cap \Gamma) \backslash N} \int_{(M^1_1 \cap \Gamma)\backslash M^1_1/(M^1_1 \cap K)} \int_{(M^1_2 \cap \Gamma)\backslash M^1_2/(M^1_2 \cap K)} \int_{y > c} |f(x,y,u,h)|^2 \; \, dg \\
& \enspace \ll  \frac{1}{c^2} \int_{(N \cap \Gamma) \backslash N} \int_{(M^1_1 \cap \Gamma)\backslash M^1_1/(M^1_1 \cap K)} \int_{(M^1_2 \cap \Gamma)\backslash M^1_2/(M^1_2 \cap K)} \int_{y > c} (1 -\Delta) f \cdot \bar{f} \;\; \, dg \leq \frac{1}{c^2} \cdot |f|^2_{\frak{B}^1_a}
\end{align*}
as claimed. $\square$
\end{proof}

%
%

\subsection{$\frak{B}^1$ norms of tails are bounded by global $\frak{B}^1$ norms}\label{sec:Global-B1a-bounds-of-smooth-trunc-L2a-tail-norms}
The previous inequality did not directly apply to smooth truncations of $f$ in $\frak{B}^1_a$ near height $c > a$, nor establish that a collection of smooth truncations $\phi_\infty \cdot f$ over all heights $c > a$ can be chosen with $\frak{B}^1$-norms uniformly bounded for $f \in B$.

The following conventions will be used in this section: for fixed height $\eta$, for $t \geq 1$, the smoothly cut-off tail $f^{[t]}$ is described as follows. Let $\phi$ be a smooth function such that $0 \leq \phi \leq 1$ on $(0,\infty)$
$$
\begin{cases} 
0 = \phi (y) & (\text{for } 0 \leq y \leq 1) \\
0 \leq \phi (y) \leq 1  & (\text{for } 1 \leq y \leq 2) \\
1 =\phi (y)  & (\text{for } 2 \leq y) 
\end{cases}
$$
Since $\phi$ is smooth and constant outside the compact interval $[1,2]$, there is a common pointwise bound $C_\phi < \infty$ for $|\phi|$, $|\phi^\prime|$ and $|\phi^{\prime\prime}|$. For $t > 0$, define a smooth cut-off function by
$$
\phi_t(y) = \phi(y/t)
$$
so that $\phi_t(y) \to 0 \enspace \forall y \text{ as } t \to \infty$. Use Iwasawa coordinates on $G/K$:
$$
g= n_x a_y m_1(u) m_2(h) \quad n_x \in N \enspace a_y \in A^+ \enspace m_1(u) \in M^1_1 \enspace m_2(h) \in M^1_2
$$
which we write as $(x,y,u,h)$ for convenience.
For $f \in H^1 = H^1(\Gamma \backslash G/K)$, let $f^{[t]}(x,y,u,h) = \phi_t(y)\cdot f(x,y,u,h)$. 
\begin{lemma}
$|f^{[t]}|_{H^1} \ll |f|_{H^1} \; (\text{implied constant independent of } f \text{ and } t \geq 1)$
\end{lemma}
\begin{proof}
We have
$$
\Omega |_{G/K} = \Delta = L^y + \Delta^{\scriptscriptstyle M^1}_N + \Omega^\prime
$$
Since $0 \leq \phi_t \leq 1$, clearly $|f^{[t]}|_{L^2} = |\phi_t \cdot f|_{L^2} \leq |f|_{L^2}$. For the other part of the $H^1$ norm
\begin{align*}
& \langle -\Delta f^{[t]}, f^{[t]} \rangle =\\
&	 \int_{(N \cap \Gamma) \backslash N} \int_{(M^1_1 \cap \Gamma)\backslash M^1_1/(M^1_1 \cap K)} \int_{(M^1_2 \cap \Gamma)\backslash M^1_2/(M^1_2 \cap K)} \int_{y \geq c} -(L^y + \Delta^{\scriptscriptstyle M^1}_N + \Omega^\prime) f^{[t]} f^{[t]} \, \, dg \\
	=& \int_{(N \cap \Gamma) \backslash N} \int_{(M^1_1 \cap \Gamma)\backslash M^1_1/(M^1_1 \cap K)} \int_{(M^1_2 \cap \Gamma)\backslash M^1_2/(M^1_2 \cap K)} \int_{y \geq c} -(\phi_t(y)^2)(\Delta^{\scriptscriptstyle M^1}_N + \Omega^\prime)f \cdot f \, \, dg \\
	-&  \int_{(N \cap \Gamma) \backslash N} \int_{(M^1_1 \cap \Gamma)\backslash M^1_1/(M^1_1 \cap K)} \int_{(M^1_2 \cap \Gamma)\backslash M^1_2/(M^1_2 \cap K)} \int_{y \geq c} L_y f^{[t]} \cdot f^{[t]} \, dg
\end{align*}
since the $(\Delta^{\scriptscriptstyle M^1}_N + \Omega^\prime)$ factor treats $y$ as a constant. Temporarily ignore the first term in $(\Delta^{\scriptscriptstyle M^1}_N + \Omega^\prime)$ to expand the term $L^y$ with the derivatives in $y$
\begin{align*}
	 & \int_{(N \cap \Gamma) \backslash N} \int_{(M^1_1 \cap \Gamma)\backslash M^1_1/(M^1_1 \cap K)} \int_{(M^1_2 \cap \Gamma)\backslash M^1_2/(M^1_2 \cap K)} \int_{y \geq c} - L_y f^{[t]} \cdot f^{[t]} \, dg \\
	 &=  \int_{(N \cap \Gamma) \backslash N} \int_{(M^1_1 \cap \Gamma)\backslash M^1_1/(M^1_1 \cap K)} \int_{(M^1_2 \cap \Gamma)\backslash M^1_2/(M^1_2 \cap K)} \int_{y \geq c} - (\phi_t(y)^2) L_y f \cdot f \, dg \\
	 &+ \int_{(N \cap \Gamma) \backslash N} \int_{(M^1_1 \cap \Gamma)\backslash M^1_1/(M^1_1 \cap K)} \int_{(M^1_2 \cap \Gamma)\backslash M^1_2/(M^1_2 \cap K)} \int_{y \geq c} \\
	  & \Bigg( 4 \frac{\partial^2 \phi}{\partial y^2}(\frac{y}{t}) \phi(\frac{y}{t})  +  8 \bigg( \frac{\partial \phi}  {\partial y}(\frac{y}{t})\bigg)^2  + C(r) \frac{\partial \phi}{\partial y}(\frac{y}{t})  \phi(\frac{y}{t}) \Bigg) f(x,y)^2 \, dg
\end{align*}
Where $C(r)$ is a constant depending on $r$ analogous to the earlier cases (i.e., since all the factor groups are either $O(p,q)$ or $O(r-1,\mathbb{C})$). Using the common bound $|\phi|, |\phi^\prime|, |\phi^{\prime\prime}| < C_\phi$ we can estimate the last term above by
\begin{align*}
& \bigg| \; \int_{(N \cap \Gamma) \backslash N} \int_{(M^1_1 \cap \Gamma)\backslash M^1_1/(M^1_1 \cap K)} \int_{(M^1_2 \cap \Gamma)\backslash M^1_2/(M^1_2 \cap K)} \int_{y \geq c} \\
& \Bigg( 4 \frac{\partial^2 \phi}{\partial y^2}(\frac{y}{t}) \phi(\frac{y}{t})  +  8 \bigg( \frac{\partial \phi}  {\partial y}(\frac{y}{t})\bigg)^2  + C(r) \frac{\partial \phi}{\partial y}(\frac{y}{t})  \phi(\frac{y}{t}) \Bigg) f(x,y)^2 \, dg \; \bigg| \\
&\leq   \int_{(N \cap \Gamma) \backslash N} \int_{(M^1_1 \cap \Gamma)\backslash M^1_1/(M^1_1 \cap K)} \int_{(M^1_2 \cap \Gamma)\backslash M^1_2/(M^1_2 \cap K)} \int_{y \geq c} \\
& \bigg| \, \Bigg( 4 \frac{\partial^2 \phi}{\partial y^2}(\frac{y}{t}) \phi(\frac{y}{t})  +  8 \bigg( \frac{\partial \phi}  {\partial y}(\frac{y}{t})\bigg)^2  + C(r) \frac{\partial \phi}{\partial y}(\frac{y}{t})  \phi(\frac{y}{t}) \Bigg)\, \bigg| \; \bigg| \, f(x,y)^2 \, \bigg| \, dg 
\end{align*}
\begin{align*}
&\leq   \int_{(N \cap \Gamma) \backslash N} \int_{(M^1_1 \cap \Gamma)\backslash M^1_1/(M^1_1 \cap K)} \int_{(M^1_2 \cap \Gamma)\backslash M^1_2/(M^1_2 \cap K)} \int_{y \geq c} \\
&  \Bigg( 4 \bigg| \, \frac{\partial^2 \phi}{\partial y^2}(\frac{y}{t}) \phi(\frac{y}{t})\, \bigg| +  8 \bigg| \,\bigg( \frac{\partial \phi}  {\partial y}(\frac{y}{t})\bigg)^2\, \bigg|  + \bigg| \,C(r) \frac{\partial \phi}{\partial y}(\frac{y}{t}) \phi(\frac{y}{t}) \, \bigg|\Bigg) \; \bigg| \, f(x,y)^2 \, \bigg| \, dg \\
&\leq   \int_{(N \cap \Gamma) \backslash N} \int_{(M^1_1 \cap \Gamma)\backslash M^1_1/(M^1_1 \cap K)} \int_{(M^1_2 \cap \Gamma)\backslash M^1_2/(M^1_2 \cap K)} \int_{y \geq c} 
 \Bigg( 4 C_\phi^2 +  8 C_\phi^2  + C_\phi^2 \Bigg) \; \bigg| \, f(x,y)^2 \, \bigg| \, dg \\
&\ll  C_\phi^\prime \int_{(N \cap \Gamma) \backslash N} \int_{(M^1_1 \cap \Gamma)\backslash M^1_1/(M^1_1 \cap K)} \int_{(M^1_2 \cap \Gamma)\backslash M^1_2/(M^1_2 \cap K)} \int_{y \geq c} \, \bigg| \, f(x,y)^2 \, \bigg| \, dg 
\end{align*}

The other terms that did not involve $y$ can be added back and $L_y f$ combined with them to get $-\Delta$. The remaining term is bounded above by a constant given by a polynomial in $C_\phi$ and $|f|_{L^2(\Gamma\backslash G/K)}^2$. 
\begin{align*}
& |f^[t]|_{\frak{B}^1_a}^2 = \langle (1 -\Delta) f^{[t]}, f^{[t]} \rangle =\\
&	 \int_{(N \cap \Gamma) \backslash N} \int_{(M^1_1 \cap \Gamma)\backslash M^1_1/(M^1_1 \cap K)} \int_{(M^1_2 \cap \Gamma)\backslash M^1_2/(M^1_2 \cap K)} \int_{y \geq c} f^{[t]} f^{[t]} \, \, dg \\
&	 \int_{(N \cap \Gamma) \backslash N} \int_{(M^1_1 \cap \Gamma)\backslash M^1_1/(M^1_1 \cap K)} \int_{(M^1_2 \cap \Gamma)\backslash M^1_2/(M^1_2 \cap K)} \int_{y \geq c} -(L^y + \Delta^{\scriptscriptstyle M^1}_N + \Omega^\prime) f^{[t]} f^{[t]} \, \, dg 
\end{align*}
\begin{align*}
&=	 \int_{(N \cap \Gamma) \backslash N} \int_{(M^1_1 \cap \Gamma)\backslash M^1_1/(M^1_1 \cap K)} \int_{(M^1_2 \cap \Gamma)\backslash M^1_2/(M^1_2 \cap K)} \int_{y \geq c} (\phi_t(y)^2)|f|^2 \, \, dg \\
& \enspace \int_{(N \cap \Gamma) \backslash N} \int_{(M^1_1 \cap \Gamma)\backslash M^1_1/(M^1_1 \cap K)} \int_{(M^1_2 \cap \Gamma)\backslash M^1_2/(M^1_2 \cap K)} \int_{y \geq c} -(\phi_t(y)^2)(\Delta^{\scriptscriptstyle M^1}_N + \Omega^\prime)f \cdot f \, \, dg \\
	-&  \int_{(N \cap \Gamma) \backslash N} \int_{(M^1_1 \cap \Gamma)\backslash M^1_1/(M^1_1 \cap K)} \int_{(M^1_2 \cap \Gamma)\backslash M^1_2/(M^1_2 \cap K)} \int_{y \geq c} L_y f^{[t]} \cdot f^{[t]} \, dg
\end{align*}
\begin{align*}	
&\ll	 \int_{(N \cap \Gamma) \backslash N} \int_{(M^1_1 \cap \Gamma)\backslash M^1_1/(M^1_1 \cap K)} \int_{(M^1_2 \cap \Gamma)\backslash M^1_2/(M^1_2 \cap K)} \int_{y \geq c} (\phi_t(y)^2)|f|^2 \, \, dg \\
& \enspace -\int_{(N \cap \Gamma) \backslash N} \int_{(M^1_1 \cap \Gamma)\backslash M^1_1/(M^1_1 \cap K)} \int_{(M^1_2 \cap \Gamma)\backslash M^1_2/(M^1_2 \cap K)} \int_{y \geq c} (\phi_t(y)^2)(\Delta^{\scriptscriptstyle M^1}_N + \Omega^\prime)f \cdot f \, \, dg \\
& \enspace - \int_{(N \cap \Gamma) \backslash N} \int_{(M^1_1 \cap \Gamma)\backslash M^1_1/(M^1_1 \cap K)} \int_{(M^1_2 \cap \Gamma)\backslash M^1_2/(M^1_2 \cap K)} \int_{y \geq c} (\phi_t(y)^2) L_y f \cdot f \, dg \\
&+  C_\phi^\prime \int_{(N \cap \Gamma) \backslash N} \int_{(M^1_1 \cap \Gamma)\backslash M^1_1/(M^1_1 \cap K)} \int_{(M^1_2 \cap \Gamma)\backslash M^1_2/(M^1_2 \cap K)} \int_{y \geq c} \, \bigg| \, f(x,y)^2 \, \bigg| \, dg 
\end{align*}
Using that $|\phi| \leq 1$, using a common bound $C$ to combine the first and last terms above, rewrite as $(1 -\Delta)f \cdot f$
\begin{align*}	
&\ll	 \int_{(N \cap \Gamma) \backslash N} \int_{(M^1_1 \cap \Gamma)\backslash M^1_1/(M^1_1 \cap K)} \int_{(M^1_2 \cap \Gamma)\backslash M^1_2/(M^1_2 \cap K)} \int_{y \geq c} (\phi_t(y)^2)(1 -\Delta )f \cdot f \, \, dg \\
& \enspace = |f|_{\frak{B}^1_a}^2
\end{align*}
giving $|f^{[t]}|_{H^1}^2 \ll |f|_{H^1}^2 $.  $\square$
\end{proof}

%
%

\subsection{$\frak{B}^1_a(\Gamma\backslash G/K)$ includes compactly into $L^2_a(\Gamma\backslash G/K)$}\label{sec:G(R)-frak-B-includes-compactly}
Let $k$ be a number field and $G = G(S)_\mathbb{R} = G(\mathbb{R})$ a real Lie group, where $S$ is a $k$-valued $k$-bilinear form on a $k$-$(r+1)$-dimensional vector space with maximal totally isotropic subspace of dimension one. Also let $G$ have subgroups: $P$ a minimal parabolic, $K$ a maximal compact subgroup, and $\Gamma \subset G$ an  arithmetic subgroup. So that we have an Iwasawa-Levi-Mal'cev decomposition of $G$
$$
G = PK = NMK \approx N \cdot A^+ \cdot M^1_1 \cdot M^1_2 \cdot K
$$
with model of $G/K$
$$
G/K \approx N \cdot A^+ \cdot M^1_1/(M^1_1 \cap K) \cdot M^1_2/(M^1_2 \cap K)
$$
We Use Iwasawa coordinates on $G/K$:
$$
g= n_x a_y m_1(u) m_2(h) \quad n_x \in N \enspace a_y \in A^+ \enspace m_1(u) \in M^1_1 \enspace m_2(h) \in M^1_2
$$
which we typically write as $(x,y,u,h)$ for convenience. 
\begin{theorem}
For $G$, $K$ and $\Gamma$ as above, then with $a \gg 1$ the inclusion $\frak{B}^1_a(\Gamma \backslash G/K) \to L^2_a(\Gamma \backslash G/K)$ is compact
\end{theorem}

\begin{proof}
Let $(x,y,u,h)$ be coordinates from the Iwasawa decomposition. By reduction theory, let $\frak{S}$ be a fixed Siegel set that surjects to the quotient $\Gamma \backslash G$. Then, given $c \geq a$, let $Y_o$ and $Y_\infty$ be the respective images of $\{g \in \frak{S} \; : \; \eta(g) = \eta(y) \leq c+1\}$ and $\{g \in \frak{S} \; : \; \eta(g) = \geq c \}$ in $\Gamma \backslash G/K$. By construction, the interiors of $Y_o$ and $Y_\infty$ cover $\Gamma \backslash G/K$. 

Let $U_i$ be a cover of $Y_o$ by open sets. By continuity of the projection to $\Gamma \backslash G/K$, $Y_o$ is compact so that we can take a finite sub-cover $U_1, \ldots ,U_m$ of $Y_o$ in $\Gamma \backslash G/K$ with small compact closures in the sense that the $U_i$ have closures in a small collar neighborhood of $Y_o$. Let $U_\infty$ be an open set in $\Gamma \backslash G/K$ covering $Y_\infty$ so that $\{U_1, \ldots ,U_m\} \cup U_\infty$ is an open cover of $\Gamma \backslash G /K$.

From general considerations we can take a smooth partition of unity $\{ \phi_i \}$ subordinate to the finite subcover $U_i$  on $Y_o$; that is,  $\phi_i$ has compact support inside the open $U_i$ for $i \in \{1, \ldots ,m\}$ and
so $\sum_i \phi_i = 1$ on $Y_o$. Let $\phi_\infty$ be a smooth function that is identically 1 for $\eta \geq c$ and that $\text{support}(\phi_\infty) \subset U_\infty$. The smooth function $\phi_\infty$ should be chosen in a fashion similar to the smooth cut-offs used previously: i.e., $\phi_\infty$ is $0$ for $0 \leq \eta \leq c-1$, $\phi_\infty$ is $1$ for $\eta \geq c$ and all values of $\phi_\infty$ lie between $0$ and $1$. In particular, $\phi_\infty$ is zero on a non-empty open contained in $Y_o$.
It will also be useful later that $\phi_\infty$ and $1 - \phi_\infty$ form a two-element partition of unity.

Iwasawa coordinates provide an explicit way of separating $\Gamma \backslash G/K$ into a cusp part $Y_\infty$ and a compact body $Y_o$. Let $B$ be the unit ball in $\frak{B}^1_a(\Gamma\backslash G/K)$. Given $\varepsilon > 0$, take $c - 1 > c^\prime \gg a$ sufficiently large (with associated $Y_o$, $Y_\infty$ and smooth cut-off $\phi_\infty$ determined by $c$) so that $\phi_\infty \cdot B$ lies in a single $\varepsilon \backslash 2$-ball in $L^2_a(\Gamma \backslash G /K)$. That is, using \cref{sec:L2a-tail-norm-bded-by-B1a-norm}, choose $c^\prime \gg a$ so that
$$
\int_{(N \cap \Gamma) \backslash N} \int_{(M^1_1 \cap \Gamma)\backslash M^1_1/(M^1_1 \cap K)}  \int_{(M^1_2 \cap \Gamma)\backslash M^1_2/(M^1_2 \cap K)}\int_{y > c^\prime} |f|^2 \; dg \ll \frac{1}{(c^\prime)^2} \cdot |f|^2_{\frak{B}^1_a(\Gamma \backslash G /K)}
$$
Choose $\phi_\infty$, such that $0 \leq \phi_\infty \leq 1$ and $\phi_\infty = 0$ for $y \leq c-1$:
\begin{align*}
 \int_{(N \cap \Gamma) \backslash N} & \int_{y \geq a} |\phi_\infty \cdot f|^2 \; dg \\
 & \leq \int_{(N \cap \Gamma) \backslash N} \int_{y > c^\prime} |f|^2 \; dg \ll \frac{1}{(c^\prime)^2} \cdot |f|^2_{\frak{B}^1_a(\Gamma \backslash G /K)}
\end{align*}
That is,
$$
|\phi_\infty \cdot f|^2_{L^2_a(\Gamma \backslash G /K)} \ll \frac{1}{(c^\prime)^2} \cdot |f|^2_{\frak{B}^1_a(\Gamma \backslash G /K)}
$$
Specifically, choose large $c^\prime$ so that
$$
|\phi_\infty \cdot f|^2_{L^2_a(\Gamma \backslash G /K)}  < \frac{\varepsilon}{2} <  \frac{1}{(c^\prime)^2} \cdot |f|^2_{\frak{B}^1_a(\Gamma \backslash G /K)}
$$
Next, note that $(1 - \phi_\infty) \cdot f$ has support in $Y_o$ for any $f \in \frak{B}^1_a(\Gamma \backslash G /K)$ and the map
$$
(1 - \phi_\infty) \cdot \; : \: \frak{B}^1_a(\Gamma \backslash G /K) \to \frak{B}^1_a(Y_o)
$$
is continuous and therefore bounded by \cref{sec:bdedness-smooth-projections}.

Thus, the image of $B$ under the restriction of this map
$$
B \to (1 -\phi_\infty) \cdot B \to \frak{B}^1_a(Y_o) \to L^2(Y_o) \subset L^2(\Gamma \backslash G /K)
$$
is pre-compact and can be covered by finitely many $\varepsilon/2$ balls in $L^2(\Gamma \backslash G /K)$. Finally, $B = \phi_\infty \cdot B + (1 - \phi_\infty) \cdot B$ is therefore covered by finitely-many $\varepsilon/2$ balls in $L^2_a(\Gamma \backslash G /K)$ so the image of $B$ in $L^2(\Gamma \backslash G/K)$ is pre-compact. $\square$
\end{proof}
%

%
%

\subsection{$\widetilde{\Delta}_a$ has purely discrete spectrum}\label{sec:Delta-a-G(R)-has-discrete-spectrum}
%
%
%


The Rellich-like result showing the compactness of the inclusion $\frak{B}^1_a(\Gamma\backslash G/K) \to L^2_a(\Gamma\backslash G/K)$ for $a \gg 1$ is used to show that the Friedrichs self-adjoint extension $\widetilde{\Delta}_a$ of the restriction $\Delta_a$ of $\Delta$ to test functions $\mathscr{D}_a$ in $L^2_a$ has compact resolvent, thus establishing that $\widetilde{\Delta}_a$ has purely discrete spectrum.

We briefly recap the exposition in the section \hyperref[sec:Friedrichs-extension-defn]{Friedrichs' canonical self-adjoint extensions} (following the development in \cite{PBG} \S\S 10.9 and 9.4). Let $T: V \to V$ be a positive, semi-bounded operator on a Hilbert space $V$ with dense domain $D$ and define a hermitian form $\langle ,\rangle_1$ and corresponding norm $| \cdot |_1$ by
$$
\langle v,w \rangle_1 = \langle v,w\rangle + \langle Tv,w\rangle = \langle v,(1+T)w \rangle = \langle (1+T)v,w \rangle \quad (\text{for } v,w \in D)
$$
The symmetry and positivity of $T$ make $\langle , \rangle_1$ positive-definite hermitian on $D$, and $\langle v, w \rangle_1$ is defined if at least one of $v$, $w$ is in $D$. Let Let $V^1$ be the Sobolev-like Hilbert-space defined by the completion of $D$ with respect to the metric induced by the norm $| \cdot |_1$ on $D$ (the definition of $V^1$ is analogous to standard definitions of $L^2$ Sobolev spaces via $\Delta$ or the Fourier transform). Since the norm $| \cdot |_1$ dominates the norm on $V$ (by positivity of $T$), the completion $V^1$ maps continuously to $V$. Since these are Hilbert spaces, the map is injective. Friedrichs' Theorem then tells us there is a positive self-adjoint extension $\widetilde{T}$ with domain $\widetilde{D} \subset V^1$. In particular, there is an important corollary: When the inclusion $V^1 \to V$ is compact, the resolvent $(1 + \widetilde{T})^{-1} : V \to V$ is compact. 

Substituting $L^2_a(\Gamma\backslash G/K)$ for $V$, $\mathscr{D}_a$ for $D$ and $\frak{B}^1_a$ for $V^1$ in the above, we have that the compactness of the map $\frak{B}^1_a(\Gamma\backslash G/K) \to L^2_a(\Gamma\backslash G/K)$ implies the compactness of the resolvent $(1 - \widetilde{\Delta}_a)^{-1}$ of the Friedrichs extension $\widetilde{\Delta}_a$. Recall also the association between eigenvalues of an operator and its (compact) resolvent: for $z$ off a discrete set $X$ in $\mathbb{C}$, the inverse $(z - \widetilde{\Delta}_a)^{-1}$ exists, is a compact operator and
$$
z \to \Bigg( (z - \tilde{\Delta}_a)^{-1} : \enspace L^2_a(\Gamma \backslash G /K) \to L^2_a(\Gamma \backslash G /K) \Bigg)
$$
is a meromorphic operator-valued function of $z \in \mathbb{C} - X$. The eigenvectors of $\widetilde{\Delta}_a$ are eigenvectors of $(z - \widetilde{\Delta}_a)^{-1}$ and eigenvalues $\lambda$ of $\widetilde{\Delta}_a$ are in bijection with non-zero eigenvalues of $(z -\widetilde{\Delta}_a)^{-1}$ by $\lambda \leftrightarrow (z-\lambda)^{-1}$. Here, the discrete set $X$ is the point-spectrum of $\widetilde{\Delta}_a$ and the domain (as a function of $z$) of $(z - \widetilde{\Delta}_a)^{-1}$ is the resolvent set (cf. \hyperref[sec:spectrum-T-spectrum-compact-resolvent]{recovering the spectrum of an operator from its resolvent}). Thus, we have that $\widetilde{\Delta}_a$ has discrete spectrum and the spectrum of $\widetilde{\Delta}_a$ matches that of its resolvent $(1 - \widetilde{\Delta}_a)^{-1}$ by the above association.

%
%

\subsection{Discrete decomposition of pseudo-cuspforms on $\Gamma\backslash G/K$}\label{sec:G-discrete-decomposition-pseudo-cuspforms}

To establish the discrete decomposition of pseudo-cuspforms on $\Gamma\backslash G/K$, it suffices to show the existence of a comparable estimate on smooth cut-offs on tails already used to establish the Rellich compactness inclusion of $\frak{B}^1_a(\Gamma\backslash G/K)$ into $L^2_a(\Gamma\backslash G/K)$. The norm on $\frak{B}^1_a(\Gamma\backslash G/K)$ is given by
$$
\int_{(N \cap \Gamma)\backslash N}  \int_{(M^1_1 \cap \Gamma)\backslash M^1_1/(M^1_1 \cap K)}  \int_{(M^1_2 \cap \Gamma)\backslash M^1_2/(M^1_2 \cap K)}  \int_{(c, \infty)} (1 -\Delta) f \cdot f \, dg
$$
Where, since the operators all have real coefficients the function $f$ can be assumed to be real-valued and where $dg$ is Haar measure on $G$ (the product measure of Haar measures on the factors) descended to the quotient.

Because $\widetilde{\Delta}_a$ is self-adjoint, its spectrum is real and there is a discrete decomposition of $L^2_a(\Gamma \backslash G /K)$ by an orthogonal basis of $L^2_a(\Gamma \backslash G /K)$ consisting of $\widetilde{\Delta}_a$-eigenvectors.

%
%

\subsection{Meromorphic continuation of Eisenstein series}\label{sec:Meromorphic-continuation-of-Eisenstein-series}

The meromorphic continuation argument here is an appropriately generalized version of that given in \hyperref[sec:Mero-Contn-Eis]{full meromorphic continuation of Eisenstein series for $O(r,1)$, $U(r,1)$ and $Sp^\ast(r,1)$}). As sketched in the introduction, in more general settings, an Eisenstein series is of the form $E_{s, f}$ where $f$ is a function on $(\Gamma \cap M^1)\backslash M^1/(M^1 \cap K)$ and represents {\it cuspidal data}. To establish notation, we recall the definition of cuspidal-data Eisenstein series in our context, following \cite{PBG}:
\begin{definition}
Let G have Iwasawa-Levi-Mal'cev decomposition
$$
G = PK = NMK \approx N \times A^+ \times M^1 \approx N \times A^+ \times M^1_1 \times M^1_2
$$
where $M^1$ is the complement to $A^+$ in $M$, let $\Omega^1$ be the Casimir operator on $M^1$, and let $f = \chi \otimes f_2$ be an $\Omega^1$ eigenfunction on 
$$
(\Gamma \cap M^1)\backslash M^1/(M^1 \cap K) \approx (\Gamma \cap M^1_1)\backslash M^1_1/(M^1_1 \cap K) \times (\Gamma \cap M^1_2)\backslash M^1_2/(M^1_2 \cap K)
$$
where $\chi$ is an unramified Hecke character on $k$ (i.e., on $M^1_1 \approx GL(1,k_\infty)$ invariant under $GL(1,\frak{o}))$.
Let
$$
\phi(g) = \phi_{s,f}(g) = y^s \cdot \chi(m_1) \cdot f_2(m_2)
$$
where 
$$
g = n \cdot m \cdot k = n \cdot a_y \cdot m_1 \cdot m_2 \cdot k \qquad n \in N, \enspace m = m_1 \cdot m_2 \in M^1 = M^1_1 \cdot M^1_2, \enspace k \in K
$$  
then the corresponding {\it genuine} Eisenstein series is:
$$
E_{s,f}(g) = \sum_{\gamma \in \Gamma} \phi_{s,f}(\gamma \cdot g)
$$\end{definition}

\begin{theorem}
$E_{s, f}$ has a meromorphic continuation in $s \in \mathbb{C}$, as a smooth function of moderate growth on $\Gamma\backslash G$. As a function of $s$, $E_{s, f}(g)$ is of at most polynomial growth vertically, which is uniform in bounded strips and for $g$ in compact subsets of $G$.	
\end{theorem}
Some consequences of the meromorphic continuation can be inferred quickly.
\begin{corollary}
The eigenfunction property 
\begin{align*}
\Delta E_{s, f} =& \lambda_{s,f} \cdot E_{s, f} \text{ where } \lambda_{s,f} = \lambda_s + \lambda_f \\
&(\text{where } \lambda_s = c \cdot s(s - 1) \text{ for suitable } c \in \mathbb{R} \text{ and } \lambda_f \text{ is the eigenvalue of } f \text{ on } M^1)
\end{align*}
persists under meromorphic continuation.	
\end{corollary}
\begin{proof}
Both $\Delta E_{s, f}$ and $\lambda_{s,f} \cdot E_{s, f}$ are holomorphic function-valued functions of $s$, taking values in the topological vector space of smooth functions. They agree in the region of convergence $\text{Re}(s) > 1$, so by the vector-valued form of the identity principle (\cite{PBG} \S 15.2) they agree on their mutual domain of convergence. $\square$
\end{proof}
\begin{corollary}
The meromorphic continuation of $E_{s, f}$ implies the meromorphic continuation of the constant term $c_P E_{s, f} = (\eta^s + c_{s,f} \, \eta^{1-s}) \cdot f(m^\prime)$ where the cuspidal data $f$ is a $\Delta_{M^1}$ eigenfunction on $(\Gamma \cap M^1)\backslash M^1/(M^1 \cap K)$. In particular, this establishes the meromorphic continuation of the function $c_{s,f}$.
\end{corollary}
\begin{proof}
Since $E_{s, f}$ meromorphically continues at least as a smooth function, the integral over the compact set $(N \cap \Gamma)\backslash N$ giving a pointwise value $c_P E_{s, f}(g)$ of the constant term certainly converges absolutely. That is, the function-valued function 
$$
n \longrightarrow (g \to E_{s, f}(ng))
$$
is a continuous, smooth-function-valued function and has a smooth-function-valued Gelfand-Pettis integral 
$$
g \longrightarrow c_P E_{s, f}(g) 
$$
(see \cite{PBG} \S 14.1). Thus, the constant term $c_P E_{s, f}$ of the continuation of $E_{s, f}$ must still be of the form $A_{s,f} \,\eta^s + B_{s,f} \, \eta^{1-s}$ for some smooth functions $A_{s,f}$ and $B_{s,f}$, since (at least for $s \neq 1$ ) $\eta^s$ and $\eta^{1-s}$ are the two linearly independent solutions of
$$
y^2 \cdot \frac{\partial^2 u}{\partial y^2} = s(1-s) \cdot u = \lambda_s \cdot u
$$ 
for functions $u$ on $A^+$. In the region of convergence $\text{Re}(s) > 1$,  direct computation gives $A_{s,f} = f(m^\prime)$ and $B_{s,f} = c_{s,f} \cdot f(m^\prime)$. The vector-valued form of analytic continuation implies that $A_{s,f} = f(m^\prime)$ throughout, and that $B_{s,f} = c_{s,f} \cdot f(m^\prime)$ throughout. In particular, this establishes the meromorphic continuation of $c_{s,f}$. $\square$
\end{proof}
Let $\widetilde{\Delta}_a$  be the Friedrichs extension of the restriction of the Laplacian $\Delta$ to $\mathscr{D}_a = C^\infty_c(\Gamma\backslash G/K) \cap L^2_a(\Gamma\backslash G/K)$. The Friedrichs construction shows that the domain of $\widetilde{\Delta}_a$  is contained in a Sobolev space:
$$
\text{domain } \widetilde{\Delta}_a \subset \frak{B}^1 =  \text{ completion of } C^\infty_c(\Gamma\backslash G/K) \text{ relative to } \langle v, w \rangle_{\frak{B}^1} = \langle (1 - \Delta ) v, w \rangle 
$$
The domain of $\widetilde{\Delta}_a$  contains the smaller Sobolev space
$$
\frak{B}^2 =  \text{ completion of } C^\infty_c(\Gamma\backslash G/K) \text{ relative to } \langle v, w\rangle_{\frak{B}^2} = \langle (1 - \Delta)^2 v, w \rangle 
$$
As before, use conditions on the height $\eta$ to decompose the quotient $\Gamma\backslash G/K$ as a union of a compact part $Y_{cpt} = Y_o$, whose geometry does not matter, and a geometrically simple non-compact part $Y_\infty$:
$$
\Gamma\backslash G/K = Y_o \cup Y_\infty \qquad \text{(compact  } Y_o, \text{ ``tubular'' cusp neighborhood } Y_\infty)
$$
relative to a condition on the normalized height function $\eta$ with $\eta(n \cdot m_y \cdot k) = a \gg 1$.
$$
Y_\infty = \text{ image of } \{ g \in G/K : \eta(g) \geq a \} = \Gamma_\infty \backslash \{g \in G/K : \eta (g) \geq a \}
$$

Define a smooth cut-off function $\tau$ as usual: fix $a^{\prime\prime} < a^\prime$ large enough so that the image of $\{ (x, y) \in G/K : y > a^{\prime\prime} \}$ in the quotient is in $Y_\infty$, and let
$$
\tau(g) = 
\begin{cases} 
1 & (\text{for } \eta(g) > a^\prime \\
0   & (\text{for } \eta(g) < a^{\prime\prime}
\end{cases}
$$
Form a pseudo-Eisenstein series $h_{s,f}$ by winding up the smoothly cut-off function $\tau(g) \cdot \phi_{s,f}(g)$:
$$
h_{s,f}(g) =  \sum_{\gamma \in \Gamma_\infty \backslash \Gamma} \tau(\gamma g) \cdot \phi_{s,f}(\gamma g)^s
$$
Since $\tau$ is supported on $\eta \geq a^{\prime\prime}$ for large $a^{\prime\prime}$, for any $g \in G/K$ there is at most one non-vanishing summand in the expression for $h_{s,f}$, and convergence is not an issue. Thus, the pseudo-Eisenstein series $h_{s,f}$ is entire as a function-valued function of $s$. Let 
\begin{align*}
\widetilde{E}_{s,f} &= h_{s,f} - (\widetilde{\Delta}_a - \lambda_{s,f})^{-1} (\Delta - \lambda_{s,f})h_{s,f}  \\
&(\lambda_{s,f} = \lambda_s + \lambda_f  \text{ where } \lambda_s = c \cdot s(s - 1) \text{ for suitable } c \in \mathbb{R} \text{ and } \lambda_f \text{ is the eigenvalue of } f \text{ on } M^1) 
\end{align*}
\begin{claim}
$\widetilde{E}_{s,f} - h_{s,f}$ is a holomorphic $\frak{B}^1$-valued function of $s$.	
\end{claim} 
\begin{proof} (as earlier)
From Friedrichs' construction, the resolvent $(\widetilde{\Delta}_a - \lambda_{s,f})^{-1}$ exists as an everywhere-defined,
continuous operator for $s \in \mathbb{C}$ for $\lambda_{s,f}$ not a non-positive real number, because of the non-positive-ness of $\Delta$. Further, for $\lambda_{s,f}$ not a non-positive real, the resolvent is a holomorphic operator-valued function. In fact, for such $\lambda_{s,f}$, the resolvent $(\widetilde{\Delta}_a - \lambda_{s,f})^{-1}$ injects from $L^2(\Gamma\backslash G/K)$ to $\frak{B}^1$. $\square$	
\end{proof}
\begin{remark}
The smooth function $(\Delta - \lambda_{s,f}) h_{s,f}$ is supported on the image of $a^{\prime\prime} \leq y \leq a^\prime$ in $\Gamma\backslash G/K$, which is compact. Thus, it is in $L^2(\Gamma\backslash G/K)$. Note that $\widetilde{E}_{s,f}$ does not vanish since the resolvent maps to the domain of $\Delta$ inside $L^2(\Gamma\backslash G/K)$, and that $h_{s,f}$ is not in $L^2(\Gamma\backslash G/K)$ for $\text{Re}(s) > \frac{1}{2}$. Thus, since $h_{s,f}$ is not in $L^2(\Gamma\backslash G/K)$ and $(\widetilde{\Delta}_a - \lambda_{s,f})^{-1}(\Delta - \lambda_{s,f}) h_{s,f}$ is in $L^2(\Gamma\backslash G/K)$, the difference cannot vanish.
\end{remark}
\begin{theorem}
For $\lambda_{s,f} = \lambda_s + \lambda_f $ not non-positive real, $u = \widetilde{E}_{s,f} - h_{s,f}$ is the unique element of the domain of $\widetilde{\Delta}_a$  such that 
$$
(\widetilde{\Delta}_a - \lambda_{s,f}) u = -(\Delta - \lambda_{s,f}) h_{s,f}
$$
Thus, $\widetilde{E}_{s,f}$ is the usual Eisenstein series $E_{s, f}$ for $\text{Re}(s) > 1$.
\end{theorem}
\begin{proof} (as in the earlier section \hyperref[sec:Mero-Contn-Eis]{meromorphic continuation of Eisenstein series to $\text{Re}(s) > \frac{1}{2}$})
\end{proof}
\textbf{\emph{Proof of full meromorphic continuation}}: since the resolvent $(\widetilde{\Delta}_a - \lambda_{s,f})^{-1}$ is a compact operator, $\widetilde{\Delta}_a$ has purely discrete spectrum. Thus, the resolvent $(\widetilde{\Delta}_a - \lambda_{s,f})^{-1}$ is meromorphic in $s$ in $\mathbb{C}$, and thus $E_{s,f} = h_{s,f} + (\widetilde{\Delta}_a - \lambda_{s,f})^{-1}(\Delta - \lambda_{s,f}) h_{s,f}$ is meromorphic in $s$ in $\mathbb{C}$. $\square$



%
%

\section{Appendix: Compactness of anisotropic quotients}\label{cptness-anisotropic-quotients}

Following \cite{PBG-anisotropic}, we reproduce some of the results covering the beginning of reduction theory for classical groups. See also Tamagawa-Mostow \cite{Mostow-Tamagawa}, Godement's Bourbaki article \cite{Godement-1962}, Borel Harish-Chandra\cite{Borel-HarishChandra-1962} and Borel \cite{Borel-1969}.

%
%
%

\subsection{Background material}

\subsubsection{Affine heights}

Let $K_\nu$ be the standard compact of $GL(n,\mathbb{Q}_\nu)$: for the archimedean completion $\mathbb{Q}_\infty = \mathbb{R}$, this is the usual orthogonal group $O(n)$, and for finite places $\nu$ and non-archimedean $\mathbb{Q}_\nu$, it is $GL(n,\mathbb{Z}_\nu)$. The fact that these are maximal compact subgroups will not be needed. Let $V= V_\mathbb{Q} = \mathbb{Q}^n$, and $V_\mathbb{A} = V \otimes_\mathbb{Q} \mathbb{A}$. Let $GL(n,\mathbb{A})$ act on the right of $V$ by matrix multiplication.
\\
\\
For the real prime $\nu = \infty$ of $\mathbb{Q}$ define the local height function $\eta_\nu$ on $x = (x_1, \ldots , x_n) \in V_\infty = \mathbb{R}^n$ by
$$
\eta_\infty = \sqrt{x_1^2 + \ldots x_n^2}
$$
For a non-archimedean (finite) prime $\nu$ of $\mathbb{Q}$ define the (local) height function $\eta_\nu$ on $x = (x_1, \ldots , x_n) \in V_{\mathbb{Q}_\nu} = \mathbb{Q}_\nu^n$ by
$$
\eta_\nu(x) = \text{sup}_i |x_i|_\nu
$$
A vector $x \in V_\mathbb{A}$ is primitive if it is of the form $x = x_o g$ where $g \in GL(n,\mathbb{A})$ and $x_o \in V_\mathbb{Q}$. That is, primitive vectors are the image of a rational vector under the adele group. For $x = (x_1, \ldots , x_n) \in V_\mathbb{Q}$, at almost all finite primes $\nu$, the $x_i's$ are in $\mathbb{Z}_\nu$ and have (local) greatest common divisor $1$.
\\
\\
For a given $x = (x_1, \ldots , x_n) \in V_\mathbb{Q}$, let $D$ be the (finite) set of the distinct primes dividing the denominators of the $x_i's$ (i.e., the union of all such primes), so for $\nu \notin D$, the $x_i's$ will be integral. Similarly, let $N$ be the (finite) set consisting of the primes dividing the various numerators of the $x_i's$; for $\nu \notin N \cup D$, the $x_i's$ are units and are relatively prime.
\\
\\
For primitive $x \in V_\mathbb{A}$ define the global height
$$
\eta(x) = \eta_\infty(x) \times \prod_{\nu \text{ prime}} \eta_\nu (x)
$$
By the above, since $x$ is primitive, at almost all finite primes $\nu$ the local height is $1$, so the product trivially converges.
\begin{lemma}
\hfill
\begin{itemize}
	\item For $t \in \mathbb{J}$ and primitive $x \in \mathbb{A}$, $\eta(tx) = |t|\eta(x)$ where $|t|$ is the idele norm.
	\item If a sequence of vectors in $\mathbb{A}$ goes to zero, then their heights go to zero as well.
	\item If the heights of a collection of (primitive) vectors $x_i$ go to zero, then there are scalars $t_i \in \mathbb{Q}^\times$ so that $t_i x_i$ goes to zero in $\mathbb{A}$.
	\item For $g \in GL(n, \mathbb{A})$ and $c > 0$, the set of non-zero vectors $x \in \mathbb{Q}^n$ so that $\eta(x g) < c$ is finite modulo $\mathbb{Q}^\times$. In particular, for each such $g$, the infimum of $\{ \eta(xg) \; : \; x \in \mathbb{Q}^n - 0\}$ is positive and is assumed.
	\item For a compact set $E$ of $GL(n, \mathbb{A})$ there are constants $c, c^\prime > 0$ so that for all primitive vectors $x$ and for all $g \in E$
	$$
	c \eta(x) \leq \eta(xg) \leq c^\prime \eta(x)
	$$
\end{itemize}
\end{lemma}
\begin{proof}
\hfill
\begin{itemize}
	\item Recall $t \in \mathbb{J}$, $t = \{y_\nu\}$ where $y_\nu \in Z^\times_\nu$ for almost all $\nu$ and the archimedean factor is in $\mathbb{R}^\times$. 
$$	
\eta(tx) = \prod_\nu \eta_\nu (y_\nu x_\nu) =  \prod_\nu \eta_\nu (y_\nu) \prod_\nu \eta_\nu(x_\nu) = |\{y_\nu\}| \prod_\nu \eta_\nu(x_\nu) = |t| \eta(x)
$$
since all products are finite.
	\item If a sequence of vectors $x_i$ goes to $0 \in \mathbb{A}$, then for every $\varepsilon > 0$, there is an $i_o$ and $N$ such that $i > i_o$ implies $|\eta_\nu(x_\nu)| < \varepsilon$ for all places so that $|x_\nu| < \varepsilon$ for all infinite primes and $x_\nu \in p_\nu^{N} \mathbb{Z}_\nu^n$ for all finite $\nu$ where $p_\nu^{-N} < \varepsilon$. Thus, letting $\ell$ be the number of infinite places
	$$
	\eta(x) \leq \varepsilon^\ell \times \prod_\nu \eta_\nu(x_\nu) = \varepsilon^\ell \times \prod_\nu p_\nu^{-N}
	$$
	so the heights go to zero.
	\item Suppose that $\eta(x_i)$ goes to $0$ for a sequence of primitive vectors $\{x_i\}$. At almost all finite $\nu$ the vector $x_i$ is in $\mathbb{Z}^n_\nu$ and the entries have local gcd $1$. Since $\mathbb{Z}$ is a principal ideal domain, we can choose $s_i \in \mathbb{Q}$ to that at every finite prime $\nu$ the components of $s_i x_i$ are locally integral and have greatest common divisor $1$. Then the local contribution to the height function from all finite primes is $1$. Therefore, the archimedean height of $s_i x_i$, Euclidean distance, goes to $0$. Finally, we need some choice of trick to make the vectors go to $0$ in $\mathbb{A}^n$. For example, for each index $i$ let $N_i$ be the greatest integer so that $\eta_\infty(s_i x_i) < \frac{1}{(N_i!)^2}$ Let $t_i = s_i \cdot N_i !$. Then $t_i x_i$ goes to $0$ in $\mathbb{A}^n$.
	\item Fix $g \in GL(n,\mathbb{A})$. Since $K$ preserves heights, via the Iwasawa decomposition we may suppose that $g$ is in the group $P_\mathbb{A}$ of upper triangular matrices in $GL(n,\mathbb{A})$. Let $g_{ij}$ be the $(i, j)^{th}$ entry of $g$. Choose representatives $x = (x_1, \ldots , x_n)$ for non-zero vectors in $\mathbb{Q}^n$ modulo $\mathbb{Q}^\times$ such that, letting $\mu$ be the first index with $x_\mu \neq 0$, then $x_\mu = 1$. That is, x is of the form
$$
x = (0, \dots ,0,1,x_{\mu+1},...,x_n)
$$
To illustrate the idea of the argument with a light notation, first consider $n = 2$, let $g = 
\begin{pmatrix} 
a & b \\ 
0 & d  
\end{pmatrix} 
$
 and $x = (1, y)$. Thus, 
$$
x \cdot g = (1,y)
\begin{pmatrix} 
a & b \\ 
0 & d  
\end{pmatrix} 
 = (a,b+yd)
$$
From the definition of the local heights, at each place $v$ of $k$
$$
 max(|a|_\nu,|b+yd|_v) \leq h_v(xg)
$$
from which 
$$
|b+yd|_v \prod_{w \neq v} |a|_w \leq \prod_{\text{all }w} h_w(xg) = h(xg)
$$
Since $g$ is fixed, $a$ is fixed, and at almost all places $|a|_w = 1$. Thus, for $h(xg) < c$ there is a uniform $c^\prime$ such that 
$$
|b+yd|_v \leq c^\prime \qquad (\text{for all } v)
$$
Since for almost all $v$ the residue class field cardinality $q_v$ is strictly greater than $c$
$$
|b + yd|_v \leq 1 \qquad (\text{for almost all } v)
$$
Therefore, $b + yd$ lies in a compact subset $C$ of $A$. Since $b$, $d$ are fixed, and since $\mathbb{Q}$ is discrete and closed in $\mathbb{A}$, the collection of images $\{b + dy \; : \; y \in k\}$ is discrete in $\mathbb{A}$. Thus, the collection of $y$ such that $b + dy$ lies in $C$ is finite. Now consider general $n$ and $x \in \mathbb{Q}^n$ such that $h(xg) < c$. Let $\mu - 1$ be the least index such that $x_\mu \neq 0$.  Adjust by $k^\times$ such that $x_\mu = 1$. For each $v$, from $h(xg) < c$
$$
|g_{\mu -1,\mu} +x_\mu g_{\mu,\mu} |_v \prod_{w \neq v} |g_{\mu -1, \mu -1} |_w \leq h(gx) < c
$$
For almost all places $v$ we have $|g_{\mu -1,\mu -1}|_v = 1$, so there is a uniform $c^\prime$ such that
$$
|g_{\mu -1, \mu} + x_\mu g_{\mu,\mu}|_v < c^\prime \qquad (\text{for all } v)
$$
For almost all $v$ the residue field cardinality $q_v$ is strictly greater than $c^\prime$, so for almost all $v$
$$
|g_{\mu -1,\mu} + x_\mu g_{\mu, \mu}|_v \leq 1
$$
Therefore, $g_{\mu -1,\mu} + x_\mu g_{\mu , \mu}$ lies in a compact subset $C$ of $\mathbb{A}$. Since $\mathbb{Q}$ is discrete, the collection of $x_\mu$ is finite.

Continuing similarly, there are only finitely many choices for the other entries of $x$. Inductively, suppose $x_i = 0$ for $i < \mu - 1$, and $x_\mu , \ldots , x_{\nu - 1}$ fixed, and show that $x_\nu$ has only finitely many possibilities. Looking at the $\nu^{th}$ component $(xg)_\nu$ of $xg$,
$$
|g_{\mu -1, \nu} +x_\mu g_{\mu , \nu} + \ldots + x_{\nu -1} g_{\nu - 1, \nu} + x_\nu g_{\nu, \nu} |_v \prod_{w \neq v} |g_{\mu -1, \mu -1}|_w \leq h(xg) \leq c
$$
For almost all places $v$ we have $|g_{\mu -1, \mu -1}|_w = 1$, so there is a uniform $c^\prime$ such that for all $v$
$$
|(xg)_\nu|_v =|g_{\mu -1, \nu} + x_\mu g_{\mu , \nu} + \ldots + x_{\nu -1} g_{\nu -1, \nu} + x_\nu g_{\nu,\nu}|_v < c^\prime
$$
For almost all $v$ the residue field cardinality $q_v$ is strictly greater than $c^\prime$, so for almost all $v$
$$
|g_{\mu -1, \nu} + x_\mu g_{\mu ,\nu} + \ldots + x_{\nu -1} g_{\nu -1, \nu} + x_\nu g_{\nu ,\nu}|_v \leq 1 
$$
Therefore,
$$
g_{\mu -1, \nu} + x_\mu g_{\mu ,\nu} + \ldots + x_{\nu -1} g_{\nu -1, \nu} + x_\nu g_{\nu ,\nu}
$$
lies in the intersection of a compact subset $C$ of $\mathbb{A}$ with a closed discrete set, so lies in a finite set. Thus, the number of possibilities for $x_\nu$ is finite. By induction we obtain the finiteness.

\item Let $E$ be a compact subset of $GL(n,\mathbb{A})$, and let $K = \prod_v K_v$. Then $K \cdot E \cdot K$ is compact, being the continuous image of a compact set. So without loss of generality $E$ is left and right $K$-stable. By Cartan decompositions the compact set $E$ of $GL(n,\mathbb{A})$ is contained in a set
$$
K \Delta K
$$
where $\Delta$ is a compact set of diagonal matrices in $GL(n,\mathbb{A})$. Let $g = \theta_1 \delta \theta_2$ with $\theta_i \in K$, and $x$ a primitive vector. By the $K$-invariance of the height,
$$
\frac{\eta(xg)}{\eta(x)} = \frac{\eta(x \theta_1 \delta \theta_2)}{\eta(x)} = \frac{\eta(x \theta_1 \delta)}{\theta(x)} = \frac{\eta((x\theta_1 )\delta)}{\eta((x\theta))} = 
$$
Thus, the set of ratios $\eta(xg)/\eta(x)$ for $g$ in a compact set and $x$ ranging over primitive vectors is exactly the set of values $\eta(x\delta)/\eta(x)$ where $\delta$ ranges over a compact set and $x$ varies over primitives. With diagonal entries $\delta_i$ of $\delta$,
$$
0 \; < \; \inf_{\delta \in \Delta} \inf_i |\delta_i| \; \leq \; \eta(x\delta)/\eta(x) \; \leq \; \sup_{\delta \in \Delta} \sup_i |\delta_i| \; < \; \infty
$$
by compactness of $\Delta$.
$\square$
\end{itemize}
	
\end{proof}

%
%
%
%
\subsection{Imbeddings and compactness criteria}

\subsubsection{Imbeddings of arithmetic quotients}\label{Imbeddings of arithmetic quotients}
Let $k$ be a number field. Let $Q = \langle , \rangle$ be a non-degenerate quadratic form on a $k$-vectorspace $V$, and $G = O(Q)$ the corresponding orthogonal group. We have the natural imbedding $G \to GL(V)$.
\begin{proposition}
The inclusion $G_k \to GL(V)_k$ induces an inclusion 
$$
G_k\backslash G_\mathbb{A}  \to GL(V)_k\backslash GL(V)_\mathbb{A}
$$
with closed image.
\end{proposition}
\noindent
A general topological lemma is necessary.
\begin{lemma}
Let $X$, $Y$ be locally compact Hausdorff topological spaces. Further, $X$ has a countable open cover $\{U_i \}$ such that every $U_i$ has compact closure. Let $G$ be a group acting continuously on $X$ and $Y$, transitively on $X$. Let $f : X \to Y$ be a continuous injective $G$-set map whose image is a closed subset of $Y$. Then f is a homeomorphism of $X$ to its image in $Y$.
\end{lemma}
\begin{proof}
This is a version of the Baire Category argument. Since $f(X)$ is closed in $Y$ the image $f(X)$ is itself (with the subset topology) a locally compact Hausdorff space. Therefore, without loss of generality, $f$ is surjective. Let $C_i$ be the closure of $U_i$. The images $f(C_i)$ of the $C_i$ are compact, hence closed, by Hausdorff-ness. We claim that some $f(C_i)$ must have non-empty interior. If not, we do the usual Baire argument: fix a non-empty open set $V_1$ in $Y$ with compact closure. Since $f(C_1)$ contains no non-empty open set, $V_1$ is not contained in $f(C_1)$, so there is a non-empty open set $V_2$ whose closure is compact and whose closure is contained in $V_1 - f(C_1)$. Since $f(C_2)$ cannot contain $V_2$, there is a non-empty open set $V_3$ whose closure is compact and whose closure is contained in $V_2 - f(C_2)$. A descending chain of non-empty open sets is produced:
$$
V_1 \supset \text{clos}(V_2) \supset V_2 \supset \text{clos}(V_2) \supset V_3 \supset \ldots
$$
By construction, the intersection of the chain of compact sets $\text{clos}(V_i)$ is disjoint from all the sets $f(C_i)$. Yet the intersection of a descending chain of compact sets is non-empty. Contradiction. Therefore, some $f(C_i)$ has non-empty interior. In particular, for $y_o$ in the interior of $f(C_i)$, the map $f$ is open at $x_o = f-1(y_o)$.

Now use the $G$-equivariance of $f$. For an open $U_o$ containing $x_o$ such that $f(U_o)$ is open in $Y$, for any $g \in G$ the set $g U_o$ is open containing $g x_o$. By the $G$-equivariance,
$$
f(g U_o) = gf(U_o) = \text{ continuous image of open set } = \text{ open}
$$
Therefore, since $G$ is transitive on $X$, $f$ is open at all points of $X$. $\square$
\end{proof}
\begin{proof}
By definition of the quotient topologies, $GL(V)_k G_\mathbb{A}$ must be shown closed in $GL(V)_\mathbb{A}$. Let $X$ be the $k$-vectorspace of $k$-valued quadratic forms on $V$. We have a linear action $\rho$ of $g \in GL(V)_k$ on $q \in X$ by
$$
\rho(g)q(v,v) = q(g^{-1} v, g^{-1}v)
$$
(with inverses for associativity). This extends to give a continuous group action of $GL(V)_\mathbb{A}$ on $X_\mathbb{A} = X \otimes \mathbb{A}$. Note that $G_k$ is the subgroup of $GL(V)_k$ fixing the point $Q \in X$, essentially by definition. Let $Y$ be the set of images of $Q$ under $GL(V)_k$, then
$$
GL(V)_k G_\mathbb{A} = \{ g \in GL(V)_\mathbb{A} \; : \; g(Q) \in Y \}
$$
That is, $GL(V)_k G_\mathbb{A}$ is the inverse image of $Y$. By the continuity of the group action, to prove that $GL(V)_k G_\mathbb{A}$ is closed in $GL(V)_\mathbb{A}$ it suffices to prove that the orbit
$$
Y = GL(V)_k G_\mathbb{A}(Q)
$$
is closed in $X_\mathbb{A}$. Indeed, $Y$ is a subset of $X \subset X_\mathbb{A}$, which is a (closed) discrete subset of $X_\mathbb{A}$. This proves the proposition, invoking the previous lemma. $\square$
\end{proof}
If the global base field is not $\mathbb{Q}$, we need more preparation:
\begin{proposition}
Let $k$ be a number field and $K$ a finite extension of $k$. Let $V$ be $K^n$ viewed as a $k$-vectorspace. Let $H = GL(n, K)$ viewed as a $k$-group, and $G = GL_k(V)$. Then the natural inclusion
$$
i : GL_K(K^n) = H \to G=GL_k(V)
$$
gives a homeomorphism of $H_k\backslash H_\mathbb{A}$ to its image in $G_k\backslash G_\mathbb{A}$, and this image is closed.
\end{proposition}
%
%
%
%
%
%
\subsubsection{Mahler's criterion for compactness}
\begin{theorem}
Mahler's criterion for compactness: Let $G$ be an orthogonal group attached to an $n$-dimensional non-degenerate $k$-valued quadratic form. For a subset $X$ of $G_\mathbb{A} \subset GL(n, \mathbb{A})$ to be compact left modulo $G_k$, it is necessary and sufficient that, given $x_i \in X$ and $v_i \in k^n$ such that $x_i v_i \to 0$ in $\mathbb{A}^n$, $v_i = 0$ for sufficiently large $i$.
\end{theorem}
\begin{proof}
The propositions above the problem to proving an analogue for $G = GL(n,k)$ with $k = \mathbb{Q}$. In particular, for $GL(n)$ suppose there are positive constants $c^\prime$ and $c^{\prime\prime}$ such that
$$
X \subset \{g \in GL(n,\mathbb{A}) \; : \; c^\prime \leq |detg| \leq c^{\prime\prime} \}
$$
The serious direction of implication is to show that, if the condition is satisfied, then $X$ is compact modulo $G_k$. Let $\eta$ be the affine height function on $k^n$. Then $\eta(x v) \geq c_1$ for some $c_1$ for any non-zero $v \in k^n$. By the Iwasawa decomposition, can write $x = p \theta$ with $\theta \in GL(n,\frak{o}_k)$ and $p$ upper-triangular, where $\frak{o}_k$ is the ring of integers in $k$. Further, since we consider $x$ modulo $G_k$, and using the fact that actually $k = \mathbb{Q}$, the Minkowski reduction allows us to suppose that the diagonal entries $p_i$ of $p$ satisfy $|p_i/p_{i+1}| \geq c$ for some $c > 0$. Therefore, letting $e_i$ be the usual basis vectors in $k^n$, $c1 \leq |p_i| = \eta(x e_1)$. And our extra hypothesis gives us
$$
c^\prime \leq |p_1 \ldots p_n| \leq c^{\prime\prime}
$$
Thus, for instance by Fujisaki's lemma, the diagonal entries of elements p coming from elements of $X$ lie inside some compact subset of $\mathbb{J}/k^\times$. Certainly the superdiagonal entries, left-modulo $k$-rational upper-triangular matrices, can be put into a compact set. Therefore, $X$ is compact left modulo $GL(n, k)$, for $k = \mathbb{Q}$. But, as remarked at the outset, the propositions above about imbeddings of arithmetic quotients reduce the general case and the orthogonal group case to this. $\square$
\end{proof}
%
%
%
%
%
%
\subsubsection{Compactness of anisotropic quotients of orthogonal groups}\label{Compactness-of-Anisotropic-Quotients-of-Orthogonal-Groups}
\begin{theorem}
Let $G$ be the orthgonal group of a non-degenerate quadratic form $Q = \langle , \rangle$ on a vectorspace $V \approx k^n$ over a number field $k$. Then $G_k\backslash G_\mathbb{A}$ is compact if and only if $Q$ is $k$-anisotropic.
\end{theorem}
\begin{proof}
On one hand, suppose $Q$ is $k$-anisotropic. If $g_n v_n \to 0$ in $\mathbb{A}^n$ with $g_n \in G_\mathbb{A}$ and $v_n \in \mathbb{A}^n$, then $Q(v_n g_n)$ also goes to $Q(0) = 0$, by the continuity of $Q$. But $Q(g_n v_n) = Q(v_n)$, because $G_\mathbb{A}$ preserves values of $Q$. Since $Q$ has no non-zero $k$-rational isotropic vectors and $k^n$ is discrete in $\mathbb{A}^n$, this means that eventually $v_n = 0$. By Mahler's criterion this implies that the quotient is compact. On the other hand, suppose that $Q$ is isotropic. Then there is a non-zero isotropic vector $v \in k^n$. Let $H$ be the subgroup of $G_\mathbb{A}$ fixing $v$. For all indices $i$ let $v_i = v$. So certainly $v_i$ does not go to $0$. Now we'll need to exploit the fact that the topology on $\mathbb{J}$ is not simply the subspace topology from $\mathbb{A}$, but is inherited from the imbedding $\alpha \to (\alpha,\alpha^{-1})$ of $\mathbb{J} \to \mathbb{A}\times \mathbb{A}$: we can find a sequence $t_i$ of ideles which go to $0$ in the $\mathbb{A}$-topology (but certainly not in the $\mathbb{J}$-topology). Then $t_i v_i \to 0$. And certainly still $Q(t_i v_i) = 0$, so by Witt's theorem there is $g_i \in G_\mathbb{A}$ so that $g_i v_i = t_i v_i$. Thus, $g_i v_i \to 0$, but certainly $v_i$ does not do so. Thus, Mahler's criterion says that the quotient is not compact. $\square$
\end{proof}

\vfill\break



\addtocontents{toc}

\end{document}